\documentclass[11pt,letterpaper]{article}
\usepackage{amsfonts, amsmath, amssymb, amscd, amsthm, graphicx, mathrsfs, wasysym, setspace, mdwlist, color}
\usepackage{hyperref}

\newcommand{\mc}[1]{\mathcal{{#1}}}

\newtheorem{thm}{Theorem}[section]
\newtheorem{cor}[thm]{Corollary}
\newtheorem{lem}[thm]{Lemma}
\newtheorem{prop}[thm]{Proposition}
\newtheorem{prob}[thm]{Problem}
\newtheorem{conj}[thm]{Conjecture}
\theoremstyle{definition}
\newtheorem{defn}[thm]{Definition}
\theoremstyle{remark}
\newtheorem{rem}[thm]{Remark}
\newtheorem{ex}[thm]{Example}

\newfont{\eufm}{eufm10}

\newcommand{\calC}{\mathcal{C}}

\newcommand{\calG}{\mathcal{G}}
\newcommand{\calH}{\mathcal{H}}

\newcommand{\calM}{\mathcal{M}}

\newcommand{\calP}{\mathcal{P}}
\newcommand{\calQ}{\mathcal{Q}}
\newcommand{\calR}{\mathcal{R}}

\newcommand{\X}{\mathbb{X}}

\newcommand {\calc} {{\mathcal {C}}}

\newcommand {\calh} {{\mathcal {H}}}

\newcommand {\calq} {{\mathcal {Q}}}
\newcommand {\calr} {{\mathcal {R}}}

\newcommand {\bbH} {{\mathbb {H}}}

\newcommand {\bbN} {{\mathbb {N}}}

\newcommand {\bbR} {{\mathbb {R}}}

\newcommand {\bbX} {{\mathbb {X}}}

\newcommand {\bbZ} {{\mathbb {Z}}}

\newcommand{\acosh}{\mathop{\mathrm{acosh}}}
\newcommand{\MCG}{\calM\calC\calG (\Sigma)}

\makeatletter
\edef\@tempa#1#2{\def#1{\mathaccent\string"\noexpand\accentclass@#2 }}
\@tempa\rond{017}
\makeatother

\newcommand{\es}{\emptyset}
\renewcommand{\phi}{\varphi}
\newcommand{\m} {^{-1}}
\newcommand{\eps} {\varepsilon}

\newcommand {\ra} {\rightarrow}

\newcommand {\onto} {\twoheadrightarrow}

\newcommand{\actson}{\curvearrowright}

\newcommand{\ol}[1]{\overline{#1}}

\newcommand{\normal} {\vartriangleleft}

\newcommand{\ie} {i.~e.\ }

\newcommand{\grp}[1]{\langle #1 \rangle}
\newcommand{\ceil}[1]{\lceil#1\rceil}

\newcommand{\Stab} {{\mathrm{Stab}}}

\newcommand{\Out} {{\mathrm{Out}}}

\newcommand{\Axis} {{\mathrm{Axis}}}

\newcommand{\inj} {{\mathrm{inj}}}

\newcommand {\Cone} {Cone} \newcommand {\Rot} {Rot}

\newcommand{\RCH}{R_{\mathrm{CH}}}
\newcommand{\dCH}{\delta_{\mathrm{CH}}}

\newcommand{\ru}{r_U}
\newcommand{\du}{\delta_U}

\renewcommand{\AA }{\overline{Area}^{rel}}
\newcommand{\G }{\Gamma (G, X\sqcup \mathcal H)}
\newcommand{\Ga }{\Gamma (G, \mathcal A)}
\newcommand{\Gk}{\Gamma (G, Y\sqcup \mathcal E)}
\newcommand{\dxh }{{\rm d}_{X\cup\mathcal H}}
\newcommand{\dx }{{\rm d}_X}
\newcommand{\dol }{{\rm d}_{Y_\lambda}}
\newcommand{\dl }{\widehat {\rm d}_{\lambda}}
\newcommand{\Ker }{{\rm Ker }}
\newcommand{\e }{\varepsilon }
\renewcommand{\kappa }{\varkappa}
\renewcommand{\P }{\mathcal P}

\newcommand{\Hl }{\{ H_\lambda \} _{\lambda \in \Lambda } }
\newcommand{\Km}{\{ K_{\lambda\mu} \} _{\mu \in M_{\lambda } } }
\newcommand{\N }{\mbox{\eufm N}}
\let\LL\ll
\renewcommand{\ll }{\left\langle\hspace{-.7mm}\left\langle }
\newcommand{\rr }{\right\rangle\hspace{-.7mm}\right\rangle }
\renewcommand{\d }{{\rm d} }

\newcommand{\he }{hyperbolically embedded }
\newcommand{\Lab }{{\bf Lab}}
\newcommand{\h}{\hookrightarrow _{h}}

\newcommand{\diam}{{\rm diam}}

\newcommand{\Ul }{\{ U_\lambda \} _{ \lambda \in \Lambda  }}
\newcommand{\Am }{\{ A_\mu  \} _{ \mu \in {\rm M} }}
\newcommand{\Bn }{\{ B_\nu \} _{ \nu\in {\rm N}  }}

\newcommand{\pky}{$\mathcal P_K(\mathbb Y)$}
\newcommand{\pr}{{\rm proj}}
\newcommand{\Nl}{\{N_\lambda \}_{\lambda \in \Lambda }}

\begin{document}

\title{Hyperbolically embedded subgroups and rotating families in groups acting on hyperbolic spaces}
\author{F. Dahmani, V. Guirardel, D. Osin}
\date{}
\maketitle

\begin{abstract}
We introduce and study the notions of hyperbolically embedded and very rotating families of subgroups. The former notion can be thought of as a generalization of the peripheral structure of a relatively hyperbolic group, while the later one provides a natural framework for developing a geometric version of small cancellation theory. Examples of such families naturally occur in groups acting on hyperbolic spaces including hyperbolic and relatively hyperbolic groups, mapping class groups, $Out(F_n)$, and the Cremona group. Other examples can be found among groups acting geometrically on $CAT(0)$ spaces, fundamental groups of graphs of groups, etc. We obtain a number of general results about rotating families and hyperbolically embedded subgroups; although our technique applies to a wide class of groups, it is capable of producing new results even for well-studied particular classes. For instance, we solve two open problems about mapping class groups, and obtain some results which are new even for relatively hyperbolic groups.
\end{abstract}

\tableofcontents

\newpage \section{Introduction}

The notion of a hyperbolic space was introduced by Gromov in his seminal paper \cite{Gro} and since then hyperbolic geometry has proved itself to be one of the most efficient tools in geometric group theory. Gromov's philosophy suggests that groups acting ``nicely" on hyperbolic spaces have properties similar to those of free groups and fundamental groups of closed hyperbolic manifolds. Of course not all actions, even free ones, are equally good for implementing this idea. Indeed every group $G$ acts freely on the complete graph with $|G|$ vertices, which is a hyperbolic space. Thus, to derive meaningful results, one needs to impose certain properness conditions.

Groups acting on hyperbolic spaces geometrically (i.e., properly and cocompactly) constitute the class of hyperbolic groups. Replacing properness with its relative analogue modulo a fixed collection of subgroups leads to the notion of a relatively hyperbolic group. These classes turned out to be wide enough to encompass many examples of interest, while being restrictive enough to allow building  an interesting theory, main directions of which were outlined by Gromov \cite{Gro}.

On the other hand, there are many examples of non-trivial actions of
non-relatively hyperbolic groups on hyperbolic spaces: the action of
the fundamental group of a graph of groups on the corresponding
Bass-Serre tree, the action of the mapping class group of a closed oriented
surface on the curve complex, and the action of the outer automorphism
group of a free group on the free factor (or free splitting) complex, just to name a few.
In general, these actions are very far from being proper. Nevertheless, they can be (and were) used to prove  interesting results.

The main goal of this paper is to suggest a general approach which allows to study hyperbolic and relatively hyperbolic groups, examples mentioned in the previous paragraph, and many other classes of groups acting on hyperbolic spaces in a uniform way. To achieve this generality, we have to sacrifice ``global properness" (in any reasonable sense). Instead we require the actions to satisfy a ``properness-like" condition that only applies to a selected collection of subgroups.

We suggest two ways of formalizeing this idea. The first way leads to the notion of a \emph{hyperbolically embedded collection of subgroups}, which can be thought of as a generalization of the peripheral structure of relatively hyperbolic groups. The other formalization is based on Gromov's rotating families \cite{Gro_cat} of special kind, which we call \emph{very rotating families of subgroups}; they provide a suitable framework to study collections of subgroups satisfying small cancellation conditions. At first glance, these two ways seem quite different: the former is purely geometric, while the latter has rather dynamical flavor. However, they turn out to be closely related to each other and many general results can be proved using either of them. On the other hand, each approach has its own advantages and limitations, so they are not completely equivalent.

Groups acting on hyperbolic spaces provide the main source of examples in our paper. Loosely speaking, we show that if a group $G$ acts on a hyperbolic space $\X$ so that the action of some subgroup $H\le G$ is proper, orbits of $H$ are quasi-convex, and    distinct translates of $H$-orbits quickly diverge, then $H$ is hyperbolically embedded in $G$. If $K\normal H$ is a normal subgroup of $H$ and all nontrivial elements of $K$ act on $\X $ with large translation length, then the set of conjugates of $K$ in $G$ forms a very rotating family. The main tools used in the proofs of these results are the projection complexes introduced in a recent paper by Bestvina, Bromberg, and Fujiwara \cite{BBF} and the hyperbolic cone-off construction suggested by Gromov in \cite{Gro_cat}. This general approach allows us to construct hyperbolically embedded subgroups and very rotating families in many particular classes of groups, e.g., hyperbolic and relatively hyperbolic groups, mapping class groups, $Out(F_n)$, the Cremona group, many fundamental groups of graphs of groups, groups acting properly on proper $CAT(0)$ spaces and containing rank one isometries, etc.

Many results previously known for hyperbolic and relatively hyperbolic groups can be uniformly reproved in the general context of groups with hyperbolically embedded subgroups, and very rotating families often provide the most convenient way of doing that. As an illustration of this idea we generalize the group theoretic analogue of Thurston's hyperbolic Dehn surgery theorem proved for relatively hyperbolic groups by the third-named author in \cite{Osi07} (and independently by Groves and Manning \cite{Gr_Ma} in the particular case of finitely generated and torsion free relatively hyperbolic groups).

This and other general results from our paper have many particular applications.  Despite its generality, our approach  is capable of producing new results even for well-studied particular classes of groups. For instance, we answer two well-known questions about normal subgroups of mapping class groups.  We also show that the sole existence of a non-degenerate (in a certain precise sense) hyperbolically embedded subgroup in a group $G$ imposes strong restrictions on the algebraic structure of $G$, complexity of its elementary theory, the structure of operator algebras associated to $G$, etc. However, we stress that the main goal of this paper is to build a general theory for the future use rather than to consider particular applications. Some further results
can be found in \cite{AnMS,BW,Hull,HO,MacSis,MO,Osi13}.

The paper is organized as follows. In the next section we provide a detailed outline of the paper and discuss the main definitions and results. We believe it useful to state most results in a simplified form there, as in the main body of the paper we stick to the ultimate generality which makes many statements rather technical.  Section 3 establishes notation and recalls some well-known results used throughout the paper. In Sections 4 and 5 we develop a general theory of hyperbolically embedded subgroups and rotating families, respectively. Most examples are collected in Section 6. Section 7 is devoted to the proof of the Dehn filling theorem. Applications are collected in Section 8. Finally we discuss some open questions and directions for the future research in Section 9.

\paragraph{Acknowledgments.} We are grateful to Mladen Bestvina, Brian
Bowditch, Montse Casals-Ruiz, Remi Coulon, Thomas Delzant, Pierre de
la Harpe, Ilya  Kazachkov, Ashot Minasyan,  Alexander Olshanskii, Mark
Sapir, and Alessandro Sisto  with whom we discussed various topics
related to this paper, and to the referee. We benefited a lot from these discussions. The
first two authors were partially supported by the ANR grant
ANR 2011-BS01-013 and the IUF. The research of the third author was supported by the NSF grants DMS-1006345, DMS-1308961, and by the RFBR grant 11-01-00945.


\section{Main results}


\subsection{Hyperbolically embedded subgroups}

The first key concept of our paper is the notion of a hyperbolically embedded collection of subgroups. For simplicity, we only discuss the case when the collection consists of a single subgroup here and refer to Section \ref{sec:GRH} for the general definition.

Let $G$ be a group, $H$ a subgroup of $G$, $X$ a (not necessary finite) subset of $G$. If $G=\langle X\cup H\rangle $, we denote by $\Gamma(G, X\sqcup H)$ \label{i-Gxh1} the Cayley graph of $G$ with respect to the generating set $X\sqcup H$. Here we think of $X$ and $H$ as disjoint alphabets; more precisely, disjointness of the union $X\sqcup H$ means that if some $x\in X$ and $h\in H$ represent the same element $g\in G$, then $\Gamma(G, X\sqcup H)$ contains two edges connecting every vertex $v\in G$ to the vertex $vg$: one edge is labelled by $x$ and the other is labelled by $h$. Let also $\Gamma _H$ denote the Cayley graph of $H$ with respect to the generating set $H$. Clearly $\Gamma _H$ is a complete subgraph of $\Gamma (G, X\sqcup H)$.

We say that a path $p$ in $\Gamma (G, X\sqcup H)$ is \emph{admissible} if $p$ does not contain edges of $\Gamma _H$. Note that we do allow $p$ to pass through vertices of $\Gamma _H$. \label{i-dhat} Given two elements $h_1,h_2\in H$, define $\widehat\d(h_1,h_2)$ to be the length of a shortest admissible path $p$ in $\Gamma (G, X\sqcup H) $ that connects $h_1$ to $h_2$. If no such path exists we set $\widehat\d(h_1, h_2)=\infty $. Since concatenation of two admissible paths is an admissible path, it is clear that $\widehat\d\colon H\times H\to [0, \infty]$ is a metric on $H$. (For the triangle inequality to make sense we extend addition from $[0, \infty)$ to $[0, \infty]$ in the obvious way.)

\begin{defn} \label{i-hypemb} We say that $H$ is {\it hyperbolically embedded in $G$ with respect to} a subset $X\subseteq G$ (and write $H\hookrightarrow _h (G,X) $) if the following conditions hold.
\begin{enumerate}
\item[(a)] $G$ is generated by $X\cup H$.
\item[(b)] The Cayley graph $\Gamma (G, X\sqcup H)$ is hyperbolic.
\item[(c)] $(H,\widehat\d)$ is a  proper metric space, i.e., every ball (of finite radius) is finite.
\end{enumerate}
We also say that $H$ is \textit{hyperbolically embedded in $G$} (and write $H\hookrightarrow _h G$) if $H\hookrightarrow _h (G,X) $ for some $X\subseteq G$.
\end{defn}

\begin{ex}\label{basic-ex-he}
\begin{enumerate}
\item[(a)] Let $G$ be any group. Then $G\h G$.  Indeed take $X=\emptyset $. Then the Cayley graph $\Gamma(G, X\sqcup H)$ has diameter $1$ and $\widehat d(h_1, h_2)=\infty $ whenever $h_1\ne h_2$. Further, if $H$ is a finite subgroup of a group $G$, then $H\h G$. Indeed $H\h (G,X)$ for $X=G$. These cases are referred to as {\it degenerate}. In what follows we are only interested in non-degenerate examples.

\item[(b)] Let $G=H\times \mathbb Z$, $X=\{ x\} $, where $x$ is a generator of $\mathbb Z$. Then $\Gamma (G, X\sqcup H)$ is quasi-isometric to a line and hence it is hyperbolic. However $\widehat\d(h_1, h_2)\le 3$ for every $h_1, h_2\in H$. Indeed in the shift $x\Gamma _H$ of $\Gamma _H$ there is an edge (labelled by $h_1^{-1}h_2\in H$) connecting $h_1x$ to $h_2x$, so there is an admissible path of length $3$ connecting $h_1$ to $h_2$ (see Fig. \ref{fig1}). Thus if $H$ is infinite, then $H\not\h (G,X)$. Moreover it is not hard to show that $H\not\h G$.

\item[(c)] Let $G=H\ast \mathbb Z$, $X=\{ x\} $, where $x$ is a generator of $\mathbb Z$. In this case $\Gamma (G, X\sqcup H)$ is quasi-isometric to a tree (see Fig. \ref{fig1}) and $\widehat\d(h_1, h_2)=\infty $ unless $h_1=h_2$. Thus $H\h (G,X)$.
\end{enumerate}
\end{ex}

\begin{figure}
  \centering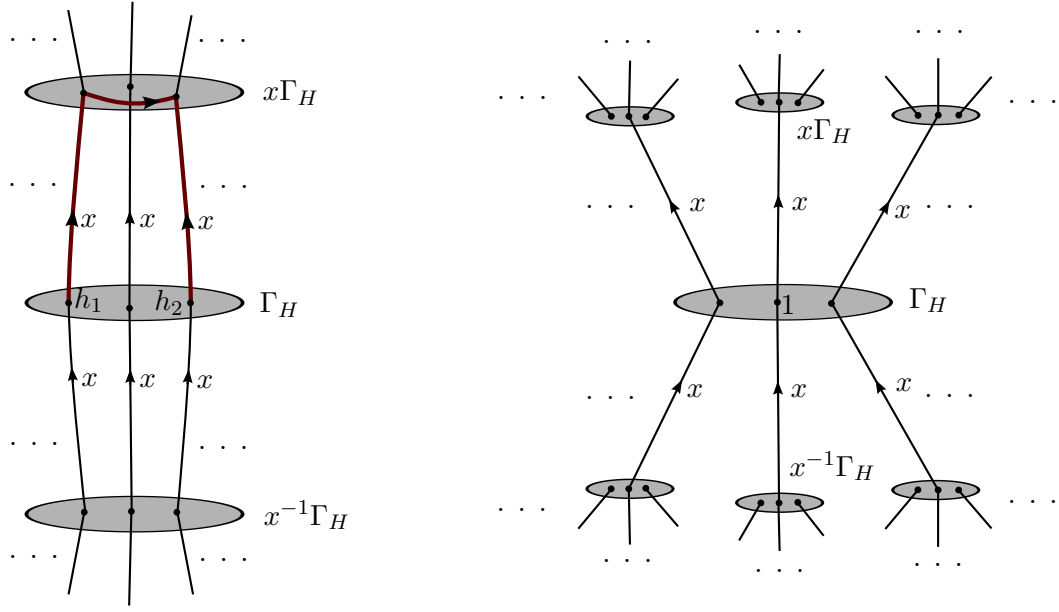\\
  \caption{Cayley graphs $\Gamma(G, X\sqcup H)$ for $G=H\times \mathbb Z$ and $G=H\ast \mathbb Z$.}\label{fig1}
\end{figure}

Our approach to the study of hyperbolically embedded subgroups is inspired by \cite{Osi06a}. In particular, we first provide an isoperimetric characterization of hyperbolically embedded subgroups, which resembles the corresponding characterization of relatively hyperbolic groups.

Recall that a {\it relative presentation} of a group $G$ with respect to a subgroup $H\le G$ and a subset $X\subseteq G$ is a presentation of the form
\begin{equation}\label{relpres-intr}
G=\langle H,\, X \mid \mathcal R\rangle,
\end{equation}
which is obtained from a presentation of $H$ by adding the set of generators $X$ and the set of relations $\mathcal R$. Thus $G=H\ast F(X)/\ll \mathcal R\rr $, where $F(X)$ is the free group with basis $X$ and $\ll \mathcal R\rr $ is the normal closure of $\mathcal R$ in $H\ast F(X)$.

The relative presentation (\ref{relpres-intr}) is {\it bounded}, if all elements of $\mathcal R$ have uniformly bounded length being considered as words in the alphabet $X\sqcup H$; further it is {\it strongly bounded} if, in addition, the set of letters from $H$ appearing in words from $\mathcal R$ is finite. For instance, if $H$ is an infinite group with a finite generating set $A$, then the relative presentation $$\langle H,\, \{ x\} \mid [x,h]=1, \, h\in H\rangle $$ of the group $G=H\times \mathbb Z$ is bounded but not strongly bounded. On the other hand, the presentation $$\langle H,\, \{ x\}\mid [a,x]=1,\, a\in A\rangle $$ of the same group is strongly bounded.

The relative isoperimetric function of a relative presentation is defined in the standard way. Namely we say that $f\colon \mathbb N\to \mathbb N$ is a \emph{relative isoperimetric function} of a relative presentation (\ref{relpres-intr}), if for every $n\in \mathbb N$ and every word $W$ of length at most $n$ in the alphabet $X^{\pm 1} \sqcup H$ which represents the trivial element in $G$, there exists a decomposition
$$
W=\prod\limits_{i=1}^k f_i^{-1}R_i^{\pm 1}f_i
$$
in the free product $H \ast F(X)$, where for every $i=1, \ldots , k$, we have $f_i\in H \ast F(X)$, $R_i\in \mathcal R$, and $k\le f(n)$.

\begin{thm}[Theorem \ref{ipchar}]\label{ipchar-intr}
Let $G$ be a group, $H$ a subgroup of $G$, $X$ a subset of $G$ such that $G=\langle X\cup H\rangle $. Then $H\h (G,X)$ if and only if there exists a strongly bounded relative presentation of $G$ with respect to $X$ and $H $ with linear relative isoperimetric function.
\end{thm}

This theorem and the analogous result for relatively hyperbolic groups (see \cite{Osi06a}) imply that the notion of a hyperbolically embedded subgroup indeed generalizes the notion of a peripheral subgroup of a relatively hyperbolic group,
  where one requires $X$ to be finite. More precisely, we have the following.

\begin{prop}[Proposition \ref{he-rh}]\label{he-rh-intr}
Let $G$ be a group, $H\le G$ a subgroups of $G$. Then $G$ is hyperbolic relative to $H$ if and only if $H \h (G,X)$ for some (equivalently, any) finite subset $X$ of $G$.
\end{prop}

On the other hand, by allowing $X$ to be infinite, we obtain many other examples of groups with  hyperbolically embedded subgroups. A rich source of such examples is provided by groups acting on hyperbolic spaces. More precisely, we introduce the following.

\begin{defn}\label{GeomSep_mainres}
Let $G$ be a group acting on a space $S$. Given an element $s\in S$ and a subset $H\subseteq G$, we define the {\it $H$-orbit } of $s$ by
$$
H(s)=\{ h(s)\mid h\in H\} .
$$
We say that (the collection of cosets of) a subgroup $H\le G$ is {\it geometrically separated} if for every $\e >0$ and every $s\in S$, there exists $R>0$ such that the following holds. Suppose that for some $g\in G$ we have
$$
{\rm diam} \left(H (s)\cap \mathcal (gH(s))^{+\e}\right)\ge R,
$$
where $(gH(s))^{+\e}$ denotes the $\e$-neighborhood of the $gH$-orbit of $s$ in $S$. Then $g \in H$.
\end{defn}

Informally, the definition says that distinct translates of the $H$-orbit of $s$ rapidly diverge. It is also fairly easy to see that replacing ``every $s\in S$" with ``some $s\in S$" yields an equivalent definition (see Remark \ref{rem-gs}).

\begin{ex}
Suppose that $G$ is generated by a finite set $X$. Let $S=\Gamma (G, X)$, and $H$ a subgroup of $G$.
Then geometric separability of $H$ with respect to the natural action on $S$ implies that $H$ is almost malnormal in $G$, i.e., $|H^g\cap H|<\infty $ for any $g\notin H$. (The converse is not true in general.)
\end{ex}

\begin{thm}[Theorem \ref{crit}]\label{geom-sep-intr}
Let $G$ be a group acting by isometries on a hyperbolic space $S$, $H$ a geometrically separated subgroup of $G$. Suppose that $H$ acts on $S$ properly and there exists $s\in S$ such that the $H$-orbit of $s$ is quasiconvex in $S$. Then $H \h G$.
\end{thm}

This theorem is one of the main technical tools in our paper.  In Section \ref{ex-intr}, we will discuss many particular examples of groups with hyperbolically embedded subgroups obtained via Theorem \ref{geom-sep-intr}. To prove the theorem, we first use the Bestvina-Bromberg-Fujiwara projection complexes (see Definition \ref{projc}) to construct a hyperbolic space on which $G$ acts coboundedly. Then a refined version of the standard Milnor-Svar\v c argument allows us to construct a (usually infinite) subset $X\subseteq G$ such that $H\h (G,X)$.

Here we mention just one application of Theorem \ref{geom-sep-intr}, which makes use of the following notion introduced by Bestvina and Fujiwara \cite{BF}.

\begin{defn}\label{WPD-intr}
Let $G$ be a group acting on a hyperbolic space $S$, $h$ an element of $G$.  One says that $h$ satisfies the {\it weak proper discontinuity} condition (or $h$ is a {\it WPD element}) if for every $\e >0$ and every $x\in S$, there exists $N=N(\e )$ such that
\begin{equation}\label{eq: wpd}
|\{ g\in G \mid \d (x, g(x))<\e, \;   \d (h^N(x), gh^N(x))<\e \} |<\infty .
\end{equation}
Recall also that an element $h\in G$ is {\it loxodromic} if the map $\mathbb Z\to S$ given by $n\mapsto h^n(s)$ is a quasi-isometric embedding for some (equivalently any) $s\in S$.
\end{defn}

This corollary summarizes Lemma \ref{elem1} and a particular case of Theorem \ref{wpd}. To prove it, we verify that $H=E(h)$ satisfies all assumptions of Theorem \ref{geom-sep-intr}.

\begin{cor}\label{wpd-he-intr}
Let $G$ be a group acting on a hyperbolic space and let $h$ be a loxodromic WPD element. Then $h$ is contained in a unique maximal virtually cyclic subgroup  of $G$, denoted $E(h)$, and $E(h)\h G$.
\end{cor}

Let us mention one restriction which is useful in proving that a subgroup is \emph{not} hyperbolically embedded in a group. In fact, it is a generalization of Example \ref{basic-ex-he} (b).

\begin{prop}[Proposition \ref{malnorm}]\label{malnorm-intr}
Let $G$ be a group, $H$ a hyperbolically embedded subgroup of $G$. Then $H$ is almost malnormal, i.e., $|H\cap H^g| <\infty $ whenever $g\notin H$.
\end{prop}

Yet another obstruction for being hyperbolically embedded is provided by homological invariants. It was proved by the first and the second author \cite{DG} that every peripheral subgroup of a finitely presented relatively hyperbolic group is
finitely presented. This was generalized by Gerasimov-Potyagailo \cite{GP} to a class of quasiconvex subgroups of  relatively hyperbolic groups. In this paper we generalize the result of \cite{DG} in another direction, namely to hyperbolically embedded subgroups. Our argument is geometric, inspired by \cite{GP}, and allows us to obtain several finiteness results in a uniform way. It is worth noting that for $n>2$ parts b) and c) of Theorem 2.9 below are new even for peripheral subgroups of relatively hyperbolic groups.

Recall that a group $G$ is said to be of \emph{type $F_n$} ($n\ge 1$) if it admits an Eilenberg-MacLane space $K(G,1)$ with finite $n$-skeleton. Thus conditions $F_1$ and $F_2$ are equivalent to $G$ being finitely generated and finitely presented, respectively. Further $G$ is said to be of \emph{type $FP_n$} if the trivial $G$-module $\mathbb Z$ has a projective resolution which is finitely generated in all dimensions up to $n$. Obviously $FP_1$ is equivalent to $F_1$. For $n=2$ these conditions are already not equivalent; indeed there are groups of type $FP_2$ that are not finitely presented \cite{BB}. For $n\ge 2$, $F_n$ implies $FP_n$ and is equivalent to $FP_n$ for finitely presented groups. For details we refer to the book \cite{Bro}.

Recall also that for $n\ge 1$, the $n$-dimensional Dehn function of a group $G$ is defined whenever $G$ has type $F_{n+1}$; it is denoted by $\delta^{(n)}_G$. In particular $\delta ^{(1)} _G=\delta _G$ is the ordinary Dehn function of $G$. The definition can be found in \cite{APW} or \cite{Bri}; we stick to the homotopical version here and refer to \cite{Bri} for a brief review of other approaches. As usual we write $f\preceq g$ for some functions $f,g\colon \mathbb N\to \mathbb N$ if there are $A,B,C,D\in \mathbb N$ such that $f(n)\le Ag(Bn)+Cn+D$ for all $n\in \mathbb N$.

In Section 4.3, we prove the following.

\begin{thm}[Corollary \ref{cor-fn}]\label{thm-fn}
Let $G$ be a finitely generated group and let $H$ be a hyperbolically embedded subgroup of $G$. Then the following conditions hold.
\begin{enumerate}
\item[(a)] $H$ is finitely generated.
\item[(b)] If $G$ is of type $F_n$ for some $n\ge 2$, then so is $H$. Moreover, we have $\delta ^{n-1} _H\preceq \delta ^{n-1}_G$. In particular, if $G$ is finitely presented, then so is $H$ and $\delta _H\preceq \delta _G$.
\item[(c)] If $G$ is of type $FP_n$, then so is $H$.
\end{enumerate}
\end{thm}

Many other results previously known for relatively hyperbolic groups can be reproved in the general context of hyperbolically embedded subgroups. One of the goals of this paper is to help making this process ``automatic". More precisely, in Section \ref{sec:GRH} we generalize some useful technical lemmas originally proved for relatively hyperbolic groups in \cite{Osi07, Osi06a} to the case of hyperbolically embedded subgroups. Then proofs of many  results about relatively hyperbolic groups work in the general context of hyperbolically embedded subgroups almost verbatim after replacing references. This approach is illustrated by the proof of the group theoretic analogue of Thurston's Dehn filling theorem discussed in Section \ref{subsec:DF}.

\subsection{Rotating families.}
The other main concept used in our paper is that of an $\alpha$-rotating family of subgroups, which we again discuss in the particular case of a single subgroups here. It is based on the notion of a rotating family (or rotation family, or rotation schema), which was introduced by Gromov in \cite[\S 26---28]{Gro_cat} in the context of groups acting on $CAT(\kappa)$ spaces with $\kappa \leq 0$. It allows to envisage a small-cancellation like property for a family of subgroups in a group through a geometric configuration of a space upon which the group acts and in which the given subgroups fix different points.
 The essence of this definition is that we have a $G$-invariant collection of points, and for each
point $c$ in this collection, a subgroup $G_c$ of $G$ whose non-trivial elements act as rotations around $c$ with
a large angle. This angle condition would make sense in a $CAT(0)$ or $CAT(-1)$ space,
and the definition we give mimics this situation in the coarser setting of a Gromov-hyperbolic space (see Figure \ref{fig_rotating}). Because of this coarseness, we need to assume that the points $c$ in our family are sufficiently far away from each other compared to the hyperbolicity constant.

\begin{figure}
  \centering\includegraphics[width=5cm]{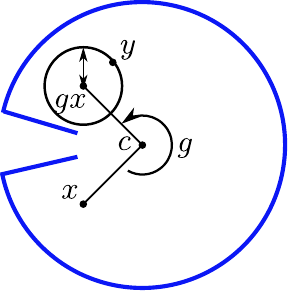}
  \caption{In a very rotating family, $g\in G_c\setminus\{1\}$ rotates by a large angle}\label{fig_rotating}
\end{figure}

\begin{defn}\label{defi;rotatingF}\label{i-rot}
\begin{enumerate}
 \item[(a)]  (Gromov's rotating families) Let $G\actson \X$ be an action of a group on a metric space.  A rotating family  $\calC = (C, \{G_c, c\in C\}) $ consists of a
   subset $C\subset \X$, and a collection  $\{G_c, c\in C\}$ of subgroups of $G$ such that
  \begin{enumerate}
  \item[(a-1)] $C$ is $G$-invariant,
  \item[(a-2)] each $G_c$ fixes $c$,
  \item[(a-3)] $\forall g\in G \; \forall c\in C\;  G_{gc}= gG_{c} g^{-1} $.
  \end{enumerate}
  The set $C$ is called the set of \emph{apices} of the family, and the groups $G_c$ are called the \emph{rotation subgroups} of the family.

\item[(b)] \label{i-sep} (Separation) One says that $C$ (or  $\calC$)  is $\rho$-\emph{separated} if any two distinct apices are at distance  at least $\rho$.

\item[(c)] \label{i-vrot} (Very rotating condition) When $\X$ is $\delta$-hyperbolic for some $\delta >0$, one says that $\calC$ is \emph{very rotating} if, for all $c\in C, g\in G_c\setminus \{1\}$, and all $x, y \in \X$ with
  both $\d (x,c),\d (y,c)$ in the interval $[20\delta ,  40\delta]$
      and $\d (gx,y)\leq 15\delta$,
          any geodesic between $x$ and $y$ contains $c$.
\item[(d)] \label{i-arot} ($\alpha$-rotating subgroup) A subgroup $H$ of a group $G$ is called \emph{$\alpha$-rotating} if there is an $ \alpha\delta$-separated very rotating family of $G$ acting on a $\delta$-hyperbolic space $\X$ for some $\delta >0$ whose rotation subgroups are exactly the conjugates of $H$. When we want to stress a particular action, we will say that $H$ is \emph{$\alpha$-rotating with respect to the given action} of $G$ on $\X$.
\end{enumerate}
\end{defn}

\begin{ex}\label{ex: freeprod}
Suppose that $G=H\ast K$. Let $C$ be the set of vertices of the corresponding Bass-Serre tree $\X $ and let $G_c$ denote the stabilizer of $c\in C$ in $G$. Then we obtain a rotating family $\calC= (C, \{G_c, c\in C\}) $ of subgroups of $G$.
Since $\X$ is $\delta$-hyperbolic for any $\delta>0$, we see that $H$ and $K$ are $\alpha $-rotating subgroups of $G$ for every $\alpha >0$.
\end{ex}

These definitions come with three natural problems. First, study the structure of the subgroups generated by rotating families. Second, study the quotients of groups and spaces by the action of rotating families. Third, provide a way to construct spaces with rotating families in different contexts. We will show that the first two questions can be answered for $\alpha$-rotating collections of subgroups if $\alpha $ is large enough and provide many examples of such collections.

The main structural result about rotating families is a partial converse of Example \ref{ex: freeprod}. Recall that given a subset $S$ of a group $G$, we denote by $\ll S\rr^G$ the normal closure of $S$ in $G$, i.e., the minimal normal subgroup of $G$ containing $S$.

\begin{thm} [Theorem \ref{theo;app_wind}]\label{free-intr}
Let $G$ be a group acting on a hyperbolic space $\X$, $H$ an $\alpha $-rotating subgroup of $G$ with respect to this action for some $\alpha \geq 200$. Then the following holds.
\begin{enumerate}
\item[(a)] There exists a (usually infinite) subset $T\subseteq G$ such that $\ll H\rr ^G=\ast _{t\in T} t^{-1}Ht$.
\item[(b)] Every element $h\in \ll H\rr^G$ either is conjugate to an element of $H$, or is loxodromic with respect to the action on $\X$
\end{enumerate}
\end{thm}

 The proof of this result is inspired by \cite[\S 26---28]{Gro_cat}, where $\X$ is assumed to be $CAT(0)$.  Claiming that this context has ``rather limited application'', Gromov indicates that a generalization to spaces with ``approximately negative'' curvature is desirable and sketches it in \cite{Gro_Meso} and \cite{G-rnd}. A result similar to Theorem \ref{free-intr} was stated in \cite[1/6 theorem]{G-rnd}. For the proof, Gromov refers to the Delzant's paper \cite{Del_Duke}, which deals with the particular case of hyperbolic groups.  Delzant did not use rotating families there, and his argument, which can indeed be generalized,  is quite technical. We propose here an alternative proof inspired by the more geometric settings of  \cite[\S 26--28]{Gro_cat}; our approach is based on the notion of a \emph{windmill} introduced in Section \ref{wind}.

A standard way of producing very rotating families is through a coning-off construction proposed by Gromov  \cite[\S 29--32]{Gro_cat} for $CAT(\kappa)$  spaces with $\kappa <0$. It was later adapted to ``approximate'' negative curvature in \cite{ArzhDel, Coulon,Del_Gro}. The general idea is to start with some action of a group $G$ on a hyperbolic space and then glue ``hyperbolic cones" to orbits of a family of subgroups to make these subgroups elliptic. In general, this does not yield a very rotating family; in fact, the resulting space may not be even hyperbolic. In order to be able to proceed with the coning-off construction while getting a suitable space, we introduce a condition of small cancellation flavor (see Definition \ref{dfn_sc_subgroup} and Proposition \ref{prop_sc_subgroup}).

We mention here a particular application of this idea to acylindrical actions and, more generally, group actions with WPD loxodromic elements.

\begin{defn}
Following Bowditch \cite{B_acyl}, we say that an action of a group $G$ on a metric space $\X $ is \emph{acylindrical} if for any $\e\ge 0$, there exist $R=R(\e)>0$ and $N=N(\e)>0$ such that for all $x,y \in \X $ with $\d(x,y)\geq R$, the set $$\{ g\in G \mid \d(x,gx)\leq \e, \d(y,gy)\leq \e\}$$ contains at most $N$ elements.
\end{defn}

It is easy to see that if the action of $\G $ on $\X$ is acylindrical, then every loxodromic element $g\in G$ is WPD (see Definition \ref{WPD-intr}). Thus part (a) of the proposition applies to a more general situation than part (b), while the conclusion in part (b) is more uniform.

\begin{prop}[Proposition \ref{prop;Acyl_free}]\label{cor-acyl-intr}
Let $G$ be a group acting on a hyperbolic space $\X$ and let $\alpha$ be a positive number.
\begin{enumerate}
\item[(a)] For any loxodromic WPD element $g\in G$, there exists $m=m(\alpha, g)\in \mathbb N$ such that the subgroup $\langle g^m\rangle$ is $\alpha$-rotating with respect to the induced action of $G$ on a certain cone-off of $\X$.
\item[(b)] If the action of $G$ on $\X$ is acylindrical, then there exists $n=n(\alpha)$ such that for every loxodromic $g\in G$ the subgroup $\langle g^n\rangle$ is $\alpha$-rotating with respect to the induced action of $G$ on a certain cone-off of $\X$.
\end{enumerate}
\end{prop}

After obtaining a rotating family acting on a suitable space, one may want to quotient this space by the group  normally generated by the very rotating family. A typical result of this type would assert that  hyperbolicity is preserved, possibly in an effective way (compare to \cite[Theorem 1/7]{G-rnd}). This is indeed what we obtain in Propositions \ref{prop;quotient_hyp} and  \ref{prop;quotient_isom}. In addition, we show that, under certain mild assumptions,  acylindricity is preserved through coning-off and taking such a quotient (see Propositions   \ref{prop;cone_acyl} and \ref{prop;quotient_acyl}, respectively). These results can be summarized as follows.

\begin{prop}\label{prop;intro_quotient}
 Let $G$ be a group acting on a hyperbolic graph $\X$.
 Let $H$ be an $\alpha$-rotating subgroup of $G$ with respect to this action, where $\alpha$ in large enough. Then for any , the following conditions hold.
\begin{enumerate}
  \item[(a)] $\X/\ll H\rr^G$ is Gromov-hyperbolic
  \item[(b)] The quotient map $\X \to \X/\ll H\rr^G$ is a local isometry away from the apices of the rotating family.
  \item[(c)] Any elliptic isometry of $\X/\ll H\rr^G$  in $G/\ll H\rr^G$ has a preimage in $G$ that is elliptic.
  \item[(d)] Suppose that the action of $G$ on $\X$ is acylindrical. Suppose also that, for a fixed point $c$ of $H$ in $\X$, the action of the stabilizer ${\rm Stab}_G(c)$ on the sphere centered at $c$ satisfies a certain properness assumption (see \ref{prop;quotient_acyl}). Then the action of $G/\ll H\rr^G$ on  $\X/\ll H\rr^G$  is acylindrical.
\end{enumerate}
\end{prop}

Preservation of acylindricity allows us to iterate applications of  Corollary \ref{cor-acyl-intr} and Proposition \ref{prop;intro_quotient} infinitely many times and construct some interesting quotient groups (see Section \ref{sec-MCG}).

\subsection{Examples}\label{ex-intr}
We now discuss some examples of hyperbolically embedded and $\alpha $-rotating subgroups. Before looking at particular groups, let us mention one general result, which allows us to pass from hyperbolically embedded subgroups to $\alpha$-rotating ones.

\begin{thm}[Theorem \ref{he-vrf}]\label{he-vrf-intr}
Suppose $H$ is a hyperbolically embedded subgroup of a group $G$. Then for every $\alpha >0$, there exists a finite subset $\mathcal F\subseteq H\setminus\{ 1\}$ such that the following holds. Let $N\lhd H$ be a normal subgroup of $H$ that contains no elements of $\mathcal F$. Then $N$ is an $\alpha $-rotating subgroup of $G$.
\end{thm}

In particular, this theorem together with Proposition \ref{he-rh-intr} allows us to construct $\alpha $-rotating subgroups in relatively hyperbolic groups. Yet another typical application is the case when $H$ is infinite and virtually cyclic. In this case, it is easy to show that for every infinite order element $h\in H$ and every $\alpha >0$, there exists $k\in \mathbb N$ such that the subgroup $\langle h^k\rangle$ is $\alpha$-rotating (see Corollary \ref{theo;WPD_free}).

Below we consider some examples where the groups are not, in general, relatively hyperbolic.

The first class of examples consists of mapping class groups. Let $\Sigma$ be a (possibly punctured) orientable closed surface. The mapping class group $\MCG$ is the group of isotopy classes of orientation preserving homeomorphisms of $\Sigma$; we do allow homeomorphisms to permute the punctures. By Thurston's classification, an element of $\MCG$ is either of finite order, or reducible (i.e., fixes a multi-curve), or pseudo-Anosov. Recall that all but finitely many mapping class groups are not relatively hyperbolic, essentially because of large ``degree of commutativity" \cite{AAS}. However, we prove the following result (see Theorem \ref{ex-MCG} and Theorem \ref{thm_MCG_HE}).

\begin{thm}\label{ex-MCG-intr}
Let $\Sigma$ be a (possibly punctured) orientable closed surface and let $\calM\calC\calG(\Sigma)$ be its mapping class group.
\begin{enumerate}
\item[(a)] For every pseudo-Anosov element $a\in \calM\calC\calG(\Sigma)$, we have $E(a)\h \calM\calC\calG(\Sigma)$, where $E(a)$ is the unique maximal virtually cyclic subgroup containing $a$.
\item[(b)] For every $\alpha >0$, there exists $n\in \mathbb N$ such that for every pseudo-Anosov element $a\in \calM\calC\calG(\Sigma)$, the cyclic subgroup $\langle a^{n}\rangle $  is $\alpha$-rotating.
\item[(c)] Every subgroup of $\calM\calC\calG(\Sigma) $ is either virtually abelian or virtually surjects onto a group with a non-degenerate hyperbolically embedded subgroup.
\end{enumerate}
\end{thm}

We explain the proof in the case on non-exceptional surfaces. That is, we assume that $3g+p> 4$, where $g$ and $p$ are the genus and the number of punctures of $\Sigma$, respectively. In all exceptional cases, $\calM\calC\calG(\Sigma)$ is a hyperbolic group, and the proof is much easier.

Associated to $\Sigma$ is its \emph{curve complex} $\cal C$, on which $\calM\calC\calG(\Sigma)$ acts by isometries; pseudo-Anosov elements of $\calM\calC\calG(\Sigma)$ are exactly loxodromic elements with respect to this action.
Recall that $\mathcal C$ is hyperbolic for all non-exceptional surfaces  $\Sigma$  \cite{MM99} and the action of $\calM\calC\calG(\Sigma)$ on $\mathcal C$ is acylindrical \cite{B_acyl} (the WPD property for pseudo-Anosov elements was established earlier in \cite{BF}). Thus we obtain (a) by applying Corollary \ref{wpd-he-intr} to the action of the mapping class group on the corresponding curve complex. We note that the subgroup $\langle a\rangle $ is not necessarily hyperbolically embedded. In fact, Proposition \ref{malnorm-intr} easily implies that no proper infinite subgroup of $E(a)$ is hyperbolically embedded in $\calM\calC\calG(\Sigma)$. Further, part (b) follows immediately from Proposition \ref{cor-acyl-intr} (b). Finally, part (c) is easy to derive from (a) and Ivanov's trichotomy, stating that every subgroup of a mapping class group is either finite, or reducible, or contains a pseudo-Anosov element \cite{Iva92}.

\label{i-iwip}
A similar result can be proved for the group $Out(F_n)$ of outer automorphisms of a free group. Recall that an element $g\in Out(F_n)$ is {\it irreducible with irreducible powers} (or {\it iwip}, for brevity) if none  of its non-trivial powers preserve the conjugacy class of a proper free factor of $F_n$. These automorphisms play the role of pseudo-Anosov mapping classes in the usual analogy between mapping class groups and  $Out(F_n)$.
As an analogue of the curve complex, we can use the free factor complex \cite{BFe2}, or a specially crafted
hyperbolic complex  \cite{BFe} on which a given iwip element
 $g$ acts loxodromically and satisfies the WPD condition.
Arguing as above, we obtain:

\begin{thm}[Theorem \ref{outfn}]
Let $F_n$ be the free group of rank $n$, $g$ an iwip element. Then  $E(g)\h Out(F_n)$, where $E(g)$ is the unique maximal virtually cyclic subgroup of $Out(F_n)$ containing $g$. Furthermore, for every $\alpha >0$, there exists $k\in \mathbb N$ such that the cyclic subgroup $\langle g^{k}\rangle $  is $\alpha$-rotating.
\end{thm}

\begin{rem}\label{rem:outFn}
Note that this theorem is significantly weaker than Theorem \ref{ex-MCG-intr}. There are two reasons: first, currently there is no known hyperbolic space on which $Out(F_n)$ acts acylindrically so that all iwip elements are exactly the loxodromic ones; second, no analogue of Ivanov's trichotomy is known for subgroups of $Out(F_n)$.  The recent papers \cite{HanMos,HiHo_hyperbolicity} provide another hyperbolic space on which $Out(F_n)$ acts. It is very well possible that further study of the action of $Out(F_n)$ on this or other similar spaces will lead to a stronger version of the theorem.
\end{rem}

The same argument also works for the group ${\bf Bir}(\mathbb P^2_{\mathbb C})$ of birational transformations of the projective plane $\mathbb P^2_{\mathbb C}$, called the \emph{Cremona group}. For the definition and details about  ${\bf Bir}(\mathbb P^2_{\mathbb C})$ we refer to the survey \cite{C13}. In \cite{CL}, Cantat and Lamy introduced the notion of a \emph{tight} element (see Definition \ref{tightdef}) and proved that ``most" elements of ${\bf Bir}(\mathbb P^2_{\mathbb C})$ are tight.  They use this notion to prove that the Cremona group is not simple, using a generalization of small cancellation arguments by Delzant \cite{Del_Duke}.
Tight elements act loxodromically on a hyperbolic space naturally associated to ${\bf Bir}(\mathbb P^2_{\mathbb C})$.
 Existence of tight elements and some additional results from \cite{BC,CL} allow us to apply Theorem \ref{geom-sep-intr} to certain virtually cyclic subgroups containing these elements.

\begin{thm}[Corollary \ref{crem}] \label{Crem-intr}
Let $g$ be a tight element of the Cremona group ${\bf Bir}(\mathbb P^2_{\mathbb C})$. Then there exists a virtually cyclic subgroup $E(g)$ of ${\bf Bir}(\mathbb P^2_{\mathbb C})$ which contains $g$ and is hyperbolically embedded in ${\bf Bir}(\mathbb P^2_{\mathbb C})$. Furthermore, for every $\alpha >0$, there exists $k\in \mathbb N$ such that the cyclic subgroup $\langle g^{k}\rangle $  is $\alpha$-rotating.
\end{thm}

The next example is due to Sisto \cite{Sis}. It answers a question from the first version of this paper.

\begin{thm}[Sisto \cite{Sis}]\label{Sis}
Let $G$ be a group acting properly on a proper $CAT(0)$ space. Suppose that $g\in G$ is a rank one isometry. Then $g$ is contained in a unique maximal virtually cyclic subgroup of $G$, which is hyperbolically embedded in $G$.
\end{thm}
This result also gives rise to rotating families via Theorem \ref{he-vrf-intr}.

The notion of a rank one isometry originates in the Ballman's paper \cite{Bal82}. Recall that an axial isometry $g$ of a $CAT(0)$ space $S$ is \emph{rank one} if there is an axis for $g$ which does
not bound a flat half-plane. Here a flat half-plane means a totally geodesic
embedded isometric copy of an Euclidean half-plane in $S$. For details see \cite{Bal95,Ham09} and references therein.

Theorem \ref{Sis} provides a large source of groups with non-degenerated \he subgroups. For instance let $M$ be an irreducible Hadamard manifold that is not a higher rank symmetric space. Suppose that a group $G$ acts on $M$ properly and cocompactly. Then $G$ always contains a rank one isometry \cite{BBE,BBS,BS87}. Conjecturally, the same conclusion holds for any locally compact geodesically complete irreducible $CAT(0)$ space that is not a higher rank symmetric space or a Euclidean building of dimension at least $2$ \cite{BB08}. Recall that a $CAT(0)$ space is called \emph{geodesically complete} if every geodesic segment can be extended to some bi-infinite geodesic. This conjecture was settled by Caprace and Sageev \cite{CS} for $CAT(0)$ cube complexes. Namely, they show that for any locally compact geodesically complete $CAT(0)$ cube complex $Q$ and any infinite discrete group $G$ acting properly and cocompactly on $Q$, $Q$ is a product of two geodesically complete unbounded convex subcomplexes or $G$ contains a rank one isometry. For instance, this applies to right angled Artin and Coxeter groups acting on the universal covers of their Salvetti complexes and their Davis complexes, respectively (see \cite{Cha} for details).

In most examples discussed above, the hyperbolically embedded subgroups are elementary (i.e. virtually cyclic). The next result allows us to construct non-elementary hyperbolically embedded subgroups starting from any non-degenerate (but possibly elementary) ones. The proof is also based on Theorem \ref{geom-sep-intr} and a small cancellation like argument. This theorem has many applications, e.g., to the proof of SQ-universality  and $C^\ast$-simplicity of groups with non-degenerate hyperbolically embedded subgroups, discussed in Section \ref{subsec:appl}.

\begin{thm}[Theorem \ref{vf}]\label{vf-intr}
Suppose that a group $G$ contains a non-degenerate hyperbolically embedded subgroup. Then the following hold.
\begin{enumerate}
\item[(a)] There exists a (unique) maximal finite normal subgroup of $G$, denoted $K(G)$.
\item[(b)] For every infinite subgroup $H\h G$, we have $K(G)\le H$.
\item[(c)] For any $n\in \mathbb N$, there exists a subgroup $H\le G$ such that $H\h G$ and $H\cong F_n \times K(G)$, where $F_n$ is a free group of rank $n$.
\end{enumerate}
\end{thm}

Given groups with hyperbolically embedded subgroups, we can combine them using amalgamated products and HNN-extensions. In Section \ref{sec:ex} we discuss generalizations of some combination theorems previously established for relatively hyperbolic groups by Dahmani \cite{Dah}. Here we state our results in a simplified form and refer to Theorem \ref{HNN} and Theorem \ref{Am} for the full generality. The first part of the theorem requires the general definition of a hyperbolically embedded collection of subgroups, which we do not discuss here (see Definition \ref{hes}).

\begin{thm}
\begin{enumerate}
\item[(a)]
Let $G$ be a group, $\{H, K\}$ a hyperbolically embedded collection of subgroups, $\iota :K\to H$ a monomorphism. Then $H$ is hyperbolically embedded in the HNN--extension
\begin{equation}\label{HNN-pres}
 \langle G,t\; |\; t^{-1}kt=\iota (k),\; k\in K\rangle .
\end{equation}
\item[(b)] Let $H$ and $K$ be hyperbolically embedded isomorphic subgroups of groups $A$ and $B$, respectively. Then $H=K$ is hyperbolically embedded in the amalgamated product $A\ast _{H=K}B$.
\end{enumerate}
\end{thm}

Finally we note that many non-trivial examples of hyperbolically embedded subgroups can be constructed via group theoretic Dehn filling discussed in the next section.

\subsection{Group theoretic Dehn filling}\label{subsec:DF}

Roughly speaking, Dehn surgery on a 3-dimensional manifold consists of cutting of a solid torus from the manifold (which may be thought of as ``drilling" along an embedded knot) and then gluing it back in a different way. The study of these ``elementary transformations" is partially motivated by the  Lickorish-Wallace theorem,
which states that every closed orientable connected 3-manifold can be obtained by
performing finitely many surgeries on the $3$-dimensional sphere.

\label{i-dfgeom} The second part of the surgery, called {\it Dehn filling}, can be formalized as follows. Let $M$ be a compact orientable 3--manifold with toric boundary. Topologically
distinct ways to attach a solid torus to $\partial M$ are parameterized
by free homotopy classes of unoriented essential simple closed curves in $\partial M$, called {\it slopes}.
For a slope $\sigma $, the corresponding   Dehn filling  $M(\sigma )$ of $M$ is the manifold
obtained from $M$ by attaching a solid torus $\mathbb D^2\times
\mathbb S^1$ to $\partial M$ so that the meridian
$\partial \mathbb D^2$ goes to a simple closed curve of the slope
$\sigma $.

The fundamental theorem of Thurston \cite[Theorem 1.6]{Th} (see \cite{HK,PP} for proofs) asserts that if  $M\setminus\partial M$ admits a complete finite volume hyperbolic
structure, then the resulting closed manifold $M(\sigma )$ is
hyperbolic provided $\sigma $ is  not in a finite set of exceptional slopes.
Algebraically this means that for all but finitely many primitive elements $x\in \pi _1(\partial M)\le \pi _1(M)$, the quotient group of $\pi_1(M)$ by the normal closure of $x$ (which is isomorphic to $\pi _1(M(\sigma))$ by the Seifert-van Kampen theorem) is non-elementary hyperbolic.
Modulo the geometrization conjecture proved by Perelman,
this algebraic statement is equivalent to the Thurston theorem.

Dehn filling can be generalized in the context of abstract group theory as follows. Let $G$ be a group and let $H$ be a subgroup of $G$. One can think of $G$ and $H$ as the analogues of $\pi_1(M)$ and $\pi _1(\partial M)$, respectively. Instead of considering just one element $x\in H$, let us consider a normal subgroup $N\lhd H$. By $\ll N\rr ^G$ we denote its normal closure in $G$. Associated to this data is the quotient group $G/\ll N\rr ^G$, which we call the \emph{group theoretic Dehn filling of $G$}.\label{i-df1}

\begin{thm}[Osin \cite{Osi07}] Suppose that a group $G$ is hyperbolic relative to a subgroup $H$. Then for any subgroup $N\lhd H$ avoiding a fixed finite set of nontrivial elements, the natural map from $H/N$ to $G/\ll N\rr ^G$ is injective and $G/\ll N\rr ^G$ is hyperbolic relative to $H/N$. In particular, if $H/N$ is hyperbolic, then so is $G/\ll N\rr ^G$.
\end{thm}

This theorem was proved in \cite{Osi07}; an independent proof for finitely generated torsion free relatively hyperbolic groups was later given in \cite{Gr_Ma}. Since the fundamental group of a complete finite volume hyperbolic manifold $M$ with toric boundary is hyperbolic relative to the subgroup $\pi_1 (\partial M)$ \cite{F} (which does embed in $\pi _1(M)$ in this case), the above result can be thought of as a generalization of the Thurston theorem.

In this paper we further generalize these results to groups with hyperbolically embedded subgroups. We also study the kernel of the filling and obtain some other results which are new even for relatively hyperbolic groups. We state a simplified version of our theorem here and refer to Theorem \ref{CEP} for a more general and stronger version.

\begin{thm}\label{CEP-simple}
Let $G$ be a group, $H$ a subgroup of $G$. Suppose that $H\h (G,X)$ for some $X\subseteq G$. Then there exists a finite subset $\mathcal F$ of nontrivial elements of $H$ such that for every subgroup $N\lhd H$ that does not contain elements from $\mathcal F$, the following hold.
\begin{enumerate}
\item[(a)] The natural map from $H /N$ to $G/\ll N\rr ^G$ is injective (equivalently, $H\cap \ll N\rr ^G=N$).

\item[(b)] $H/N\h (G/\ll N\rr ^G, \bar X)$, where $\bar X$ is the natural image of $X$ in $G/\ll N\rr ^G$.

\item[(c)] Every element of $\ll N\rr ^G$ is either conjugate to an element of $N$ or acts loxodromically on $\Gamma (G, X\sqcup H)$. Moreover, translation numbers of loxodromic elements of $\ll N\rr ^G$ are uniformly bounded away from zero.

\item[(d)] $\ll N\rr ^G=\ast_{t\in T } N^t $ for some subset $T\subseteq G$.
\end{enumerate}
\end{thm}

The proof of parts (a) and (b) makes use of van Kampen diagrams and Theorem \ref{ipchar-intr}, while parts (c) and (d) are proved using rotating families, namely Theorem \ref{he-vrf-intr} and Theorem \ref{free-intr}. Note that parts (c) and (d) of this theorem (as well as some other parts of Theorem \ref{CEP}) are new even for relatively hyperbolic groups.

\subsection{Applications}\label{subsec:appl}
We start with some results about mapping class groups. The following question is Problem 2.12(A) in  Kirby's list. It was asked in the early '80s and is often attributed to Penner,  Long, and  McCarthy. It is also recorded by  Ivanov \cite[Problems 3]{Iv}, and  Farb refers to it in \cite[\S 2.4]{F_book} as a ``well known open question". Recall that a subgroup of a mapping class group is called \emph{purely pseudo-Anosov}, if all its non-trivial elements are pseudo-Anosov.

\begin{prob}
Let $\Sigma $ be a closed orientable surface. Does $\MCG$ contain a non-trivial purely pseudo-Anosov normal subgroup?
\end{prob}

The abundance of finitely generated (non-normal) purely pseudo-Anosov free subgroups of $\MCG$ is well known, and follows from an easy ping-pong argument. However, this method does not elucidate the case of infinitely generated normal subgroups.  For a surface of genus $2$ the problem was answered by  Whittlesey \cite{Whi} who proposed an example based on Brunnian
braids (see also the study of Lee and Song of the kernel of a variation of the Burau representation \cite{LeeSong}). Unfortunately methods of \cite{Whi, LeeSong} do not generalize to closed surfaces of higher genus.

Another question was probably first asked by Ivanov (see \cite[Problem 11]{Iv}). Farb also recorded this question in \cite[Problem 2.9]{F_book}, and qualified it as a ``basic test question" for understanding normal subgroups
of $\MCG$.
\begin{prob}
Let $\Sigma $ be a closed orientable surface. Is the normal closure of a certain nontrivial power of a pseudo-Anosov element of $\MCG$ free?
\end{prob}

We answer both questions positively. In fact, our approach can be used in more general settings. Namely, we derive the following from Proposition \ref{cor-acyl-intr} and Theorem \ref{free-intr}. Part (a) can be alternatively derived from Corollary \ref{wpd-he-intr} and Theorem \ref{CEP-simple}.

\begin{thm}[Theorem \ref{wpd-free}]\label{wpd-free-intr}
Let $G$ be a group acting on a hyperbolic space $\X$.
\begin{enumerate}
\item[(a)] For every loxodromic WPD element $g\in G$, there exists $n\in \mathbb N$ such that the normal closure $\ll g^n\rr$ in $G$ is free.
\item[(b)] If the action of $G$ is acylindrical, then there exists $n\in \mathbb N$ such that for every loxodromic element $g\in G$, the normal closure $\ll g^n\rr$ in $G$ is free.
\end{enumerate}
Moreover, in both cases every non-trivial element of $\ll g^n\rr $ is loxodromic with respect to the action on $\X$.
\end{thm}

This result can be viewed as a generalization of Delzant's theorem \cite{Del_Duke} stating that for a hyperbolic group $G$ and every element of infinite order $g\in G$, there exists $n\in \mathbb N$ such that $\ll g^n\rr $ is free (see also \cite{Chay} for a clarification of certain aspects of Delzant's proof).

Applying  Theorem \ref{wpd-free-intr} to mapping class groups acting on the corresponding curve complexes, we obtain:

\begin{thm}[Theorem \ref{theo;MCG}]
Let $\Sigma$ be a (possibly punctured) closed orientable surface. Then there exists $n\in \mathbb N$ such that for any pseudo-Anosov element  $a\in \MCG$, the normal closure of $a^{n}$ is free and purely pseudo-Anosov.
\end{thm}

For $Out(F_n)$, we cannot achieve such a uniform result for the reason explained in Remark \ref{rem:outFn}. Newertheless, we obtain the following theorem by applying Theorem \ref{wpd-free-intr} to the action of $Out(F_n)$ on the free factor complex.

\begin{thm}[Theorem \ref{theo;Out}]
Let $f$ be an iwip element of $Out (F_n)$. Then there exists $n\in \mathbb N$ such that the normal closure of $f^{n}$ is free and purely iwip.
\end{thm}

Using techniques developed in our paper it is not hard to obtain many general results about groups with hyperbolically embedded subgroups. We prove just some of them to illustrate our methods and leave others for future papers. We start with a theorem, which shows that a group containing a non-degenerate hyperbolically embedded subgroup is ``large" in many senses. Recall that the class of groups with non-degenerate hyperbolically embedded subgroups includes non-elementary hyperbolic and relatively hyperbolic groups with proper peripheral subgroups, all but finitely many mapping class groups, $Out(F_n)$ for $n\ge 2$, the Cremona group, and many other examples.

\begin{thm}[Theorem \ref{large}]\label{large-intr}
Suppose that a group $G$ contains a non-degenerate hyperbolically embedded subgroup. Then the following hold.
\begin{enumerate}
\item[(a)] The group $G$ is SQ-universal. Moreover, for every finitely generated group $S$ there is a quotient group $Q$ of $G$ such that $S\h Q$.
\item[(b)] ${\rm dim\, }H_b^2 (G, \mathbb R)=\infty $\footnote[1]{After the first version of this paper was completed, M. Hull and the third author proved in \cite{HO} a more general extension theorem for quasi-cocycles, which also implies that ${\rm dim\, }H_b^2 (G, \ell^p(G))=\infty $ for any $p\in [1, +\infty)$. Later Bestvina, Bromberg and Fujiwara \cite{BBF13} proved this result (in different terms) even for more general coefficients. For the relation between these and some older results, see the discussion before Theorem 8.3 in \cite{Osi13}. In particular, we have ${\rm dim\, }H_b^2 (G, \ell^2(G))=\infty $, which allows one to apply orbit equivalence and measure equivalence rigidity results of Monod an Shalom \cite{MS} to groups with hyperbolically embedded subgroups.} In particular, $G$ is not boundedly generated.
\item[(c)] The elementary theory of $G$ is not superstable.
\end{enumerate}
\end{thm}

Recall that a group $G$ is called
{\it SQ-universal} \label{i-SQ} if  every countable group can be
embedded into a quotient of $G$ \cite{Sch}. It is straightforward
to see that any SQ-universal group contains an infinitely
generated free subgroup. Furthermore, since the set of all
finitely generated groups is uncountable and  every single
quotient of $G$ contains at most countably many finitely
generated subgroups,  every SQ-universal group has uncountably
many non-isomorphic quotients.

The first non-trivial example of an SQ-universal group was
provided by Higman, Neumann and Neumann \cite{HNN}, who proved
that the free group of rank $2$ is SQ-universal. Presently many
other classes of groups are known to be SQ-universal:
various HNN-extensions and amalgamated products
\cite{FT,Los,Sas}, groups of deficiency $2$ (it follows from \cite{BP}), most $C(3)\,
\& \, T(6)$-groups \cite{How}, non-elementary hyperbolic groups \cite{Del_Duke,Ols}, and non-elementary groups hyperbolic relative to proper subgroups \cite{AMO}. However our result is new, for instance, for mapping class groups, $Out(F_n)$, the Cremona group and some other classes. The proof is based on Theorem \ref{vf-intr} and part (a) of Theorem \ref{CEP-simple}.

The next notion of ``largeness" comes from model theory. We briefly recall some definitions here and refer to \cite{Mar} for details. An algebraic structure  $M$ for a first order language is called \emph{$\kappa $-stable} for an infinite cardinal $\kappa $, if for every subset $A\subseteq M$ of cardinality $\kappa $ the number of complete types over $A$ has cardinality $\kappa $. Further, $M$ is called \emph{stable}, if it is $\kappa $-stable for some infinite cardinal $\kappa $, and \emph{superstable} if it is $\kappa $-stable for all sufficiently large cardinals $\kappa $. A theory $T$ in some language is called {\it stable} or {\it superstable}, if all models of $T$ have the respective property. The notions of stability and superstability were introduced by Shelah \cite{Sh69}. In \cite{Sh85}, he showed that superstability is  a necessary condition for a countable complete theory to permit a reasonable classification of its models. Thus the absence of superstability may be considered, in a very rough sense, as an indication of logical complexity of the theory. For other results about stable and superstable groups we refer to the survey \cite{Wag}.

Sela \cite{Sel} showed that free groups and, more generally, torsion free hyperbolic groups are stable. On the other hand, non-cyclic free groups are known to be not superstable \cite{P}. More generally, Ould Houcine \cite{Ould} proved that a superstable torsion free hyperbolic group is cyclic. It is also known that a free product of two nontrivial groups is  superstable if and only if both groups have order $2$ \cite{P}. Our theorem can be thought as a generalization of these results.

For the definition and main properties of bounded cohomology we refer to \cite{Mon}. It is known that $H_b^2 (G, \mathbb R)$ vanishes for amenable groups and all irreducible lattices in higher rank semi-simple algebraic groups over local fields. On the other hand, according to Bestvina and Fujiwara \cite{BF}, groups which admit a ``non-elementary" (in a certain precise sense) action on a hyperbolic space have infinite-dimensional space of nontrivial quasi-morphisms $\widetilde{QH} (G)$, which can be identified with the kernel of the canonical map $H_b^2(G, \mathbb R)\to H^2(G, \mathbb R)$. Examples of such groups include non-elementary hyperbolic groups \cite{EF}, mapping class groups of surfaces of higher genus \cite{BFe},  and $Out(F_n)$ for $n\ge 2$ \cite{BFe}. In Section \ref{appl} we show that the action of any group $G$ with a non-degenerate hyperbolically embedded subgroup on the corresponding relative Cayley graph is non-elementary in the sense of \cite{BF}, which implies ${\rm dim\, }H_b^2 (G, \mathbb R)=\infty $.

Recall also that a group $G$ is {\it boundedly generated}, if there are elements $x_1, \ldots , x_n$ of $G$ such that for any $g\in G$ there exist integers $\alpha _1,\ldots , \alpha _n$ satisfying the equality $g=x_1^{\alpha _1}\ldots x_n^{\alpha _n}$. Bounded generation is closely related to the Congruence Subgroup Property of arithmetic groups \cite{Rap}, subgroup growth \cite{LubSeg}, and Kazhdan Property (T) of discrete groups \cite{Sh}. Examples of boundedly generated groups include $SL_n(\mathbb Z)$ for $n\ge 3$ and many other lattices in semi-simple Lie groups of $\mathbb R$-rank at least $2$ \cite{CK,T}. There also exists a finitely presented boundedly generated group which contains all recursively presented groups as subgroups \cite{Osi05}. It is well-known and straightforward to prove that for every boundedly generated group $G$, the space $\widetilde{QH} (G)$ is finite dimensional, which implies the second claim of (c).

We mention one particular application of Theorem \ref{large-intr} to subgroups of mapping class groups. It follows immediately from part (c) of Theorem \ref{ex-MCG-intr} together with the fact a group that has an $SQ$-universal subgroup of finite index or an $SQ$-universal quotient is itself $SQ$-universal.

\begin{cor}[Corollary \ref{coro;HJI}]\label{SQ-sub-MCG}
Let $\Sigma$ be a (possibly punctured) closed orientable surface. Then every subgroup of $\MCG$ is either virtually abelian or $SQ$-universal.
\end{cor}

It is easy to show that every $SQ$-universal group $G$ contains non-abelian free subgroup; if, in addition, $G$ is finitely generated, then it has uncountably many normal subgroups. Thus Corollary \ref{SQ-sub-MCG} can be thought of as a simultaneous strengthening of the Tits alternative \cite{Iva92} and various non-embedding theorems of lattices into mapping class groups \cite{Farb_Masur}. Indeed we recall that if $\Gamma$ is an irreducible lattice in a connected higher rank semi-simple Lie group with finite center, then every normal subgroup of $\Gamma$ is either finite or of finite index by the Margulis theorem. In particular, $\Gamma$ has only countably many normal subgroups. Hence the image of every such  a lattice in $\MCG$ is finite.

We also obtain some results related to von Neumann algebras and reduced $C^\ast$-algebras of groups with hyperbolically embedded subgroups. Recall that a non-trivial group $G$ is \textit{ICC} (Infinite Conjugacy Classes) if every nontrivial conjugacy class of $G$ is infinite. By a classical result of Murray and von Neumann \cite{MvN} a countable discrete group $G$ is ICC if and only if  the von Neumann algebra $W^\ast (G)$ of $G$ is a $II_1$ factor. Further a group $G$ is called {\it inner amenable}, if there exists a finitely additive measure $\mu\colon \mathcal P (G\setminus \{ 1\})\to [0,1]$  defined on the set of all subsets of $G\setminus \{ 1\}$ such that $\mu (G\setminus \{ 1\})=1$ and $\mu $ is conjugation invariant, i.e., $\mu (g^{-1}Ag)=\mu (A)$ for every $A\subseteq G\setminus \{ 1\}$ and $g\in G$. This property was first introduced by Effros \cite{Efr}, who proved that if $G$ is a countable group and $W^\ast (G)$ is a $II_1$ factor which has property $\Gamma $ of Murray and von Neumann, then $G$ is inner amenable. (The converse is not true as was recently shown by Vaes \cite{V}.)

It is easy to show that every group with a nontrivial finite conjugacy class is inner amenable. It is also clear that every amenable group is inner amenable. Other examples of inner amenable groups include R. Thompson's group $F$, its generalizations \cite{J,Pic}, and some HNN-extensions \cite{Stal}. On the other hand, the following groups are known to be not inner amenable: ICC Kazhdan groups (this is straightforward to prove), lattices in connected real semi-simple Lie groups with trivial center and without compact factors \cite{HaSk}, and non-cyclic torsion free hyperbolic groups \cite{Har88}. To the best of our knowledge, the question of whether every non-elementary ICC hyperbolic group is not inner amenable was open until now (see the discussion in Section 2.5 of \cite{Har95}). In this paper we prove a much more general result.

\begin{thm}[Theorem \ref{cstar}]\label{vN-intr}
Suppose that a group $G$ contains a non-degenerate hyperbolically embedded subgroup. Then the following conditions are equivalent.
\begin{enumerate}
\item[(a)] $G$ has no nontrivial finite normal subgroups.
\item[(b)] $G$ is ICC.
\item[(c)] $G$ is not inner amenable.
\end{enumerate}
If, in addition, $G$ is countable, the above conditions are also equivalent to
\begin{enumerate}
\item[(d)] The reduced $C^\ast $-algebra of $G$ is simple.
\item[(e)] The reduced $C^\ast $-algebra of $G$ has  a unique normalized trace.
\end{enumerate}
\end{thm}

The study of groups with simple  reduced $C^\ast$-algebras have begun with the Power's paper \cite{Pow}, where he proved that the reduced $C^\ast $-algebra of a non-abelian free group is simple. Since then many other examples of groups with simple  reduced $C^\ast$-algebras have been found, including centerless mapping class groups, $Out(F_n)$ for $n\ge 2$, many amalgamated products and HNN-extensions \cite{BrHa,HaPre}, and free Burnside groups of sufficiently large odd exponent \cite{OO}. For a comprehensive survey we refer to \cite{Har07}. Recall also that a (normalized) \emph{trace} on a unitary C$^\ast $-algebra $A$ is a linear map $\tau\colon A\to \mathbb C$ such that $$\tau (1)=1,\;\;\; \tau (a^\ast a)\ge 0,\;\;\; {\mathrm{and}} \;\;\; \tau(ab)=\tau (ba)$$ for all $a, b \in A$.

Equivalence of (a), (d), and (e) was known before for relatively hyperbolic groups \cite{AM}.  Note however that in \cite{AM} properties (d) and (e) are derived from the fact that the corresponding group satisfies the property $P_{nai}$, which says that for every finite subset $\mathcal F\subseteq G$, there exists a nontrivial element $g\in G$ such that for every $f\in \mathcal F$ the subgroup of $G$ generated by  $f$ and $g$ is isomorphic to the free product of the cyclic groups generated by $f$ and $g$. In this paper we choose a different approach:  Theorems \ref{vf-intr} and \ref{CEP-simple} are used to show that if a group $G$ contains a non-degenerate hyperbolically embedded subgroup and satisfies (a), then it is a group of the so-called Akemann-Lee type, which means that $G$ contains a non-abelian normal free subgroup with trivial centralizer. For countable groups, this implies (d) and (e) according to \cite{AL}.


\section{Preliminaries}


\subsection{General conventions and notation}

Throughout the paper we use the standard notation $[a,b]=a^{-1}b^{-1}ab$ and $a^b=b^{-1}ab$ for elements $a,b$ of a group $G$. Given a subset $R\subseteq G$, by $\ll R\rr ^{G}$ (or simply by $\ll R\rr$ if no confusion is possible) we denote the normal closure of $R$ in $G$, i.e., the smallest normal subgroup of $G$ containing $R$.

We say that a group is \emph{elementary} \label{i-elem} if it is virtually cyclic.

Given a path $p$ in a metric space, we denote by $p_-$ and $p_+$ its beginning and ending points, respectively. If $p$ is a combinatorial path in a labeled directed graph (e.g., a Cayley graph or a van Kampen diagram), $\Lab (p)$ denotes its label \label{i-lab1}.

When talking about metric spaces, we allow the distance function to take infinite values. Algebraic operations and relations $<$, $>$, etc., are extended to $[-\infty, +\infty]$ in the natural way. Say, $c+\infty=\infty$ for any $c\in (-\infty , +\infty]$ and $c\cdot \infty = \infty$ for any $c\in [0, +\infty]$, while $-\infty+\infty$ and $\infty /\infty $ are undefined. Whenever we write any expression potentially involving $\pm \infty$, we assume that it is well defined.

If $S$ is a geodesic metric space and $x,y\in S$, $[x,y]$ denotes a geodesic in $S$ connecting $x$ and $y$. For two subsets $T_1, T_2$ of a metric space $S$ with metric $\d$, we denote by $\d (T_1, T_2)$ and $\d _{Hau} (T_1, T_2)$ the usual and the Hausdorff distance between $T_1$ and $T_2$, respectively. That is,
$$
\d (T_1, T_2)=\inf \{\d (t_1, t_2)\mid t_1\in T_1,\, t_2\in T_2\}
$$
and
$$
\d _{Hau}(T_1, T_2)=\sup \{ \d (t_1, T_2), \d (T_1, t_2) \mid t_1\in T_1,\, t_2\in T_2\} \} .\label{i-dHau}
$$
For a subset $T\subseteq S$, $T^{+\e} $ denotes the closed $\e$-neighborhood of $T$, i.e.,
$$
T^{+\e }= \{ s\in S\mid \d (s,T)\le \e \}.
$$

Given a word $W$ in an alphabet $\mathcal A$, we denote by $\| W\| $ its length. We
write $W\equiv V$ to express the letter-for-letter equality of
words $W$ and $V$. If $\mathcal A$ is a generating set of a group $G$, we do not distinguish between words in $\mathcal A$ and elements of $G$ represented by these words if no confusion is possible. Recall that a subset $X$ of a group $G$ is said to be {\it symmetric} if for any $x\in X$, we have $x^{-1}\in X$.
In this paper all generating sets of groups under consideration
are supposed to be symmetric, unless otherwise is stated explicitly.

If $G$ is a group and $X\subseteq G$, we denote by $|g|_X$ the {\it (word) length}  \label{i-wl} of an element $g\in G$. Note that we do not require $G$ to be generated by $X$ and we will often work with word length with respect to non-generating subsets of $G$.  By definition, $|g|_X$ is the length of a shortest word in $X$ representing $g$ in $G$ if $g\in \langle X\rangle $ and $\infty $ otherwise. Associated to this length function is the \emph{word metric} \label{i-wd} $\d _{X}\colon G\times G\to [0, \infty ]$ defined in the usual way:
$$
\d_X (f,g)=|f^{-1}g|_X
$$
for any $f,g\in G$. If $G$ is generated by $X$, we also denote by $\d _X$ the natural extension of this metric to the corresponding Cayley graph.

To deal with infinite values, we extend addition and multiplication to $[0, \infty]$ in the following way:
$$
c+\infty=\infty+c=\infty, \;\;\; d\cdot \infty =\infty \cdot d=\infty, \;\;\; 0\cdot \infty = 0
$$
for every $c\in [0, \infty]$ and $d\in (0, \infty)$. We also order $[0, +\infty]$ in the natural way.

\subsection{Hyperbolic spaces and group actions}

A geodesic metric space $S$
is {\it $\delta $--hyperbolic} \label{i-hyps} for some $\delta \ge 0$ (or simply
{\it hyperbolic}) if for any geodesic triangle with vertices $x$, $y$, $z$ in $S$,
and any points $p\in [x,y]$, $q\in [x,y]$ with
$\d(x,p)=\d(x,q)\leq (y.z)_x$, we have
$\d(p,q)\leq \delta$. Here by $y.z)_x$ we denote the Gromov's product of $y$ and $z$ with respect to $x$, that is,
$$
(y.z)_x=\frac{\d(x,y)+\d(x,z)-\d(y,z)}{2}.
$$
In particular, any
side of the triangle belongs to the union of the closed $\delta
$--neighborhoods of the other two sides \cite{Gro}. A finitely generated group is called {\it hyperbolic} \label{defhypg} if its Cayley
graph with respect to some (equivalently, any) generating set is a hyperbolic metric space.

By $\partial S$ we denote the Gromov boundary of a hyperbolic space $S$. Note that we do not assume, in general, that $S$ is proper and thus we have to employ the Gromov's definition of the boundary via sequences convergent at infinity (see \cite[Section 1.8]{Gro}).

Given a group $G$ acting on a hyperbolic space $S$, an element $g\in G$ is called \emph{elliptic} \label{i-ell} if some (equivalently, any) orbit of $g$ is bounded, and \emph{loxodromic} \label{i-lox} if the map $\mathbb Z\to S$ defined by $n\mapsto g^n(s)$ is a quasi-isometric embedding for some (equivalently any) $s\in S$. Equivalently, an element $g\in G$ is loxodromic if it has exactly $2$ limit points on the Gromov boundary $\partial S$. Finally, an element $g$ is \emph{parabolic} \label{i-par} if it has exactly one limit point on the boundary  $\partial S$. Every isometry of a hyperbolic space is either elliptic, or loxodromic, or parabolic. For details we refer to \cite{Gro}; a clarification of some of Gromov's arguments in the case of non-proper spaces can be found in \cite{Ham}.

Given a path $p$ in a metric space, we denote by $p_-$ and $p_+$ the origin and the terminus of $p$, respectively. We also denote $\{ p_-, p_+\}$ by $p_\pm$. The length of $p$ is denoted by $\ell (p)$. A path $p$ in a metric space $S$ is called
{\it $(\lambda , c)$--quasi--geodesic} \label{i-qg} for some $\lambda \ge 1$,
$c\ge 0$ if
$$\ell (q)\le  \lambda dist(q_-, q_+)+c$$ for any subpath $q$ of $p$.
The following property of quasi-geodesics in a hyperbolic space is well known and will be widely used in this paper.

\begin{lem}\label{qg}
For any $\delta \ge 0$, $\lambda \ge 1$, $c\ge 0$, there exists a
constant $\kappa =\kappa (\delta , \lambda , c)\ge 0$ such that
\begin{enumerate}
\item[(a)] Every two $(\lambda , c)$--quasi--geodesics in a $\delta $-hyperbolic space with the same endpoints belong to the closed $\kappa $--neighborhoods of each other.

\item[(b)] For every two bi-infinite $(\lambda , c)$--quasi--geodesics $a,b$ in a $\delta $-hyperbolic space,  $\d _{Hau} (a,b)<\infty $ implies  $\d _{Hau} (a,b)<\kappa $.
\end{enumerate}
\end{lem}

\begin{proof}
The first assertion follows for instance from \cite[Th.1.7 p.401]{BH_metric}.
We could not find a reference for the second assertion when the space is not proper
(we cannot use the existence of a bi-infinite geodesic).
Let $M=\d_{Hau}(a,b)$, and $x=a(t)$.
Let $\kappa_0$ be a constant as in the first assertion. We claim that we can take $\kappa=2\kappa_0+2\delta$
in the second assertion.

Consider $t_1<t<t_2$, $x_1=a(t_1), x_2=a(t_2)$
so that $\d(x_1,x),\d(x_2,x)\geq M+\kappa_0+10\delta$.
Consider $y\in [x_1,x_2]$ at distance
$\leq \kappa_0$ from $x$.
Let $x'_1,x'_2$ be two points on $b$ at distance at most $M$ from $x_1$ and $x_2$.
Considering the quadrilateral $x_1,x_2,x'_2,x'_1$, we see that $y$ is at distance
at most $2\delta$ from $[x_1,\cup,x'_1]\cup[x'_1,x'_2]\cup [x'_2,x_2]$,
but since $d(y,\{x_1,x_2\})\geq M+10\delta$, there is a point $y'\in [x'_1,x'_2]$
such that $d(y,y')\leq 2\delta$.
Using the first assertion, $y'$ is at distance at most $\kappa_0$ from $b$,
so $d(x,b)\leq 2\kappa_0+2\delta$.
\end{proof}

The next lemma is a simplification of Lemma 10 from \cite{Ols92}. We say that two paths $p$ and $q$ in a metric space are {\it $\e $-close} for some $\e >0$ if $\d _{Hau} \{p_\pm, q_\pm)\le \e$.

\begin{lem}\label{Ols}
Suppose that the set of all sides of a geodesic $n$--gon $\mathcal P=p_1p_2\ldots p_n$ in a $\delta $--hyperbolic space is divided into two subsets $S$, $T$. Assume that the total lengths of all sides from $S$ is at least $10^3 cn$ for some $c\ge 30\delta $. Then there exist two distinct sides $p_i$, $p_j$, and $13\delta $-close subsegments $u$, $v$ of $p_i$ and $p_j$, respectively, such that $p_i\in S$ and $\min \{ \ell(u), \ell(v)\} > c$.
\end{lem}

\newcommand{\sqc}{strongly quasiconvex}
From now on, let $S$ be a geodesic metric space.
\begin{defn}\label{dfn_sqc}
 A subset  $Q\subset S$ is \textit{$\sigma $-quasiconvex} if any geodesic in $S$ between any two points of $Q$ is contained in the closed $\sigma$-neighborhood of $Q$. We say that $Q\subset S$ is \textit{$\sigma $-\sqc } if for any two points $x,y\in Q$, there exist $x^\prime,y^\prime\in Q$
and geodesics $[x^\prime,y^\prime],[x,x^\prime],[y,y^\prime]$ of $S$
  such that $\max \{ \d (x, x^\prime ), \d (y, y^\prime )\} \le \sigma $ and $[x^\prime,y^\prime]\cup[x,x^\prime]\cup[y,y^\prime]\subset Q$.
\end{defn}

We will need the following remarks. The proofs are elementary and we leave them to the reader.

\begin{lem}\label{lem_4delta}
Let $S$ be a $\delta$-hyperbolic space, and $Q$ a subset of $S$.
\begin{enumerate}
\item[(a)] If $Q$ is $\sigma$-quasiconvex, then for all $r\geq \sigma$, $Q^{+r }$ is $2\delta$-\sqc.
\item[(b)] If $Q$ is $\sigma $-\sqc, then the induced path-metric $\d _Q$ on $Q$ satisfies for all $x,y\in Q$, $\d _S(x,y)\leq \d _Q(x,y) \leq \d _S(x,y)+2\sigma $.
\item[(c)] If $Q$ is $2\delta$-\sqc then $Q$ is $4\delta$-quasiconvex.
\end{enumerate}
\end{lem}

\subsection{Relative presentations and isoperimetric functions}\label{sec-GPIF}

\paragraph{Van Kampen Diagrams and isoperimetric functions} \label{i-vK} A {\it van Kampen diagram} $\Delta $ over a presentation
\begin{equation}
G=\langle \mathcal A\; | \; \mathcal O\rangle \label{ZP}
\end{equation}
is a finite oriented connected planar 2--complex endowed with a
labeling function $\Lab : E(\Delta )\to \mathcal A \cup\{1\}$,
where $E(\Delta ) $ denotes the set of oriented edges of $\Delta $,
such that $\Lab(e^{-1})\equiv(\Lab(e))^{-1}$. (Recall that we always assume that generating sets are symmetric; thus $\mathcal A=\mathcal A^{-1}$). We identify $1$ with the empty word in $A$; thus $1=1^{-1}$. It is convenient to assume that the empty word represents the identity element of $G$.

Given a path $p=e_1\ldots e_k$ in a van Kampen diagram, where $e_1, \ldots, e_k$ are edges, we define $\Lab (p)$ \label{i-lab2} to be the concatenation of labels of $e_1, \ldots , e_k$. Note that we remove all $1$'s from the label since $1$ is identified with the empty word. Thus the label of every path in a van Kampen diagram is a word in $\mathcal A$.

We call edges labelled by letters from $\mathcal A$ \emph{essential}; edges labelled by $1$ are called \emph{$0$-edges}. \label{i-0edge}  Since   $1^{-1}=1$, we will often drop the orientation of $0$-edges in illustrations.

By a \emph{cell} of a van Kampen diagram, we always mean a $2$-cell. Given a cell $\Pi $ of $\Delta $, we denote by $\partial \Pi$ the
boundary of $\Pi $. Similarly, $\partial \Delta $ denotes the
boundary of $\Delta $. The labels of $\partial \Pi $ and $\partial
\Delta $ are defined up to cyclic permutations. An additional
requirement is that for any cell $\Pi $ of $\Delta $, one of the following two conditions holds.

\begin{enumerate}
 \item[(a)] $\Lab (\partial \Pi )$ is equal to (a cyclic permutation of) a word $P^{\pm 1}$, where $P\in \mathcal O$.
 \item[(b)] The boundary path of $\Pi$ either entirely consists of $0$-edges or has exactly two  essential edges (in addition to $0$-edges) with mutually inverse labels. (In both cases the boundary label of such a cell is equal to $1$ in the free group generated by $\mathcal A$.) Such cells are called \emph{$0$-cells} and all other cells are called {\it essential}. \label{i-0c}
\end{enumerate}

A diagram $\Delta $ over (\ref{ZP}) is called a {\it disk diagram} if it is homeomorphic to a disc. Note that every simply connected van Kampen diagram can be made homeomorphic to a disk by adding $0$-cells. \label{i-0ref} This can be done by the so-called \emph{$0$-refinement}, which is illustrated on Fig. \ref{0-ref}. For a more formal discussion we refer to \cite[Section 11]{Ols-book}.

Similarly, using $0$-refinement we can ensure the following condition, which will be assumed throughout the paper.
\begin{enumerate}
\item[(c)] Every cell is homeomorphic to a disk, i.e., its boundary do not self intersect.
\end{enumerate}

By the well-known van Kampen Lemma, a word $W$ over an alphabet
$\mathcal A$ represents the identity in the group given by
(\ref{ZP}) if and only if there exists a disc diagram $\Delta $ over (\ref{ZP}) such
that $\Lab (\partial \Delta )\equiv W$ (see \cite[Ch. 4]{Ols-book}).

\begin{figure}
  \label{0-ref}
  \centering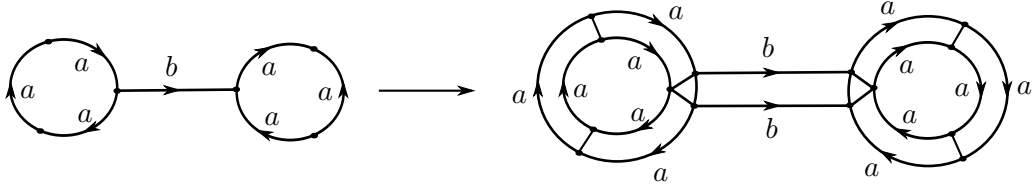
  \caption{A 0-refinement of a van Kampen diagram over the presentation $G=\langle a,b\mid a^3=1\rangle$}
\end{figure}

\begin{rem}\label{rem-mu}
It is easy to show that for any vertex $O$ of a disc van Kampen diagram $\Delta $ over (\ref{ZP}), there is a natural continuous map $\mu$ from the $1$-skeleton of $\Delta $ to the Cayley graph $\Ga $
that maps $O$ to the identity vertex of $\Ga $, collapses $0$-edges to points,  and preserves labels and
orientation of essential edges.
\end{rem}

Let
\begin{equation}\label{standpres}
G=\langle X\mid \mathcal R\rangle
\end{equation}
be a group presentation. Given a word $W$ in the alphabet $X\cup \X^{-1}$ representing $1$ in $G$, denote by $Area (W)$ the minimal number of cells in a van Kampen diagram with boundary label $W$. A function $f\colon \mathbb N\to \mathbb N$ is called an {\it isoperimetric function} \label{i-ip} of (\ref{standpres}) if $Area (W)\le f(n)$ for every word $W$ in $X\cup X^{-1}$ of length at most $n$ representing $1$ in $G$.

\paragraph{Relative presentations.} \label{i-rp} Let $G$ be a group, $\Hl $ a collection of subgroups of $G$. A subset $X$ is called a {\it relative generating set of
$G$ with respect to $\Hl $} \label{i-rgs} if $G$ is generated by $X$ together with
the union of all $H_\lambda $'s. In what follows we always assume relative generating sets
to be symmetric, i.e., if $x\in X$, then $x^{-1}\in X$.

Let us fix a relative generating set $X$ of
$G$ with respect to $\Hl $. The group $G$ can be regarded as
a quotient group of the free product
\begin{equation}
F=\left( \ast _{\lambda\in \Lambda } H_\lambda  \right) \ast F(X),
\label{F}
\end{equation}
where $F(X)$ is the free group with the basis $X$.

Suppose that kernel of the natural homomorphism $F\to G$ is a normal closure of a subset
$\mathcal R$ in the group $F$. The set $\mathcal R$ is always supposed to be {\it symmetrized}. This means that if $R\in \mathcal R$ then every cyclic shift of $R^{\pm 1}$ also belongs to $\mathcal R$.
Let $$\mathcal H= \bigsqcup\limits_{\lambda \in \Lambda } H_\lambda.$$ We think of $\mathcal H$ as a subset of $F$. Let us stress that the union is disjoint, i.e., for every nontrivial element $h\in G$ such that $h\in H_\lambda\cap H_\mu $ for some $\lambda\ne \mu$, the set $\mathcal H$ contains two copies of $h$, one in $H_\lambda $ and the other in $H_\mu $. Further for every $\lambda \in \Lambda $, we denote by $\mathcal S_\lambda
$ the set of all words over the alphabet $ H_\lambda $ that represent the identity in $H_\lambda $. Then the group $G$ has the presentation
\begin{equation}
\langle X, \mathcal H  \mid  \mathcal S\cup \mathcal R\rangle,
\label{rp}
\end{equation}
where $\mathcal S=\bigcup\limits_{\lambda\in \Lambda } \mathcal S_\lambda $. In what follows, presentations of this type are called {\it relative presentations} of $G$ with respect to $X$ and $\Hl$. Sometimes we will also write (\ref{rp}) in the form
$$
G=\langle X, \Hl\mid \mathcal R\rangle .
$$

Let $\Delta $ be a van Kampen diagram over (\ref{rp}). As usual, \label{i-RSc} a cell of $\Delta $ is called an $\mathcal R$-cell (respectively, a $\mathcal S$-cell) if its boundary is labeled by a (cyclic permutation of a) word from $\mathcal R$ (respectively $\mathcal S$).

Given a word $W$ in the alphabet $X\cup \mathcal H$ such that $W$
represents $1$ in $G$, there exists an expression
\begin{equation}
W=_F\prod\limits_{i=1}^k f_i^{-1}R_i^{\pm 1}f_i \label{prod}
\end{equation}
where the equality holds in the group $F$, $R_i\in \mathcal R$, and
$f_i\in F $ for $i=1, \ldots , k$. The smallest possible number
$k$ in a representation of the form (\ref{prod}) is called the
{\it relative area} \label{i-ra} of $W$ and is denoted by $Area^{rel}(W)$.

Obviously $Area^{rel}(W)$ can also be defined in terms of van Kampen diagrams. Given a diagram $\Delta $ over (\ref{rp}), we define its {\it relative area}, $Area^{rel} (\Delta )$, to be the number of $\mathcal R$-cells in $\Delta $. Then $Area^{rel}(W)$ is the minimal relative area of a van Kampen diagram over (\ref{rp}) with boundary label $W$.

Finally we say that $f(n)$ is a {\it relative isoperimetric function} \label{i-rip} of (\ref{rp}) if for every word $W$ of length at most $n$ in the alphabet $X\cup \mathcal H$ representing $1$ in $G$, we have $Area^{rel} (W)\le f(n)$. Thus, unlike the standard isoperimetric function, the relative one only counts $\mathcal R$-cells.

\paragraph{Relatively hyperbolic groups.}
The notion of relative hyperbolicity goes back to Gromov \cite{Gro}. There are
many definitions of (strongly) relatively hyperbolic
groups \cite{Bow,DS,F,Osi06a}. All these definitions are equivalent for finitely generated groups. The proof of the equivalence and a detailed analysis of the case of infinitely generated groups can be found in \cite{Hru}.

\label{i-rhg} We recall the isoperimetric
definition suggested in \cite{Osi06a}, which is the most suitable one for
our purposes. That relative hyperbolicity in the sense of
\cite{Bow,F,Gro} implies relative hyperbolicity in the sense of the
definition stated below is essentially due to
Rebbechi \cite{Reb}. Indeed it was proved in \cite{Reb} for finitely presented groups. The later condition is not really important and the proof from \cite{Reb} can easily be generalized to the general case (see \cite{Osi06a}). The converse implication was proved in \cite{Osi06a}.

\begin{defn}\label{rhg}
Let $G$ be a group, $\Hl $ a collection of subgroups of $G$. Recall that $G$ is {\it hyperbolic relative to a collection of subgroups} $\Hl $ if $G$ has a finite relative presentation (\ref{rp}) (i.e., the sets $X$ and  $\mathcal R$ are finite) with linear relative isoperimetric function.
\end{defn}
In particular, $G$ is an ordinary {\it hyperbolic group} if $G$ is hyperbolic relative to the trivial subgroup.


\section{Generalizing relative hyperbolicity}\label{sec:GRH}


\subsection{Weak relative hyperbolicity and bounded presentations}\label{sec-WRHBP}

Throughout this section let us fix a group  $G$, a collection of subgroups $\Hl $ of $G$, and a (not necessary finite) relative generating set $X$ of $G$ with respect to $\Hl $. That is, we assume that $G$ is generated by $X$ together with the union of all $H_\lambda$. We also assume that $X$ is symmetric, i.e., for every $x\in X$, we have $x^{-1}\in X$.

Our first goal is to extend some standard tools from the theory of relatively hyperbolic groups to a more general case.

More precisely, as in the case of relatively hyperbolic groups we define 
\begin{equation}\label{calH}
\mathcal H= \bigsqcup\limits_{\lambda \in \Lambda } H_\lambda
\end{equation}
and let $\G $ denote the Cayley graph of $G$ with respect to the alphabet $X\sqcup\mathcal H$.

Here by the \emph{Cayley graph} \label{i-Cg} of a group $G$ with respect to an alphabet $\mathcal A$ given together with a (not necessarily injective) map $\alpha \colon \mathcal A\to G$ we mean the graph with vertex set $G$ and set of edges $\{ (g,g\alpha(a),a) \mid g\in G, \; a\in \mathcal A\}$. The edge $(g, g\alpha(a),a)$ goes from $g$ to $g\alpha(a)$ and has label $a$. For $\mathcal A=X\sqcup \mathcal H$, the map $\alpha $ is the obvious one, so we omit it from the notation. Note that some letters from $X\sqcup\mathcal H$ may represent the same element in $G$, in which case $\G $ has multiple edges corresponding to these letters.\label{i-Gxh2} For example, there are at least $|\Lambda |$ loops at each vertex, which correspond to identity elements of subgroups $H_\lambda$. (We could remove these loops by considering $H_\lambda \setminus \{1\}$ instead of $H_\lambda$ in (\ref{calH}), but their presence does not cause any problems.)

\begin{defn}\label{i-wrh}
We say that $G$ is {\it weakly hyperbolic} relative to $X$ and $\Hl $ if the Cayley graph $\G $ is hyperbolic.
\end{defn}

We also denote by $\Gamma _\lambda $ the Cayley graphs $\Gamma (H_\lambda, H_\lambda )$, which we think of as complete subgraphs of $\G $.

\begin{defn} \label{maindef}
For every $\lambda \in \Lambda $, we introduce a \textit{relative metric} $\dl \colon H_\lambda \times H_\lambda \to [0, +\infty]$ as follows. Given $h,k\in H_\lambda $ let $\dl (h,k)$ be the length of a shortest path $p$ in $\G $ that connects $h$ to $k$ and has no edges in $\Gamma _\lambda $. We stress that we do alow $p$ to pass through vertices of $\Gamma_\lambda$; $p$ can also have edges $e$ labelled by elements of $H_\lambda$ if $e$ is not in $\Gamma_\lambda$ (i.e., $p$ can travel inside a coset of $H_\lambda$ other than $1H_\lambda$). If no such path exists, we set $\dl (h,k)=\infty $. Clearly $\dl $ satisfies the triangle inequality.
\end{defn}

The notion of weak relative hyperbolicity defined above is not sensitive to `finite changes' in generating sets in the following sense. Recall that two metrics $\d _1, \d _2\colon S\to [0, +\infty)$ on a set $S$ are {\it bi-Lipschitz equivalent} \label{i-bL1} (we write $\d _1\sim _{Lip} \d _2$ if $\d_1$) if the ratios $\d _1/\d _2$ and $\d _2/\d _1$ are bounded on $S\times S$ minus the diagonal.

\begin{prop}\label{x1x2}
Let $G$ be a group, $\Hl$ a collection of subgroups of $G$, $X_1$, $X_2$ two relative generating sets of $G$ with respect to $\Hl $. Suppose that $|X_1\triangle X_2|<\infty $. Then $\d_{X_1\cup \mathcal H}\sim _{Lip} \d_{X_2\cup \mathcal H}$. In particular, $\Gamma (G, X_1\sqcup \mathcal H)$ is hyperbolic if and only if $\Gamma (G, X_1\sqcup \mathcal H)$ is.
\end{prop}

\begin{proof}
The proof is standard and is left to the reader.
\end{proof}

\begin{rem}\label{dl-not-eq}
Note that the metric $\dl $ is much more sensitive. For instance, let $G$ be any finite group, $H=G$, and $X=\emptyset$. Then $\widehat \d (g,h)=\infty$ for any distinct $g,h\in H$. However if we take $X=G$, we have $\widehat \d (g,h)<\infty$ for all $g,h\in H$.
\end{rem}

\begin{defn}[{\bf Components, connected and isolated components}]\label{defcomp}
Let $q$ be a path in the Cayley graph $\G $. A (non-trivial) subpath $p$ of $q$ is called an \emph{$H_\lambda $-subpath}, if the label of $p$ is a word in the alphabet $H_\lambda   $. An $H_\lambda$-subpath $p$ of $q$ is an {\it $H_\lambda $-component} if $p$ is not contained in a longer subpath of $q$ with this property. Further by a {\it component} of $q$ we mean an $H_\lambda $-component of $q$ for some $\lambda \in \Lambda$.

Two $H_\lambda $-components $p_1, p_2$ of a path $q$ in $\G $ are called {\it connected} if there exists a
path $c$ in $\G $ that connects some vertex of $p_1$ to some vertex of $p_2$, and ${\Lab (c)}$ is a word consisting only of letters from $H_\lambda $. In algebraic terms this means that all vertices of $p_1$ and $p_2$ belong to the same left coset of $H_\lambda $. Note also that we can always assume that $c$ has length at most $1$ as every non-trivial element of $H_\lambda $ is included in the set of generators. An $H_\lambda$-component $p$ of a path $q$ in $\G $ is {\it isolated} if it is not connected to any other component of $q$.

Finally, given  a path $p$ in $\G $ labelled by a word in an alphabet $H_\lambda  $ for some $\lambda \in \Lambda $, we define $$\widehat\ell (p)=\dl (1, \Lab (p)).$$  We stress that $\widehat\ell$ not only depends on the endpoints of $p$, but also on $\lambda$. Indeed it can happen that two vertices $x,y\in G$ can be connected by paths $p$ and $q$ labelled by words in alphabets $H_\lambda  $ and $H_\mu  $ for some $\mu \ne \lambda $. In this case $\widehat \ell (p) =\dl (1, x^{-1}y)$ and $\widehat \ell (q)=\widehat\d_\mu (1, x^{-1}y)$ may be non-equal. Also note that $\widehat \ell$ is undefined for paths whose labels involve letters from more then one $H_\lambda $ or from $X$. In our paper $\widehat \ell$ will be used to ``measure" components of paths in $\G $, in which case it is always well-defined (but may be infinite).
\end{defn}

The lemma below follows immediately from Definitions \ref{maindef} and \ref{defcomp}.

\begin{lem}\label{isocomp}
Let $p$ be an isolated $H_\lambda $-component of a cycle of length $C$ in $\G $. Then $\widehat\ell (p)\le C$.
\end{lem}

\begin{defn}[{\bf Bounded and reduced presentations}]\label{i-brp}
A relative presentation
\begin{equation}\label{relpres}
\langle X, \mathcal H \mid \mathcal S\cup \mathcal R\rangle
\end{equation}
is said to be {\it bounded} if relators from $\mathcal R$ have uniformly bounded length, i.e., $\sup \{ \| R\| \mid R\in \mathcal R\} <\infty $. Further the presentation is called {\it reduced} if for every $R\in \mathcal R$ and some (equivalently any) cycle $p$ in $\G $ labeled by $R$, all components of $p$ are isolated and have length $1$ (i.e., consist of a single edge).
\end{defn}

\begin{rem}\label{rem1}
Note that whenever (\ref{relpres}) is reduced, for any letter $h\in H_\lambda  $ appearing in a word from $\mathcal R$, we have  $\dl (1, h)\le \| R\|$ by Lemma \ref{isocomp}. In particular, if (\ref{relpres}) is bounded, there is a uniform bound on $\dl (1, h)$ for such $h$.
\end{rem}

\begin{lem}\label{pres}
Suppose that a group $G$ is weakly hyperbolic relative to a collection of subgroups $\Hl $ and a subset $X$. Then there exists a bounded reduced relative presentation of $G$ with respect to $\Hl $ and $X$ with linear relative isoperimetric function.

Conversely, suppose that there exists a bounded relative presentation of $G$ with respect to $\Hl $ and $X$ with linear relative isoperimetric function. Then $G$ is weakly hyperbolic relative to the collection of subgroups $\Hl $ and the subset $X$.
\end{lem}

\begin{proof}
Let us call a word $W$ in the alphabet $X\sqcup \mathcal H$ a {\it relator} if $W$ represents the identity in $G$. Further we call $W$ {\it atomic} if the following conditions hold: a) for some (hence any) cycle $p$ in $\G $ labeled by $W$, all components of $p$ are isolated and have length $1$ (i.e., consist of a single edge); b) $W$ is not a single letter from $\mathcal H$.

Let $\mathcal R^\prime $ (respectively, $\mathcal R$) consist of all relators (respectively, atomic relators) that have length at most $16\delta $, where $\delta $ is the hyperbolicity constant of $\G $.

Let us first show that for every integer $n$, there exists a constant $C_n>0$ such that for every word $W\in \mathcal R^\prime $ of length $\| W\| \le n$, there is a van Kampen diagram $\Delta $ with boundary label $W$ and $Area ^{rel}(\Delta )\le C_n$ over the presentation
\begin{equation}\label{rp49}
\langle X, \mathcal H \mid \mathcal S\cup \mathcal R\rangle ,
\end{equation}
where $\mathcal S=\bigcup_{\lambda\in \Lambda} \mathcal S_\lambda  $ as in Section \ref{sec-GPIF}. We proceed by induction on $n$. If $\| W\| =1$, then $W$ is either atomic or consists of a single letter from $\mathcal H$. Thus we can tale $C_1=1$. Suppose now that $\| W\| =n>1$ and $W$ is not atomic. Let $p$ be a cycle in $\G$ labeled by $W$. There are two possibilities to consider.

\begin{figure}
  \centering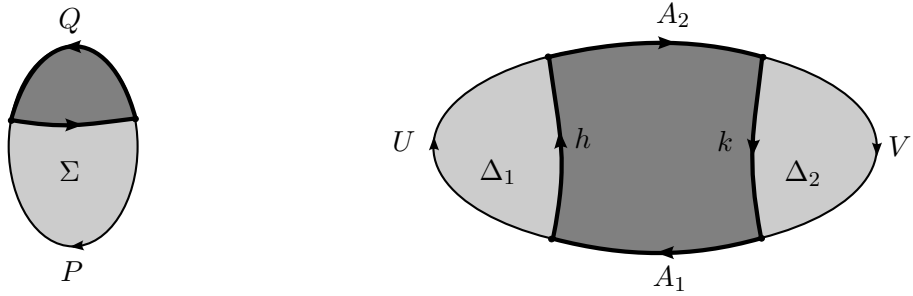\\
  \caption{Two cases in the proof of Lemma \ref{pres}.}\label{non-atomic}
\end{figure}

First assume that some $H_\lambda $--component $q$ of $p$ has length greater than $1$. Up to a cyclic permutation, we have $W\equiv AQ$, where $Q=\Lab (q)$. Let $h\in H_\lambda $ be the element represented by $Q$. Then $Ah$ is a relator of lengths at most $\| W\| -1$ and by the inductive assumption there is a van Kampen diagram $\Sigma $ over (\ref{rp49}) with boundary label $Ah$ and area at most $C_{n-1}$. Gluing this diagram and the $\mathcal S$-cell with boundary label $h^{-1}Q$ in the obvious way (Fig. \ref{non-atomic}), we obtain a van Kampen diagram over (\ref{rp49}) with boundary label $W$ and area at most $C_{n-1}+1$.

Now assume that the cycle $p$ decomposes as $a_1ua_2v$, where $a_1, a_2$ are connected $H_\lambda $ components for some $\lambda \in \Lambda $. Let $A_1UA_2V$ be the corresponding decomposition of $W$. Since the components $a_1$ and $a_2$ are connected to each other, $U$ and $V$ represent some elements $h$ and $k$ of $H_\lambda $, respectively. Note that $A_1hA_2k\in \mathcal S_\lambda $. Further the words $h^{-1}U$, $k^{-1}V$ represent $1$ in $G$ and have lengths smaller than $\| W\| $. By the inductive assumption there are disc van Kampen diagrams $\Delta _1$ and $\Delta _2$ over (\ref{rp49}) with boundary labels $h^{-1}U$ and $k^{-1}V$, respectively, and areas at most $C_{n-1}$. Gluing these diagrams and the $\mathcal S$-cell labeled $A_1hA_2k$ in the obvious way (Fig. \ref{non-atomic}), we obtain a diagram over (\ref{rp49}) with boundary label $W$ and area at most $2C_{n-1}+1$.  Thus we can set $C_n=2C_{n-1}+1$.

Recall that any $\delta $--hyperbolic graph
endowed with the combinatorial metric becomes $1$--connected after gluing $2$--cells along all combinatorial loops of length at most
$16\delta$ and moreover the combinatorial isoperimetric function of the resulting $2$-complex is linear (see \cite[Ch. III.H, Lemma 2.6]{BH_metric} for details). In our settings this means that the presentation
\begin{equation}\label{rp01}
\langle X, \mathcal H\mid \mathcal R^\prime \rangle
\end{equation}
represents the group $G$ and (\ref{rp01}) has a linear isoperimetric function $f(n)=An$ for some constant $A$. According to the previous paragraph every diagram over (\ref{rp01}) can be converted to a diagram over (\ref{rp49}) by replacing every cell with a van Kampen diagram over (\ref{rp49}) having the same boundary label and at most $C_{16\delta}$ cells.  Thus (\ref{rp49}) represents $G$ and $C_{16\delta }An$ is a relative isoperimetric function of (\ref{rp49}). Clearly (\ref{rp49}) is bounded and reduced.

To prove the converse, let (\ref{rp49}) be a relative presentation of $G$ with respect to $\Hl $ and $X$ with  relative isoperimetric function $f(n)=Cn$. Obviously we also have
\begin{equation}\label{rp49-1}
G=\langle X, \mathcal H \mid \mathcal S^\prime\cup \mathcal R\rangle ,
\end{equation}
where $\mathcal S^\prime =\bigcup_{\lambda\in \Lambda} \mathcal S_\lambda^\prime  $ and  $\mathcal S_\lambda ^\prime$ consists of all words of length $\le 3$ in the alphabet $H_\lambda  $ representing $1$ in $H_\lambda$. The idea is to show that the (non-relative) isoperimetric function of (\ref{rp49-1}) is linear. Since the length of relators in (\ref{rp49-1}) is uniformly bounded, the combinatorial isoperimetric function of $\G$ is also linear by Remark \ref{rem-mu}. Hence $\G $ is hyperbolic (for the definition of an isoperimetric function of a general space and its relation to hyperbolicity see Sec. 2 of Ch. III.H  and specifically Theorem 2.9 in \cite{BH_metric}).

Let $W$ be a word in $X\sqcup \mathcal H$ representing $1$ in $G$. Suppose that $\| W\|=n$. Let $\Delta $ be a van Kampen diagram with boundary label $W$ over (\ref{rp49}) such that a) $\Delta $ has at most $Cn$ $\mathcal R$-cells; and b) $\Delta $ has minimal number of $\mathcal S$-cells among all diagrams satisfying a). In particular, b) implies that no two $\mathcal S$-cells can have a common boundary edge as otherwise we could replace these $\mathcal S$-cells with a single one. Hence every boundary edge of every $\mathcal S$-cell either belongs to $\partial \Delta $ or to a boundary of an $\mathcal R$-cell. Thus the total length of boundaries of all $\mathcal S$-cells in $\Delta $ is at most $(CM+1)n$, where $M=\max \{ \| R\| : R\in \mathcal R\}$. Triangulating every $\mathcal S$-cell of $\Delta $ in the obvious way, we obtain a van Kampen diagram $\Delta ^\prime $ over (\ref{rp49-1}) with less than $(CM+C+1)n$ cells. Hence the (non-relative) isoperimetric of (\ref{rp49-1}) is linear and we are done.
\end{proof}

Given a subset $Y\le G $ and a path $p$ in a Cayley graph $\G$, let $\ell _Y (p)$ \label{i-lY} denote the word length of the element represented by $\Lab (p)$ with respect to $Y$. Recall that $\ell _Y(p)=\infty $ if $\Lab (p)\notin \langle Y\rangle $.

\begin{lem}\label{Ylambda}
Let
\begin{equation}\label{rp1}
\langle X, \mathcal H  \mid  \mathcal S\cup \mathcal R\rangle
\end{equation}
be a bounded presentation of a group $G$ with respect to a collection of subgroups $\Hl $. Let $Y_\lambda$ be the set of all letters from $H_\lambda   $ that appear in words from $\mathcal R$. Suppose that (\ref{rp1}) has relative isoperimetric function $f(n)$. Then for every  cycle $q$ in $\Gamma (G, X\sqcup H)$ and every set of isolated components $p_1, \ldots , p_n$ of $q$, where $p_i$ is an $H_{\lambda _i}$-component, we have
\begin{equation}\label{sumoflengths}
\sum\limits_{i=1}^n \ell_{Y_{\lambda _i}} (p_i) \le M f(\ell (q)),
\end{equation}
where
\begin{equation}\label{M}
M=\max\limits_{R\in \mathcal R} \| R\| .
\end{equation}
\end{lem}

\begin{proof}
Consider a van Kampen diagram $\Delta $ over (\ref{rp1}) whose boundary label is $\Lab (q)$. In what follows
we identify $\partial \Delta $ with $q$. Assume that $q=p_1r_1\cdots p_nr_n$. For $i=1,
\ldots , n$, let $\mathcal D_i$ denote the set of all subdiagrams
of $\Delta $ bounded by $p_i(p^{\prime }_i)^{-1}$, where
$p^{\prime }_i$ is a path in $\Delta $ without self intersections such that
$(p^{\prime }_i)_-=(p_i)_-$, $(p^{\prime }_i)_+=(p_i)_+$, and
$\Lab (p^{\prime }_i)$ is a word in the alphabet $H_{\lambda
_i}$. We choose a subdiagram $\Sigma _i \in
\mathcal D_i$ that has maximal number of cells among all
subdiagrams from $\mathcal D_i$ (see Fig. \ref{fig4}).

\begin{figure}
  \centering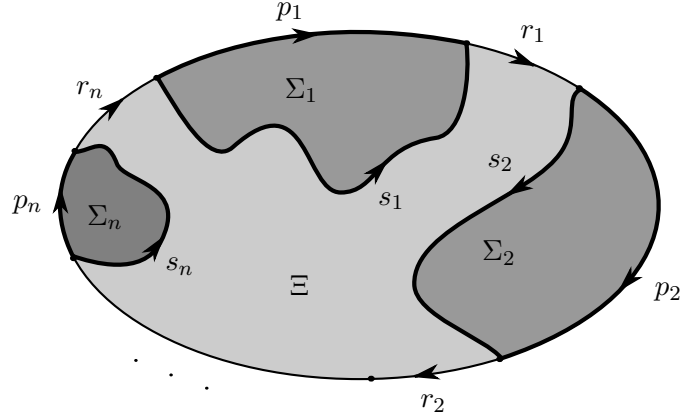\\
  \caption{Decomposition of $\Delta $.}\label{fig4}
\end{figure}

Let $\partial \Sigma _i= p_is_i^{-1}$. Since $p_i$ is an isolated
component of $q$, the path $s_i$ has no common edges with $r_i$,
$i=1, \ldots k$, and the sets of edges of $s_i$ and $s_j$ are
disjoint whenever $j\ne i$. Therefore each edge $e$ of $s_i$ belongs
to the boundary of some cell $\Pi $ of the subdiagram $\Xi $ of
$\Delta $ bounded by $s_1r_1\cdots s_kr_k $.

If $\Pi $ is an $\mathcal S$--cell, then $\Lab (\Pi )$ is a word in the alphabet $H_{\lambda
_i} $. Hence by joining $\Pi $ to $\Sigma _i$ we get
a subdiagram $\Sigma _i^\prime \in \mathcal D_i$ with bigger number
of cells that contradicts the choice of $\Sigma _i $. Thus each edge
of $s_i$ belongs to the boundary of an $\mathcal R$--cell.

The total (combinatorial) length of $s_i$'s does not exceed the number of $\mathcal R$--cells in $\Xi $
times the maximal number of edges in boundary of an $\mathcal
R$--cell. Therefore,
\begin{equation}\label{Lomega}
\sum\limits_{i=1}^k \ell _{Y_{\lambda _i}} (p_i)= \sum\limits_{i=1}^k
\ell _{Y_{\lambda _i}}(s_i) \le M Area ^{rel} (\Lab(\partial \Delta ))\le
Mf(\ell(q)).
\end{equation}
\end{proof}

We extend the definition of bi-Lipschitz equivalence \label{i-bL2} to metrics with possibly infinite values as follows. Two metrics $\d_1, \d_2\colon H\to [0, +\infty]$ on a set $H$ are \emph{bi-Lipschitz equivalent} (we write $\d_1 \sim_{Lip} \d_2$), if there is a constant $C$ such that for any $h_1,h_2\in H$, $\d_1 (h_1,h_2)$ is finite if and only if $\d_2(h_1,h_2)$ is, and if both ratios are finite we have $\d_1(h_1,h_2)/\d_2(h_1,h_2)<C$ and $\d_2(h_1,h_2)/\d_1(h_1,h_2)<C$.

Recall that given a path $p$ in $\G $, $\widehat\ell (p)$ is only defined if $p$ is labelled by elements of some $H_\lambda $ and equals $\dl (1, \Lab (p))$ in this case. For weakly relatively hyperbolic groups we obtain the following.

\begin{lem}\label{omega}
Suppose that $G$ is weakly hyperbolic relative to $X$ and $\Hl $. Then the following hold.
\begin{enumerate}
\item[(a)] There exists a constant $L$ such that for every  cycle $q$ in $\G$ and every set of isolated components $p_1, \ldots , p_n$ of $q$, we have
$$
\sum\limits_{i=1}^n \widehat\ell (p_i) \le L \ell (q).
$$
\item[(b)] For every $\lambda \in \Lambda $, there exists a subset $Y_\lambda \subseteq H_\lambda $ such that $\dol \sim_{Lip} \dl$. More precisely, if (\ref{rp1}) is a reduced bounded relative presentation
of $G$ with respect to $X$ and $\Hl $ with linear relative isoperimetric function, then one can take  $Y_\lambda $ to be the set of all letters from $H_\lambda   $ that appear in words from $\mathcal R$.
\end{enumerate}
\end{lem}

\begin{proof}
By Lemma \ref{pres} there exists a reduced bounded relative presentation of $G$ with respect to $X$ and $\Hl $ with linear relative isoperimetric function. Let (\ref{rp1}) be any such a presentation. Note that since (\ref{rp1}) is reduced, we have
\begin{equation}\label{lly}
\dl (1,y)\le M
\end{equation}
for every $y\in Y_\lambda$ by Lemma \ref{isocomp}, where $M$ is defined by (\ref{M}). This and the inequality (\ref{sumoflengths}) implies (a).

To prove (b), take any $h\in H_\lambda$. Notice that $\dl (1,h)\le M |h|_{Y_\lambda }$ by (\ref{lly}). It remains to prove the converse inequality. In case $\dl (1,h) =\infty$ we obviously have $|h|_{Y_\lambda }\le \dl (1,h)$. Suppose now that $\dl (1,h)=n<\infty $. Let $p$ be a path in $\G $ of length $n$ such that $p_-=1$, $p_+=h$, and $p$ contains no edges of $\Gamma _{H_\lambda }$. Let $e$ be the edge of $\G $ connecting $1$ to $h$ and labeled by $h$. Then $e$ is an isolated $H_\lambda $-component of the cycle $ep^{-1}$ and by part (a) we have $$|h|_{Y_\lambda }=\ell _{Y_\lambda } (e)\le L \ell (q)=L (n+1)\le 2Ln=2L\dl (1,h).$$ Thus $\dl $ and $\dol$ are Lipschitz equivalent.
\end{proof}

In many cases the subsets $Y_\lambda $ can be described explicitly. Here are some elementary examples. Note that in these cases changing the relative presentation significantly affects the corresponding relative metric (cf. Remark \ref{dl-not-eq}).

\begin{ex}\label{ex: HNN}
\begin{enumerate}
\item[(a)] Let $G=H_1\ast _{A=B} H_2$ be the amalgamated product of groups $H_1,H_2$ corresponding to an isomorphism $\iota\colon A\to B$ between subgroups $A\le H_1$ and $B\le H_2$. Then $G$ is weakly hyperbolic relative to $\{ H_1, H_2\}$ and $X=\emptyset$. Indeed it is easy to verify that $\G $ is quasi-isometric to the Bass-Serre tree of $G$ (see, e.g., \cite{Osi04}). The natural relative presentation
    \begin{equation}\label{amprod}
    G=\langle H_1, H_2\mid a=\iota (a), \, a\in A\rangle
    \end{equation}
     is obviously bounded. Moreover it easily follows from the normal form theorem for amalgamated products \cite[Ch. IV, Theorem 2.6]{LS} that (\ref{amprod}) has linear relative isoperimetric function. The definition of $Y_\lambda $ from Lemma \ref{Ylambda} gives $Y_1=A$, $Y_2=B$ in this case. Hence by part (b) of Lemma \ref{omega}, for the corresponding relative metrics on $H_1$ and $H_2$ we have $\widehat\d _1\sim _{Lip} \d _A$ and $\widehat\d _2\sim _{Lip} \d _B$, where $d_A$ and $d_B$ are the word metrics on $H_1$ and $H_2$ with respect to the subsets $A$ and $B$, respectively (note that these metrics only take values in $0,1,\infty$).

\item[(b)] Similarly if $G$ is an HNN-extension of a group $H$ with associated subgroups $A,B\le H$, then $G$ is weakly hyperbolic relative to $H$ and $X=\{ t\}$, where $t$ is the stable letter. The corresponding relative metric on $H$ is bi-Lipschitz equivalent to the word metric with respect to the set $A\cup B$.

\item[(c)] More generally, it is not hard to show that for every finite graph of groups $\mathcal G$, its fundamental group $\pi_1(\mathcal G)$ is weakly hyperbolic relative to the collection of vertex groups and the subset $X$ consisting of stable letters (i.e., generators corresponding to edges of $\mathcal G\setminus T$, where $T$ is a spanning subtree of $\mathcal G$). The corresponding relative metric on a vertex group $H_v$ corresponding to a vertex $v$ will be bi-Lipschitz equivalent to the word metric with respect to the union of the edge subgroups of $H_v$ corresponding to edges incident to $v$. The proof is essentially the same as above. We leave this as an exercise for the reader. For details about fundamental groups of graphs of groups, their presentations, and the normal form theorem we refer to \cite{Ser}.
\end{enumerate}
\end{ex}

\subsection{Isolated components in geodesic polygons}

Throughout this section let $G$ be a group, $\Hl $ a collection of subgroups of $G$, $X$ a subset of $G$.
Our next goal is to generalize some useful results about quasi-geodesic polygons in Cayley graphs of relatively hyperbolic groups proved in \cite{Osi07}. We start with a definition which is an analogue of \cite[Definition 3.1]{Osi07}.

\begin{defn}
For $\mu \ge 1$, $c\ge 0$, and $n\ge 2$, let $\mathcal
Q_{\mu , c}(n)$ denote the set of all pairs $(\P ,\, I)$,
where $\P =p_1\ldots p_n$ is an $n$--gon in $\G $ and $I$ is a
distinguished subset of the set of sides $\{ p_1, \ldots, p_n\} $
of $\P $ such that:
\begin{enumerate}
\item Each side $p_i\in I$ is an isolated $H_{\lambda_i}$-component of $\P $ for some $\lambda_i\in \Lambda$.

\item Each side $p_i\notin I$ is $(\mu ,c)$--quasi--geodesic.
\end{enumerate}
For technical reason, it is convenient to allow some of the sides
$p_1, \ldots , p_n$ to be trivial. Thus we have $\mathcal
Q_{\mu , c}(2)\subseteq\mathcal Q_{\mu , c}(3)\subseteq
\ldots $.
Given $(\P , I)\in \mathcal Q_{\mu , c}(n)$, we set
$$s(\P , I)=\sum\limits_{p_i\in I} \widehat \ell (p_i)=\sum\limits_{p_i\in I} \widehat d_{\lambda_i} (1, \Lab(p_i))$$ and $$ s_{\mu , c}(n)=\sup\limits_{(\P , I) \in
\mathcal Q_{\mu , c}(n)} s(\P , I).$$\label{i-smcn}
\end{defn}

A priori, it is not even clear whether the quantity $s_{\mu , c}(n)$ is finite for fixed values of $n$, $\mu $, and $c$. However
a much stronger result holds. It is the analogue of Proposition 3.2 from \cite{Osi07}.

\begin{prop}\label{sn}
Suppose that $G $ is weakly hyperbolic relative to $X$ and $\Hl $. Then for any $\mu \ge 1$, $c\ge 0$, there exists a constant $D=D(\mu , c)>0$ such that $s_{\mu , c}(n)\le Dn $ for any $n\in \mathbb N$.
\end{prop}

The proof of this proposition repeats the proof of its relatively hyperbolic analogue,  Proposition 3.2 in Section 3 of \cite{Osi07},  almost verbatim after few changes in notation and terminology. In fact, the key tool in the proof of Proposition 3.2 in \cite{Osi07} was Lemma 2.7 from the same paper, which has a direct analogue, namely Lemma \ref{omega}, in our situation. Apart from this lemma, the proof in \cite{Osi07} only uses general facts about hyperbolic spaces, so all arguments remain valid. Since Proposition \ref{sn} plays a central role in our paper, we reproduce here the proof for convenience of the reader.

The following obvious observation will often be used without special references. If $q_1$, $q_2$ are two components of some path in $\G $ that are connected, then for any two vertices $u\in
q_1$ and $v\in q_2$, we have $\dxh (u,v)\le 1$. Note also that replacing $p_i$ for each $i\in I$ with a single edge labelled by a letter from the corresponding alphabet $H_\lambda  $ does not change $\widehat \ell (p_i)$. Thus we assume that for each $i\in I$, $p_i$ is a single edge. Below we also use the following notation for vertices of
$\P $:
$$ x_1=(p_n)_+=(p_1)_-, \; x_2=(p_1)_+=(p_2)_-,\; \ldots , \;
x_n=(p_{n-1})_+=(p_{n})_-.$$

The following immediate corollary of Lemma \ref{qg} will be used several times.
\begin{lem}\label{rect}
For any $\delta \ge 0$, $\mu \ge 1$, $c\ge 0$, there exists a
constant $\theta =\theta(\delta , \mu , c)\ge 0$ with the
following property. Let $Q$ be a quadrangle in a
$\delta$--hyperbolic space whose sides are $(\mu ,
c)$--quasi--geodesic. Then each side of $Q$ belongs to the closed
$\theta $--neighborhood of the union of the other three sides.
\end{lem}
\begin{proof}
Obviously $\theta =\kappa (\mu , c) + 2\delta$, where $\kappa (\mu , c)$ is the constant provided by Lemma \ref{qg}, works.
\end{proof}

From now on, we fix $\mu $ and $c$.  Without loss of generality we may
assume $\theta = \theta(\delta , \mu , c)$ to be a positive integer.
The proof of Proposition \ref{sn} is by induction on $n$. We
begin with the case $n\le 4$.

\begin{lem}\label{l1}
For any $\mu \ge 1$, $c\ge 0$, and $n\le 4$,  $s_{\mu ,
c}(n)$ is finite.
\end{lem}

\begin{proof}
Suppose that $(\P ,\, I)\in \mathcal Q_{\mu , c}(4)$, $\P
=p_1p_2p_3p_4$.  According to Lemma
\ref{omega}, it suffices to show that for each $p_i\in I$, there
is a cycle $c_i$ in $\G $ of length at most $K$, where $K$ is a constant which depends on $\mu $, $c$, and the hyperbolicity
constant $\delta $ of the graph $\G $ only, such that $p_i$ is an isolated component of
$c_i$. We will show that $$K=100(\mu
\theta +c+\theta )$$ works. There are $4$ cases to consider.

\noindent {\bf Case 1.}  Suppose $\sharp\, I=4$. Then the
assertion of the lemma is obvious. Indeed $\ell(\P )=4<K$ as each
$p_i\in I$ has lengths $1$, and we can set $c_i=\P $ for all $i$.

\noindent {\bf Case 2.}  Suppose $\sharp\, I=3$, say $I=\{
p_1,p_2,p_3\} $. Since $p_4$ is $(\mu, c)$--quasi--geodesic,
we have
$$\ell(p_4)\le \mu \dxh (x_4, x_1)+c\le 3\mu +c$$ by
the triangle inequality. Hence $\ell(\P )\le 3\mu +c+3< K$ and we
can set $c_i=\P $ again.

\medskip

\begin{figure}
\hspace{31mm} \includegraphics{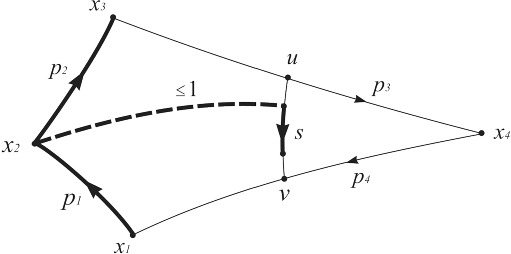} \\
\hspace*{36mm} \includegraphics{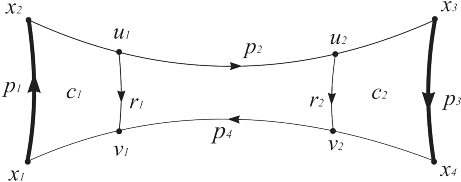}\\
\vspace*{-8mm} \caption{Cases 3 a) and b)}\label{7-f1}
\end{figure}

\noindent {\bf Case 3.}  Assume now that $\sharp\, I=2$. Up to
renumbering
the sides, there are two possibilities to consider.

\smallskip

a) First suppose $I=\{ p_1,p_2\} $. If $\dxh (x_3, x_4)< \theta
+2$, we have
$$
\ell(p_3)\le \mu \dxh (x_3, x_4) +c<\mu (\theta +2)+c,
$$
$$
\begin{array}{cl}
\ell(p_4)\le & \mu \dxh (x_4, x_1) +c \le \\ &  \mu (\dxh
(x_1, x_2) +\dxh (x_2, x_3) +\dxh (x_3, x_4))+ c< \\ & \mu (1
+1 + \theta +2) +c \le \mu (\theta +4)+c,
\end{array}
$$
and hence $$\ell(\P )<2+ \ell(p_3)+\ell(p_3) <\mu (2\theta +6)+2c+2<
K.$$ Thus we may assume $\dxh (x_3, x_4)\ge \theta +2$. Let $u$ be
a vertex on $p_3$ such that $\dxh (x_3, u)=\theta +2$. By Lemma
\ref{rect} there exists a vertex $v\in p_1\cup p_2\cup p_4 $ such
that $\dxh (u,v)\le \theta $. Note that, if fact, $v\in p_4$.
Indeed otherwise $v=x_2$ or $v=x_3$ and we have
$$\dxh (x_3, u)\le \dxh (x_3, v)+\dxh (u, v)\le 1+\theta $$ that
contradicts the choice of $u$.

Let $r$ be a geodesic path in $\G $ connecting $u$ to $v$. We wish
to show that no component of $r$ is connected to $p_1$ or $p_2$.
Indeed suppose that a component $s$ of $r$ is connected to $p_1$
or $p_2$ (Fig.\ref{7-f1}). Then $\dxh (x_2, s_-)\le 1$ and we
obtain
$$
\begin{array}{rl}
\dxh (u, x_3)\le & \dxh (u, s_-)+\dxh (s_-, x_2)+ \dxh (x_2,
x_3)\le \\ & (\theta -1)+1 +1=\theta +1.
\end{array}
$$
This contradicts the choice of $u$ again. Note also that $p_1$,
$p_2$ can not be connected to a component of $p_3$ or $p_4$ as
$p_1$, $p_2$ are isolated components in $\P $. Therefore $p_1$ and
$p_2$ are isolated components of the cycle
$$c=p_1p_2[x_3, u]r[v, x_1],$$ where $[x_3, u]$ and $[v, x_1]$ are
segments of $p_3$ and $p_4$ respectively. Using the triangle
inequality, it is easy to check that $\ell([v, x_1]) \le
\mu(2\theta +4)$ and $\ell(c)\le \mu (3\theta +6)+2c+\theta
+2< K$.

\smallskip

 b) Let $I=\{ p_1,p_3\}  $. If $\dxh (x_2, x_3)< 2\theta +2$,
we obtain $\ell(\P )< K$ arguing as in the previous case. Now assume
that $\dxh (x_2, x_3)\ge 2\theta +2$. Let $u_1$ (respectively
$u_2$) be the vertex on $p_2$ such that $\dxh (x_2, u_1)=\theta
+1$ (respectively $\dxh (x_3, u_2)=\theta +1$). By Lemma
\ref{rect} there exist vertices $v_1, v_2$ on $p_1\cup p_3\cup
p_4$ such that $\dxh (v_i,u_i)\le \theta $, $i=1,2$. In fact,
$v_1, v_2$ belong to $p_4$ (Fig.\ref{7-f1}). Indeed the reader can
easily check that the assumption $v_1=x_2$ (respectively
$v_1=x_3$) leads to the inequality $\dxh (x_2, u_1)\le \theta $
(respectively $\dxh (x_2, x_3)\le 2\theta +1$). In both cases we
get a contradiction. Hence $v_1\in p_4$ and similarly $v_2\in
p_4$.

Let $r_i$, $i=1,2$, be a geodesic path in $\G $ connecting $u_i$
to $v_i$. We set
$$c_1=p_1[x_2,u_1]r_1[v_1, x_1]$$ and
$$c_3=p_3[x_4, v_2]r_2^{-1}[u_2, x_3].$$
Arguing as in Case 3a) we can easily show that $p_i$ is an
isolated component of $c_i$ and $\ell(c_i)<K$ for $i=1,2$.

\medskip

\begin{figure}
   \hspace{12mm} \includegraphics{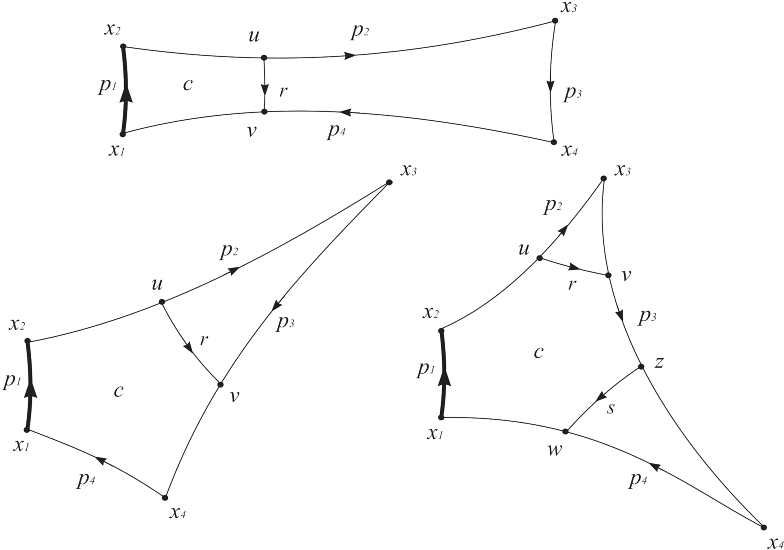}\\
  \vspace*{-10mm}
  \caption{Cases 4 a),  b1), and b2). }\label{7-f3}
\end{figure}

\noindent {\bf Case 4.}  Finally assume $\sharp I=1$. To be
definite, let $I=\{ p_1\} $. If $\dxh(x_2,x_3)< \theta +1$ and
$\dxh(x_4,x_1)< \theta +1$, we obtain $\ell(\P )< K$ as in the
previous cases. Thus, changing the enumeration of the sides if
necessary, we may assume that $\dxh(x_2,x_3)\ge \theta +1$. Let
$u$ be a point on $p_2$ such that $\dxh (x_2, u)=\theta +1$, $v$ a
point on $p_1\cup p_3\cup p_4$ such that $\dxh (u,v)\le \theta $,
$r$ a geodesic path in $\G $ connecting $u$ to $v$. As above it is
easy to show that $v\in p_3\cup p_4$. Let us consider two
possibilities (see Fig. \ref{7-f3}).

\smallskip

a) $v\in p_4$. Using the same arguments as in Cases 2 and 3 the
reader can easily prove that $p_1$ is an isolated component of the
cycle
\begin{equation}\label{c4a}
c=p_1[x_2,u]r[v, x_1].
\end{equation}
It is easy to show that $\ell(c)<K$.

\smallskip

b) $v\in p_3$. Here there are 2 cases again.

\smallskip

b1) If $\dxh (x_1,x_4)< \theta +1$, then we set
$$
c=p_1[x_2,u]r[v,x_4]p_4.
$$
The standard arguments show that $\ell(c)<K$ and $p_1$ is isolated in
$c$.

\smallskip

b2) $\dxh (x_1,x_4)\ge \theta  +1$.  Let $w$ be a vertex on $p_4$
such that $\dxh (x_1, w)=\theta +1$, $z$ a vertex on $p_1\cup
p_2\cup p_3$ such that $\dxh (z,w)\le \theta $. Again, in fact,
our assumptions imply that $z\in p_2\cup p_3$. If $z\in p_2$, the
lemma can be proved by repeating the arguments from the case 4a)
(after changing enumeration of the sides). If $z\in p_3$, we set
$$c=p_1[x_2,u]r[v,z]s[w,x_1],$$ where $s$ is a geodesic in $\G $
connecting $z$ to $w$. It is straightforward to check that $p_1$
is an isolated component of $c$ and $\ell(c)<K$. We leave details to
the reader.

\end{proof}

\begin{lem}\label{l2}
For any $n\ge 4 $, we have
\begin{equation}\label{ge5}
s_{\mu , c}(n)\le n (s_{\mu , c}(n-1)+s_{\mu , c}(4)).
\end{equation}
\end{lem}

\begin{proof}
We proceed by induction on $n$. The case $n=4$ is obvious, so we
assume that $n\ge 5$. Let $(\P , I)\in \mathcal Q_{\mu ,
c}(n)$, $p_i\in I$, and let $q$ be a geodesic in $\G $ connecting
$x_i$ to $x_{i+3}$ (indices are taken $mod\, n$). If $p_i$ is
isolated in the cycle $p_ip_{i+1}p_{i+2}q^{-1}$, we have $\widehat \ell
(p_i)\le s_{\mu , c}(4)$. Assume now that the component $p_i$
is not isolated in the cycle $p_ip_{i+1}p_{i+2}q^{-1}$. As $p_i$
is isolated in $\P $, this means that $p_i$ is connected to a
component $s$ of $q$. Hence $\dxh (x_i, s_+)\le 1$. Since $q$ is
geodesic in $\G $, this implies $s_-=x_i$ (see Fig. \ref{7-f4}).

Let $q=ss^\prime $ and let $e$ denote an edge in $\G $ such that $e_-=x_{i+1}$, $e_+=s_+$, and $\phi (e)$ is
a word in $\mathcal H$. We note that $e$ is
an isolated component of the cycle
$r=p_{i+1}p_{i+2}(s^{\prime}e)^{-1}$. Indeed if $e$ is connected to
a component of $p_{i+1}$ or $p_{i+2}$, then $p_i$ is not isolated
in $p$, and if $e$ is connected to a component of $s^\prime $,
then $q$ is not geodesic.  Similarly $s$ is an isolated component
of $p_{i+3}\ldots p_{i-1} ss^\prime $. Hence $\widehat \ell (s)\le
s_{\mu , c}(n-1)$ by the inductive assumption and $\widehat \ell
(e)\le s_{\mu , c}(4)$. Therefore we have $\widehat \ell (p_i)\le
s_{\mu , c}(4)+s_{\mu , c}(n-1)$. Repeating these
arguments for all $p_i\in I$, we get (\ref{ge5}).
\end{proof}

\begin{figure}
  \hspace{37mm}\includegraphics{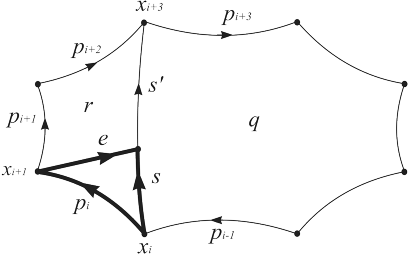}
  \caption{}\label{7-f4}
\end{figure}

\begin{cor}
$s_{\mu , c}(n)$ is finite for any $n$.
\end{cor}

The proof of the next lemma is a calculus exercise. We do not copy it and refer the reader to \cite[Lemma 3.6]{Osi07}.

\begin{lem}\label{linfunct}
Let $f\colon \mathbb N\to \mathbb N$. Suppose that there exist
constants $C, N>0$, and $\alpha \in (0,1)$ such that for any $n\in
\mathbb N$, $n>N$, there are $n_1, \ldots , n_k\in \mathbb N$
satisfying the following conditions:

a) $k\le C \ln n$;

b) $f(n)\le \sum\limits_{i=1}^k f(n_i)$;

c) $n\le \sum\limits_{i=1}^k n_i\le n+C \ln n$;

d) $n_i\le \alpha n$  for any $i=1, \ldots , k$.

\noindent Then $f(n)$ is bounded by a linear function from above.
\end{lem}

The next lemma was proved by Olshanskii \cite[Lemma 23]{Ols92} for
geodesic polygons. In \cite{Ols92}, the inequality (\ref{distuv})
had the form $dist (u,v)\le 2\delta ( 2+\log_2 n)$. Passing to
quasi--geodesic polygons we only need to add a constant to the right
hand side according to the above--mentioned property of
quasi--geodesics in hyperbolic spaces.

\begin{lem}\label{P}
For any $\delta \ge 0$, $\mu \ge 1$, $c\ge 0$, there exists a
constant $\eta =\eta (\delta , \mu , c)$ with the following
property. Let $\P =p_1\ldots p_n$ be a $(\mu ,
c)$--quasi--geodesic $n$--gon in a $\delta $--hyperbolic space. Then
there are points $u$ and $v$ on sides of $\mathcal P$ such that
\begin{equation}\label{distuv}
dist (u,v)\le 2\delta ( 2+\log_2 n)+\eta
\end{equation}
and the geodesic segment connecting $u$ to $v$ divides $\P $ into an
$m_1$--gon and $m_2$--gon such that $n/4 < m_i < 3n/4 +2$.
\end{lem}

\begin{figure}
\hspace{-3mm} \includegraphics[width=160mm]{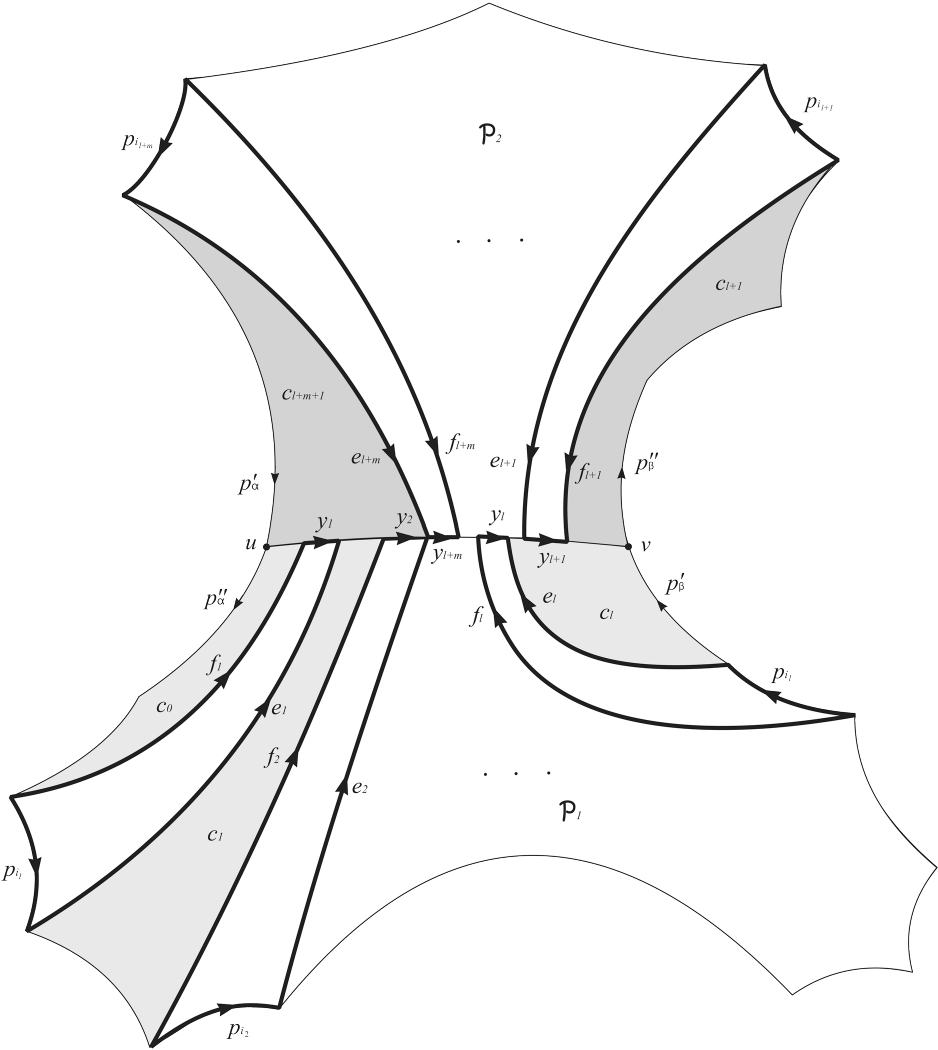}\\
\vspace{3mm}  \caption{Decomposition of the $n$--gon in the proof of
Proposition \ref{sn}} \label{7-f5}
\end{figure}

Now we are ready to prove the main result of this section.

\begin{proof}[Proof of Proposition \ref{sn}]
We are going to show that for any fixed $\mu \ge 1$, $c\ge 0$,
the function $s_{\mu , c}(n)$ satisfies the assumptions of
Lemma \ref{linfunct}. Let $(\P , I)\in \mathcal Q_{\mu ,
c}(n)$, where $\P =p_1\ldots p_n$. As in the proof of Lemma
\ref{l1}, we may assume that every $p_i\in I$ consists of a single
edge. We also assume $n\ge N$, where the constant $N$ is big
enough. The exact value of $N$ will be specified later.

Let $u,v$ be the points on $\P $ provided by Lemma \ref{P}.
Without loss of generality we may assume that $u,v$ are vertices
of $\G $. Further let $t$ denote a geodesic path in $\G $ such
that $t_-=u$, $t_+=v$. According to Lemma \ref{P},
\begin{equation}\label{lr}
\ell(t)\le 2\delta (2+\log _2 n)+\eta ,
\end{equation}
where $\eta $ is a constant depending only on $\delta $,
$\mu $, and $c$, and $t$ divides $\P $ into an $m_1$--gon $\P
_1$ and $m_2$--gon $\P _2$ such that
\begin{equation}\label{ni}
m_i\le 3n/4 +2<n
\end{equation}
for $i=1,2$. To be precise we assume that $u\in p_\alpha $, $v\in
p_\beta $, and $p_\alpha =p_\alpha ^\prime p_\alpha
^{\prime\prime}$,  $p_\beta =p_\beta ^\prime p_\beta
^{\prime\prime}$, where $(p_\alpha ^\prime)_+=(p_\alpha
^{\prime\prime})_-=u$, $(p_\beta ^\prime)_+=(p_\beta
^{\prime\prime})_-=v$. Then $$ \P _1 = p_\alpha ^{\prime\prime }
p_{\alpha +1} \ldots p_{\beta -1}p_\beta ^\prime t^{-1}$$ and $$
\P _2 = p_\beta ^{\prime\prime }p_{\beta +1}\ldots p_{\alpha
-1}p_\alpha ^\prime t.$$ (Here and below the indices are taken
modulo $n$.) Since each $p_i\in I$ consists of a single edge, one
of the paths $p_\alpha ^\prime $, $p_\alpha ^{\prime\prime }$
(respectively $p_\beta ^\prime $, $p_\beta ^{\prime\prime }$) is
trivial whenever $p_\alpha \in I$ (respectively $p_\beta \in I$).
Hence the set $I$ is naturally divided into two disjoint parts
$I=I_1\sqcup I_2$, where $I_i$ is a subset of $I$ consisting of
sides of $\P _i$, $i=1,2$.

Let us consider the polygon $\P _1$ and construct cycles $c_0,
\ldots , c_l$ in $\G $ as follows. If each $p_i\in I_1$ is isolated
in $\P _1$, we set $l=0$ and $c_0=\P _1$. Further suppose this is
not so. Let $p_{i_1}\in I_1$, be the first component (say, an
$H_{\lambda _1}$--component) in the sequence $p_\alpha , p_{\alpha
+1}, \ldots $ such that $p_{i_1}$ is not isolated in $\P _1$. As
$p_{i_1}$ is isolated in $\P $, this means that $p_{i_1}$ is
connected to an $H_{\lambda _1}$--component $y_1$ of $t$. Let $f_1$
(respectively $e_1$) be an edge in $\G $ labelled by an element of
$H_{\lambda _1}$ such that
$(f_1)-=(p_{i_1})_-$, $(f_1)_+=(y_1)_-$ (respectively
$(e_1)_-=(p_{i_1})_+$, $(e_1)_+=(y_1)_+$). We set $$
c_0=p_\alpha^{\prime\prime }p_{\alpha +1}\ldots p_{i_1-1}f_1 [
(y_1)_-,u],$$ where $[(y_1)_-,u]$ is the segment of $t^{-1}$ (see
Fig. \ref{7-f5}).

Now we proceed by induction. Suppose that the cycle $c_{k-1}$ and
the corresponding paths $f_{k-1},e_{k-1},y_{k-1}, p_{i_{k-1}}$ have
already been constructed. If the sequence $p_{i_{k-1} +1},
p_{i_{k-1} +2}, \ldots $ contains no component $p_i\in I_1$ that is
not isolated in $\P _1$, we set $l=k$,
$$ c_k=e_{k-1}^{-1}p_{i_{k-1}+1}\ldots p_{\beta -1}p_\beta ^\prime
[v,(y_{k-1})_+], $$ where $[v,(y_{k-1})_+]$ is the segment of
$t^{-1}$, and finish the procedure. Otherwise we continue as
follows. We denote by $p_{i_k}$ the first component in the
sequence $p_{i_{k-1} +1}, p_{i_{k-1} +2}, \ldots $ such that
$p_{i_k}\in I_1$  and $p_{i_k}$ is connected to some component
$y_k$ of $t$. Then we construct $f_k$, $e_k$ as above and set $$
c_k=e_{k-1}^{-1}p_{i_{k-1}+1}\ldots
p_{i_k-1}f_k[(y_k)_-,(y_{k-1})_+].$$ Observe that each path
$p_i\in I_1$ is either included in the set $J_1=\{ p_{i_1},
\ldots, p_{i_l}\} $ or is an isolated component of some $c_j$.
Indeed a path $p_i\in I_1\setminus J_1$ can not be connected to a
component of $t$ according to our choice of $p_{i_1}, \ldots,
p_{i_l}$. Moreover $p_i\in I_1\setminus J_1$ can not be connected
to some $f_j$ or $e_j$ since otherwise $p_i$ is connected to
$p_{i_j}$ that contradicts the assumption that sides from the set
$I$ are isolated components in $\P $.

By repeating the ``mirror copy" of this algorithm for $\P _2$, we
construct cycles $c_{l+1}, \ldots , c_{l+m+1}$, $m\ge 0$, the set
of components $J_2=\{ p_{i_{l+1}},\ldots , p_{i_{l+m}}\}\subseteq
I_2$, components $y_{l+1}, \ldots , y_{l+m}$ of $t$, and edges $f_{l+1}, e_{l+1}, \ldots,  f_{l+m}, e_{l+m}$ in
$\G $ such that $f_j$ (respectively $e_j$) goes from $(p_{i_j})_-$
to $(y_j)_+$ (respectively from $(p_{i_j})_+$ to $(y_j)_-$)  (see
Fig. \ref{7-f5}) and each path $p_i\in I_2$ is either included in
the set $J_2$ or is an isolated component of $c_j$ for a certain
$j\in \{ l+1, \ldots, l+m+1\} $.

Each of the cycles $c_j$, $0\le j\le l+m+1$, can be regarded as a
geodesic $n_j$--gon whose set of sides consists of paths of the
following five types (up to orientation):

\begin{enumerate}
\item[(1)] Components from the set $I\setminus (J_1\cup J_2)$.

\item[(2)] Sides of $\P _1$ and $\P _2$ that do not belong to the
set $I$.

\item[(3)] Paths $f_j$ and $e_j$, $1\le j\le l+m$.

\item[(4)] Components $y_1 , \ldots , y_{l+m}$ of $t$.

\item[(5)] Maximal subpaths of $t$ lying ``between" $y_1, \ldots ,
y_{l+m}$, i.e. those maximal subpaths of $t$ that have no common
edges with $y_1, \ldots , y_{l+m}$.
\end{enumerate}

It is straightforward to check that for a given $0\le j\le l+m+1$,
all sides of $c_j$ of type (1), (3), and (4) are isolated
components of $c_j$. Indeed we have already explained that sides
of type (1) are isolated in $c_j$. Further, if $f_j$ or $e_j$ is
connected to $f_k$, $e_k$, or $y_k$ for $k\ne j$, then $p_{i_j}$
is connected to $p_{i_k}$ and we get a contradiction. For the same
reason $f_j$ or $e_j$ can not be connected to a component of a
side of type (2). If $f_j$ or $e_j$ is connected to a component
$x$ of a side of type (5), i.e., to a component of $t$, then $y_j$
is connected to $x$. This contradicts the assumption that $t$ is
geodesic. Finally $y_j$ can not be connected to a component of a
side of type (2) since otherwise $p_{i_j}$ is not isolated in $\P
$, and $y_j$ can not be connected to another component of $t$ as
notified in the previous sentence.

Observe that (\ref{lr}) and (\ref{ni}) imply the following
estimate of the number of sides of $c_j$:
$$
n_j\le \max\{ m_1, m_2\} + \ell(t)\le 3n/4 +2 +2\delta (\log _2 n +2)
+\eta .
$$
Assume that $N$ is a constant such that $3n/4 +2 +2\delta (\log _2
n +2)+\eta \le 4n/5 $ for all $n\ge N$. Then for any $n\ge N$,
we obtain the following.
$$
\sum\limits_{p_i\in I} \widehat \ell (p_i)\le \sum\limits_{p_i\in
I\setminus (J_1\cup J_2)} \widehat \ell (p_i) +\sum\limits_{j=1}^{l+m}
\big( \widehat \ell (y_j) + \widehat \ell (e_j) +\widehat \ell (f_j)\big) \le
\sum\limits_{j=0}^{l+m+1} s_{\mu , c}(n_j)
$$
Here the last inequality follows from the fact that every component appearing in its left side is an isolated component of $c_j$ for some $j$.

Further there is a constant $C>0$ such that
$$
\sum\limits_{j=0}^{m+l+1} n_j \le n + 6\ell(t)\le n +12\delta (\log
_2 n +2) +6\eta \le n+C\log _2 n
$$
and
$$
m+l+2\le 2\ell(t)+2\le C\log _2 n.
$$
Therefore, for any $n\ge N$, the function $s_{\mu , c}(n)$
satisfies the assumptions of Lemma \ref{linfunct} for $k=m+l+2$
and $\alpha =4/5$. Thus $s(n, \mu , c)$ is bounded by a linear
function from above.
\end{proof}

\subsection{Paths with long isolated components}

In this section we prove a technical lemma, which will be used several times in this paper. Informally it says the following. Let $p$ be a path in $\G$ such that at least every other edge is a long (with respect to $\widehat\ell $) component and no two consecutive components are connected. Then $p$ is quasi-geodesic. Further if two such paths are long and close to each other, then there are many consecutive components of one of them which are connected to consecutive components of the other. For relatively hyperbolic groups, similar lemmas were proved in \cite{AMO,Min}. Recall that relative generating sets are always assumed symmetric, so $X=X^{-1}$ in the following lemma.

\begin{lem}\label{w}
Let $G$ be a group weakly hyperbolic relative to $X$ and $\Hl $ and let
$\mathcal W$ be the set consisting of all words $U$ in $X\sqcup \mathcal H$ such that:
\begin{enumerate}
\item[(W$_1$)] $U$ contains no subwords of type $xy$, where $x,y\in X$.
\item[(W$_2$)] If $U$ contains a letter $h\in H_\lambda  $ for some $\lambda \in \Lambda $, then $\dl (1, h)>50D $, where $D=D(1,0)$ is given by Proposition \ref{sn}.
\item[(W$_3$)] If $h_1xh_2$  (respectively, $h_1h_2$) is a subword of $U$, where $x\in X$, $h_1\in H_\lambda $, $h_2\in H_\mu $, then either $\lambda \ne \mu $ or the element represented by $x$ in $G$ does not belong to $H_\lambda $ (respectively, $\lambda \ne \mu $).
\end{enumerate}
Then the following hold.
\begin{enumerate}
\item[(a)] Every path in $\G $ labelled by a word from $\mathcal W$ is $(4,1)$-quasi-geodesic.

\item[(b)] For every $\e>0$ and every integer $K>0$, there exists $R=R(\e, K)>0$ satisfying the following condition. Let $p, q$ be two paths in $\G$ such that $\ell (p)\ge R$, $\Lab (p), \Lab (q)\in \mathcal W$, and $p$, $q$ are oriented $\e$-close, i.e., $$\max \{ \d (p_-,q_-), \, \d(p_+, q_+)\} \le \e .$$ Then there exist $K$ consecutive components of $p$ which are connected to $K$ consecutive components of $q$. That is, $$p=x_0a_1\ldots x_{K-1}a_Kx_K, \;\;\; q=y_0b_1\ldots y_{K-1}b_Ky_K,$$ where $x_i, y_i$ are edges labelled by elements of $X$ or trivial paths for $i=1, \ldots, K-1$, and $a_j$, $b_j$ are connected components for every $j=1, \ldots, K$.
\end{enumerate}
\end{lem}

\begin{proof}
Let $p$ be a path in $\G $ such that $\Lab (p)\in \mathcal W$. Then according to (W$_1$) and (W$_3$),  $p=r_0p_1r_1\ldots p_mr_m$, where $p_i$'s are edges labelled by elements of $\mathcal H$ while $r_i$'s are either edges labelled by elements of $X$ or trivial paths. Further (W$_3$) guarantees that  no two consecutive components of $p$ are connected.

We start by showing that all components of $p$ are isolated. Suppose that two $H_\lambda $-components, $p_i$ and $p_j$, are connected for some $j>i$ and $j-i$ is minimal possible (Fig. \ref{43-f1}). Note that  $j=i+1+k$ for some $k\ge 1$, as no two consecutive components of $p$ are connected. Let $t$ denote the segment of $p$ with $t_-=(p_i)_+$ and $t_+=(p_j)_-$, and let $c$ be an empty path or an edge in $\G $ labelled by an element of $H_\lambda  $ such that $c_-=(p_i)_+$, $c_+=(p_j)_-$. Note that the components $p_{i+1}, \ldots , p_{i+k}$ are isolated in the cycle $tc^{-1}$. Indeed otherwise we can pass to another pair of components connected to each other with smaller value of $j-i$. By Proposition \ref{sn} we have
$$
\sum\limits_{l=1}^k \widehat\ell (p_{i+l})\le  D\ell (tc^{-1})\le D (2k+4).
$$
Hence $\widehat\ell (p_{i+l})\le D (2+4/k)\le 6D$ for some $l$ which
contradicts (W$_2$). Thus all components of $p$ are isolated.

To prove (a) we have to show that $p$  is $(4,1)$-quasi-geodesic. If $\ell (p)=1$, then this is obvious, so we assume that $\ell(p)>1$ and hence $m\ge 1$. Let $u$ be a geodesic connecting $p_+$ and $p_-$. Consider the geodesic $(2m+2)$-gon $\mathcal P =pu$ whose sides are $u$ and edges of $p$. Let $I$ be any subset of components of $p$ that are isolated in $\mathcal P$. By Proposition \ref{sn} we have
$$
\sum\limits_{s\in I} \widehat\ell (s)\le D(2m+2).
$$
Since $\widehat \ell(s)> 50D$ for every $s\in I$ by (W$_2$), we have $|I|<(2m+2)/50\le m/10$. Hence at least $9m/10$ components of $p$ are not isolated in $\mathcal P$. As no two distinct components of $p$ are connected, these $9m/10$ components are connected to distinct components of $u$.  In particular, $$\ell(u)\ge 9m/10> 3m/4 \ge (2m+1)/4\ge \ell (p)/4.$$ As this argument works for any subpath of $p$ as well, $p$ is $(4,1)$-quasi-geodesic.

\begin{figure}
  \centering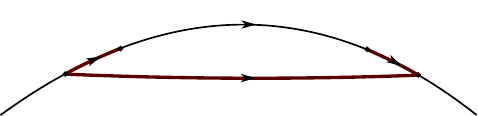\\
  \caption{}\label{43-f1}
\end{figure}

Let us prove (b). Fix $\e>0$ and an integer $K>0$. Let $p$ be as above  and let $q=s_0q_1s_1\ldots q_ns_n$, where $q_j$'s are edges labelled by elements of $\mathcal H$ while $s_j$'s are either edges labelled by elements of $X$ or trivial paths. As above, $q_j$'s are isolated components of $q$. Since $p$ is $(4,1)$-quasi-geodesics, we can choose $R$ such that
\begin{equation}\label{sctl-1}
R\ge 8\e +3
\end{equation}
and the inequality $\ell (p)\ge R$ guarantees the existence of a subpath $w$ of $p$ such that
\begin{equation}\label{sctl-2}
\d (w, p_\pm)> \e
\end{equation}
and
\begin{equation}\label{dwppm}
\ell (w)\ge 4K+ \ell (p)/2.
\end{equation}
Let $\mathcal Q=u_1pu_2q^{-1}$ be a loop in $\G $ such that $u_i$ is geodesic and
\begin{equation}\label{lui}
\ell (u_i)\le \e ,\; i=1,2.
\end{equation}
We can think of $\mathcal Q$ as a geodesic $k$-gon for $k=\ell (p)+\ell(q) +2$ whose sides are $u_1,u_2$ and the edges of $p$ and $q$. Since $q$ is $(4,1)$-quasi-geodesic, we have $\ell (q)\le 4(2\e + \ell (p))+1$. Hence $k\le 5\ell (p) +8\e +3\le 6 \ell (p)$ by (\ref{sctl-1}). Since $\ell (w)\ge \ell(p)/2 +1$, $w$ contains at least $\ell (p)/4$ components. Using Proposition \ref{sn},  (W$_2$), and arguing as above, we can show that every set $I$ of isolated components of $w$ satisfies $|I|\le 6\ell(p)/50<\ell(p)/4$ and hence not all components of $w$ are isolated in $\mathcal Q$.

Let $p_i$ be an $H_\lambda$-component of $w$ that is not isolated in $\mathcal Q$. We can assume that the segment $v$ of $w$ starting from $(p_i)_+$ and ending at $w_+$ has length at least $(\ell (w)-1)/2$. (The case when the initial subsegment of $w$ ending at $(p_i)_-$ has length at least $(\ell (w)-1)/2$ is symmetric.) By (\ref{lui}) and  (\ref{sctl-2}), $p_i$ can not be connected to a component of $u_1$ or $u_2$. Hence $p_i$ is connected to an $H_\lambda$-component $q_j$ of $q$.

Let $e$ be the edge  connecting $(q_j)_+$ to $(p_i)_+$ and labelled by a letter from $H_\lambda $. Note that $v$ has at least
$$
(\ell (v)-1)/2 \ge ((\ell (w)-1)/2-1)/2=(\ell(w)-3)/4> \ell (p)/8
$$
components by (\ref{dwppm}). We consider the polygon $$\mathcal Q^\prime=r_{i+1}p_{i+1}\ldots r_{m-1}p_mr_m u_2 (s_{j+1}q_{j+1}\ldots s_{n-1}q_ns_n)^{-1} e,$$ where the only side that is not an edge is $u_2$. Clearly the total number of sides of $\mathcal Q^\prime$ is less than $k\le 6 \ell (p)$. Again by (W$_2$) and Proposition \ref{sn} every set $I$ of isolated components of $v$ satisfies $|I|\le 6\ell(p)/50<\ell(p)/8$ and therefore not all components of $v$ are isolated in $\mathcal Q^\prime$. Let $p_{i+a}$ be an $H_\mu $-component of $v$ which is  not isolated in $\mathcal Q^\prime$ and such that $a$ is minimal possible. Note that $p_{i+a}$ can not be connected to $e$ as otherwise it is connected to $p_i$ as well, which contradicts the fact that all components of $p$ are isolated. Again by (\ref{sctl-2}) and (\ref{lui}), $p_{i+a}$ can not be connected to a component of $u_2$. Hence  $p_{i+a}$ is connected to an $H_\mu $-component $q_{j+b}$ of $q$. Let $f$ be an edge (or an empty path) connecting $(p_{i+a})_-$ to $(q_{j+b})_-$ and labelled by a letter from $H_\mu $ (Fig. \ref{43-f2}). Routinely applying Proposition \ref{sn} to the polygon $\mathcal Q^{\prime\prime}$ whose sides are $e$, $f$, and edges of $p$ (respectively, $q$) between $(p_{i})_+$ and $(p_{i+a})_-$ (respectively, $(q_{j})_+$ and $(q_{j+b})_-$), we conclude that if $a>1$, then there is a component $p_{i+a^\prime}$ of $p$, $0<a^\prime <a $, which is not isolated in $\mathcal Q^{\prime\prime}$. As above $p_{i+a^\prime}$ can not be connected to $e$ or $f$. Hence it is connected to $q_{j+b^\prime}$ for some $b^\prime >0$. However this contradicts minimality of $a$. Hence $a=1$ and similarly $b=1$. Thus $p_{i+1}$ is connected to $q_{j+1}$.

\begin{figure}
  \centering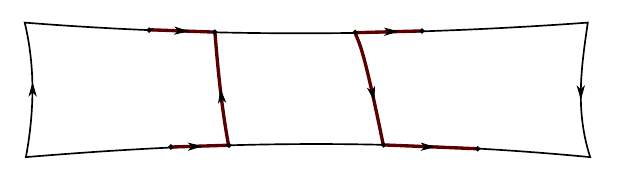\\
  \caption{}\label{43-f2}
\end{figure}

Repeating the arguments from the previous paragraph, we can show that components $p_{i}, p_{i+1},\ldots, p_{i+K-1}$ are connected to $q_j, q_{j+1},\ldots, q_{j+K-1}$. The key point here is that, for every $1\le l \le K-2$, the segment $[(p_{i+l})_+, w_+]$ of $p$ contains at least $$(\ell(w)-3)/4 -l> (\ell (w) -4l-3)/4 > (\ell(w) -4K)/4 \ge \ell(p)/8$$ components while at most $6\ell(p)/50<\ell(p)/8$ of them are not connected to components of the segment $[(q_{j+l})_+, q_+]$ of $q$. Thus there exists a component $p_{i+l+a}$ of $[(p_{i+l})_+, w_+]$ connected to a component $q_{j+l+b}$ of $[(q_{j+l})_+, q_+]$ and then the same argument as above shows that $a=b=1$. Thus part (b) is proven.
\end{proof}

\subsection{Hyperbolically embedded subgroups}

Our next goal is to introduce the notion of a hyperbolically embedded collection of subgroups. Assume that the group $G$ has a relative presentation

\begin{equation}
\langle X, \mathcal H  \mid  \mathcal S\cup \mathcal R\rangle
\label{rp1_rep}
\end{equation}
with respect to a collection of subgroups $\Hl$ and a relative generating set $X$. (The reader may want to review Section 3.3 before reading the rest of this section.)

\begin{defn}[{\bf Strongly bounded presentations}]\label{i-sbp}
We say that a relative presentation (\ref{rp1_rep}) is {\it strongly bounded} if  it is bounded
(that is, words in $\mathcal R$ have bounded length),
and for every $\lambda \in \Lambda$, the set of letters from $H_\lambda $ that appear in relators $R\in \mathcal R$ is finite.
\end{defn}

One easily checks that this definition agrees with the one given in the introduction, when
there is a single subgroup $H_\lambda$ (i.e.\ $|\Lambda|=1$).

\begin{ex}
Let $K$ be the free group with countably infinite basis $X=\{ x_1,x_2, \ldots \}$ and let $H=\langle t\rangle $. The group $G=K\times H$ has relative presentation $G=\langle X, H \mid  \mathcal R \rangle $ with respect to $X$ and $H$, where $\mathcal R=\{ [t, x_n]=1 \mid n=1,2,\ldots \} $. This relative presentation is strongly bounded. There is another relative presentation $\langle t, K \mid \mathcal R\rangle $ of $G$ with respect to the generating set $\{ t\} $ and the subgroup $K$. This presentation is bounded but not strongly bounded.
\end{ex}

\begin{thm}\label{ipchar}
Let $G$ be a group, $\Hl $ a collection of subgroups of $G$, $X$ a relative generating set of $G$ with respect to $\Hl $. The following conditions are equivalent.
\begin{enumerate}
\item[a)] The Cayley graph $\G $ is hyperbolic and for every $\lambda\in \Lambda $, the metric space $(H_\lambda , \dl)$ is locally finite.

\item[b)] There exists a strongly bounded relative presentation of $G$ with respect to $X$ and $\Hl $ with linear relative isoperimetric function.
\end{enumerate}
\end{thm}

\begin{proof}
Suppose first that for every $\lambda\in \Lambda $, the metric space $(H_\lambda , \dl)$ is locally finite. Let
\begin{equation}
\langle X, \mathcal H  \mid  \mathcal S\cup \mathcal R\rangle
\label{rp12}
\end{equation}
be a reduced bounded presentation with linear relative isoperimetric function provided by Lemma \ref{pres}. By Remark \ref{rem1} the letters from $\mathcal H$ that appear in relators $R\in \mathcal R$ have uniformly bounded length with respect to $\dl$. Since $(H_\lambda , \dl)$ is locally finite, the later condition means that the set of letters from $\mathcal H$ that appear in relators $R\in \mathcal R$ is finite. Thus (\ref{rp12}) is strongly bounded.

Now suppose that (\ref{rp12}) is a  strongly bounded relative presentation of $G$ with respect to $X$ and $\Hl $ with linear relative isoperimetric function. Let $Y_\lambda\subseteq H_\lambda $ be the subset consisting of all letters from $\mathcal H$ that appear in relators $R\in \mathcal R$. Suppose that $\dl (1,h)=n<\infty $ for some $h\in H_\lambda $. Let $p$ be a path in $\G $ of length $n$ such that $p_-=1$, $p_+=h$, and $p$ contains no edges of $\Gamma _{H_\lambda }$, $h\ne 1$. Let $e$ be the edge of $\G $ connecting $1$ to $h$ and labeled by $h\in H_\lambda $. Since $p$ contains no edges of $\Gamma _{H_\lambda }$, $e$ is an isolated $H_\lambda $-component of the cycle $ep^{-1}$. By Lemma \ref{Ylambda}, we obtain
\begin{equation}\label{ep}
\ell _{Y_{\lambda _i}} (p_i)\le MC \ell (ep^{-1})=MC(n+1),
\end{equation}
where $C$ is the isoperimetric constant of (\ref{rp12}) and $M=\max\limits_{R\in \mathcal R} \| R\| $. Since (\ref{rp12}) is strongly bounded, $Y_\lambda$ is finite and $M<\infty $. Therefore there are only finitely many $h\in H_\lambda $  satisfying (\ref{ep}) and thus $(H_\lambda , \dl)$ is locally finite.
\end{proof}

\begin{defn}\label{hes}
If either of the conditions from Theorem \ref{ipchar} holds, we say that the collection $\Hl $ is {\it hyperbolically embedded} in $G$ with respect to $X$ and write $\Hl \h (G,X)$. Further we say that $\Hl $ is \he in $G$ and write $\Hl \h G$ if $\Hl \h (G,X)$ for some relative generating set $X$.
\end{defn}

\begin{rem} Note that if $\Hl\h G$, then $H_\lambda \h G$. Indeed let $\mathcal H_\lambda =\bigcup\limits_{\mu\in \Lambda\setminus\{ \lambda \}} H_\mu$. Then it follows immediately from the definition that $H_\lambda \h (G, X\cup \mathcal H_\lambda )$ for every $\lambda\in \Lambda $.  However the converse does not hold. For example, let $H_1=G=F(x,y)$ be the free group of rank $2$ and let $H_2=\langle x\rangle$. Then one has $H_1\h G$ and $ H_2\h G$. However $\{H_1, H_2\}$ is not hyperbolically embedded in $(G, X)$ for any $X$ as $(H_2,\widehat\d_2 )$ is always bounded.
\end{rem}

We record a useful corollary of Theorem \ref{ipchar} (cf. Proposition \ref{x1x2}).

\begin{cor}\label{he-indep}
Let $G$ be a group, $\Hl$ a collection of subgroups of $G$, $X_1, X_2\subseteq G$ relative generating sets of $G$ with respect to $\Hl$. Suppose that $|X_1\triangle X_2|<\infty $. Then  $\Hl\h (G, X_1)$ if and only if
$\Hl \h (G, X_2)$.
\end{cor}

\begin{proof}
It is convenient to use both
definitions of hyperbolically embedded subgroups from Theorem \ref{ipchar}.
Suppose that $\Hl\h (G, X_1)$. We first note that $G$ is weakly hyperbolic relative to $\Hl$ and $X_2$ by Proposition \ref{x1x2}. Further observe that if $(H_\lambda , \d_\lambda ^X)$ is locally finite, where the relative metric $\d_\lambda ^X$ on $H_\lambda $ is defined using some subset $X\subseteq G$, then for every $Y\subseteq X$, $(H_\lambda , \d_ \lambda ^Y)$ is also locally finite, where $\d_ \lambda ^Y$ is defined using $Y$. Indeed this follows directly from the definition of the relative metric. Hence it suffices to prove that $\Hl\h (G, X_1\cup X_2)$. Thus without loss of generality, we can assume that $X_1\subseteq X_2$. By induction, we can further reduce this to the case $X_2=X_1\cup \{ t\}$. The proof in this case will be done using the isoperimetric characterization of hyperbolically embedded subgroups.

Let
\begin{equation}\label{rp-X1}
G=\langle X_1, \mathcal H \mid \mathcal S\cup \mathcal R\rangle.
\end{equation}
be a strongly bounded relative presentation of $G$ with respect to $X_1$ and $\Hl $ with relative isoperimetric function $Cn$. Let $V$ be a word in $X_1\sqcup \mathcal H $ representing $t$ in $G$. Then
\begin{equation}\label{rp-X2}
G=\langle X_2, \mathcal H \mid \mathcal S\cup (\mathcal R\cup \{ tV^{-1}\}) \rangle.
\end{equation}
and it is routine to check that (\ref{rp-X2}) has linear relative isoperimetric function.

Indeed let $W$ be a word in $X_2\sqcup \mathcal H $ of length $\| W\| \le n$ representing $1$ in $G$. Let $$W=W_1t^{\e_1}\cdots W_kt^{\e _k}W_{k+1},$$ where words $W_1, \ldots , W_{k+1}$ do not contain $t^{\pm 1}$ and $\e_i=\pm 1$ for $i=1, \ldots, k$. Obviously we have $W=_GU$, where
$$
U=W_1V^{\e_1}\cdots W_kV^{\e _k}W_{k+1}.
$$
Let $Area_1^{rel}$ and $Area_2^{rel}$ denote the relative areas with respect to presentations (\ref{rp-X1}) and (\ref{rp-X2}), respectively. Obviously $$Area_2^{rel}(W)\le Area_1^{rel}(U) + k\le C\| U\| +k \le C\|V\|n+n.$$
Thus the relative isoperimetric function of (\ref{rp-X2}) is also linear and hence  $\Hl \h (G, X_2)$.
\end{proof}

The next result shows that Definition \ref{hes} indeed generalizes the notion of a relatively hyperbolic group.

\begin{prop}\label{he-rh}
Let $G$ be a group, $\Hl$ a collection of subgroups of $G$.
\begin{enumerate}
\item[a)] Suppose that $G$ is hyperbolic relative to $\Hl$. Then $\Hl \h (G,X)$ for some (equivalently, any) finite relative generating set $X$ of $G$.
\item[b)] Conversely if $\Hl \h (G,X)$ for some (equivalently, any) finite relative generating set $X$ of $G$ and $\Lambda $ is finite, then $G$ is hyperbolic relative to $\Hl $.
\end{enumerate}
\end{prop}

\begin{proof}
Since every finite relative presentation
is strongly bounded, a) follows immediately from Definitions \ref{rhg} and \ref{hes}. Conversely if $X$ and $\Lambda $ are finite, then every strongly bounded relative presentation of $G$ with respect to $X$ and $\Hl $ is finite, and the claim follows from the definitions again.
\end{proof}

We are now going to discuss some general results about hyperbolically embedded subgroups. Our first goal is to prove that many finiteness properties pass from groups to hyperbolically embedded subgroups. We will need the following.

\begin{lem}\label{lem-retrLip}
Let $G$ be a group, $H$ a subgroup of $G$, $X$ a generating set of $G$.
Suppose that $\Gamma (G, X\sqcup H)$ is hyperbolic. Then there is a map $r\colon G\to H$ such that the restriction of $r$ to $H$ is the identity map for some fixed constant $C>0$ we have
\begin{equation}\label{retrLip}
\widehat\d  (r(f),r(g))\le C \dx (f,g)
\end{equation}
for every $f,g\in G$.
\end{lem}

\begin{proof}
Given $g\in G$ we define $r(g)$ to be
any element of $H$ such that $$\dxh (g,h)=\dxh (g, H).$$ Obviously $r(g)=g$ for every $g\in H$.

Assume first that $f,g\in G$ and $\dx (f,g)=1$. Consider a geodesic $4$-gon $Q$ in $\Gamma (G, X\sqcup H)$ with consecutive vertices $f,g, r(g), r(f)$ (some sides of $Q$ may be trivial) such that the side $[f,g]$ is labelled by some $x\in X$ and the side $p=[r(g), r(f)]$ is labelled by some $h\in H$. By the definition of $r$, the sides $[f,r(f)]$ and $[g, r(g)]$ intersect $H$ only at $r(f)$ and $r(g)$, respectively. Hence $p$ is a component of $Q$ which is not connected to any $H$-component of $[f,r(f)]$ or $[g, r(g)]$. Since $[f,g]$ is labelled by some $x\in X$ and thus has no $H$-components at all, $p$ is isolated in $Q$. Hence $\widehat\d  (p_-, p_+)\le 4D$, where $D=D(1,0)$ is the constant from Proposition \ref{sn}. Now (\ref{retrLip}) follows for any $f,g\in G$ and $C=4D$ by the triangle inequality.
\end{proof}

The next definition is inspired by \cite{Alo}.

\begin{defn}
Let $S$, $T$ be metric spaces. We say that $S$ is a {\it Lipschitz quasi-retract} \label{i-Lqr} of $T$ if there exists a sequence of Lipschitz maps $$ S\stackrel{i}\longrightarrow T \stackrel{r}\longrightarrow S $$ such that $r\circ i\equiv id _{S}$.
\end{defn}

Given a finitely generated group $A$ and a group $B$, we say that $B$ is a \emph{Lipschitz quasi-retract} of $A$ if $B$ is finitely generated and $(B, \d_Y)$ is a Lipschitz quasi-retract of $(A, \d _X)$, where $\d_X$ and $\d _Y$ are word metrics corresponding to some finite generating sets $X$ and $Y$ of $A$ and $B$ respectively. (Obviously replacing `some finite generating sets $X$ and $Y$' with `any finite generating sets $X$ and $Y$' leads to an equivalent definition.) We stress that the maps $i$ and $r$ do not need to preserve the group structure, so our definition does not imply that $B$ is a retract of $A$ in the group theoretic sense. On the other hand, it is easy to see that if $B$ is a retract of $A$ in the group theoretic sense and $A$ is finitely generated, then $B$ is a Lipschitz quasi-retract of $A$.

\begin{thm}\label{he-retr}
Let $G$ be a finitely generated group and let $H$ be a hyperbolically embedded subgroup of $G$. Then $H$ is a Lipschitz quasi-retract of $G$.
\end{thm}

\begin{proof}
Let $X_0$ be a finite generating set of $G$. Suppose that $H\h (G, X)$. By Corollary \ref{he-indep} we can assume that $X_0\subseteq X$. Lemma \ref{lem-retrLip} easily implies that $H$ is generated by the set $$Y=\{ y\in H \mid \widehat\d (1,y)\le C\}.$$ Indeed for any $h\in H$ there is a path $q$ in $\Gamma (G, X\sqcup H)$ labelled by a word in the alphabet $X_0$ and connecting $1$ to $h$. Let $h_0=1$, $h_1$, \ldots , $h_n=h$ be the images of consecutive vertices of $q$ under the map $r$ provided by Lemma \ref{lem-retrLip}. Then for $1\le i\le n$, we have $\widehat\d (1, h_{i-1}^{-1}h_i )= \widehat\d  (h_{i-1}, h_i)\le C$ by Lemma \ref{lem-retrLip}. Hence $h_{i-1}^{-1}h_i \in Y$ and $h= (h_0^{-1}h_1) \cdots (h_{n-1}^{-1}h_n) \in \langle Y\rangle $. Thus $Y$ generates $H$. Moreover, our argument shows that for every $h\in H$, we have
\begin{equation}\label{retr-undist}
|h|_Y \le |h|_X.
\end{equation}

Let $i\colon (H, \d_Y)\to (G, \d_X)$ be the map induced by the natural embedding $H\to G$. Then $i$ is Lipschitz by (\ref{retr-undist}). Further it is obvious that the composition $r\circ i$ is identical on $H$. Since $r$ is also Lipschitz, we conclude that $(H, \d_Y)$ is a Lipschitz quasi-retract of $(G, \d_X)$. It remains to note that $Y$ is finite since $H\h G$.
\end{proof}

Note that every Lipschitz quasi-retract in our sense is a quasi-retract in the sense of \cite{Alo}. It was proved in \cite{Alo} and \cite{APW} that if a finitely generated group $H$ is a quasi-retract of a finitely generated group $G$, then $H$ inherits some finiteness properties and upper bounds on Dehn functions (including higher dimensional ones) from $G$.
Combining Theorem \ref{he-retr} with these results we obtain the following.

\begin{cor}\label{cor-fn}
Let $G$ be a finitely generated group and let $H$ be a hyperbolically embedded subgroup of $G$. Then the following conditions hold.
\begin{enumerate}
\item[(a)] $H$ is finitely generated.
\item[(b)] If $G$ is of type $F_n$ for some $n\ge 2$, then so is $H$. Moreover, the corresponding $(n-1)$-dimensional Dehn functions satisfy $\delta ^{n-1} _H\preceq \delta ^{n-1}_G$. In particular, if $G$ is finitely presented, then so is $H$ and $\delta _H\preceq \delta _G$.
\item[(c)] If $G$ is of type $FP_n$, then so is $H$.
\end{enumerate}
\end{cor}

Let us mention some other elementary results generalizing well-known properties of relatively hyperbolic groups. These results will be used later in this paper.

\begin{prop}\label{malnorm}
Suppose that a group $G$ is weakly hyperbolic relative to $X$ and $\Hl $. Then there exists a constant $A>0$ such that following conditions hold.
\begin{enumerate}
\item[a)] For any distinct $\lambda,\mu \in \Lambda$, and any $g\in G$, the intersection $H_\lambda ^g\cap H_\mu$ has diameter at most $A$ with respect to $\dl $. In particular, if $\Hl\h G$, then $|H_\lambda^g\cap H_\mu|<\infty $.
\item[b)] For any $\lambda \in \Lambda $ and any $g\in G\setminus H_\lambda $, the intersection $H_\lambda ^g\cap H_\lambda$ has diameter at most $A$ with respect to $\dl $. In particular, if $\Hl\h G$, then $|H_\lambda \cap H_\lambda ^g|<\infty $.
\end{enumerate}
\end{prop}

\begin{proof}
We first prove (a). Consider a shortest word $W$ in the alphabet $X\sqcup \mathcal H$ that represents $g$ in $G$. Assume that
$W=W_1W_2$, where $W_1$ is the maximal (may be empty) prefix of
$W$ consisting of letters from $H_\lambda $. Denote by
$f$ the element of $G$ represented by $W_2$. It is clear that
$H^g_\lambda =H^{f}_\lambda $. Thus passing from $g$ to $f$ if necessary, we can assume that the first letter of $W$ does not belong to $H_\lambda $.

Let us take an arbitrary element $h\in H_\lambda ^g \cap H_\mu $
and denote by $h_1$, $h_2$ the letters from $H_\lambda $ and $H_\mu $ that represent elements $h^{g^{-1}}\in H_\lambda $ and $h\in H_\mu $, respectively. Since $W^{-1}h_1W$ and $h_2$ represent the same element $h$, there is a geodesic quadrilateral $c=a^{-1}pbq$ in $\G $, where $a$ and $b$ are paths labelled by $W$, and $p,q$ are edges labelled by $h_1\in H_\lambda $ and $h_2^{-1}\in H_\lambda $, respectively. Note that $p$ is an isolated component of $c$. Indeed as $\lambda \ne \mu $, $p$ can not be connected to $q$. Further suppose that a component of $a^{-1}$ is connected to $p$. Since $a$ is geodesic this component must be the last edge of $a^{-1}$, which contradicts our assumption that the first letter of $W$ does not belong to $H_\lambda $. Hence $p$ can not be connected to a component of $a$. The same argument applies to $b$.  Thus $p$ is isolated in $c$ and $\widehat \ell (p)\le 4L$, where $L$ is the constant provided by Proposition \ref{sn}. This proves (a).

The proof of (b) is similar. The only difference is that $p$ can not be connected to $q$ in this case since $g\notin H_\lambda $.
\end{proof}

\begin{cor}\label{center}
Suppose that $G$ is a group with infinite center. Then $G$ contains no proper infinite hyperbolically embedded subgroups.
\end{cor}

\begin{proof}
Assume that there exists a non-degenerate hyperbolically embedded subgroup $H$ of $G$ and let $Z$ denote the center of $G$. Then $H^z=H$ for every $z\in Z$. Since $H$ is infinite, we obtain $Z\le H$ from part b) of the proposition (and the fact that $H$ is a proper space with respect to the relative metric $\widehat d$). Since $H\ne G$, there exists $g\in G\setminus H$. Then $H^g\cap H$ must be finite by part b) of the proposition. Obviously this intersection contains $Z$. Hence $Z$ is finite.
\end{proof}

The next proposition shows that ``being a hyperbolically embedded subgroup" is a transitive property. The analogous property of relatively hyperbolic groups can be found in \cite{Osi06a}.

\begin{prop}\label{transitive}
Let $G$ be a group, $\Hl$ a finite collection of subgroups of $G$, $X\subseteq G$, $Y_\lambda\subseteq H_\lambda$. Suppose that $\Hl\h (G, X)$ and, for each $\lambda \in \Lambda$, there is a collection of subgroups $\Km $ of $H_\lambda $ such that $\Km\h (H_\lambda, Y_\lambda)$. Then $\bigcup_{\lambda \in \Lambda } \Km\h (G,Z)$, where
$Z=X\cup \left(\bigcup_{\lambda \in \lambda }Y_\lambda\right).$
\end{prop}

\begin{proof}
Let us fix some strongly bounded relative presentations with linear relative isoperimetric functions:
\begin{equation}\label{GrelH}
G=\left.\left\langle X, \mathcal H \left| \left(\bigcup_{\lambda\in \Lambda} \mathcal S_\lambda \right) \cup \mathcal R\right.\right.\right\rangle
\end{equation}
and
\begin{equation}\label{HrelK}
H_\lambda = \langle Y_\lambda , \Km   \mid  \mathcal P_\lambda\rangle .
\end{equation}
Here, as usual, $\mathcal H=\bigsqcup_{\lambda \in \Lambda } H_\lambda $, and $\mathcal S_\lambda $ is the set of all words in $H_\lambda $ representing $1$ in $H_\lambda $. Clearly $G$ can be also represented by the relative presentation
\begin{equation}\label{GrelK}
G=\left\langle \left.X\cup Y, \bigcup_{\lambda\in \Lambda}\Km  \right| \mathcal P \cup \mathcal R\right\rangle ,
\end{equation}
where $Y=\bigcup_{\lambda \in \lambda }Y_\lambda $ and $\mathcal P=  \bigcup_{\lambda\in \Lambda} \mathcal P_\lambda$. It is clear that (\ref{GrelK}) is strongly bounded. To prove the proposition it suffices to show that it has linear relative isoperimetric function. We define the notions of $\mathcal S_\lambda $-, $\mathcal S$-, $\mathcal R$-, $\mathcal P_\lambda $-, and $\mathcal P$-cells in diagrams over (\ref{GrelH})-(\ref{GrelK}) in the obvious way.

Since $\Lambda $ is finite, there exists $C>0$ such that $f(n)=Cn$ is a relative isoperimetric function of the presentations (\ref{GrelH}) and (\ref{HrelK}) for all $\lambda $. Let $\preceq $ be the lexicographic order on $\mathbb N\times\mathbb N$, that is $(a,b)\preceq (c,d)$ if and only if $a<c$ or $a=c$ and $b\le d$. We say that a diagram $\Delta $ over (\ref{GrelH}) has type $(a,b)$ if $a$ and $b$ are the numbers of $\mathcal R$-cells and $\mathcal S$-cells in $\Delta $, respectively. Let $W$ be any word in $X\cup Y\sqcup\mathcal K$, where $\mathcal K=\bigsqcup_{\lambda \in \Lambda }\bigsqcup_{\mu \in M_\lambda } (K_{\lambda \mu}  )$, and suppose that $W$ represents $1$ in $G$. Let $\Delta $ be the diagram over (\ref{GrelH}) of minimal type with $\partial \Delta \equiv W$.

Observe that no two $S_\lambda $-cells have a common edge in $\Delta $. Indeed otherwise we could replace these two cells with one, which contradicts the minimality of the type of $\Delta $. Hence every edge of in $\Delta $ either belongs to the boundary of $\Delta $ or to the boundary of some $\mathcal R$-cell. Let $E$ be the total number of edges in $\Delta $. Then $E\le (CM+1) \| W\| $ , where $M=\max\limits_{R\in \mathcal R} \| R\| $.

For every $S_\lambda $-cell $\Xi $ in $\Delta $, there is a diagram over (\ref{HrelK}) with the same boundary label and the number of $\mathcal P_\lambda $-cells at most $C\ell(\partial \Xi)$. After replacing all $S_\lambda $-cells (for all $\lambda $) with such diagrams, we obtain a diagram $\Delta ^\prime$ over (\ref{GrelK}), where the total number of $\mathcal P$-cells is at most $CE\le C(CM+1)\| W\| $. Note that the number of $\mathcal R$-cells does not change and is at most $C\| W\| $ by the minimality of the type of $\Delta $ and the choice of $C$. Hence the total number of $\mathcal P$ and $\mathcal R$-cells in $\Delta ^\prime$ is at most $C(CM+2)\| W\| $. Thus (\ref{GrelK}) has a linear relative isoperimetric function.
\end{proof}

The next result shows that the property of being hyperbolically embedded is conjugacy invariant. Moreover, we have the following.
\begin{prop}\label{heconj}
Let $G$ be a group, $\Hl $ a collection of subgroups of $G$, $X$ a subset of $G$ such that $\Hl\h (G,X)$. Let $t$ be an arbitrary element of $G$ and let $M$ be any subset of $\Lambda $. Then we have $\{ H_\lambda^t\} _{\lambda\in M}\cup \{ H_\lambda\}_{\lambda \in \Lambda \setminus M}\h (G,X)$.
\end{prop}

\begin{proof}
Let $$\mathcal H=\bigsqcup \limits_{\lambda \in \Lambda } H_\lambda$$ and $$\mathcal H^\prime =\left(\bigsqcup\limits_{\lambda\in M} H_\lambda^t\right) \sqcup \left(\bigsqcup \limits_{\lambda \in \Lambda\setminus M } H_\lambda\right).$$ By Corollary \ref{he-indep} we can assume that $t\in X$ without loss of generality.

Let $\dl$ and (respectively, $\dl^\prime $) be the metric defined on $H_\lambda $ (respectively, $H_\lambda ^t$ for $\lambda \in M$ and $H_\lambda $ for $\lambda \in \Lambda \setminus M$) using the graph $\G $ (respectively, $\Gamma (G, X\sqcup \mathcal H^\prime)$). Note that every word $W$ in the alphabet $X\sqcup \mathcal H^\prime $ can be turned into a word in the alphabet $X\sqcup \mathcal H$ by replacing each letter $h^t\in H_\lambda^t $, $\lambda\in M$, with the word $t^{-1}ht$ of length $3$, where $h\in H_\lambda$. We will denote the resulting word by $\pi (W)$. Note that $W$ and $\pi (W)$ represent the same element in the group $G$.

Let $x^t\in H_\lambda ^t$ for some $\lambda \in M $. Let $p$ be a in $\Gamma (G, X\sqcup \mathcal H^\prime)$ between $1$ and the vertex $x^t$. Let $q$ be the path connecting $1$ and $x\in H_\lambda$ with label $t\pi (\Lab(p))t^{-1}$. It is straightforward to verify that if $p$ does not contain any edges of the subgraph $\Gamma (H_\lambda ^t, H_\lambda^t )$ of $\Gamma (G, X\sqcup \mathcal H^\prime)$, then $q$ does not contain any edges of the subgraph $\Gamma (H_\lambda , H_\lambda )$ of $\G $. Since $\ell (q)\le 2+3\ell (p)$, we conclude that $\dl (1,x)\le 3\dl^\prime (1, x^t)$ for every $x\in H_\lambda $. Hence local finiteness of $(H_\lambda , \dl)$ implies local finiteness of $(H_\lambda ^t, \dl ^\prime)$ for $\lambda \in M$. Further for $\lambda \in \Lambda \setminus M$, the local finiteness of $(H_\lambda ^t, \dl ^\prime)$ can be obtained in the same way. The only difference is that we have to use the label $\pi(\Lab (p))$ instead of $t\pi (\Lab(p))t^{-1}$ in the definition of $q$. Thus $\{ H_\lambda^t\} _{\lambda\in M}\cup \{ H_\lambda\}_{\lambda \in \Lambda \setminus M}\h (G,X)$ by the definition.
\end{proof}

\subsection{Projection complexes and geometrically separated subgroups}

Our main goal in this section is to propose a general method of constructing hyperbolically embedded subgroups in groups acting on hyperbolic spaces. Our approach is based on projection complexes introduced by Bestvina, Bromberg, and Fujiwara in \cite{BBF}. We begin by recalling the definitions.

\begin{defn}\label{projc}
Let $\mathbb Y$ be a set. Assume that for each $Y \in \mathbb Y$ we have a
function $$\d_Y^\pi  \colon (\mathbb  Y\setminus\{Y\}) \times (\mathbb Y\setminus\{Y\})
\longrightarrow [0,\infty),$$ called \textit{projection on $Y$}, and a constant $\xi>0$ that
  satisfy the following axioms for all $A,B \in \mathbb Y\setminus \{ Y\}$:
\begin{enumerate}
\item[(A$_1$)] $\d_Y^\pi (A,B) = \d_Y^\pi (B,A)$;
\item[(A$_2$)] $\d_Y^\pi(A,B) + \d_Y^\pi(B,C) \ge \d_Y^\pi(A,C)$;
\item[(A$_3$)] \label{chris}   $\min\{\d_Y^\pi(A,B), \d_B^\pi(A,Y) \} < \xi $;
\item[(A$_4$)] \label{finite}  $\# \{Y |
  \d_Y^\pi(A,B)\ge \xi\}$ is finite.
\end{enumerate}

Let also $K$ be a positive constant. Associated to this data is the {\it projection complex}, {\pky }, which is a graph constructed as follows. The set of vertices of {\pky } is the set $\mathbb Y$ itself. To describe the set of edges, one first defines a new function $\d_Y   \colon (\mathbb  Y\setminus\{Y\}) \times (\mathbb Y\setminus\{Y\}) \longrightarrow [0,\infty)$ as a small perturbation of $\d _Y^\pi $. The exact definition can be found in \cite{BBF} and is not essential for our goals. The only essential property of $\d _Y$ is the following inequality, which is an immediate corollary of \cite[Proposition 2.2]{BBF}. For every $Y\in \mathbb Y$ and every $A,B\in \mathbb  Y\setminus\{Y\}$, we have
\begin{equation}\label{dpd}
|\d^\pi _Y(A,B) -\d _Y(A,B)|< 2\xi .
\end{equation}

Two vertices $A,B\in \mathbb Y$ are connected by an edge if and only if for every $Y\in \mathbb Y\setminus \{A,B\}$, the projection $\d_Y (A,B)$ satisfies $\d_Y(A,B)\le K$. Note that this construction strongly depends on $K$ and, in general, the complexes corresponding to different $K$ are not quasi-isometric. We also remark that if $\mathbb Y$ is endowed with an action of a group $G$ that preserves projections (i.e., $d^\pi_{g(Y)}(g(A), g(B))=\d^\pi _Y (A,B)$), then it extends to an action of $G$ on {\pky }.
\end{defn}

The following is the main example (due to Bestvina, Bromberg, and Fujiwara \cite{BBF}), which motivates the terminology.

\begin{ex}\label{pk-ex}
Let $G$ be a discrete group of isometries of $\mathbb H^n$ and $g_1, \ldots , g_k$
a finite collection of loxodromic elements of $G$. Denote by $X_i$ the axis of
$g_i$ and let
$$
\mathbb Y = \{ gX_i \mid g\in G, \; i=1, \ldots , k\}
$$
It is easy to check that there exists  $\nu > 0$ such that the projection ${\rm proj }_YX$
(i.e. the image under the nearest point projection map) of any geodesic
$X\in \mathbb Y$ to any other geodesic $Y\in \mathbb Y$ has diameter bounded by $\nu $. Thus we can define
$\d^{\pi}_Y (X,Z)$ to be ${\rm diam}({\rm proj }_Y (X\cup Z))$. The reader may check that all axioms
hold.
\end{ex}

Later on we will apply the above construction in a situation which can be viewed as a generalization of Example \ref{pk-ex}.

The following was proved in \cite[Lemma 2.4 and Theorem 2.9]{BBF} under the assumptions of Definition \ref{projc}.
\begin{prop}\label{BBF}
There exists $K>0$ such that {\pky } is connected and quasi-isometric to a tree.
\end{prop}

Given a group $G$ acting on a set $S$, an element $s\in S$, and a subset $H\le G$, we define the {\it $H$-orbit } of $s$ by
$$
H(s)=\{ h(s)\mid h\in H\} .
$$

\begin{defn}\label{GeomSep}
Let $G$ be a group acting on a space $(S, d)$. A collection of subgroups $\Hl $ of $G$ is called {\it geometrically separated} if for every $\e >0$ and every $s\in S$, there exists $R>0$ such that the following holds. Suppose that for some $g\in G$ and $\lambda,\mu \in \Lambda $ we have
\begin{equation}\label{GS-eq}
{\rm diam} \left(H_\mu (s)\cap \mathcal (gH_\lambda (s))^{+\e}\right)\ge R.
\end{equation}
Then $\lambda =\mu $ and $g \in H_\lambda $.
\end{defn}

Informally, the definition says that the orbits of distinct cosets of subgroups from the collection $\Hl $ rapidly diverge. In the next section, we will also show that geometric separability can be thought of as a generalization of the weak proper discontinuity condition introduced by Bestvina and Fujiwara \cite{BF}.

\begin{figure}
  \centering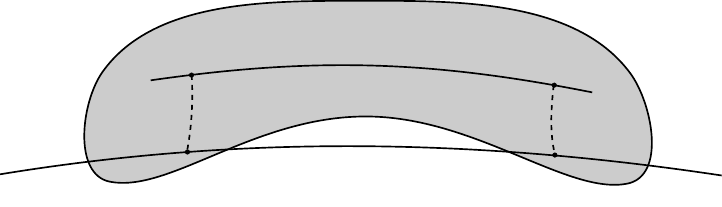\\
  \caption{}\label{44-f3}
\end{figure}

\begin{rem}\label{rem-gs}
Note that in order to prove that $\Hl $ is geometrically separated it suffices to verify that for every $\e >0$ and \emph{some} $s\in S$, there exists $R=R(\e)>0$ satisfying the requirements of the Definition \ref{GeomSep}. Indeed then for every $\e>0$ and every $s^\prime \in S$, we can take $$R^\prime =2R(\e + 2\d (s, s^\prime))+ 4\d (s, s^\prime).$$ Now if
$$
{\rm diam} \left(H_\mu (s^\prime )\cap \mathcal (gH_\lambda (s^\prime ))^{+\e}\right)\ge R^\prime ,
$$
then there exist $h_1,h_2\in H_\mu$ and $k_1, k_2\in gH_\lambda $ such that $$\d (h_1(s^\prime), h_2(s^\prime))\ge R^\prime/2=R(\e + 2\d (s, s^\prime))+ 2\d (s, s^\prime)$$ and $\d (h_i (s^\prime), k_i(s^\prime))\le \e $ for $i=1,2$ (Fig. \ref{44-f3}). This implies
$$
\d (h_1(s), h_2(s))\ge \d (h_1(s^\prime), h_2(s^\prime)) -\d (h_1(s^\prime), h_1(s))- \d (h_2(s^\prime),h_2(s)) \ge R(\e + 2\d (s, s^\prime))
$$
and similarly
$$
\d (h_i (s), k_i(s))\le \e +2\d (s, s^\prime).
$$
Therefore,
$$
{\rm diam} \left(H_\mu (s)\cap \mathcal (gH_\lambda (s))^{+\e+2\d (s, s^\prime)}\right)\ge R(\e + 2\d (s, s^\prime)),
$$
which implies $\lambda =\mu $ and $g \in H_\lambda $.
\end{rem}

 The main result of this section is the following.

\begin{thm}\label{crit}
Let $G$ be a group, $\Hl $ a finite collection of distinct subgroups of $G$. Suppose that the following conditions hold.
\begin{enumerate}
\item[(a)] $G$ acts by isometries on a hyperbolic space $(S, \d)$.

\item[(b)] There exists $s\in S$ such that for every $\lambda \in \Lambda $, the $H_\lambda $-orbit of $s$ is quasiconvex in $S$.

\item[(c)] $\Hl $ is geometrically separated.
\end{enumerate}
Then there exists a relative generating set $X$ of $G$ with respect to $\Hl $ and a constant $\alpha >0$ such that the Cayley graph $\G $ is hyperbolic, and for every $\lambda \in \Lambda $ and $h\in H_\lambda $ we have
\begin{equation}\label{hec0}
\dl (1,h) \ge \alpha \d (s, h(s)).
\end{equation}
In particular, if every $H_\lambda $ acts on $S$ properly, then $\Hl \h (G,X)$.
\end{thm}

\begin{rem}
Note that the assumptions of the theorem imply that each $H_\lambda $ acts properly and coboundedly on a hyperbolic space, namely on a suitable neighborhood of the orbit $H_\lambda (s)$. This implies that $H_\lambda $ is hyperbolic. Thus all applications of Theorem \ref{crit} lead to \he families of hyperbolic subgroups.
\end{rem}

Note that if $ {\rm diam} (H_\lambda (s))<\infty $ for some $\lambda \in \Lambda $, then the inequality (\ref{hec0}) holds for $\alpha = 1/{\rm diam} (H_\lambda (s) )$ for {\it any} generating set $X$. (In particular, if $ {\rm diam} (H_\lambda (s))<\infty $ for all $\lambda \in \Lambda $, then we can take $X=G$.) Thus it suffices to prove the theorem assuming that
\begin{equation}\label{diamHl}
{\rm diam} (H_\lambda (s))=\infty
\end{equation}
for all $\lambda \in \Lambda $.

We present the proof as a sequence of lemmas. Throughout the rest of the section we work under the assumptions of Theorem \ref{crit}. We also assume (\ref{diamHl}).

Let us first introduce some auxiliary notation.
Let $\delta >0$ be a hyperbolicity constant of $S$. Given a point $a\in S$ and a subset $Y\subseteq S$, we define the {\it projection} of $a$ to $Y$ by
\begin{equation}\label{projdef}
{\rm proj }_Y(a)=\{ y\in Y\mid \d (a,y)\le \d (a, Y)+\delta \}.
\end{equation}
Further, given two subsets $A,Y\subseteq S$, we define
$$
{\rm proj }_Y(A)=\{ {\rm proj }_Y(a)| a\in A\} .
$$
The proof of following lemma is a standard exercise in hyperbolic geometry.

\begin{figure}
  \centering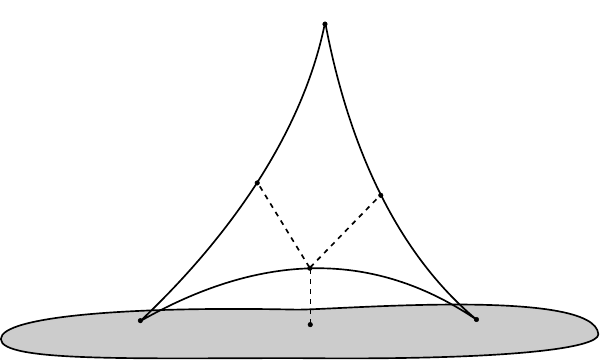\\
  \caption{}\label{44-f2}
\end{figure}

\begin{lem}\label{PrOfPoint}
Suppose that $Y$ is a $\sigma$-quasiconvex subset of $S$. Then for every $a\in S$, we have ${\rm diam} (\pr _Y(a))\le 6\delta + 2\sigma $.
\end{lem}
\begin{proof}
Let $x,y\in \pr _Y(a)$, let $z$ be the point of the geodesic segment $[x,y]$ such that $\d (z, z_1)=\d (z, z_2)\le \delta $ for some $z_1\in [a,x]$ and $z_2\in [a,y]$, and let $w$ be a point from $Y$ such that $\d (z,w)\le \sigma $ (Fig. \ref{44-f2}). By the definition of $\pr _Y(a)$, we have $\d (z_1, x) \le \d (z_1, w)+\delta \le 2\delta +\sigma$. Hence $\d (x, z)\le \d (x, z_1)+\d (z_1, z)\le 3\delta + \sigma $ and similarly for $\d (y,z)$.
\end{proof}

We define $\mathbb Y$ to be the set of orbits of all cosets of $H_\lambda$'s. That is, let
$$
\mathbb Y=\{ gH_\lambda (s) \mid \lambda\in \Lambda ,\; g\in G \} .
$$
In what follows the following observation will be used without any references.
\begin{lem}\label{abuse}
Suppose that for some $f,g\in G$ we have $gH_\lambda (s)=fH_\mu (s)$. Then $gH_\lambda =fH_\mu$.
\end{lem}
\begin{proof}
If $gH_\lambda (s)=fH_\mu (s)$, then $f^{-1}gH_\lambda (s)=H_\mu (s)$. Now the geometric separability condition together with (\ref{diamHl}) imply that $\lambda =\mu $ and $f^{-1}g\in H_\lambda$, hence the claim.
\end{proof}

Recall that $|\Lambda |<\infty $ in Theorem \ref{crit}. Let us denote by $\sigma $ a common quasiconvexity constant of all $H_\lambda (s)$, $\lambda \in \Lambda $. Thus all subsets $Y\in \mathbb Y$ of $S$ are $\sigma$-quasiconvex.

\begin{lem}\label{boundproj}
There exists a constant $\nu $ such that for any distinct $A,B\in \mathbb Y$ we have
\begin{equation}\label{nu}
{\rm diam} ( {\rm proj}_B(A) )\le \nu .
\end{equation}
\end{lem}

\begin{proof}
Let $$\e = 13\delta + 2\sigma $$ and let $R=R(\e )$ be the constant given by Definition \ref{GeomSep}. Let $$c=\max\{ R+ 2\sigma , 30\delta +\sigma\} .$$ We will show that (\ref{nu}) is satisfied for $\nu = 4000c$.

Indeed let $a_1,a_2\in A$ and let $b_1\in {\rm proj }_B(a_1)$, $b_2\in {\rm proj }_B(a_2)$ and suppose that $\d (b_1, b_2)> 4000 c$. Note that by our definition of projections, we have
\begin{equation}\label{b1b2}
\d (a_i, b_i)\le \d(a_i, B) +\delta, \; i=1,2.
\end{equation}
By Lemma \ref{Ols} applied to the geodesic $4$-gon with consecutive vertices $a_1, a_2, b_2, b_1$, there is a subsegment $u$ of the geodesic segment $[b_1, b_2]$ and a subsegment $v$ of one of the geodesic segments $[a_1, b_1], [a_1, a_2], [a_2, b_2]$ such that $\min\{\ell (u),\ell (v)\} \ge c$ and $u,v$ are $13\delta $-close.

\begin{figure}
  \centering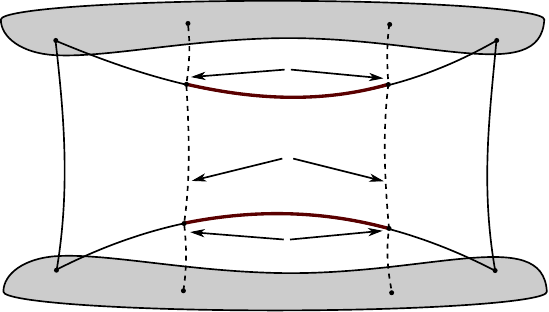\\
  \caption{}\label{44-f6}
\end{figure}

It easily follows from our definition of projections that $v$ can not belong to $[a_1,b_1]$ or $[a_2, b_2]$. Indeed if $v$ is an (oriented) subsegment of $[a_i,b_i]$ for some $i\in \{ 1,2\}$, then
$$
 \d (a_1,B)\le \d (a_1, v_-)+\d (v_-, u)+\d (u,B)\le \d (a_1, v_-)+13\delta +\sigma,
$$
while
$$
\d (a_1, b_1)\ge \d (a_1, v_-) +\ell (v) \ge \d (a_1, v_-)+30\delta +\sigma.
$$
This contradicts (\ref{b1b2}). Hence $v$ is a subsegment of $[a_1, a_2]$ (Fig. \ref{44-f6}).

Since $A$ and $B$ are $\sigma$ -quasiconvex, we obtain
$$
{\rm diam} \left(B\cap A^{+\e }\right)\ge \ell (u) - 2\sigma \ge c-2\sigma \ge R.
$$
By the geometric separability condition, this inequality implies $A=B$. A contradiction.
\end{proof}

Let us now define, for any $Y\in \mathbb Y$ and $A,B\in \mathbb Y\setminus \{ Y\} $,
$$
\d ^\pi _Y (A,B)={\rm diam } ({\rm proj}_Y  (A)\cup {\rm proj}_Y  (B)).
$$
The quantity $\d ^\pi _Y$ is finite by Lemma \ref{boundproj}.

\begin{lem}\label{A14}
The functions $\d ^\pi _Y$ satisfy axioms (A$_1$)-(A$_4$) from Definition \ref{projc}.
\end{lem}

\begin{proof}
Axioms (A$_1$) and (A$_2$) obviously hold. The  nontrivial part of the proof is to verify  (A$_3$) and (A$_4$).

\begin{figure}
  \centering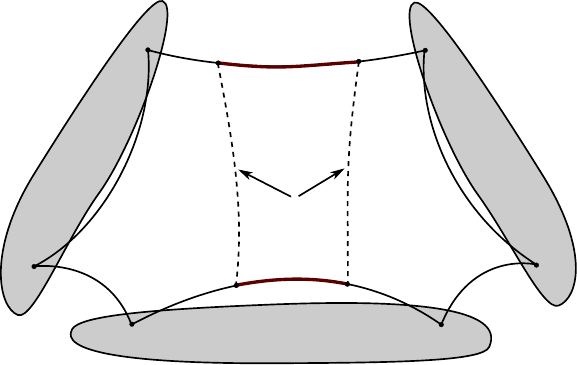\\
  \caption{}\label{44-f4}
\end{figure}

Let us start with  (A$_3$). Let $\e $, $R$, and $c$ be as in the proof of Lemma \ref{boundproj}. We will show that (A$_3$) hold for any $\xi > 6000 c+2\nu$, where $\nu $ is given by Lemma \ref{boundproj}.
Indeed suppose that $\d _Y (A,B)\ge  \xi $. Let $a\in A$, $b\in B$, $x,y\in Y$ be points such that
\begin{equation}\label{dABY}
\d (a,x)\le \d (A,Y)+\delta ,\;\;\; \d(b,y)\le \d(B,Y)+\delta .
\end{equation}
In particular, $x\in {\rm proj}_Y(a)$, $y\in {\rm proj}_Y(b) $, and hence $$\d (x,y)> \d _Y (A,B) -{\rm diam }({\rm proj }_Y A)  -{\rm diam }({\rm proj }_Y B)  \ge \xi -2 \nu > 6000 c.$$
By Lemma \ref{boundproj} it suffices to show that for any $a^\prime\in A$ and any $b^\prime \in {\rm proj} _B(a^\prime)$, we have
\begin{equation}\label{dbbp}
\d (b^\prime , b)\le 6000c .
\end{equation}

Consider the geodesic hexagon $P$ with consecutive vertices $a^\prime, a, x,y, b, b^\prime $ (Fig. \ref{44-f4}). By Lemma \ref{Ols}, there exists a subsegment $u$ of $[x,y]$ and a subsegment $v$ of one of the other $5$ sides of $P$ such that $u$ and $v$ are $13\delta $-close and $\min\{\ell (u),\ell (v)\} \ge c$. As in the proof of Lemma \ref{boundproj}, we can show that $v$ can not be a subsegment of $[x,a]$ or $[y,b]$ and $v$ can not be a subsegment of $[a, a^\prime]$ or $[b,b^\prime]$ by the geometric separability condition  as $A\ne Y$ and $B\ne Y$. Hence $v$ is a subsegment of $[a^\prime , b^\prime]$. For definiteness, assume that $\d (u_+, v_+)\le 13\delta $.

We now consider the geodesic pentagon $Q$ with consecutive vertices $u_+, v_+, b^\prime, b, y$. If $\d (b, b^\prime )>5000 c$, then applying Lemma \ref{Ols} we obtain $13\delta$-close subsegments $w$ and $t$ of $[b, b^\prime]$ and one of the other $4$ sides of $Q$, respectively, which have length at least $c\ge 30\delta $. This leads to a contradiction since $t$ can not be a subsegment of $[v_+, u_+] $ as $\d(v_+, u_+)\le 13\delta$, and $t$ can not be a subsegment of the other $3$ sides for the same reasons as above. Hence  $\d (b, b^\prime )\le 5000 c$. In particular, (\ref{dbbp}) holds. This completes the proof of (A$_3$).

\begin{figure}
  \centering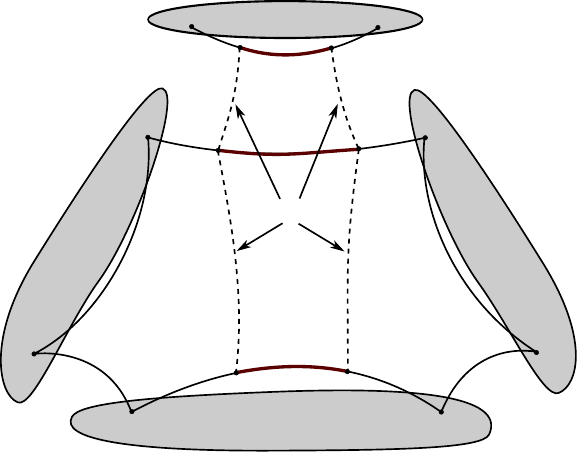\\
  \caption{}\label{44-f5}
\end{figure}

To verify (A$_4$), we take $$\e = 26\delta + 2\sigma $$ and modify $R=R(\e)$ and $$c=\max\{ R+ 2\sigma , 30\delta +\sigma \} $$ accordingly. Again we will prove (A$_4$) for any $\xi > 6000 c+2\nu$. Fix any $a^\prime \in A$ and any $b^\prime \in {\rm proj} _B(a^\prime)$. As above if $\d _Y (A,B)\ge  \xi $, then for any $a\in A$, $b\in B$, $x\in {\rm proj}_Y(a)$, $y\in {\rm proj}_Y(b) $ we have $\d (x,y)> 6000 c$. Consider the geodesic $6$-gon with consecutive vertices $a, a^\prime, b^\prime , b, y,x$ (Fig. \ref{44-f5}). Arguing as in the prof of (A$_3$) we can find subsegments $u$ of $[x,y]$ and $v$ of $[a^\prime, b^\prime]$ such that $u$ and $v$ are $13\delta $-close and $\min\{\ell (u),\ell (v)\} \ge c$. Note that $Y$ is uniquely defined by the subsegment $v$. Indeed if for some $Y^\prime \in \mathbb Y$ and $x^\prime, y^\prime \in Y^\prime $, we also have a subsegment $u^\prime $ of $[x^\prime,  y^\prime]$, which is $13\delta $-close to $v$, then $u$ and $u^\prime $ are $26\delta $-close. Hence $Y=Y^\prime $ by the geometric separability condition  as in the proof of Lemma \ref{boundproj}. Thus the number of $Y$'s satisfying the inequality in (A$_4$) is bounded by the number of subsegments of $[a^\prime, b^\prime]$, which is finite.
\end{proof}

Let {\pky } be the projection complex associated to the set $\mathbb Y$ and the family of projections defined above. We will denote by $\d_{\mathcal P}$ the combinatorial metric on {\pky }. Our definition of projections is $G$-equivariant and hence the (cofinite) action of the group $G$ on $\mathbb Y$ extends to a (cobounded) action on {\pky }. Let $\Lambda =\{1, \ldots , k\} $ and let $$\Sigma =\{ s_1, \ldots , s_k\}\subseteq \mathbb Y,$$ where $s_\lambda=H_\lambda (s)$.

Our next goal is to construct a special generating set of $G$. We proceed as follows.  For every $\lambda\in \Lambda $ and every edge $e\in {\rm Star} (s_\lambda)$ going from $s_\lambda$ to another vertex $v=gH_\mu (s)=g(s_\mu)$, we choose any element $x_e\in H_\lambda gH_\mu $ such that
\begin{equation}\label{dcoset}
\d (s, x_e(s))\le \inf \{ \d (s, y(s)) \mid y\in H_\lambda gH_\mu\} +\delta .
\end{equation}
We will say that $x_e$ has {\it type} $(\lambda,\mu )$.

\begin{rem}\label{edge}
Note that for every $x_e$ as above there is an edge in {\pky } going from $s_\lambda$ to $x_e (s_\mu)$. Indeed $x_e=h_1gh_2$ for some $h_1\in H_\lambda$, $h_2\in H_\mu$, hence $$\d _{\mathcal P}(s_\lambda , x_es_\mu) =\d _{\mathcal P}(h_1^{-1} (s_\lambda ), gh_2(s_\mu ))= \d _{\mathcal P}(s_\lambda , g(s_\mu ))=1.$$
\end{rem}

For every edge $e$ connecting $s_\lambda$ and $g(s_\mu)$, there exists a dual edge, $f=g^{-1}(e)$, connecting $g^{-1}(s_\lambda)$ and $s_\mu $. In addition to (\ref{dcoset}), we can (and will) choose the elements $x_e$ and $x_f$ to be mutually inverse. In particular, the following set
$$
X=\left\{ x_e\ne 1 \;\left| \; e\in \bigcup\limits_{\lambda=1}^k {\rm Star }(s_\lambda)\right.\right\}
$$
is symmetric (i.e., closed under taking inverses). Let also $\mathcal H=\bigsqcup\limits_{\lambda=1}^k H_\lambda$.

\begin{lem}\label{qipk}
The union $X\cup \left(\bigcup\limits_{\lambda=1}^k H_\lambda\right)$ generates $G$ and the Cayley graph $\G $ is quasi-isometric to $\mathcal P _K(\mathbb Y)$.
\end{lem}

\begin{proof} We define a map $\iota\colon  G\to \mathbb Y$ by the rule $\iota (g)=g(s_1)$.

Note that if $x_e\in X$ is of type $(\lambda,\mu)$, then we have
\begin{equation}\label{diam1}
\d _{\mathcal P}(x_e(s_1), s_1)\le \d _{\mathcal P}(x_e (s_1), x_e(s_\mu))+\d _{\mathcal P}(x_e(s_\mu), s_\lambda) +\d _{\mathcal P}(s_\lambda, s_1)\le 2\diam (\Sigma )+1
\end{equation}
(see Remark \ref{edge}). Similarly for every $\lambda\in \Lambda $ and every $h\in H_\lambda$ we have $h(s_\lambda )=s_\lambda $ and hence
\begin{equation}\label{diam2}
\d _{\mathcal P}(h(s_1), s_1)\le \d _{\mathcal P}(h(s_1), h(s_\lambda)) +\d _{\mathcal P}(s_\lambda, s_1)\le 2 \diam (\Sigma ).
\end{equation}
Inequalities (\ref{diam1}) and (\ref{diam2}) can be summarized as $\d _{\mathcal P}(a(s_1), s_1)\le 2\diam (\Sigma )+1$ for any $a\in X\cup\mathcal H$. This immediately implies
$$
\d _{\mathcal P}(\iota(1), \iota (g))\le (2\diam (\Sigma )+1) |g|_{X\cup \mathcal H}.
$$
Thus the map $\iota $ is Lipschitz.

On the other hand, suppose that for some $g\in G$ we have $\d _{\mathcal P}(\iota (1), \iota (g))=r$. If $r=0$, then $gH_1(s)=g(s_1)=s_1=H_1(s)$ and hence $g\in H_1$ by Lemma \ref{abuse}. In particular, $|g|_{X\cup \mathcal H}\le 1$. Let now $r>0$ and let $p$ be a geodesic in {\pky } connecting $s_1$ to $\iota (g)=g(s_1)$. Let $$v_0=s_1,\; v_1,\; \ldots ,\; v_r=g(s_1)$$ be consecutive vertices of $p$. Suppose that $v_i=g_iH_{\lambda_i}(s)=g_i(s_{\lambda_i})$ for some $g_i\in G$ and $\lambda_i\in \Lambda $. We assume that $g_0=1$ and $g_r=g$. Since $g_i(s_{\lambda_i})$ is connected by an edge to $g_{i+1}(s_{\lambda_{i+1}})$, the vertex $s_{\lambda_i}$ is connected to the vertex $g_i^{-1}g_{i+1}(s_{\lambda_{i+1}})$. This means that $g_i^{-1}g_{i+1}=h_iy_ih_i^\prime $ for some $y_i\in X$ and $h_i\in H_{\lambda_{i}}$, $h_i^\prime \in H_{\lambda_{i+1}}$. In particular, $|g_i^{-1}g_{i+1}|_{X\cup\mathcal H} \le 3$. Hence
$$
|g|_{X\cup\mathcal H} = \left| \prod\limits_{i=1}^r g_{i-1}^{-1}g_i\right|_{X\cup\mathcal H} \le \sum\limits_{i=1}^r |g_{i-1}^{-1}g_i|_{X\cup\mathcal H} \le 3r= 3\d (\iota (1), \iota (g)).
$$
As {\pky } is connected, we obtain that $X\cup\mathcal H$ generates $G$ and  $\iota $ is a quasi-isometric embedding of $(G, |\cdot |_{X\cup\mathcal H})$ into {\pky }. Finally note that the vertex set of {\pky } is contained in $(\iota(G))^{+{\rm diam} (\Sigma )}$. Therefore,  $\G $ is quasi-isometric to {\pky }.
\end{proof}

Note that so far we have not used (\ref{dcoset}). However this condition is essential for the next lemma.

\begin{lem}\label{projest}
There exists a constant $\alpha $ such that if for some $Y\in \mathbb Y$ and $x\in X\cup\mathcal H$, we have
\begin{equation}\label{sxs}
{\rm diam} (\pr _Y \{ s, x(s)\} )> \alpha ,
\end{equation}
then $x\in H_\lambda $ and $Y=H_\lambda (s)$ for some $\lambda \in \Lambda $.
\end{lem}

\begin{figure}
  \centering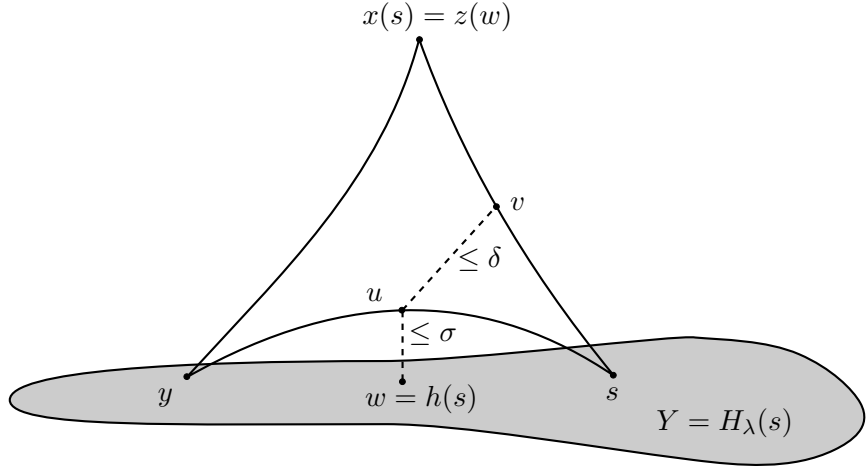\\
  \caption{Case 2 in the proof of Lemma \ref{projest}.}\label{44-f1}
\end{figure}

\begin{proof}
Let $$\alpha =\max \{ K+2\xi , 6\sigma +19\delta, \nu \}, $$ where $\xi $ is the constant from Definition \ref{projc}, and $\nu $ is given by Lemma \ref{boundproj}.

Assume first that $x\in X$. Let $x$ be of type $(\lambda, \mu)$, i.e., there is an edge in {\pky } connecting $H_\lambda (s)$  and $xH_\mu (s)$ (see Remark \ref{edge}). There are three cases to consider. We will arrive at a contradiction in each case thus showing that $x$ cannot belong to $X$.

{\it Case 1.} If $H_\lambda (s)\ne Y\ne xH_\mu (s)$, then $${\rm diam} (\pr _Y \{ s, x(s)\} )\le \d _Y^{\pi} (H_\lambda (s) , xH_\mu (s))\le \d _Y (H_\lambda (s) , xH_\mu (s)) +2\xi \le K +2\xi \le \alpha $$ by the definition of {\pky } and (\ref{dpd}). This contradicts (\ref{sxs}).

{\it Case 2.}  Further suppose that $H_\lambda (s)=Y$. Let $y\in \pr _Y(x(s))$. If $\d (s, y)\le 2\sigma +7\delta $, then by Lemma \ref{PrOfPoint}, we have $${\rm diam} (\pr _Y \{ s, x(s)\} )\le 6\sigma +19\delta\le \alpha.$$ Thus  $$\d (s, y)> 2\sigma +7\delta .$$ Consider the geodesic triangle with vertices $s, x(s), y$. Let $u$ be a point on the geodesic segment $[s, y]$ such that
\begin{equation}\label{duv}
\d (u,y)=\sigma +4\delta
\end{equation}
and let $v\in [s,x(s)]\cup [x(s), y]$ be such that $\d (u,v)\le \delta $. Using the definition of projection and (\ref{duv}) it is easily to show that, in fact, $v\in [s,x(s)]$ (see Fig. \ref{44-f1}). Let $w\in Y$ be such that $\d (u, w)\le \sigma$. Let $h\in H_\lambda $ and $z\in G$ be such that  $h(s)=w$ and $z(w)=x(s)$. We obviously have $x=zht$ for some $t\in Stab_G(s)$. Note that $$\d (s, v)\ge \d (s,u) -\delta > \d (s,y) -\d (y,u) -\delta >\sigma +2\delta .$$ Hence
$$
\begin{array}{rcl}
\d (s, h^{-1}zht(s))& = & \d (h(s), zh(s))=\d (w, x(s))\le \d (w,v) +\d (v, x(s))= \\ && \\
&& \d (w,v)+\d (s, x(s)) - \d (s,v) < \d (s,x(s))  -\delta.
\end{array}
$$
This contradicts (\ref{dcoset}) as $y=h^{-1}zht\in H_\lambda x$.

{\it Case 3.}  The last case when $H_\lambda (s)\ne Y$, but $xH_\mu (s)=Y$ can be reduced to the previous one by translating everything by $x^{-1}$. Indeed in this case ${\rm diam} (\pr _Y \{ s, x(s)\} )= {\rm diam} (\pr _{H_\mu (s)} \{ s, x^{-1}(s)\} )$, $x^{-1}\in X$ as $X$ is symmetric, and $x^{-1}$ has type $(\mu, \lambda )$. So the same arguments apply.

Thus if (\ref{sxs}) holds, then $x\notin X$, i.e., $x\in H_\lambda $ for some $\lambda \in \Lambda $. If $H_\lambda (s)\ne Y $, then ${\rm diam} (\pr _Y \{ s, x(s)\} )\le \nu <\alpha $ again by Lemma \ref{boundproj}. Thus  $H_\lambda (s) =Y$.
\end{proof}

\begin{proof}[Proof of Theorem \ref{crit}]
The Cayley graph $\G $ is hyperbolic by Lemma \ref{qipk} and Proposition \ref{BBF}. It only remains to prove (\ref{hec0}).

Let us take $h\in H_\lambda $ such that $\dl (1, h)=r$. Let $e$ be the edge in $\G $ connecting $h$ to $1$ and labelled by $h^{-1}$. Then by the definition of $\dl $ there exists a path $p$ in $\G $ of length $r$ such that $e$ is an isolated component of the cycle $ep$ in $\G $. Let $\Lab (p)\equiv x_1\ldots x_r$ where $x_1, \ldots , x_r\in X\cup\mathcal H$ and let $$v_0=s,\; v_1=x_1 (s),\; \ldots ,\; v_r=x_1\ldots x_r(s) = h(s).$$ Note that for every $i=1, \ldots , r$, we have
$$
{\rm diam} (\pr _{H_\lambda (s)} \{ v_{i-1}, v_{i}\}) = {\rm diam} (\pr _{Y} \{ s, x_{i}(s)\}),
$$
where $Y=(x_1\ldots x_{i-1})^{-1} H_\lambda (s)$. By Lemma \ref{projest}, we have
\begin{equation}\label{dvi}
{\rm diam} (\pr _{H_\lambda (s)} \{ v_{i-1}, v_{i}\})\le \alpha
\end{equation}
unless $x_{i}\in H_\lambda (s)$ and $(x_1\ldots x_{i-1})^{-1} H_\lambda(s) =H_\lambda (s) $, i.e., $x_1\ldots x_{i-1}\in H_\lambda $. However this would mean that $e$ is not isolated in $ep$. Hence (\ref{dvi}) holds for all $1\le i\le r$ and we obtain $$
\d (s, h(s)) \le {\rm diam} (\pr _{H_\lambda (s)} \{ v_0, v_r\})\le \sum\limits_{i=1}^r {\rm diam} (\pr _{H_\lambda (s)} \{ v_{i-1}, v_{i}\}) \le \alpha r.
$$
\end{proof}


\section{Very rotating families}


In the context of relatively hyperbolic groups, an important space to consider is the cone-off of a Cayley graph, first used by Farb \cite{F} for this purpose. In this graph, each left coset of each parabolic subgroup has diameter $1$. One can also use another type of cone-off, by hyperbolic horoballs, as in Bowditch's definitions \cite{Bow}. This time, the left cosets of parabolic subgroups still have infinite diameter, but their word  metric is exponentially distorted in the new ambient metric.  There are also  mixtures of both choices (see \cite{Gr_Ma}).  In all these spaces, each conjugate of a parabolic subgroup fixes a point (usually unique, possibly at infinity for Bowditch's model), and the rest of the space ``rotates'' around this point, under its action.

Now consider a group, which possibly is no longer relatively hyperbolic, but with some hyperbolically embedded subgroup.  When one suitably cones off such a subgroup,   one may obtain an interesting space, and,   if the residual properties of this subgroup allow it, some interesting dynamics (see Corollary \ref{cor-he-vrf}). This is captured by the definition of rotating families.

On the other hand, given a suitable space with a suitable rotating family, one may infer that the rotating groups are hyperbolically embedded. This is made precise in Corollary \ref{cor;RFtoHE}.

In this section, we first establish a structural result on the
group generated by a suitable rotating family in the spirit of Greendlinger's
lemma. This allows us to show that under relevant assumptions,
quotienting the space by a rotating family preserves hyperbolicity
and acylindricity. Finally, we also provide conditions, and
constructions in the literature leading to such rotating families.

\subsection{Rotating families and windmills}\label{wind}

\subsubsection{Definitions and main results}
In this section, we recall the definition of  rotating families (\ref{defi;rotatingF}), the very rotating assumption,
and the main results we prove about them.

\begin{figure}[htbp]
  \centering\includegraphics[width=5cm]{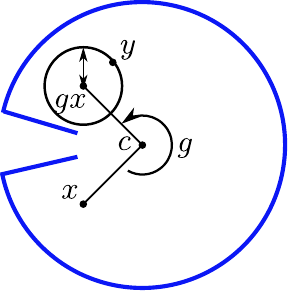}
  \caption{In a very rotating family, $g\in G_c\setminus\{1\}$ rotates by a large angle. Any geodesic $[x,gx]$ contains $c$, and more generally, so does $[x,y]$ for any $y$ close enough to $gx$.}\label{fig_rotating2}
\end{figure}

Assuming that  $\X$  is CAT(0), one can think of the very rotating assumption below
in terms of \emph{large rotation angles} as follows (see Figure \ref{fig_rotating2}).
Assume that
any $g\in G_c\setminus\{1\}$ fixes $c$ and rotates any $x\in \X\setminus\{c\}$
by an angle larger that $\pi$, i.e.\ that the angle between $[c,x]$ and $[c,gx]$ is larger than $\pi$ in the link of $c$.
Then the geodesic joining $x$ to $y=gx$ has to go through $c$, and this is still true
if $y$ is any point close enough to $gx$.
The very rotating condition is a version of this large angle assumption
that makes sense in a hyperbolic space. We only ask it to hold for $x$ such that $d(x,c)\in [20\delta,40\delta]$:
we don't care about what happens to $x$ too close to $c$,
and we will see in Lemma \ref{lem_very_global} that the very rotating condition implies that a similar condition holds for $x$ at distance $>40\delta$ from $c$.

\begin{defn}
\begin{enumerate} \label{defi;Rotating_and_virtues}\label{i-rot1}
 \item[(a)]  (Gromov's rotating families) Let $G\actson \X$ be an action of a group on a metric space.  A rotating family  $\calC = (C, \{G_c, c\in C\}) $ consists of a
   subset $C\subset \X$, and a collection  $\{G_c, c\in C\}$ of subgroups of $G$ such that
  \begin{enumerate}
  \item[(a-1)] $C$ is $G$-invariant,
  \item[(a-2)] each $G_c$ fixes $c$,
  \item[(a-3)] $\forall g\in G \; \forall c\in C\;  G_{gc}= gG_{c} g^{-1} $.
  \end{enumerate}
  The set $C$ is called the set of \emph{apices} of the family, and the groups $G_c$ are called the \emph{rotation subgroups} of the family.

\item[(b)] \label{i-sep} (Separation) One says that $C$ (or  $\calC$)  is $\rho$-\emph{separated} if any two distinct apices are at distance  at least $\rho$.

\item[(c)] \label{i-vrot} (Very rotating condition) When $\X$ is $\delta$-hyperbolic for some $\delta >0$, one says that $\calC$ is \emph{very rotating} if, for all $c\in C, g\in G_c\setminus \{1\}$, and all $x, y \in \X$ with
  both $\d (x,c),\d (y,c)$ in the interval $[20\delta ,  40\delta]$
      and $\d (gx,y)\leq 15\delta$,
          any geodesic between $x$ and $y$ contains $c$.
\item[(d)] \label{i-arot} ($\alpha$-rotating subgroup) A subgroup $H$ of a group $G$ is called \emph{$\alpha$-rotating} if there is an $ \alpha\delta$-separated very rotating family of $G$ acting on a $\delta$-hyperbolic space $\X$ for some $\delta >0$ whose rotation subgroups are exactly the conjugates of $H$. When we want to stress a particular action, we will say that $H$ is \emph{$\alpha$-rotating with respect to the given action} of $G$ on $\X$.
\end{enumerate}
\end{defn}

Depending on the context, it might be more relevant to identify the property of admitting such an action, rather than the action itself. This motivates the following definition.

\begin{defn}
      \label{def;VR}\label{i-arot1}
            A collection of subgroups  $\Nl $  of a group $G$ is  called $\alpha$-rotating if there is   a  $ \alpha\delta$-separated very rotating family of $G$ on a $\delta$-hyperbolic space $\X$, whose rotation subgroups are exactly the conjugates of  elements of $\Nl $. When we want to stress a particular action, we will say that $\Nl$ is \emph{$\alpha$-rotating with respect to the given action} of $G$ on $\X$
\end{defn}

Our goal is to prove the following structure theorem, analogous to
\cite{Del_Duke}.

\begin{thm}\label{theo;app_wind}
Let $G\actson \X$  be a group acting on a $\delta$-hyperbolic geodesic space, 
and $\calC= (C, \{G_c, c\in C\}) $ be a $\rho$-separated very rotating family for some $\rho\geq 200\delta$.
Then the normal subgroup $\Rot=\grp{ G_c|c\in C}\normal G$ satisfies
\begin{enumerate}
      \item[(a)]       $\Rot=\displaystyle *_{c\in C^\prime} G_c$ for some (usually infinite) subset $C^\prime\subset C$.
      \item[(b)] For any $g\in \Rot$, either $g\in G_c$ for some $c\in C$, or $g$ is loxodromic with respect to the action $G\actson \X$ and has an invariant geodesic line on which $g$ acts by translation of length at least $\rho$.
\end{enumerate}
\end{thm}

As a particular case, we get

\begin{cor}\label{coro;app_wind}
    Let $H$ be a $200$-rotating subgroup of a group $G$. Then the normal subgroup of $G$ generated by $H$
is a free product of a  (usually infinite) family of conjugates of $H$.\qed
\end{cor}

Before stating additional corollaries,
we observe that the local very rotating property gives a global condition:
    \begin{lem}[Global very rotating condition]\label{lem_very_global}
      Assume that $\calC = (C, \{G_c, c\in C\})$ is a very rotating family on a $\delta$-hyperbolic space $\X$.

Consider $x_1,x_2\in \X$ such that there exists $q_i\in[c,x_i]$ with $\d (q_i,c)\geq 20\delta$ and $h\in G_c\setminus \{1\}$, such that $\d (q_1,h q_2)\leq 10\delta$.
Then any geodesic between $x_1$ and $x_2$ contains $c$. In particular, for any choice of geodesics $[x_1,c]$, $[c,x_2]$,
their concatenation $[x_1,c]\cup [c,x_2]$ is geodesic.
    \end{lem}

One immediately deduces:

    \begin{cor}\label{cor_free}
     Under the previous condition, for each $c\in C$, $G_c$ acts freely and discretely on $\X\setminus B(c,20\delta)$.
\qed
    \end{cor}

 \begin{proof}[Proof of Lemma \ref{lem_very_global}]
Let $d=\d(q_1,hq_2)$.
We claim that there exists $q'_i\in [c,q_i]$ such that $21\delta\leq \d(c,q'_i)\leq 39\delta$
and such that  $\d(q'_1,hq'_2) \leq d+2\delta$. Indeed, if $d(c,q_1)\geq 39\delta$ or $d(c,q_2)\geq 39\delta$
we can take for $q'_i$ the point at distance $21\delta$ from $c$, and in this case $\d(q'_1,hq'_2)\leq \delta$.
Otherwise, one can take $q'_i$ at distance at most $\delta$ from $q_i$ to ensure that $\d(c,q'_i)\geq 21\delta$.
      The fact that $\calC$ is very rotating implies that every geodesic from $q^\prime_1$ to $q^\prime_2$ contains $c$.

      Let $[x_1,x_2]$ be any geodesic.
      Looking at the triangle $(c,x_1,x_2)$, we see that there are points
$q^{\prime\prime}_1,q^{\prime\prime}_2\in[x_1,x_2]$
such that $\d(q'_i,q''_i)\leq \delta$.
Thus, $\d (q^{\prime\prime}_1,hq^{\prime\prime}_2)\leq d+4 \delta\leq 15\delta $, and $20\delta\leq \d (c,q^{\prime\prime}_i) \leq 40\delta$.
      By the very rotating hypothesis, $[q^{\prime\prime}_1,q^{\prime\prime}_2]\subset [x_1,x_2]$ contains $c$.
    \end{proof}

Using Theorem \ref{theo;app_wind}, we deduce:

    \begin{cor}\label{cor_free2} Under the assumptions of Theorem \ref{theo;app_wind},
      the group $\Rot=\grp{ G_c|c\in C}$ acts freely and discretely on the complement of the $20\delta$-neighborhood of $C$ in $\X$.

      If $h\in \Rot\setminus\{1\}$ and $x_0\in \X$ are such that $\d (x_0,hx_0)< \rho$,
      then $h\in G_c$ for
      some $c\in C$,
      and either $d(x_0,c)\leq 20\delta$ or $\d (c,x_0)=\d (x_0,hx_0)/2$. \qed
    \end{cor}

\label{i-Greend}
Additionally, we are going to prove a refinement of the last assertion of Theorem \ref{theo;app_wind},
which is a qualitative analogue of the classical Greendlinger lemma in small
cancellation theory.
Recall that the Greendlinger lemma guarantees that, for a group
         with a small cancellation presentation, for each word $w$
         in the normal subgroup generated by the relators, there exists $r$ a conjugate of the relators
         such that $|wr|<|w|$.

\begin{figure}[htbp]
  \centering\includegraphics[width=9cm]{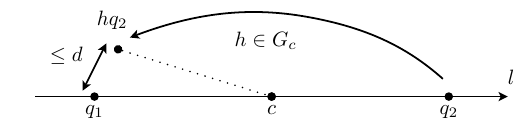}
  \caption{A $d$-shortening pair $\{q_1,q_2\}$ at $c$ on a geodesic $l$}\label{fig_shortening}
\end{figure}

\label{i-shorteningpair}
\begin{defn}\label{dfn_shortening}
Given a geodesic $l$, $d<50\delta$ and a point $c\in l\cap C$,
we say that $\{q_1,q_2\}\subset l$ is a \emph{$d$-shortening pair} at $c$ if
$c\in [q_1,q_2]$,
$\d(c,q_1),\d(c,q_2)\in [25\delta,30\delta]$, and there exists $h\in G_c\setminus\{1\}$
such that $\d (q_1,hq_2)\leq d$ (see Figure \ref{fig_shortening}).
\end{defn}

Note that the condition implies that $\d(q_1,hq_2)\leq \d(q_1,q_2)- (50\delta-d)<\d(q_1,q_2)$.
In particular, $[q_1,q_2]$ does not map to a geodesic segment in $\X/\grp{G_c}$.

\begin{lem}(Qualitative Greendlinger lemma) \label{lem;green_unpointed}
      Let $\X$  be a  hyperbolic geodesic  space, equipped with a
 $200\delta$-separated very rotating family
      $\calC= (C, \{G_c, c\in C\}) $, and consider
 $\Rot=\grp{G_c|c\in C}$ as above.

For any $g\in \Rot\setminus\{1\}$, either $g\in G_c$ for some $c\in C$, or $g$ is loxodromic in $\X$,
it has an invariant geodesic line $l$, and $l\cap C$ contains at least two distinct points in a $g$-orbit
at which there is a $3\delta$-shortening pair.
\end{lem}

We also give a pointed version:

\begin{lem}(Pointed qualitative Greendlinger lemma) \label{lem;green_pointed}
 In the situation above,
given $g\in \Rot\setminus\{1\}$ and $p_0\in \X$,
either $g\in G_c$ and $d(p_0,c)\leq 25\delta$ for some $c\in C$,
or any geodesic $[p_0,gp_0]$ contains a
 $5\delta$-shortening pair at some $c\in [p_0,gp_0]\cap C$.
\end{lem}

A consequence of the qualitative lemma is the following form of linear isoperimetric inequality:
if $g\in \Rot$ is such that its translation length is at most $l$, then it is a product
of at most $Kl$ elements of $\cup_{c\in C} G_c$ for some constant $K= \frac{1}{47\delta}$.
We will prove these lemmas in Subsection \ref{subsec_green}.


\subsubsection{Windmills and proof of the structure theorem}

\begin{figure}[htbp]
  \centering\includegraphics[width=10cm]{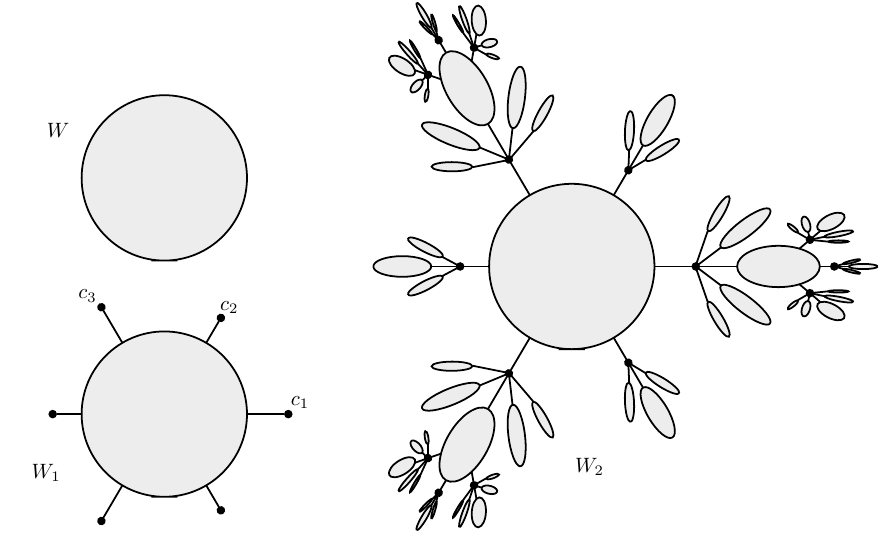}
  \caption{A windmill}\label{fig_windmill}
\end{figure}

The  goal of this section is to prove Theorem \ref{theo;app_wind}
giving the free product structure of the normal subgroup generated by a very rotating family.
Our proof  follows an argument of  Gromov in a CAT(0) setting
\cite{Gro_cat}.
Let us briefly sketch the argument. It may be helpful to think that $\X$ is CAT(0) (so that the notion of angle makes sense)
and to assume that every element in $G_c\setminus \{1\}$
rotates any point in $\X\setminus\{c\}$ by an angle larger than $\pi$.
Start with any apex $c\in C$, and consider a small ball around $c$.
Let its radius increase until it comes sufficiently close to some $c'\in C$
(like one of the $c_i$'s on Figure \ref{fig_windmill}).
Because the points in $C$ are far from each other, this is now a \emph{big} ball $W$
(see Figure \ref{fig_windmill}).
The key point implied by the convexity of $W$ and the very rotating condition at $c'$ is that all translates of $W$
under $G_{c'}$ are disjoint; even more:
for any $g\in G_{c'}\setminus\{1\}$, any geodesic joining a point in $W$ to a point in $gW$ has to go through $c'$.
Since $W$ is $G_c$-invariant, we have a similar picture at any point in the $G_c$-orbit of $c'$
(in the proof, we rather consider the collection of all points outside $W$ but close enough to $W$).
Now \emph{unfold} $W$ by taking the union of all its translates by the action of the group \emph{generated by} $G_c$ and $G_{c'}$ (this is $W_2$ on Figure \ref{fig_windmill}). The main claim is that
this collection of balls has a tree-like structure.
Indeed, consider a word $w=g_1h_1\dots g_nh_n$ with $g_i\in G_c$ and $h_i\in G_{c'}$,
and two points $a\in W$, $b\in gW$.
The word $w$ naturally defines a broken geodesic $[a,c_1]\cup[c_1,c_2]\cup\dots\cup [c_n,b]$
that starts from $a$,
goes to $c_1=g_1c'=g_1h_1c'$, then to $c_2=g_1h_1g_2c'=g_1h_1g_2h_2c'$,
etc.
Thanks to the very rotating assumption, the key point above shows that this broken geodesic is a local geodesic,
 hence a global one. This implies that our collection of balls is tree-like, and that
the group $G'$ generated by  $G_c\cup G_{c'}$ is a free product of these two groups.
Moreover, a suitable neighborhood $W'$ of $W_2$ will be convex. Because of
its shape, we call $(W',G')$ a \emph{windmill}.
This whole procedure will be applied inductively: starting from a windmill $(W,G_W)$, we produce
a larger windmill $(W',G_{W'})$ where $W'$ is convex, and $G_{W'}$ is a free product of $G_W$ with some
rotation groups. In this process, the windmills will exhaust $\X$, and the corresponding
groups will exhaust the (normal) subgroup generated by $\{G_c|c\in C\}$.
 Although not unrelated, our windmills are not the same as and McCammond and Wise's \cite{McCW_windmills}.
\\

We give an axiomatic definition of windmills in Definition \ref{dfn_windmill} below, and proposition \ref{prop_grow}
is the iterative step allowing to construct a larger windmill from an existing one.
Axioms \ref{wm_freeprod}-\ref{wm_elliptic} of this definition say
that the theorem 
applies to $G_W$.
Axiom \ref{wm_elliptic} also
implies 
a weak version of the unpointed Greendlinger's Lemma \ref{lem;green_unpointed} which
asks for $2$ distinct shortening pairs.
Axiom \ref{wm_Cloin} is a technical assumption saying that $W$ does not get too close
to any apex in $C$.

\newcounter{memowm}
    \begin{defn}[Windmill]\label{dfn_windmill}
      Let $\X$ be a $\delta$-hyperbolic metric space, and  $\calC= (C, \{G_c, c\in C\}) $
      a $\rho$-separated very rotating family on $\X$. A \emph{windmill} for $\calC$ is a subset $W$ of $\X$ satisfying the following axioms.

      \begin{enumerate}
        \item \label{wm_4delta} $W$ is $4\delta$-quasiconvex,
        \item \label{wm_Cloin} $W^{+50\delta}\cap C = W\cap C \neq \emptyset$,
        \item \label{wm_invar} The group $G_W$ generated by $\bigcup_{c\in W\cap C} G_c$ preserves $W$.
        \item \label{wm_freeprod} There exists a subset $S_W\subset
          W\cap C$ such that $G_W$ is the free product  $\displaystyle
          *_{c\in S_W} G_c$.
        \item \label{wm_elliptic} Every elliptic element of $G_W$ lies in some $G_c, c\in W\cap C$.
          Every non-elliptic element of $G_W$ is loxodromic, of translation length at least $\rho$, and has an invariant geodesic line $l\subset W$.
          Moreover, any such $l$ contains a point $c\in C$ at which there is a $\delta$-shortening pair.
      \end{enumerate}
      \setcounter{memowm}{\theenumi}
    \end{defn}

    We first note that if $C$ is $\rho$-separated with $\rho\geq 200\delta$, then for any $c\in C$, the ball $W=B(c,100\delta)$ is a windmill because $W\cap C=\{c\}$, so $G_W=G_c$.

Our iterative procedure for the proof of Theorem \ref{theo;app_wind} is contained in the following proposition.
It is illustrated on Figure \ref{fig_windmill}, where starting from a windmill $W$, one gets a new windmill $W'$
as a small thickening of $W_2$.

    Recall that if $Q\subset \X$, we write $Q^{+r}$ for the set of points within distance
    at most $r$ from $Q$.

    \begin{prop}[Growing windmills]\label{prop_grow}
       Let $G$ act on a $\delta$-hyperbolic space $\X$, and  $\calC= (C, \{G_c, c\in C\}) $ be a $\rho$-separated very rotating
       family, with $\rho\geq 200\delta$.

      Then for any windmill $W$, there exists a windmill $W^\prime$ containing $W^{+10\delta}$ and  $W^{+60\delta}\cap C$,
 such that $G_{W'}=G_W*(*_{x\in S} G_c)$ for some (maybe infinite) subset $S\subset \calc\cap(W'\setminus W)$.
    \end{prop}

\begin{proof}[Proof of Theorem \ref{theo;app_wind} using Proposition \ref{prop_grow}]
First choose $c_0\in C$. Then as noticed above, $W=B(c_0,100\delta)$ is a windmill.
Define inductively $W_{n+1}$ as the windmill obtained from $W_n$ using  Proposition \ref{prop_grow}.
Then $\bigcup_{n\in \mathbb{N}} W_n=\X$ since $W_{n+1}$ contains the $10\delta$-neighborhood of $W_n$.
Consider $C_0=\{c_0\}$, and let $S_{n+1}\subset C\cap W_{n+1}$ be such that
$G_{W_{n+1}}=G_{W_n}*\left( \displaystyle *_{c\in S_{n+1}} G_c\right)$, and $S_\infty=\cup_{n\geq 0} S_n$.
Since $\Rot=\displaystyle \cup_{n\geq 0} G_{W_n}$,
we have $\Rot=\displaystyle *_{c\in S_{\infty}} G_c$.

Given any element $g\in \Rot=\grp{G_c|c\in C}$, $g$ lies in some $\grp{G_{c_1},\dots, G_{c_k}}$,
so $g\in G_{W_n}$ as soon as $W_n$ contains $\{c_1,\dots,c_k\}$.
The last statement of Theorem \ref{theo;app_wind}
then follow from Axiom \ref{wm_elliptic} of a windmill.
\end{proof}

We now prove  Proposition \ref{prop_grow} (Lemmas from now on to \ref{lem;Wpr_4deltaqc} are dedicated to this).

Assume that $W^{+60\delta}\cap C=\es$.
Then $W^\prime=W^{+10\delta}$ is clearly a windmill with $G_{W^\prime}=G_W$, and we are done.
Therefore, we assume that the set $C_1=W^{+60\delta}\cap C$ is non-empty. By Axiom \ref{wm_Cloin},
all points of $C_1$ are at distance
at least $50\delta$ from $C$, and  
$C_1$ is $G_W$-invariant by Axiom \ref{wm_invar}.

For all $c\in C_1$, let $\bar c$ be a closest point to $c$ in $W$,
and $[c,\bar c]$ a geodesic segment.
 Note that $G_W$ acts freely on $C_1$ by Axiom \ref{wm_elliptic} and Corollary \ref{cor_free},
 so one can make this choice in a $G_W$-equivariant way.
Define 
$W_1= W\cup \left(\bigcup_{c\in C_1} [c, \bar c]\right)$.

Note that $W_1\cap C=(W\cap C)\cup C_1$ since points of $C$ are at distance at least $\rho> 60\delta$ from each other.
The group $G_{W_1}$ generated by $\{G_c| c\in W_1\}$
is the group generated by $G_W$ and by $\{G_c| c\in C_1\}$.
Finally, we define $W_2 = G_{W_1} W_1$ and $W^\prime = W_2^{+10\delta}$
(we \emph{unfold} to get $W_2$, and then thicken to get $W'$, see Figure \ref{fig_windmill}).
Note that by construction, $G_{W_1}=G_{W_2}$, and $W^\prime$ contains $W^{+10 \delta}$ and $C_1$.

It remains to check that $W^\prime$ is a windmill.

As $\calc$ is $200\delta$-separated, 
we have
$d(c,W_1)>60\delta$ for any $c\in C\setminus W_1$.
For each $c\in C\setminus W_2$, $d(c,W_2)=d(c,gW_1)=d(g\m c,W_1)$ for some $g\in G_{W_1}$,
and since $g\m c\notin W_1$, $d(c,W_2)> 60\delta$.
It follows that
$W'^{+50\delta}\cap C=W_2^{+60\delta}\cap C=W_2\cap C\subset W'\cap C$
 so $W^\prime$ satisfies Axiom \ref{wm_Cloin} of a windmill.
Since $W^\prime\cap C=W_2\cap C$, $G_{W^\prime}=G_{W_2}=G_{W_1}$.
Axiom \ref{wm_invar} follows.

\begin{lem}\label{lem;W1qc}
  $W_1$ is $6\delta$-quasiconvex.
\end{lem}

\begin{proof}
Consider $x_1,x_2\in W_1$, and $[x_1,x_2]$ a geodesic of $\X$ joining them.
Assume for instance that $x_1\in [c_1,\ol c_1]$ and $x_2\in [c_2,\ol c_2]$ for some $c_1,c_2\in C_1$.
Then $[x_1,x_2]$ is contained in the $2\delta$-neighborhood of $[c_1,\ol c_1]\cup [\ol c_1,\ol c_2]\cup [\ol c_2,c_2]$.
Since $W$ is $4\delta$-quasiconvex, $[\ol c_1,\ol c_2]$ is contained in the $4\delta$-neighborhood of $W$.
The other cases are similar, which proves the Lemma.
\end{proof}

\begin{rem}\label{rem_qc}
If $c_0$ is some point in $C_1$, the lemma also applies to $W\cup \bigcup_{c\in C_1\setminus\{ c_0\}} [c,\ol c]$.
\end{rem}

\begin{lem}\label{lem_local_C}
Consider $c\in C_1$ and $h\in G_{c}\setminus \{1\}$.
Let $[c,x]$ and $[c,y]$ be two geodesics
that intersect $W\cup (C_1\setminus \{c\})$.

Then $[x,c]\cup [c,hy]$ is geodesic, and
any geodesic joining $x$ to $hy$ contains $c$ and a $\delta$-shortening pair at $c$.

In particular,
$W_1\cup h W_1$ is $6\delta$-quasiconvex.
\end{lem}




\begin{proof}
Let $W'_1=W\cup\bigcup_{c'\in C_1\setminus \{c\}}[c',\ol c']$.
We prove the lemma under the weaker assumption that $[c,x]$ and $[c,y]$ intersect $W'_1$.
Consider $x'\in [c,x]\cap W'_1$,
and $y'\in [c,y]\cap W'_1$.
Since $W'_1$ is $6\delta$-quasiconvex by Remark \ref{rem_qc},
$[x',y']$ is contained in the $6\delta$-neighbourhood of $W'_1$.

Consider $q_1\in [c,x]$ and $q_2\in  [c,y]$ at distance $28\delta$ from $c$.
By hyperbolicity of the triangle $(c,x', y')$, if $d(q_1,q_2)>\delta$, then
there exists  $q_3\in [x',y']$ at distance $\leq \delta$ from $q_1$.
Then $d(q_1,W'_1)\leq d(q_1,q_3)+6\delta\leq 7\delta$
so $d(c,W'_1)\leq 35\delta$ contradicting Axiom \ref{wm_Cloin}.
Therefore $d(q_1,q_2)\leq \delta$.

The global very rotating property (Lemma \ref{lem_very_global})
implies that $[x,c]\cup [c,hy]$ is geodesic, and that any geodesic joining $x$ to $hy$ contains $c$.
Moreover, since $d(q_1,q_2)\leq \delta$, $\{q_1,hq_2\}$ is a $\delta$-shortening pair at $c$
in $[x,c]\cup [c,hy]$.

Consider $\gamma$ any other geodesic joining $x$ to $hy$, we know that it contains $c$, and we prove that
$\gamma$ contains a $\delta$-shortening pair at $c$. The argument is the same as the one above:
consider the points $x'',hy''\in \gamma$ defined by
$\d(x'',x)=\d(x',x)$ and $\d(hy'',hy)=\d(hy',hy)$.
In particular, $\d(x',x'')\leq \delta$ and $\d(y',y'')\leq \delta$.
Define $q'_1\in [c,x'']\subset \gamma$, $q'_2\in [c,y'']\subset h\m\gamma$
at distance $28\delta$ from $c$.
Since $[x'',y'']$ lies in the $2\delta$ neighbourhood of $[x',y']$,
hence in the $8\delta$-neighbourhood of $W'_1$,
we get as above that if $d(q'_1,q'_2)>\delta$,
$d(c,W'_1)\leq 28\delta+\delta+8\delta=37\delta$ a contradiction.

%

To prove the $6\delta$-quasiconvexity of $W_1\cup hW_1$, consider $x,y\in W_1$.
If $x,y\in W'_1$, then any geodesic $[x,hy]$ contains $c$, and we conclude using the  $6\delta$-quasiconvexity
of $W_1$. Assume that $x\in [c,\ol c]$, and $y\in W'_1$, the other cases being similar.
Then every geodesic from $\ol c$ to $hy$ contains $c$, so
we have triangle equalities
$\d(\ol c,hy)=\d(\ol c,c)+d(c,hy)=\d(\ol c,x)+d(x,hy)$,
so for any geodesic $[x,hy]$, $[\ol c,x]\cup [x,hy]$ is a geodesic, so $[x,hy]$ has to contain $c$.
We conclude as above
using the  $6\delta$-quasiconvexity
of $W_1$.
\end{proof}

We now prove that $W_2$ is tree-like.
Consider the bipartite graph $\Gamma$ whose vertices are
the images of $W$ under $G_{W_1}$, together with the points of $G_{W_1}.C_1$.
We put an edge between $gW$ and $hc$  if $d(hc,gW)\leq 60\delta$ i.e.\ if $g\m hc\in C_1\cup W$.

\begin{figure}[htbp]
  \centering\includegraphics[height=3cm]{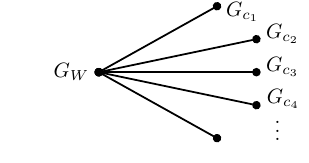}
  \caption{The graph of groups $\Lambda$}\label{fig_gog}
\end{figure}

Let $\Tilde C_1\subset C_1$ be a set of representatives of the orbits of the action $G_W$ on $C_1$ ($\Tilde C_1$ needs not be finite).
We consider a graph of groups $\Lambda$ whose fundamental group is $G_W*(*_{c\in\Tilde C_1} G_c)$ as
in Figure \ref{fig_gog}:
its underlying graph is a tree,
it has a central vertex with vertex group $G_W$,
and for each $c\in \Tilde C_1$, it has a vertex with vertex group $G_c$
joined to the central vertex by an edge with trivial edge group.
Let $\phi:\pi_1(\Lambda)\ra G_{W_2}$ be the map induced by the inclusions of the vertex groups in $G_{W_2}$.
Let $T_\Lambda$ be the Bass-Serre tree of this graph of groups,
and $v_W\in T_\Lambda$ (resp. $v_c\in T_\Lambda$) the vertex fixed by $W$ (resp.\ by $G_c$
for $c\in\Tilde C_1$).
Let $f:T_\Lambda\ra \Gamma$ the $\phi$-equivariant map sending $gv_W$ to $gW$ and $gv_c$ to $gc$.
Denote by $V_W\subset T_\Lambda$ be the set of vertices of $T_\Lambda$ in the orbit of $v_W$,
and by $V_C$ the vertices in $T_\Lambda\setminus V_W$, \ie corresponding to an element of $C$.
Note that $T_\Lambda$ is bipartite for this partition of vertices.

We are going to prove that $f$ is an isomorphism of graphs.
To each segment $[u,v]$ in $T_\Lambda$, we associate a path $\gamma_{[u,v]}$ in $\X$,
depending on some choices, as follows.
Assume first that $u,v\in V_C$.
Half of the vertices in $[u,v]$ lie in $V_C$, denote them by $u=v_0,v_1,\dots, v_n=v$.
Let $c_i=f(v_i)$ be the element of $C$ corresponding to $v_i$.
We define $\gamma_{[u,v]}$ as the concatenation of some chosen geodesics
$[c_i,c_{i+1}]$ in $\X$.
In the remaining case, $u= gv_W$ or $v= g^\prime v_W$ for some $g,g^\prime\in G_{W_2}$.
Still denote by $v_0,v_1,\dots, v_n$ the vertices of $[u,v]\cap V_C$,
and choose any point in $p\in gW$
(resp.\  $p^\prime\in g^\prime W$).
We then define $\gamma_{[u,v]}$ as the concatenation $[p,c_0]\cup \gamma_{[v_0,v_n]}\cup [c_n,p^\prime]$.
In the degenerate case where $u=v\in V_W$ we define $\gamma_{[u,v]}=[p,p^\prime]$.

\begin{lem}
  For every choice, $\gamma_{[u,v]}$ is a geodesic in $\X$.
\end{lem}

\begin{proof}
We can translate the segments $[p,c_{0}]$ and $[c_0,c_{1}]$
 by a suitable element in $G_{W_1}$ so that they
 satisfy the hypotheses of Lemma \ref{lem_local_C}.
 We thus get that the concatenation $[p,c_{0}]\cup [c_0,c_{1}]$
 is geodesic.
Applying \ref{lem_local_C} again to $[p,c_{0}]\cup [c_0,c_{1}]$ and $ [c_1,c_2]$, we get that $[p,c_{0}]\cup [c_0,c_{1}]\cup [c_1,c_2]$ is geodesic.
By induction, we see  that $\gamma_{[u,v]}$ is geodesic (whatever the choices).
\end{proof}

\begin{lem}
  $f:T_\Lambda\ra \Gamma$ is an isomorphism of graphs,
and $\phi:\pi_1(\Lambda)\ra G_{W_2}$ is an isomorphism.
\end{lem}

\begin{proof}
The maps $f$ and $\phi$ are clearly onto.

Consider $u\neq v\in T_\Lambda$, and $\gamma_{[u,v]}$ a corresponding path.
Since $\gamma_{[u,v]}$ is geodesic, its endpoints are distinct (whatever the choices).
This implies that $f$ is injective and moreover that for all $g\notin G_W$, $gW\cap W=\es$.

It follows that $\phi$ is injective: since edge stabilizers are trivial, for any $g\in G\setminus \{1\}$
there exists a vertex $x\in T_\Lambda$ such that $gx\neq x$. Since $f$ is injective, $\phi(g)f(x)=f(gx)\neq f(x)$ so
$\phi(g)$ is non-trivial.
\end{proof}

This establishes that $G_{W'}=G_{W_2}\simeq \pi_1(\Lambda)=G_W*_{c\in \Tilde C_1} G_c$,
and in particular that $W^\prime$ satisfies axiom \ref{wm_freeprod} of a windmill.

\begin{lem}\label{lem_short}
Let $[u,v]$ be a segment of $T_\Lambda$, and let $(u,v)=[u,v]\setminus \{u,v\}$.
For all $v_c\in (u,v)\cap V_C$, $\gamma_{[u,v]}$ contains a $\delta$-shortening pair at $c=f(v_c)$.
\end{lem}

\begin{proof}
Consider $v_c  \in (u,v)\cap V_C$, $w_1,w_2\in V_W$ be the two neighbors of $v_c$ in $[u,v]$.
Up to translation by some element of $G_{W_2}$, we can assume that $w_1$ is the base point $v_W$.
Consider $h\in G_c$ such that $w_2=hw_1$.
Write $\gamma_{[u,v]}$ as the concatenation $[p,c]\cup [c,p^\prime]$.
Then Lemma \ref{lem_local_C} applies to $[p,c]$ and $h\m [c,p^\prime]$, so $\gamma_{[u,v]}$ contains a
$\delta$-shortening pair at $c$.
\end{proof}

\begin{lem}\label{lem_toutes}
For all $p,p^\prime\in G_{W_2}.(W\cup C_1)$,
any geodesic $[p,p^\prime]$ coincides with some $\gamma_{[u,v]}$.
\end{lem}

\begin{proof}
Assume for instance that $p\in gW$ and $p^\prime\in g^\prime W$ for some $g,g'\in G_{W_2}$, the other cases being similar.
We fix some geodesic $[p,p^\prime ]$ of $\X$.
Let $u=g v_W$ and $v=g^\prime v_W$ be the corresponding vertices of $T_\Lambda$,
and consider $\gamma_{[u,v]}=[p,c_0]\cup[c_0,c_1]\dots,[c_{n-1},c_n]\cup [c_n,p^\prime]$  corresponding to this choice of $p,p^\prime$.
We need only to prove that $c_i\in [p,p^\prime]$.
By Lemma \ref{lem_short}, $\gamma_{[u,v]}$ contains a   
$\delta$-shortening pair $\{q_1,q_2\}$ at $c_i$.
 Consider $h\in G_{c_i}\setminus \{1\}$ such that $d(q_1,hq_2)\leq\delta$.
Since $[p,p^\prime]$ is contained in the 
$\delta$ neighborhood of $\gamma_{[u,v]}$, consider $q'_i\in [p,p']$
at distance at most $\delta$ from $q_i$.
Then $d(q'_1,hq'_2)\leq 3\delta$ and
the global very rotating condition (Lemma \ref{lem_very_global})
then implies that $c_i\in  [p,p^\prime]$.
\end{proof}

\begin{lem}\label{lem;Wpr_4deltaqc}
  $W_2$ is 
$8\delta$-quasiconvex, and $W^\prime$ is $4\delta$-quasiconvex.
\end{lem}

\begin{proof}
Given any two points $p,p^\prime$ in $G_{W_2}.(W\cup C_1)$, look at the corresponding vertices $u,v$ in $T_\Lambda$.
By Lemma \ref{lem_toutes}, any geodesic joining $p$ to $p^\prime$ coincides with some
 geodesic $\gamma_{[u,v]}$.
This geodesic is a concatenation of geodesic segments joining two points in a translate of $W_1$.
Since $W_1$ is $6\delta$-quasiconvex, $\gamma$ lies in the $6\delta$-neighborhood of $W_2$.

Now consider the case where $p\in g [c,\ol c]$ for some $g\in G_{W_2}$ and some $c\in C_1$,
and $p'\in g' [c',\ol c']$ with $g'\in G_{W_2}$, $c'\in C_1$, the remaining cases being similar.
Since $[p,p']$ lies in the $2\delta$-neighbourhood of $[p,gc]\cup [gc,g'c']\cup [g'c',p']$
with $[p,gc]\cup [p',gc']\subset W_2$, and since $[gc,g'c']$ is contained in the $6\delta$-neighbourhood of $W_2$,
$W_2$ is $8\delta$-quasiconvex.
By Lemma \ref{lem_4delta}, $W^\prime=W_2^{+10\delta}$ is $4\delta$-quasiconvex.
\end{proof}

This shows axiom \ref{wm_4delta}.
To prove axiom \ref{wm_elliptic}, consider $g\in G_{W^\prime}$, and look at its action on $T_\Lambda$.
If it fixes a vertex of $T_\Lambda$, then either $g\in G_c$ for some $c\in W^\prime$, or $g$ is contained $G_W$ up to conjugacy,
and we can conclude using that $W$ is a windmill.
If $g$ acts hyperbolically on $T_\Lambda$, let $[c,gc)\subset T_\Lambda$ be a fundamental domain of its axis for the action of $g$,
with endpoints in $V_C$. Let $\gamma_{[c,gc]}$ be a geodesic of $\X$ associated to $[c,gc]$ as above.
Then $l=g^\bbZ \gamma_{[c,gc]}$ is a $g$-invariant geodesic of $\X$ so $g$ is loxodromic in $\X$.
Since $c,gc\in W_2$, and $W_2$ is $8\delta$-quasiconvex, $l\subset W_2^{+8\delta}$ and in particular, $l\subset W'$.
Moreover, $l$ contains a $\delta$-shortening pair at the image of each vertex of $V_C$ by Lemma \ref{lem_short}.

This concludes the proof of Proposition \ref{prop_grow}.

\subsubsection{Greendlinger's lemmas}\label{subsec_green}

\label{i-Greend2}
We now prove the two Greendlinger's Lemmas \ref{lem;green_unpointed} and \ref{lem;green_pointed}.
As noted above, the windmill Axiom \ref{wm_elliptic} implies a weak version
of the unpointed Greendlinger's Lemma with one shortening pair instead of two.
We will use inductively the existence of two distinct shortening pairs to prove the pointed
Greengliger's Lemma, and the pointed Greendlinger's Lemma to prove the existence of two distinct
shortening pairs.

The following definition is a technical version of a shortening pair
that is needed for our induction
(shortening pairs were defined in Definition \ref{dfn_shortening}, see Figure \ref{fig_shortening}).
Recall that the Gromov product $(u|v)_a=\frac{1}{2}(\d(u,a)+\d(v,a)-\d(u,v))$ measures
the distance between $a$ and $[u,v]$ in a comparison tree.

\begin{figure}[htbp]
  \centering\includegraphics{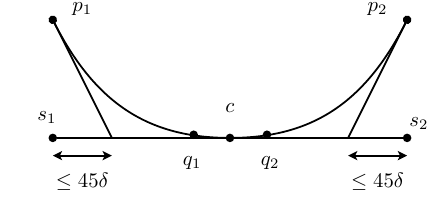}
  \caption{$\{s_1,s_2\}$ is a $d$-security pair for $c$: if $[p_1,p_2]$ comes close enough to $\{s_1,s_2\}$,
it contains a $d$-shortening pair $\{q_1,q_2\}$}\label{fig_security}
\end{figure}

\begin{defn}[Figure \ref{fig_security}]
Given $d<50\delta$,
 a $d$-security pair for $c\in C$ is a pair of points $s_1,s_2\in \X\setminus\{c\}$ such that $c\in [s_1,s_2]$,
and such that given any $p_1,p_2\in\X$ such that $(p_i|c)_{s_i}\leq 45\delta$,
any geodesic $[p_1,p_2]$ contains $c$ and a $d$-shortening pair at $c$.
\end{defn}

Informally, this means that if some geodesic $[p_1,p_2]$ comes close enough to $s_1$ and $s_2$,
then it contains $c$ and a $d$-shortening pair at $c$.
The definition implies that $[s_1,s_2]$ contains a $d$-shortening pair at $c$, taking $p_i=s_i$.

The first example of a security pair is given by rotating a point by an element of $G_c$.

\begin{lem}\label{lem_security_elliptic}
  Let $c\in C$, and $g\in G_c\setminus\{1\}$. For all $s\in \X$ such that $\d(s,c)\geq 75\delta$,
$\{s,gs\}$ is a $2\delta$-security pair for $c$.
\end{lem}

\begin{proof}
To unify notations, let $s_1=s$, $s_2=gs$.
Let $q_i\in [c,s_i]$ be the point at distance $27\delta$ from $c$, and note that $q_2=gq_1$.
Let $p_i$ be such that $(p_i|c)_{s_i}\leq 45\delta$. Then looking at the triangle $c,s_i,p_i$,
the fact that $\d(c,q_i)\leq (p_i|c)_{s_i}-\delta$ implies that there is a point $q'_i\in [c,p_i]$ at distance at most $\delta$ from $q_i$. In particular, $d(q'_2,gq'_1)\leq 2\delta$.
By the global very rotating condition in Lemma \ref{lem_very_global}, $c\in [p_1,p_2]$.
Then $\{q'_1,q'_2\}$ is a $2\delta$-shortening pair at $c$ in $[p_1,p_2]$.
It follows that $\{s_1,s_2\}$ is a $2\delta$-security pair for $c$.
\end{proof}

The definition of a security pair implies that $d(s_i,c)\geq 45\delta+25\delta$ since otherwise,
one could take $p_i$ at distance $< 25\delta$ from $c$,
preventing $[p_1,p_2]$ from containing a shortening pair at $c$.
Conversely, a $d$-shortening pair between two points that are far enough is a $d+2\delta$-security pair:

\begin{lem}\label{lem_security_axis}
  Let $l$ be a geodesic containing a $d$-shortening pair $\{q_1,q_2\}\subset l$ at $c\in l\cap C$
for some $d\leq 48\delta$.

\begin{enumerate}\item
  Then for any $s_1,s_2\in l$ with $c\in [s_1,s_2]$ and $d(c,s_i)>
  80\delta$, $\{s_1,s_2\}$ is a $(d+2\delta)$-security pair for $c$.

\item
  Similarly, for any $s_1,s_2\in \X$ whose projections $u_1,u_2$ on $l$
  are such that $c\in [u_1,u_2]$ and $d(c,u_i)\geq 40\delta$, and
  $d(s_i,l)\geq 50\delta$, then $\{s_1,s_2\}$ is a
  $(d+2\delta)$-security pair for $c$.

\item
  For any $p_1,p_2\in \X$ whose projections $u_1,u_2\in l$
  are such that $c\in [u_1,u_2]$ and $d(c,u_i)> 34\delta$, then $[p_1,p_2]$ contains $c$ and a
  $(d+2\delta)$-shortening pair at $c$.
\end{enumerate}
\end{lem}

\begin{proof}
Let us start with Assertion 3.
Looking at the quadrilateral $(p_1,u_1,u_2,p_2)$, 
we have that $[q_1,q_2]$ is in the $2\delta$-neighbourhood of $[p_1,p_2]\cup[p_1,u_1]\cup [p_2,u_2]$.
Since $d(c,u_i)> 34\delta$, and $d(c,q_i)\leq 30\delta$,
no point in $[q_1,q_2]$ can be $2\delta$-close to a point in $[p_i,u_i]$,
so $[q_1,q_2]$ lies in the $2\delta$-neighbourhood
of $[p_1,p_2]$.
Applying the local very rotating conditon to
the projections $q^\prime_1,q^\prime_2$ of $q_1,q_2$ on $[p_1,p_2]$,
we see that $c\in [q'_1,q'_2]\subset [p_1,p_2]$.
Consider the points $q''_i\in [c,p_i]$ such that $d(c,q''_i)=d(c,q_i)$;
since the triangle $(c,u_i,p_i)$ is  $\delta$-thin,
$d(q''_i,q_i)\leq \delta$.
It follows that $\{q''_1,q''_2\}$ is a $(d+2\delta)$-shortening pair in $[p_1,p_2]$.

We now prove Assertion 1.
Let $p_i$ be such that $(p_i|c)_{s_i}\leq 45\delta$.
Then the projection $u'_i$ of $p_i$ on $[c,s_i]$ is at distance $\leq 46\delta$ from $s_i$.
Since $\d(c,s_i)> 80\delta$, $d(u'_i,c)> 34\delta$. It follows that  the projection $u_i$ of $p_i$ on $l$
satisfies $d(u_i,c)> 34\delta$, and Assertion 3 shows that $[p_1,p_2]$ contains a $(d+2\delta)$-shortening pair.
This shows that $\{s_1,s_2\}$ is a $(d+2\delta)$-security pair for $c$, which concludes Assertion 1.

In a similar way, one checks that under the assumptions of Assertion 2,
given $p_i$  such that $(p_i|c)_{s_i}\leq 45\delta$,
the projection of $p_i$ on $l$ is at distance $> 34\delta$ from $c$.
Assertion 3 allows one to  conclude in the same way.
Indeed, let $v_i$ be the projection of $p_i$ on $[c,u_i]$ and assume that $\d(c,v_i)\leq 34\delta$.
In the triangle $(s_i,c,u_i)$, consider $v'_i,u'_i\in [s_i,c]$ such that
$\d(v'_i,c)=\d(v_i,c)$ and $\d(u'_i,c)=(u_i|s_i)_c$; since $u_i$ is the projection of $s_i$ and
$\d(c,v_i)\leq 34\delta \leq d(c,u_i)-\delta$,
we get $\d(u_i,u'_i)\leq \delta$ and $\d(v_i,v'_i)\leq \delta$.
In the triangle $(s,c,p_i)$, consider $u''_i,v''_i\in [p_i,c]$ with $\d(c,u''_i)=\d(c,u'_i)$,
and $\d(c,v''_i)=\d(c,v'_i)$; since $(p_i|c)_{s_i}\leq 45\delta\leq \d(s_i,u''_i)$,
we get $\d(u'_i,u''_i)\leq \delta$ and $\d(v'_i,v''_i)\leq \delta$.
Since $\d(p_i,v_i)=\d(p_i,[c,u_i])\leq \d(p_i,u''_i)+2\delta$,
$\d(p_i,v''_i)\leq \d(p_i,u''_i)+4\delta$,
so $\d(c,v''_i)\geq \d(c,u''_i)-4\delta\geq \d(c,u_i)-5\delta\geq 35\delta$.
\end{proof}

Unlike in Lemma \ref{lem_security_axis}, the constant characterizing the security pair does not increase in the following lemma.

\begin{lem}\label{lem_security_iter}
  Let $l$ be a bi-infinite geodesic, $\{s_1,s_2\}\subset l$ a $d$-security pair for $c\in l\cap C$.
Consider $s'_1,s'_2\in \X$ such that $d(s'_i,l)\geq 50\delta$ and assume that some closest point projection
$u_i$ of $s'_i$ on $l$ satisfies $d(u_i,s_i)\leq 40\delta$.

Then $\{s'_1,s'_2\}$ is a $d$-security pair for $c$.
\end{lem}

\begin{proof}
Consider  $p_i$ 
such that $(c|p_i)_{s'_i}\leq 45\delta$, and let's prove that $(c|p_i)_{s_i}\leq 45\delta$.
Since $\{s_1,s_2\}$ is a $d$-security pair for $c$, this will imply that so is $\{s'_1,s'_2\}$.

Let $\sigma\in [s'_i,c]$ be the point corresponding to the center of the tripod $(c,s'_i,p_i)$.
We have $\d(s'_i,\sigma)=(p_i|c)_{s'_i}\leq 45\delta$.

Denote by $\tau_1\in [s_i,c]$, $\tau_2\in [s_i,s'_i]$, $\tau_3\in [c,s'_i]$
the three points corresponding to the center of the tripod $(s_i,c,s'_i)$.
Since $\d(\tau_3,[s_i,c])\leq \delta$ and $\d(s'_i,[s_i,c])\geq 50\delta$,
$\d(s'_i,\tau_3)\geq 49\delta\geq \d(s'_i,\sigma)$ so $\tau_3\in [c,\sigma]$.
Let $\tau'\in [c,p_i]$ be the point such that $\d(\tau',c)=\d(\tau_3,c)$,
then $\d(\tau_3,\tau')\leq \delta$.

It is an easy fact that in any $\delta$-thin triangle, if $x\in [c,b],y\in [c,a]$ are at distance $\leq d$,
then $(b|a)_c\geq \min{\d(c,x),\d(c,y)}-d/2-\delta$.
Indeed, one may assume that $(b|a)_c\leq\min\{\d(a,x),\d(b,y)\}$;
then the points $x',y'\in [a,b]$ defined by $d(b,x)=d(b,x')$
and $d(a,y)=d(a,y')$ satisfy $d(x,x')\leq \delta$ and
 $d(y,y')\leq \delta$.
Now $2(b|a)_c= \d(c,x)+\d(x,b)+\d(c,y)+\d(y,a)-(\d(b,x')+\d(x',y')+\d(y',a))=\d(c,x)+\d(c,y)-\d(x',y')$ so
$(b|a)_c\geq \min\{\d(c,x),\d(c,y)\}-d/2-\delta$.

Applying this fact to $\tau_1,\tau'$ in the triangle $(s_i,c,p_i)$,
we get $(p_i|s_i)_{c}\geq \min\{\d(c,\tau_1),\d(c,\tau')\}-\d(\tau_1,\tau')/2-\delta
\geq \d(c,\tau_1)-2\delta$. 
It follows that $(c|p_i)_{s_i}=\d(s_i,c)-(p_i|s_i)_{c}\leq  \d(s_i,\tau_1)+2\delta=\d(s_i,\tau_2)+2\delta$.

Since $u_i$ is a projection of $s'_i$ on $l$,
$\d(s'_i,\tau_2)\geq \d(s'_i,u_i)-\delta$,
so $\d(s_i,\tau_2)=\d(s'_i,s_i)-\d(s'_i,\tau_2)\leq \d(s'_i,s_i)-\d(s'_i,u_i)+\delta\leq \d(s_i,u_i)+\delta$.

If follows that $(c|p_i)_{s_i}\leq  \d(s_i,\tau_2)+2\delta\leq \d(s_i,u_i)+3\delta \leq 43\delta$.
\end{proof}

\begin{defn}
  An improved windmill $W\subset \X$ is a windmill satisfying the following  additional axioms.
  \begin{enumerate}\setcounter{enumi}{\thememowm}

\item \label{wm_two}     For any loxodromic element $g\in G_W$ preserving a bi-infinite geodesic line $l$,
          $l\cap C$ contains a point $c$ at which $l$ has a $\delta$-shortening pair, and there
          exists $c'\in [c,gc]\cap C$
          such that $\{c,gc\}$ is a $3\delta$-security  pair for $c'$.

      \item \label{wm_security}
        If $g\in G_W\setminus \{1\}$ is loxodromic with axis $l\subset W$ as above, and if $p_0\in \X$ is such that
        $d(p_o,l)\geq 50\delta$  then $\{p_0,gp_0\}$ is a $3\delta$-security pair for some $c\in l\cap C$.
  \end{enumerate}
\end{defn}

If $c\in C$, the set $W=B(c,100\delta)$ is an improved windmill since Axioms \ref{wm_two} and \ref{wm_security}
are empty.

\begin{prop}\label{prop_grow2}
       Let $G$ act on a $\delta$-hyperbolic space $\X$, and  $\calC= (C, \{G_c, c\in C\}) $ be a $\rho$-separated very rotating
       family, with $\rho\geq 200\delta$.
Consider a windmill $W$, and $W^\prime$ the larger windmill constructed in Proposition \ref{prop_grow}.

If $W$ is an improved windmill then so is $W'$.
\end{prop}

\begin{proof}
We use the notations of the previous section.
  We first prove that $W'$ satisfies Axiom \ref{wm_two}.
Let $g\in G_{W'}$. If $g$ fixes a point in the tree $T_\Lambda$,
then either $g$ lies in some $G_c$ and there is nothing to do, or $g$ is conjugate in $G_W$,
and we conclude because $W$ satisfies Axiom \ref{wm_two}.
If $g$ is hyperbolic in $T_\Lambda$, let $l_\Lambda\subset T_\Lambda$ be its axis, $v_c\in l_\Lambda\cap V_C$,
$\gamma_{[v_c,gv_c]}\subset\X$ a corresponding geodesic,
and $l=g^\bbZ.\gamma_{[v_c,gv_c]}\subset\X$ a corresponding $g$-invariant bi-infinite geodesic
(recall that any $g$-invariant geodesic of $\X$ is of this form by Lemma \ref{lem_toutes}).
By Lemma \ref{lem_short}, $l$ has a $\delta$-shortening pair at each point of $l$ corresponding
to a point in $V_C\cap l_\Lambda$.
If the segment $[v_c,gv_c]$ in $T_\Lambda$ contains a point $v_{c'}\in V_C\setminus \{v_c,gv_c\}$,
then $l$ has a $\delta$-shortening pair at $c'$, so by Assertion 1 of Lemma \ref{lem_security_axis}, $\{c,gc\}$ is a $3\delta$-security pair for $c'$.
Because $T_\Lambda$ is bipartite, the only remaining possibility is that
$v_c$ and $gv_c$ are at distance $2$ in $T_\Lambda$, and their midpoint $w$ lies in $V_W$.
Up to conjugation, we may assume that the vertex $w$ corresponds to $W$,
so that $gc=hc$ for some $h\in G_W$.

If $h$ is elliptic in $\X$, then since $W$ satisfies Axiom \ref{wm_elliptic},
$h\in G_{c'}\setminus\{1\}$ for some $c'\in W$. Since $d(c,c')\geq 200\delta$ by
      separation of $C$, Lemma \ref{lem_security_elliptic}
implies that $\{c,hc\}$ is a $2\delta$-security pair for $c'$, and in particular,
that $c'\in [c,hc]\subset l$.

      If $h$ is not elliptic, Axiom \ref{wm_elliptic} for $W$ ensures
      that $h$ preserves an infinite geodesic $l'\subset W$, and since
      $d(c,W)\geq 50\delta$, Axiom \ref{wm_security} for $W$ ensures
      that $\{c,hc\}$ is a
      $3\delta$-security pair for some $c'\in [c,hc]\subset l$.
This concludes the proof of Axiom \ref{wm_two} for $W'$.

To prove Axiom \ref{wm_security} for $W'$, consider $g\in G_{W'}$.
If $g$ is conjugate in some $G_W$ or some $G_c$, there is nothing to do.
So we can assume that $g$ acts loxodromically on $T_\Lambda$, and let $l\subset\X$ be a corresponding
bi-infinite $g$-invariant geodesic. 
Let $p\in\X$ be at distance at least $50\delta$ from $l$.
Let $u$ be a closest point projection of $p$ on $l$.
By Axiom \ref{wm_two}, there exists $c\in [u,gu]\cap l$ at which $l$ has a $\delta$-shortening pair,
and such that $\{c,gc\}$ is a $3\delta$-security pair for some $c'\in [c,gc]$.
If $\d(c,\{u,gu\})\geq 40\delta$, then by Assertion 2 of Lemma \ref{lem_security_axis}, $\{p,gp\}$ is a $3\delta$-security pair for $c$, and we are done.
Otherwise, we can assume for instance that $\d(u,c)\leq 40\delta$.
By Lemma \ref{lem_security_iter}, $\{p,gp\}$ is $3\delta$-security pair for $c'$.
\end{proof}

We can now deduce the Greendlinger lemmas.

\begin{proof}[Proof of Lemmas \ref{lem;green_unpointed} and \ref{lem;green_pointed}.]
  Let $g\in \Rot\setminus\{1\}$. Then as in the proof of Theorem \ref{theo;app_wind},
$g\in G_W$ for some improved windmill $W$.
Axiom \ref{wm_two} ensures that if $g$ is loxodromic
with axis $l$, then $l$ contains a $\delta$-shortening pair at some $c\in l\cap C$,
and $\{c,gc\}$ is a $3\delta$-security pair for some $c'\in [c.gc]\cap C$.
In particular, $l$ has a $3\delta$-shortening pair at $c'$. This proves Lemma \ref{lem;green_unpointed}.

To prove Lemma \ref{lem;green_pointed}, consider
 $p_0\in \X$, and $g\in G_{W}\setminus\{1\}$.
If $g\in G_c$ for some $c\in C$, and if $\d(c,p_0)\geq 25\delta$,
the fact that $c\in [p_0,gp_0]$ is a consequence of the global very rotating condition,
and one gets a $\delta$-shortening pair $\{q_1,q_2\}$ by taking     $q_1\in [c,p_0], q_2\in [c,gp_0]$ such that $\d(c,q_i)=25\delta$.

So assume that $g$ is loxodromic, let $l\subset\X$ be a $g$-invariant geodesic,
and $u$ a closest point projection of $p_0$ on $l$.
Using Axiom \ref{wm_two} for $W$, consider $c,c'\in l$ such that
$l$ has a $\delta$-shortening pair at $c$, and $\{c,gc\}$ is a $3\delta$ security pair
for $c'$. In particular, $l$ has a $3\delta$-shortening pair at $c$ and at $c'$.
Since $C$ is $200\delta$-separated, $[u,gu]$ contains a point in the $g$-orbit of $c$ or $c'$ at distance at least
$100\delta$ from $\{u,gu\}$.
By Assertion 3 of Lemma \ref{lem_security_axis}, $[p_0,gp_0]$
contains a $5\delta$-shortening pair at $c$, which proves Lemma \ref{lem;green_pointed}.
\end{proof}

\subsection{Quotient space by a very rotating family, hyperbolicity, isometries, and acylindricity}

 We now describe the quotient of a space by a very rotating family.

\paragraph{Cartan-Hadamard Theorem} \label{i-CarH}

     Let us recall that local hyperbolicity implies global
     hyperbolicity, as the Cartan-Hadamard theorem states.   A length space
     is $\sigma$-simply connected if its fundamental group is normally
     generated by free homotopy classes of loops of diameter less than $\sigma$.

     \begin{thm}\cite[Cartan-Hadamard Theorem 4.3.1]{Del_Gro} \cite[A.1]{Coulon_notes}\label{theo;CH}
       For all $\delta$, there exists $R=\RCH(\delta)$ ($=10^7\delta$)  and $\delta^\prime=\delta_{CH}(\delta) =300\delta$ such that, for
       all geodesic $100\delta$-simply connected space $\X$ that is $R$-locally $\delta$-hyperbolic
       (in the sense that all its balls of radius $R$ are $\delta$-hyperbolic),
       the space $\X$ is $\delta^\prime$-hyperbolic (globally).
     \end{thm}

The assumption that $\X$ is $\delta$-simply connected means that $\pi_1(\X)$ is the normal closure of free homotopy classes of loops of diameter $\leq \delta$.

The subscript \emph{CH} stands for Cartan-Hadamard.
 For a complete
proof of Cartan-Hadamard theorem, we recommend the
appendix of Coulon's notes on small cancellation and Burnside's
problem, \cite[Appendix A]{Coulon_notes}.

\paragraph{Hyperbolicity}

 We now prove the hyperbolicity of the quotient of a hyperbolic space by a separated very rotating family. Our arguments follow \cite{Del_Gro}.

\begin{prop}\label{prop;quotient_hyp}
      Let $\X$ be a $\delta$-hyperbolic space
      equipped with a very rotating family,  whose set of apices is $\rho$-separated, for $\rho>10\RCH(200\delta)$.
      Let $\Rot$ be the group of isometries generated by the rotating family. Then
\begin{enumerate}
   \item[(a)]   $\X/\Rot$ is $60000\delta$-hyperbolic.

   \item[(b)]  For all $x\in \X$, if $x$ is at distance at least $\rho/5 + 100\delta$ from the apices, the ball $B(x, \rho/10)$ in $\X$ isometrically embeds in  $\X/\Rot$.
\end{enumerate}
\end{prop}

    \begin{proof}
      We are going to prove that $\X/\Rot$ is $\rho/10$-locally $200\delta$-hyperbolic.
      In the course of the proof of this fact, we will establish and use the second assertion of the Proposition.
      As $\X$ is $\delta$-hyperbolic, it is $4\delta$-simply
      connected,
      and
     so is $\X/Rot$ since $\Rot$ is generated by elliptic elements.
      This allows to apply the Cartan-Hadamard theorem (with $200\delta$ as our local hyperbolicity constant),
      which proves the proposition.

      Denote by $(C,(G_c)_{c\in C})$ the rotating family.
      Consider $x_0$ with  $\d (x_0,C)\geq \rho/5+100\delta$
      and consider the ball $B(x_0,\rho/5)\subset \X$.
      By the pointed Greendlinger Lemma  \ref{lem;green_pointed},
      $B(x_0,\rho/5)$ is disjoint from its translates by $\Rot$. In particular the quotient map is injective on $B(x_0,\rho/5)$, and
    therefore isometric on
    $B(x_0,\rho/10)$. This already proves the second assertion.

    Therefore $\X/\Rot$ is $\rho/10$-locally
 hyperbolic on the complement of the
    $(\rho/5+100\delta)$-neighborhood of $C$.

    Let $\bar c$ be the image of $c$ in  $\X/\Rot$, and $\bar x_0$ with $\d (\bar x_0, \bar c)\leq \rho/5+100\delta$.
    It is enough to prove hyperbolicity of the ball $B(\bar x_0,\rho/10)$.
    Consider a triangle $(\bar x, \bar y, \bar z)$ in this ball.
    Note that its perimeter is at most $6\rho/10$.
    We claim that any $\bar u\in [\bar x,\bar y]$ lies in the  $50\delta$-neighborhood of $[\bar x,\bar z]\cup[\bar z,\bar y]$.
    This will imply $200\delta$-hyperbolicity by \cite[Prop 21 p.41]{GdH}

    If both $\bar x$ and $\bar y$ are at distance at most $20\delta$ from $\bar c$,
    or if $\d (\bar u,\{\bar x,\bar y\})\leq 50\delta$,
    this simply follows from triangular inequality.
    So assume $\d (\bar x, \bar c)\geq 20\delta$.

    Up to replacing $\bar u$ by some $\bar u^\prime\in [\bar x,\bar y]$ at distance $\leq 15\delta$ from $\bar u$,
    we can assume $\d (\bar u,\bar c)>6\delta$, and we need to prove the existence of $\bar w\in  [\bar x,\bar z]\cup[\bar z,\bar y]$
    with $\d (\bar w,\bar u)\leq 35\delta$.

    Lift $[\bar x, \bar y]$ as a geodesic $[x,y]$ in $\X$, then $[\bar y, \bar z ]$ as a geodesic
    $[y,z]$, and finally $[\bar z, \bar x]$ as a geodesic $[z, x^\prime]$.
    If $x=x^\prime$ the claim follows from the hyperbolicity in $\X$.
    Otherwise, the bound on the perimeter of the triangle gives $\d (x,x^\prime)\leq 6\rho/10$. By Corollary \ref{cor_free2},
    $x^\prime = gx$ with $g\in G_c$, and by the very rotating hypothesis, any geodesic $[x,gx]$ must contain $c$.
    Let $u$ be the lift of $\bar u $ in $[x,y]$.
    Hyperbolicity in the quadrilateral $(x,y,z,x^\prime)$ ensures that $u$ is $2\delta$
    close to another side. If it is  $[x,y]$ or  $[y,z]$, we are done.
    If it is $2\delta$ close to $v \in [x,x^\prime]$, consider $v^\prime\in [x,x^\prime]$
    defined by $v^\prime=gv$ or $v^\prime=g\m v$ according to whether $v\in [c,x]$ or $v\in [c,x^\prime]$.
    Let  $w\in [x,y] \cup [y,z] \cup [z,x^\prime]$ at distance $2\delta$ from $v^\prime$.
    If  $w\in [y,z] \cup [z,x^\prime]$, we are done, since
    $\d (\bar u, \bar w )\leq \d (u,v) + \d (g^{\pm 1}v, w)  \leq 4\delta$.

    Assume that   $w\in [x,y]$. Since $[x,y]$ maps to a geodesic in $\X/\Rot$,
    and since $\d (\ol u,\ol w)\leq \d (u,v)+\d (g^{\pm1}v,w)\leq 4\delta$,
    $\d (u,w)\leq 4\delta$. It follows that $\d (v,gv)\leq \d (u,w)+4\delta\leq 8\delta$,
    so $\d (c,v)\leq 4\delta$ and $\d (c,u)\leq 6\delta$, a contradiction.
\end{proof}

\paragraph{Isometries of the quotients}
The next result is about isometries produced by the quotient group on the quotient space.

\begin{prop}\label{prop;quotient_isom}
        Let $G\actson \X$  be a group acting on a $\delta$-hyperbolic geodesic space, and
      $\calC= (C, \{G_c, c\in C\}) $ be a   $\rho$-separated very
      rotating family for $\rho>  10\RCH(200\delta)  $.

If $\ol g\in G/\Rot$ acts elliptically (resp.\ parabolically) on $\X/\Rot$, then $\ol g$ has a preimage in $G$ acting  elliptically (resp.\ parabolically) on $\X$.
\end{prop}

Note that  $10\RCH(200\delta)  $ is actually $ 2\times 10^{10} \delta$.

\begin{proof}
  Denote by $\ol\delta\leq 60000\delta$ the hyperbolicity constant of
  $\X/\Rot$. It is smaller than $\rho/1000$.

We first claim that if $\ol g$ moves some point $\ol x$ by at most
$d< 4\rho/10$,
and
  $\d (\ol x,\ol c)\leq  3\rho/10$  for some $c\in C$, then $\ol g$ has an elliptic preimage.

Indeed, consider $x,c\in\X$ some preimage of $\bar x,\bar c$ with $\d (x,c)\leq 4\rho/10$, and $g$ a preimage of $\ol g$ moving
$x$ by at most $d$. Then $\d (c,gc) <\rho$,
and $gc=c$ since $C$ is $\rho$-separated. This proves the claim.

Now if $\ol g$ is elliptic in $\X/\Rot$, consider
$\ol x$ whose orbit under $\grp{g}$ has diameter at most
$10\ol\delta$.
Choose $x$ some preimage of $\ol x$, and $g$ representing $\ol g$ with   $d(x,gx)\leq 10 \ol \delta$.
Using the claim above, we can assume $\d (x,C)\geq 3\rho/10$
since $10\ol\delta$ is smaller than  $4\rho/10$.
Recall that, by Proposition \ref{prop;quotient_hyp},  $B(x,\rho/10)$ isometrically embeds in $\X$.
We claim that the orbit of $x$ under $g$ has diameter at most
$10\ol\delta$, proving ellipticity of $g$. If not, let $i$ be the smallest integer with
$\d (x,g^i x)>10 \ol\delta$. Note that $g^i x$ lies in $B(x,\rho/10)$
since $\d (x,g^i x)\leq \d (x,g^{i-1},x) + \d (x,gx)\leq 20\ol\delta <\rho/10$.
Since $B(x,\rho/10)$ isometrically embeds in $\X/\Rot$, this is a contradiction.

Recall that $[g]$ denotes $\min_x \d (x,gx)$.
If  $\ol g$ is parabolic  in $\X/\Rot$, no $g$ representing it can be elliptic in $\X$.
Let $g\in G$ representing $\ol g$ and moving some point by at most $10\ol\delta$.
Assume that $g$ is loxodromic,
and consider $n$ such that $[g^n]\geq 100\ol\delta$, and $x\in \X$ minimizing $\d (x,g^nx)$.
Then $l=\grp{g^n}.[x,g^nx]$ is a $100\ol\delta$-local geodesic (\cite[2.3.5]{Del_Gro}).
Note that $g$ moves points of $l$ by $[g]\leq 10\ol\delta$.
It follows that $l$ stays at distance $\geq 3\rho/10$ from $C$ by the initial claim.
In particular, any ball of radius $\rho/10$  centered at a point of $l$ isometrically embed in $\X/\Rot$.
Since $\rho/10\geq 100\ol\delta$, it follows that the image $\ol l$ of $l$ in $\X/\Rot$ is a $\ol g^n$-invariant
$100\ol\delta$-local geodesic. It follows that $\ol l$ is quasi-isometrically embedded in $\X/\Rot$ and that
$\ol g$ acts loxodromically on $\X/\Rot$, a contradiction.

\end{proof}

\paragraph{Acylindricity on quotients}
We conclude on the acylindricity of the action of the quotient group on the quotient space, under
some properness assumption for the action of a rotation group on the link of its apex.

We first recall some equivalent definitions of acylindricity.

Following Bowditch \cite{B_acyl}, we define an acylindrical action as follows.
\begin{defn}[Acylindricity] \label{dfn_acyl}
Let $G$ be a group acting by isometries on a space $S$.  We say that  the action
is \emph{acylindrical} if       for all $d$ there exists $R_d>0, N_d>0$ such that for all $x,y \in S$ with $\d(x,y)\geq R_d$,the set $$\{ g\in G, \; \d(x,gx)\leq d, \d(y,gy)\leq d\}$$ contains at most $N_d$ elements.
\end{defn}

\begin{prop}[Equivalence of definitions]
Assume that $S$ is $\delta$-hyperbolic,  with $\delta>0$. Then the action of $G$ is acylindrical if and only if there exists $R_0, N_0$ such that for all $x,y \in S$ with $\d(x,y)\geq R_0$, the set $$\{ g\in G, \; \d (x,gx)\leq 100\delta, \d (y,gy)\leq 100\delta\}$$ contains at most $N_0$ elements.
\end{prop}

     \begin{rem}
       If $S$ is an $\bbR$-tree,
   it is not enough, in general,
       to assume
       the condition for only $\delta =0$, and one needs it to be true for
       some $\delta >0$ in order to have acylindricity in the sense of
       Bowditch condition.
     \end{rem}

     \begin{proof}
       If the action is acylindrical, the condition is obviously true.

       Conversely, let $d$ be arbitrary, and take $R_d= R_0 + 4d +100\delta$.
       Consider  $x,y$ at distance $\geq R_d$, and a subset $S\subset G$ of
       elements  that move $x$ and $y$ by at most $d$.
       Consider a geodesic $[x,y]$, and $x^\prime$ at distance $d+10\delta$ from $x$ on $[x,y]$.
       The point $gx^\prime$ lies on $[gx,gy]$ at distance $d+10\delta$ from $gx$.
       Looking at the quadrilateral $(x,gx,y,gy)$, we get that
       $gx^\prime$ is
$2\delta$-close
       to a point $p_g\in [x,y]$ since
$d(gx',[x,gx]\cup[y,gy])\geq 10\delta$ as $R_d\geq 2d+20\delta$.
       Note that $\d (p_g,x)\leq 2d+20\delta$.

       Consider $N_1=\ceil{\frac{2d+20\delta}{10\delta}}$,
       and  for all $i=1,\dots, N_1$, consider the point $p_i\in[x,y]$ at distance $10i\delta$ from $x$ (this is where we use $\delta>0$).
       By construction, for all $g\in S$, $p_g$ is $10\delta$-close to some $p_i$, so $gx^\prime$ is $20\delta$-close to some $p_i$.

       It follows that there exists $i\in \{1,\dots,N_1\}$,
       and a set $S^\prime\subset S$ of cardinality at least $\#S/N_1$ such that for all $g\in S^\prime$, $\d (gx^\prime,p_i)\leq 20\delta$.
       Choose $g_0\in S^\prime$,
and consider $g'\in g_o\m S^\prime$. Note that $g'$ moves $x^\prime$ by at most $40\delta$, and moves $y$ by at most $2d$.

       Let $y^\prime\in [x^\prime,y]$ be at distance $d+10\delta$ from $x^\prime$. By choice of $R_d$, $\d (x^\prime,y^\prime)\geq R_0$.
       Looking at the quadrilateral $(x^\prime,g'x^\prime,y,gy)$, we see that $\d (y^\prime,g'y^\prime) \leq 50\delta$.
       Since $\d (x^\prime,g'x^\prime)<50\delta$, our assumption implies that $\#S^\prime\leq N_0$.
       Hence $\#S\leq N_1 N_0$ so we can take $N_d=N_1 N_0=N_0\ceil{\frac{2d+20\delta}{10\delta}}$.     \end{proof}

  \begin{prop}\label{prop;quotient_acyl}
Let $\X$ be a
$\delta$-hyperbolic space, $G$ a group acting on $\X$, and
    $ (C, \{G_c,c\in  C\})$ a   $\rho$-separated very rotating family, for some
$\rho >10\RCH(200\delta) $.
 Let $\Rot\normal G$ the group of isometries generated by the rotating family.

Assume moreover that there exists $K\in\bbN$ such that for all $c\in C$, and for all $x$ at distance $50\delta$ from $c$,
the set of $g\in G$ fixing $c$ and moving $x$ by at most $10\delta$ has at most $K$ elements.

If $G\actson \X$ is acylindrical, then so is $G/\Rot\actson \X/\Rot$.
  \end{prop}

\begin{proof}
Let us recall some orders of magnitude. Denote by
$\ol\delta\leq 60000\delta$ the hyperbolicity constant of
$\X/\Rot$, and $\rho$ is actually larger than $2\times 10^{10}
\delta$.
Acylindricity in $\X$
gives us  $R_0>0$ and $N_0$ such that for all $a,b\in \X$ with $\d (a,b)\geq R_0$,
there are at most $N_0$ elements $g\in G$ moving $a$ and $b$ by at most $110\ol\delta$.
Then we have $\delta <\!\!\!< \ol\delta  <\!\!\!< \rho $   and we have no control on $R_0$ so one could have $R_0\gg \rho$.

    Let $\ol a,\ol b\in \X/\Rot$ with $\d (\ol a, \ol b)\geq R_0+\rho$.
Let $\ol g\in G/\Rot$ that moves $\ol a$ and $\ol b$ by at most $100\ol\delta$.
Moving $\ol a,\ol b$ inwards, we can assume that $\ol a$ and $\ol b$ are at distance
at least $\rho/10$ from $C$.
Note that the new points $a,b$
satisfy $\d (a,b)\geq R_0+\rho-4\rho/10\geq R_0$, and are moved by at most $110\ol\delta$ by $\ol g$.

Lift the geodesic $[\ol a,\ol b]$ to a geodesic $[a,b]$ of $\X$ with $\d (a,b)=\d (\ol a, \ol b)$.
Choose a lift $g$ of $\ol g$ with $\d (b,gb)\leq 110\ol\delta$.
Choose $r\in \Rot$ such that $\d (ra,ga)=\d (\ol a, \ol g\ol a)$.
If $r$ is trivial, we can use acylindricity in $\X$ to bound the number of possible $g$.

Otherwise, by the pointed Greendlinger Lemma, there exists $c\in [a,ra]\cap C$  and
 $\{q_1,q_2\}\subset [a,ra]$ a $5\delta$-shortening pair.
In particular, $\d(q_1,q_2)\geq 40\delta$, and there exists $h\in G_c$ with $\d (q_1,hq_2)\leq 5\delta$.
Since $a,b,ra$ are far from cone points,  $c$ is at distance at least $\rho/10$ from $a$, $ra$, and $b$.
On the other hand, $\d (ra,ga),\d (b,gb)\leq 110\ol\delta\leq \rho/10$.
Looking at the pentagon $(a,b,gb,ga,ra)$,
we see that there are  $c^\prime,q^\prime_1,q^\prime_2\in [a,b]\cup [b,gb]\cup[ga,gb]\cup [ra,ga]$
with $\d (c,c^\prime),\d (q_1,q^\prime_1),\d (q_2,q^\prime_2)\leq 3\delta$.
Since $\d (b,C)\geq \rho/10$, and $\d (b,gb)\leq 110\ol\delta\leq \rho/10$, $c^\prime$, $q^\prime_1,q^\prime_2$ cannot lie in $[b,gb]$,
nor in $[ga,ra]$ for similar reasons.
Since $[a,b]$ maps to a geodesic in the quotient,
$[a,b]$ cannot contain both $q^\prime_1$ and $q^\prime_2$, and neither can $[ga,gb]$.
So we can assume that $q^\prime_1\in [a,b]$ and $q^\prime_2\in [ga,gb]$.
Let $q^{\prime\prime}_2\in [b,ga]$ at distance $\delta$ from $q^\prime_2$.

Let $p\in [a,b]$ be the center of the triangle $(a,b,ga)$.
One has $\d (c,p)\leq 100\delta$ since $\d (c,q_1)\leq 40\delta$, and $\d (p,[q^\prime_1,q^{\prime\prime}_2])\leq \delta$.
Looking at the quadrilateral $(a,b,gb,ga)$, one sees that $\d (p,gp)\leq \d (b,gb)+10\delta\leq \rho/10$,
so $\d (c,gc)\leq \rho/10+200\delta \leq 2\rho/10$. It follows that $g$ fixes $c$.
Since $\d (b,gb)\leq 110\ol\delta\LL \d (c,b)$, $g$ moves the point at
distance
$50\delta$ from $c$
on $[c,b]$ by at most $\delta$.
By hypothesis, given $c$, there are at most $K$ such elements $g$.
Since there are at most $(R_0+\rho)/(\rho-200\delta)$ elements of $C$ at distance $\leq 100\delta$ from $[a,b]$,
this bounds the number of possible elements $g$ with $r\neq 1$.

\end{proof}

\subsection{Hyperbolic cone-off} \label{sec;cstes}

In this section, we recall the cone-off construction of a hyperbolic space developed by Gromov, Delzant, and Coulon \cite{Coulon}.
This will be our main source of examples of spaces equipped with rotating families.

     We first collect a few universal constants that will be useful later.

     Let  $\du=\dCH(3) = 900$,
   and $\RCH(3)$ be the constants given by the
     Cartan-Hadamard theorem so that any $\RCH(3)$-locally
     $3$-hyperbolic simply connected space is globally $\du$-hyperbolic.

     Let also $\RCH(50\du)\geq \RCH(3)$ be given by the Cartan-Hadamard theorem
     for $\delta=50\du$.

     Finally, let us fix once and for all  $\ru >   10\RCH(50\du)$.

Note that, according to our conventions,   $10\RCH(50\du) > 5\times
10^{12}$, which is greater than $10 \RCH(3)$ and  $10^6 \du$

\paragraph{The hyperbolic cone}

Given a metric space $Y$ and $r_0>0$, define its \emph{hyperbolic cone} \label{i-hc} of radius $r_0$, denoted by $\Cone(Y,r_0)$, as the space
$Y\times[0,r_0]/\sim$ where $\sim$ is the relation collapsing $Y\times \{0\}$ to a point.
The image of $Y\times\{0\}$ in $\Cone(Y,r_0)$ is called its \emph{apex}.\label{i-apex}
We endow $\Cone(Y,r_0)$ with the metric
$$\d ((y,r),(y^\prime,r^\prime))=\acosh\left(\cosh r\cosh r^\prime -\cos \theta(y,y^\prime)\sinh r \sinh r^\prime \right)$$
where $\theta(y,y^\prime)=\min(\pi,\frac{\d (y,y^\prime)}{\sinh r_0})$.

For example, $Y$ is a circle of radius $r_0$ in $\bbH^2$, its perimeter is $2\pi\sinh(r_0)$,
and $\Cone(Y,r_0)$ is isometric to the disk of radius $r_0$ in $\bbH^2$.
If is a circle of perimeter $\theta\sinh(r_0)$, then $\Cone(Y,r_0)$ is a hyperbolic cone
of angle $\theta$ at the apex.
If $Y$ is a line, then $\Cone(Y,r_0)$ is a hyperbolic sector of radius $r_0$ and of infinite angle,
isometric to the completion of the universal cover of the hyperbolic disk of radius $r_0$ punctured at the origin.
We will always take $r_0 \geq \ru$ as defined above.

The \emph{radial projection} \label{i-radproj} is the map $p_Y$ defined on the complement of the apex in  $\Cone(Y,r_0)$
and mapping $(y,r)$ to $y$.

In what follows, we are going to assume that our initial space $Y$ is a metric
graph whose edges have constant length. This is to ensure that the cone-off
is a geodesic space \cite[I.7.19]{BH_metric}.
This is not a restriction because of the following well known fact.

\begin{lem}\label{lem_graph}
  Let $\X$ be a length space and $l>0$. Let $\Gamma_{\X,l}$ be the metric graph
with vertex set $\X$,
with an edge between $x,y$ if and only if $d_\X(x,y)\leq l$, and
where all edges are assigned the length $l$.

Then the inclusion $\X\subset \Gamma_{\X,l}$ is a $(1,l)$-quasi-isometry: any point of $ \Gamma_{\X,l}$ is at distance at most $l$ from $\X$, and
for all $x,y\in \X$,
$$d_{\X}(x,y)\leq d_{\Gamma_{\X,l}}(x,y)\leq  d_{\X}(x,y)+l.$$
\qed
\end{lem}

Note in particular that if $\X$ is $\delta$-hyperbolic, $\Gamma_{\X,l}$ is $\delta+l$-hyperbolic and  if $Y\subset \X$ is $\sigma$-strongly quasi-convex,
then the subgraph of $\Gamma_{\X,l}$ induced by $Y$ is $\sigma+2l$ quasi-convex.

\begin{prop}[{\cite[Prop I.5.10]{BH_metric}}]\label{prop;cone_BH}
For all $y\in Y$, and all $r\in [0,r_0]$, $(y,r)$ is at distance $r$ from the apex;
the radial path $\gamma:[0,r_0]\ra\Cone(y,r_0)$ defined by $r\mapsto (y,r)$ is the unique geodesic joining its endpoints.

Some geodesic joining $(y,r)$ to $(y^\prime,r^\prime)$ in $\Cone(Y,r_0)$ goes through the apex if and only if $\theta(y,y^\prime)\geq \pi$.
Such a geodesic is a concatenation
of two radial paths, and there is no other geodesic joining these points.

If $r,r^\prime>0$ and $\theta(y,y^\prime)<\pi$,
the radial projection $p_Y$ induces
a bijection between the set of (unparametrized) geodesics of $\Cone(Y,r_0)$ joining $(y,r)$ and $(y^\prime,r^\prime)$
 and the set of (unparametrized) geodesics of $Y$ joining $y$ and $y^\prime$.

In particular, if $Y$ is geodesic, so is $\Cone(Y,r_0)$.
\end{prop}

Recall that $Y\subset \X$ is $C$-\sqc\
if for any two points $x,y\in Y$,
there exists $x^\prime,y^\prime\in Y$ at distance at most $C$ from $x,y$ and
 geodesics $[x^\prime,y^\prime],[x,x^\prime],[y,y^\prime]$ that are contained in $Y$.

In general, the restriction to $Y$ of the metric $\d _\X$ of $\X$ is not a path metric.
However, if $Y$ is $C$-\sqc, the path metric of $Y$ induced by $\d _\X$ differs from
$\d _\X$ by at most $4C$.
We will assume that $\X$ is a graph with the induced path metric.
This will guarantee that $Y$, endowed with the induced path metric is a geodesic space.

\begin{prop}[{\cite[Prop. 2.2.3]{Coulon}}] \label{prop;the_cone}
  Given $r_0\geq \ru$  (as defined in the beginning of Section \ref{sec;cstes}),
  there exists a small $\delta_c>0$ such that the following holds.
  Let $Y$ be a $10\delta_c$-\sqc\ subset of a geodesic  $\delta_c$-hyperbolic
  metric graph $\X$, endowed with the induced path metric $\d _Y$.
  Then
   $\Cone(Y,r_0)$ is geodesic and $(3)$-hyperbolic.

 Moreover,
     there exists a constant $L(r_0)=\frac{\pi\sinh(r_0)}{2r_0}$ such that if $(y,r_0)$ and $(y^\prime,r_0)$ are at distance $l<2r_0$ in $\Cone(Y,r_0)$,
    then $\d _Y(y,y^\prime)\leq L(r_0)l$.

\end{prop}

\begin{proof}
The cone is geodesic because $Y$ is (it is a graph).
 Hyperbolicity is proved in \cite[Prop 2.2.3]{Coulon},
  see also \cite{Coulon_notes}, or \cite[Prop. 4.6]{Gui_pcmi}.

By \cite[Proposition 2.1.4]{Coulon}, the distance $l$ between $(y,r_0)$ and $(y^\prime,r_0)$
satisfies $l\geq\frac{ 2r_0}{\pi} \theta(x,y)$. Since $\theta(x,y)=\frac{\d _Y(x,y)}{\sinh(r_0)}$,
we get $\d (x,y)\leq \frac{\pi\sinh(r_0)}{2r_0}\d _Y(x,y)$, so we can take $L(r_0)=\frac{\pi\sinh(r_0)}{2r_0}$.
\end{proof}

\paragraph{The cone-off}
\label{sec_coneoff}
We now define the cone-off construction on a hyperbolic space along some quasi-convex subspace.

Let $\X$ be a geodesic $\delta$-hyperbolic space, and $G$ a group of isometries of $\X$.
Consider $\calQ$ a $G$-invariant system of  $10\delta$-strongly quasiconvex subsets of $\X$.

Let us define the \emph{cone-off} $C(\X,\calQ,r_0)$  of $\X$ along $\calQ$
as the space obtained from the disjoint union of $\X$ and of $\Cone(Q,r_0)$ for all $Q\in\calq$,
and by gluing each $Q$ to $Q\times\{r_0\}$ in $\Cone(Q,r_0)$.
We endow  $C(\X,\calQ,r_0)$ with the induced path metric (in principle a pseudo-metric, but a genuine metric at least when  $\X$ is a graph).

Given $Q_1,Q_2 \in \calQ$ define
 their fellow traveling constant as \label{i-ftc}
$$\Delta(Q_1,Q_2) = \diam (Q_1^{+20\delta} \cap Q_2^{+20\delta})\in \mathbb{R}\cup\{+\infty\}$$
and $$\Delta(\calQ) = \sup_{ Q_1\neq Q_2 \in \calQ} \Delta(Q_1,Q_2).  $$\label{i-ftc1}

Note that radial paths in each cone are still geodesic in $C(\X,\calQ,r_0)$.
In particular, the apices are at distance $r_0$ from $\X$.

\begin{lem}\label{lem_lip}
    \begin{enumerate}
  \item \label{it_lip} Consider $[x,y]$ some geodesic
of $C(\X,\calQ,r_0)$
of length $l$ with endpoints in $\X$. If $[x,y]$ does not contain any apex, then
the length  of its radial projection on $\X$ is at most
 $L(r_0)l$.
\item In particular, if $x,y\in Q$ are such that $\d _\X(x,y)\geq M(r_0)$ with $M(r_0)=2L(r_0)r_0=\pi\sinh r_0$,
then any geodesic of $C(\X,\calQ,r_0)$ joining them contains the apex $c_Q$.
  \end{enumerate}
\end{lem}

\begin{proof}
  The first assertion follows from Proposition \ref{prop;the_cone}, the
  second is a consequence.
\end{proof}

\begin{thm}(Gromov, Delzant-Gromov,  Coulon \cite[3.5.2]{Coulon})
  \label{theo;cone-off_hyp_loc}

  Given $r_0\geq \ru$, there exists numbers $\Delta_c  <\infty$ and $\delta_c$
  such that the following holds.

  Let $\X$ be a
  $\delta_c$-hyperbolic metric graph (whose edges all have the same length),
  $\calQ$ be a system of $10\delta_c$-\sqc\ subsets, with $\Delta(\calQ) \leq \Delta_c$.
       Then
       the cone-off  $C(\X,\calQ,r_0)$ of $\X$ along $\calQ$
       is
       geodesic and $(\ru/8)$-locally $(3)$-hyperbolic.
     \end{thm}

     \begin{proof}
The fact that $\X$ is geodesic is an easy adaptation of  Theorem I.7.19 of \cite{BH_metric}
saying that a simplicial complex with finitely many shapes is geodesic.
This result assumes that each simplex is isometric to a geodesic simplex
but the argument easily extends to our $2$-dimensional situation.

     The local hyperbolicity is stated and proved in \cite[Theorem 3.5.2]{Coulon}
     for the points far from the apices, and in the previous proposition for the points
     close to the apices. See also \cite[Theorem 5.2.1]{Del_Gro},
     and \cite[6.C, 7.B]{Gro_Meso},
 or the expositions \cite{Gui_pcmi,Coulon_notes}.
     \end{proof}

     By hyperbolicity, $\X$ is $\delta_c$-simply connected.
     It follows that so is $C(\X,\calQ,r_0)$.
We can apply the Cartan-Hadamard theorem since the cone-off is locally $3$-hyperbolic on balls of radius $\ru/8\geq \RCH(3)$  (see Section \ref{sec;cstes}).
Cartan-Hadamard Theorem \ref{theo;CH} gives global $\du$-hyperbolicity.

     Using the Cartan-Hadamard theorem, we get

     \begin{cor} \label{coro;cone-off_hyp}
       Under the  assumption of Theorem \ref{theo;cone-off_hyp_loc},
      $C(\X,\calQ,r_0)$  is globally $\du$-hyperbolic.
     \end{cor}

\paragraph{Acylindricity of the cone-off}

     In  order to make Proposition \ref{prop;quotient_acyl} useful in practice, we need to check that acylindricity is
     preserved by taking (suitable) cone-off.

     \begin{prop} \label{prop;cone_acyl}
       Let $r_0 \geq \ru$, and  $\Delta_c  <\infty$ and $\delta_c$  as in  Theorem \ref{theo;cone-off_hyp_loc}.
       Let $\X$ be a $\delta_c$-hyperbolic graph,
       $\calQ$ be a system of $10\delta_c$-quasiconvex subsets, with $\Delta(\calQ) \leq \Delta_c$.
       Consider a group $G$ acting acylindrically by isometries on $\X$, and  preserving $\calQ$.
       Then
       the natural action of  $G$ on  $\dot{\X} =  C(\X,\calQ,r_0)$ is also acylindrical.
     \end{prop}

     \begin{proof}
       By Corollary \ref{coro;cone-off_hyp},
       $\dot{\X}$ is $\du$-hyperbolic (with our notation
       $\du=\dCH(3)$). Recall that $r_0 \geq \ru >  300\du$.
       Let $M(r_0)$ be as in Lemma \ref{lem_lip}.

To prove acylindricity of $\dot \X$,
 it is sufficient to
find $R^\prime,N^\prime$ such that for all  $a,b$ in $\dot{\X}$ such that
       $\d _{\dot{\X}}(a,b)\geq R^\prime$, there are at most $N^\prime$ different
       elements $g$ of $G$ such that  $\max\{\d _{\dot{\X}}(a,ga),
       \d _{\dot{\X}}(b,gb) \} <  200\du$.

       By acylindricity of $\X$,
       there exists $R$, and $N$  such that, in $\X$, for all $a,b\in \X$, at
       distance at least $R$,
       there are at most $N$  elements $g$ of $G$ satisfying
       $\max\{\d _\X(a,ga), \d _\X(b,gb) \} \leq  220\du L(r_0) $.
       We will show that one can take $R^\prime=R+4r_0$ and $N^\prime=N$.

       Let $a,b$ in $\dot{\X}$ such that $\d _{\dot{\X}}(a,b)\geq R^\prime=R+4r_0$.
       First note that by hyperbolicity, if $\d _{\dot \X}(a,ga) \leq
       200\du$ and  $\d _{\dot \X}(b,gb) \leq 200\du$, then
       for all point in a segment $[a,b]$,  $\d _{\dot \X}(x,gx) \leq
       220\du$.
       Therefore, we can assume that $a,b \in \X\subset \dot \X$, $\d _{\dot
         \X}(a,b)\geq R$, and we need to bound the set of elements $g$ moving
       $a$ and $b$ by at most $220\du$.

       Let $g$ be such an element. Since $\d _{\dot \X}(a,ga)\leq 220\du
       < r_0$,
       a geodesic $[a,ga]$ in $\dot \X$ cannot contain an apex.
       Since $a,ga\in \X\subset \dot \X$, $\d _\X(a,ga)\leq L(r_0)\d _{\dot \X}(a,ga)\leq 220\du L(r_0)$ by Lemma \ref{lem_lip}.
       Similarly, $\d _\X(a,ga)\leq 220\du L(r_0)$.
       On the other hand, $\d _\X(a,b)\geq \d _{\dot \X}(a,b)\geq R$. There are at most $N$ such elements $g$ by acylindricity of $\X$,
       which concludes the proof.
     \end{proof}

\paragraph{Coning off quasiconvex subgroups with large injectivity radius.}

      We saw previously that coning off a nice family of quasiconvex subspaces $\calq$ provides a hyperbolic space.
      Here, we assume that a group $G$ acts on the space preserves $\calq$,
      and that to each subspace $Q\in \calq$ is equivariantly assigned a
      subgroup $G_Q\subset G$ stabilizing $Q$ and with large injectivity radius.
      We conclude that $(G_Q)_{Q\in \calq}$ defines a very rotating family on the cone-off.

      Recall that the injectivity radius of a subgroup $H\subset G$ is
      $$\inj_\X (H) =  \inf_{x\in \X \, g\in
        H\setminus\{1\}}  \d _\X(x,gx). $$ \label{i-inj}
      If $\calr$ is a family of subgroups, we define $\inj_\X(\calr)=\inf_{H\in\calr}{\inj_\X(H)}$.

     \begin{prop}  \label{prop;ex_of_VR}
       Let $r_0 \geq \ru$, and $\Delta_c,\delta_c$ be as in Theorem
       \ref{theo;cone-off_hyp_loc},
       and let $\inj_c= 4r_0L(r_0)$
       (where $L(r_0)$ is defined in Proposition \ref{prop;the_cone}   ).
       Let $\X$ be a $\delta_c$-hyperbolic graph,
       $\calQ$ a   system of $10\delta_c$-quasiconvex subsets of $\X$,
       with  $\Delta(\calQ) \leq \Delta_c$.
       Let $\dot \X =C(\X,\calQ,r_0)$ be the cone-off of $\X$, and $C\subset \dot \X$ be the set of apices.

       Consider a group $G$ acting on $\X$, and preserving $\calQ$.
       For each $Q\in \calq$, consider a subgroup $G_Q\subset G$ stabilizing $Q$,
       and such that $G_{gQ}=gG_Q g\m$.
       Assume that each $G_Q$ acts on $\X$ with injectivity radius at least $\inj_c$.
       Then  $(C,(G_Q)_{Q\in \calQ})$  is a very rotating family on $\dot{\X} =C(\X,\calQ,r_0)$, and $C$ is $2r_0$-separated.
     \end{prop}

     \begin{rem}
       In a rotating family $(C,(G_c)_{c\in C})$, the subgroups $G_c$ should be indexed by $C$.
       In the statement above, we slightly abuse notation using the natural bijection between $C$ and $\calq$.
     \end{rem}

     \begin{proof}
       By  Corollary  \ref{coro;cone-off_hyp}, $\dot{\X}$ is $\du$-hyperbolic.
       Obviously, $(G_Q)_{Q\in \calq}$ is a rotating family, and
       the distance between two distinct apices is at least $2r_0$, by construction. So we
        need to check that the family  is very rotating.

Consider $Q\in \calq$, and $c\in \dot \X$ the corresponding apex.
Since $r_0>40\du$ the ball $B(c,40\du)$ is contained in a cone.
We need the following lemma.

\begin{lem} \label{lem;proj}
  Let $x,y\in \dot \X\setminus\{c\}$ at distance $\leq r_0$ from $c$, and let $\bar x, \bar y$ be their radial projection on $Q\subset \X$.

  If some geodesic $[x,y]$ avoids $c$, then some geodesic $[\bar x, \bar y]$ avoids $c$.
\end{lem}

\begin{proof}
If  $\d (\bar x, \bar y) < \d (\bar x,c) + \d (c,\bar y)$ then the claim is obvious, so assume
$\d (\bar x, \bar y) = \d (\bar x,c) + \d (c,\bar y)$.
Since radial paths are geodesic,
we get $\d (\bar x, \bar y) = \d (\bar x,c) + \d (\bar y, c)=\d (\bar x,x)+\d (x,c)+ \d (c,y)+ \d (y,\bar y)\geq
\d (\bar x,x)+\d (x,y)+ \d (y,\bar y)$. By triangular inequality, this is an equality.
 In particular, for any geodesic $[x,y]$, the concatenation
 $[\bar x, x ]\cdot   [x,y] \cdot [y,\bar  y]$ is a geodesic.
By assumption, one of these geodesics avoids $c$, which proves the claim.
\end{proof}

We need to prove that for all $g\in G_Q\setminus \{1\}$,
and all $x\in \dot \X$ with
$20\delta \leq \d _{\dot \X}(x,c)\leq 40\delta$, and all $y\in \dot \X$ with $\d _{\dot \X}(gx,y)\leq 15\du$,
any geodesic of $\dot \X$ between $x$ and $y$ contains $c$.
Look at $\bar x,\bar y$ the radial projections of $x,y$ on $\X$, and note that $g\bar x$ is the radial projection of $gx$.
Assume by contradiction that some geodesic $[x,y]$ avoids $c$.
Note that no geodesic $[y,gx]$ can contain $c$ by triangular inequality.

By Lemma \ref{lem;proj}, there are geodesics $[\ol x,\ol y]$ and $[\ol y,g \ol x]$
avoiding $c$. By Lemma \ref{lem_lip}, $\d _\X(\ol x,\ol y)$ and $\d _\X(\ol y,g\ol x)$ are bounded by $M(r_0)=2r_0L(r_0)$.
It follows that $\inj_\X(G_Q)\leq 4r_0L(r_0)$, a contradiction.
\end{proof}


\section{Examples}\label{sec:ex}


In this section, we give examples of situations in which one finds   hyperbolically embedded subgroups and rotating subgroups. In particular, we show that if a group acts on a hyperbolic space with a so-called loxodromic WPD element, then this element is contained in a cyclic hyperbolically embedded subgroup, and a power of this element generates a rotating subgroup. There are actually two ways in which we can see the later fact. We will prove, in the subsection ``back and forth'' that any cyclic hyperbolically embedded group has a subgroup which is a (cyclic) rotating subgroup.  But we will also prove, somewhat more directly, that a certain small cancellation condition ensures the existence of rotating subgroups and WPD can be used to ensure this small cancellation condition. We think that it can be convenient to have the choice between these two ways of achieving rotating subgroups from the WPD condition, for instance depending on the expositions choices in a lecture.

We also prove that the existence of non-degenerate hyperbolically embedded subgroups implies the existence of non-elementary virtually free ones. This will become useful in certain applications, e.g., in the proof of SQ-universality (see Theorem \ref{large}). 

In the last subsection, we discuss a few specific groups, such as mapping class groups, outer automorphism groups of free groups, and the Cremona group.

\subsection{WPD elements and elementary subgroups}

The aim of this section is to show that if a non-elementary group $G$ acts on hyperbolic space and the action satisfies a certain weak properness condition, then $G$ contains non-degenerate hyperbolically embedded subgroups. The class of such groups includes, for example, all groups acting non-elementarily and acylindrically on a hyperbolic spaces.  More precisely, we recall the following definition due to Bestvina and Fujiwara \cite{BF}.

\begin{defn}\label{WPD}
Let $G$ be a group acting on a hyperbolic space $S$, $h$ an element of $G$.  One says that $h$ satisfies the {\it weak proper discontinuity} condition (or $h$ is a {\it WPD element}) if for every $\e >0$ and every $x\in S$, there exists $N=N(\e )$ such that
\begin{equation}\label{eq: wpd}
|\{ g\in G \mid \d (x, g(x))<\e, \;   \d (h^N(x), gh^N(x))<\e \} |<\infty .
\end{equation}
\end{defn}

Recall that an element $g$ of a group $G$ acting on a hyperbolic space $S$ is called \emph{loxodromic} if the map $\mathbb Z\to S$ defined by $n\mapsto g^ns$ is a quasi-isometry for some (equivalently, any) $s\in S$.

\begin{rem}
It is easy to see that the WPD property is conjugation invariant. That is, if $h_1=t^{-1}ht$ for some $h,h_1,t\in G$, then $h_1$ satisfies WPD if and only if $h$ does. Also it is clear that acylindricity implies WPD for all loxodromic elements.
\end{rem}

\begin{defn}\label{def-qga}
Given an element $h\in G$ and $x\in S$, consider the bi-infinite path $l_x$ in $S$ obtained by connecting consequent points in the orbit $\ldots ,h^{-1} (x), x, h(x), \ldots $ by geodesic segments so that the segment connecting $h^n(x)$ and $h^{n+1}(x)$ is the translation of the segment connecting $x$ and $h(x)$ by $h^n$. Clearly $l_x$ is $h$-invariant, and if $h$ is loxodromic then $l_x$ is quasi-geodesic for every $x\in S$. We call $l_x$ a \textit{quasi-geodesic axis} \label{i-qga} of $h$ (based at $x$).
\end{defn}

For technical reasons (e.g., to deal with involutions in Lemma \ref{elem1}), we will need the freedom of choosing $N$ in (\ref{eq: wpd}) sufficiently large. More precisely, we will use the following.

\begin{lem}\label{elem0}
Let $G$ be a group acting on a $\delta $-hyperbolic space $S$, $h\in G$ a loxodromic WPD element. Then for every $\e>0$ and every $x\in S$, there exists $N\in \mathbb N$ such that
\begin{equation}\label{eq: wpd0}
|\{ g\in G \mid \d (x, g(x))<\e\, \;   \d (h^M(x), gh^M(x)<\e \} |<\infty
\end{equation}
holds for any $M\ge N$.
\end{lem}

\begin{figure}
  \centering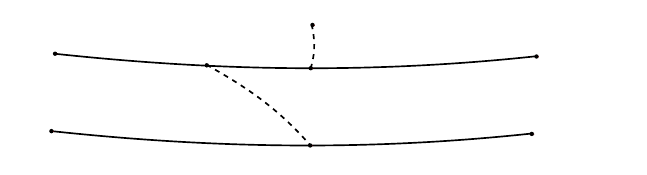\\
  \caption{}\label{61-f1}
\end{figure}

\begin{proof}
Fix $\e>0$ and $x\in S$. Let $l=l_x$ be the quasi-geodesic axis of $h$ based at $x$. Suppose that $l$ is $(\lambda , c)$-quasi-geodesic. Let $\e ^\prime = 3\e + 4\delta + 2\kappa $, where $\kappa =\kappa (\lambda, c, \delta)$ is given by Lemma \ref{qg}. Let $N=N(\e ^\prime )$ satisfy
\begin{equation}\label{eq: wpd1}
|\{ g\in G \mid \d (x, g(x))<\e^\prime , \;  \d (h^N(x), gh^N(x)<\e^\prime  \} |<\infty .
\end{equation}

Let $M\ge N$, $z=h^M(x)$, $y=h^N(x)$, and let $[x,z]$ be a geodesic segment. By Lemma \ref{qg}, there exists $t\in [x,z]$ such that $\d (y,t)\le \kappa $. Note that $g(t)$ belongs to the geodesic segment $[g(x), g(z)]=g([x,z])$ (Fig. \ref{61-f1}). As $S$ is $\delta $-hyperbolic, $g(t)$ is within $2\delta $ from the union of geodesic segments $[g(x), x]$, $[x,z]$, and $[z,g(z)]$. Hence there exists a point $u\in [x,z]$ such that  $\d (u,g(t))\le 2\delta +\e $. Since $$\d (x,u) \ge \d (g(x), g(t))- \d (x, g(x))-\d (u, g(t))\ge \d (x,t)-2\e-2\delta $$ and $[x,z]$ is geodesic,
we obtain $$\d (u,t)=\d (x,t)-\d(x,u) \le 2\e +2\delta $$ and consequently
$$ \d (g(t), t)\le \d (g(t), u)+\d (u,t)\le 3\e +4\delta .$$ This yields
$$
\d (y, g(y))\le \d(y,t)+\d (t, g(t))+\d (g(t), g(y))\le 3\e +4\delta +2\kappa =\e ^\prime.
$$
Thus (\ref{eq: wpd0}) follows from (\ref{eq: wpd1}).
\end{proof}

Bestvina and Fujiwara proved in \cite{BF} that for every loxodromic WPD element $h$ of a group $G$ acting on a hyperbolic space, the cyclic subgroup $\langle h\rangle $ has finite index in the centralizer $C_G(h)$. (Although the assumptions are stated in a slightly different form there.) We use the same idea to prove the following.

\begin{lem}\label{elem1}
Let $G$ be a group acting on a $\delta$-hyperbolic space $S$, $h\in G$ a loxodromic WPD element. Then
$h$ is contained in a unique maximal elementary subgroup of $G$, denoted $E(h)$. Moreover,
$$
E(h)=\{ g\in G\mid \d _{Hau} (l, g(l))<\infty \},\label{i-Eh1}
$$
where $l$ is a quasi-geodesic axis of $h$ in $S$.
\end{lem}

\begin{proof}
It is clear that $E(h)$ is a subgroup. Let $E^+(h)$ consist of all elements of $E(h)$ that preserve the orientation of $l$ (i.e., fixe the limit points of $l$ on the boundary). Clearly $E^+(h)$ is also a subgroup, which has index at most $2$ in $E(h)$.

Let $l=l_x$ be a $(\lambda, c)$-quasi-geodesic axis of $h$ based at some $x\in S$. It easily follows from Lemma \ref{qg} that if $\d _{Hau} (l, g(l))<\infty$ then, in fact, $\d _{Hau} (l, g(l))<\kappa=\kappa (\lambda , c)$.

Let $g\in E^+(h)$ and let $h^{k}(x)$ be the point of the $\langle h\rangle $-orbit of $x$ that is closest to $g(x)$. Thus $\d (g(x),h^{k}(x))$ is uniformly bounded from above by $\kappa$ plus the diameter of the fundamental domain for the action of $h$ on $l$. We denote this upper bound by $C$ and let
$\e = C +6\kappa$. Let $N=N(x,\e)$ be as in the definition of WPD.

We note that $g_0=h^{-k}g$ moves $x$ by at most $C$. Further, let $y=h^N(x)$, let $l^+$ be the half-line of $l$ that starts at $x$ and contains $y$, and let $t$ be the point on $l$ closest to $g_0(y)$. In particular, $\d (g_0(y), t)\le \kappa$. We can assume that $N$ (and hence $\d (x,y)$) is large enough by Lemma \ref{elem0}. This guarantees that $t\in l^+$ since $g_0$ fixes the limit points of $l$ on $\partial S$.  Let $z$ be a point on $l^+$ such that $y$ and $t$ are located between $x$ and $z$. Let also $y^\prime$ and $t^\prime$ be points on the geodesic segment $[x,z]$ closest to $y$ and $t$, respectively (Fig. \ref{61-f2}). We have
\begin{equation}\label{e+1}
|\d (x, t^\prime) - \d (x,y)|=|\d (x, t^\prime) - \d (g_0(x),g_0(y))|\le \d (x, g_0(x))+\d (g_0(y), t^\prime ) \le C + 2\kappa
\end{equation}
and
\begin{equation}\label{e+2}
|\d (x, y^\prime) - \d (x,y)|\le \d (y, y^\prime)\le \kappa .
\end{equation}
Since $[x,z]$ is geodesic, (\ref{e+1}) and (\ref{e+2}) imply
$$
\d (y^\prime , t^\prime)\le C+3\kappa .
$$
Consequently,
$$
\d (y, g_0(y))\le \d (y, y^\prime) +\d (y^\prime, t^\prime)+\d(t^\prime , g_0(y))\le C+6\kappa  =\e.
$$

\begin{figure}
  \centering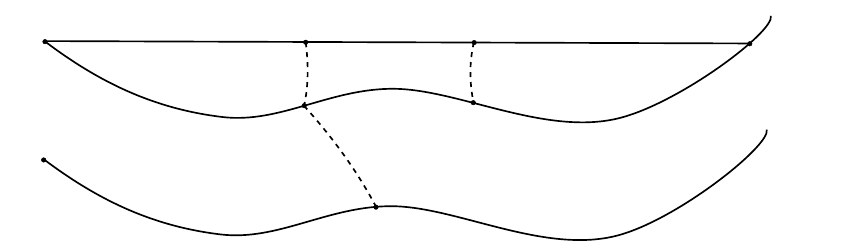\\
  \caption{}\label{61-f2}
\end{figure}

Thus $g_0$ moves both $x$ and $y=h^N(x)$ by at most $\e $. By WPD $g_0$ belongs to some finite set of elements and hence $g$ belongs to a finite set of cosets of $\langle h \rangle$ in $E^+(h)$. Since $g$ was an arbitrary element of $E^+(h)$, we have $|E(h):\langle h\rangle |<\infty $.

To prove that $E(g)$ is maximal, we note that if $E$ is another elementary subgroup containing $h$, then for every $g\in E$ we have $g^{-1}h^ng=h^{\pm n}$ for some $n\in \mathbb N$, which easily implies that $\d _{Hau} (l, g(l))<\infty $. Hence $g\in E(h)$ by definition.
\end{proof}

\begin{cor}\label{elemrem}
Let $G$ be a group acting on a hyperbolic space $S$, $h\in G$ a loxodromic WPD element. Then for every $g\in G$ the following conditions are equivalent.
\begin{enumerate}
\item[(a)] $g\in E(h)$.
\item[(b)] There exists $n\in \mathbb N$ such that $g^{-1}h^ng=h^{\pm n}$.
\item[(c)] There exist $k,m\in \mathbb Z\setminus \{ 0\}$ such that $g^{-1}h^kg=h^{m}$.
\end{enumerate}
Further, we have
$$
E^+(h)=\{ g\in G \mid \exists \, n\in \mathbb N\; g^{-1}h^ng=h^n\} =C_G (h^r).
$$
for some positive integer $r$.
\end{cor}

\begin{proof}
Since $[E(h):\langle h\rangle]<\infty$, there exists $n\in \mathbb N$ such that $\langle h^n\rangle \lhd E(h)$ and the implication (a) $\Rightarrow$ (b) follows. The implication (b) $\Rightarrow$ (c) is obvious. Now suppose that (c) holds. Let $l$ be a quasi-geodesic axes of $h$. Then $h^k$ preserves the bi-infinite quasi-geodesic $g(l)$. This easily implies  $\d_{Hau} (g(l), l)<\infty$, which in turn yields $g\in E(h)$ by Lemma \ref{elem1}. Finally we note that the statements about $E^+(h)$ follow easily from the definition of $E^+(h)$ and the fact that $[E(h):\langle h\rangle]<\infty$.
\end{proof}

The next result is part (2) of \cite[Proposition 6]{BF}. Note that although in \cite[Proposition 6]{BF} the authors assume that all elements of $G$ satisfy WPD, this condition is only used for the element involved in the claim. Note also that the proof of \cite[Proposition 6]{BF} works for any fixed constant in place of $B(\lambda, c,\delta)$.

\begin{lem}\label{BFlemma}
Let $G$ be a group acting on a hyperbolic space $S$, $h\in G$ a loxodromic WPD element. Then for any constants $B, \lambda , c>0$, and any $(\lambda , c)$-quasi-axes $l$ of $h$, there exists $M>0$ with the following property. Let $ t_1(l)$, $ t_2(l)$ be two $G$-translations of $l$. Suppose that there exist segments $p_1, p_2$ of $t_1(l)$ and $t_2(l)$, respectively, which are oriented $B$-close, i.e., $$\max \{\d ((p_1)_-, (p_2)_-), \d ((p_1)_+, (p_2)_+)\} \le B,$$ and have length $$\min\{ \ell (p_1), \ell (p_2)\} \ge M.$$   Then the corresponding conjugates $t_1ht_1^{-1}, t_2ht_2^{-1}$ of $t$ have equal positive powers.
\end{lem}

Recall that two elements $g,h$ of a group $G$ are \textit{commensurable} \label{i-comm} if some non-zero powers of them are conjugate in $G$.

\begin{thm}\label{wpd}
Let $G$ be a group acting on a hyperbolic space $(S,d)$ and let $\{h_1, \ldots , h_k\} $ be a collection of pairwise non-commensurable loxodromic WPD elements of $G$. Then $\{ E(h_1), \ldots , E(h_k)\}\h G$.
\end{thm}
\begin{proof}
Fix any point $s$ of the space $S$. We will show that the conditions (a)-(c) from Theorem \ref{crit} are satisfied. The first condition is a part of our assumption. The second one follows immediately from the fact that $h_i$'s are loxodromic and $\langle h_i\rangle $ is of finite index in $E(h_i)$. Indeed the later condition implies $\d _{Hau} (E(h_i) (s), \langle h_i\rangle (s))<\infty $ and hence each $E(h_i)(s)$ is quasi-convex. Since  each $\langle h_i\rangle $ acts on $S$ properly, so does $E(h_i)$.

It remains to verify the geometric separability condition. Fix any $\e >0$. Let $l_i$ be a quasi-geodesic axes of $h_i$ based at $s$, $i=1, \ldots , k$. Fix $\lambda\ge 1$, $c>0$ such that each $l_i$ is $(\lambda , c)$-quasi-geodesic. Let
$$\theta =\sup\{ \d (x, h_i(x))\mid i=1, \ldots , k, \; x\in l_i\}.$$
Since the action of $h_i$ on $l_i$ is cocompact, $\theta <\infty $.
Choose a constant $M$ such that the conclusion of Lemma \ref{BFlemma} holds for every $h_i$ (with the axis $l_i$), and for $$B=2\e + 4\rho + \theta,$$ where
$$\rho = \max\{ d _{Hau} (E(h_i) (s), l_i)\mid i=1, \ldots, k\}.$$
Let $$R=M+2\rho.$$

\begin{figure}
  \centering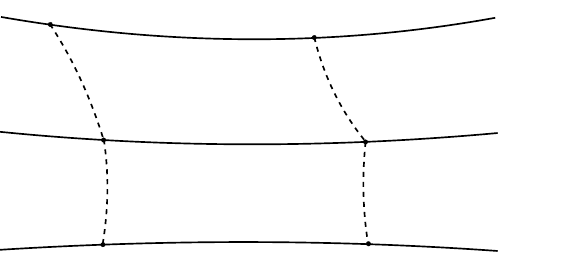\\
  \caption{}\label{61-f3}
\end{figure}

Suppose that
$$
{\rm diam} (E(h_i)(s) \cap  (gE(h_j)(s))^{+\e })\ge R
$$
for some $g\in G$, $i,j\in \{ 1, \ldots , k\}$. Then $$ {\rm diam} ( (l_i)^{+\rho} \cap  (g(l_j))^{+\e +\rho})\ge R $$ and hence there exist points $x_1, x_2 \in l_i$ and $y_1, y_2 \in g(l_j)$ such that $\max \{ \d (x_1, y_1), \d (x_2, y_2)\} \le \e +2\rho $ and $\d (x_1, x_2)\ge R -2\rho =M$ (Fig. \ref{61-f3}). Let $f=gh_jg^{-1}$. Note that for every $y\in g(l_j)$, we have $y=g(x)$ for some $x\in l_j$. Thus
$$
\d (y, f(y))\le \d (g(x), fg(x)) =\d (g(x), gh_j(x))=\d (x, h_j(x))\le \theta .
$$
Hence for $m=1,2$ we have $$\d (x_m, f(x_m)) \le \d (x_m, y_m) +\d (y_m, f(y_m))+\d (f(y_m), f(x_m)) \le 2\e +4\rho +\theta=B.$$ Thus $l_i$ and $f(l_i)$ have oriented $B$-close segments of length at least $M$.
By Lemma \ref{BFlemma}, there exist positive integers $a,b$ such that $fh_i^af^{-1}= h_i^b$. Hence $f\in E(h_i)$ by Corollary \ref{elemrem}. This implies that $h_i$ and $h_j$ are commensurable, which means that $i=j$. Similarly  $f=gh_ig^{-1}\in E(h_i)$ implies $g\in E(h_i)$. Thus the collection $\{ E(h_1), \ldots , E(h_k)\}$ is geometrically separated.
\end{proof}

Let us now show how to construct loxodromic WPD elements in weakly relatively hyperbolic groups. To state our next result, we will need the following.

\begin{defn}\label{i-oH}
Let $G$ be a group, $\Hl$ a collection of subgroups of $G$, $X$ a relative generating set of $G$ with respect to $\Hl$. Let $\dl $ denote the corresponding relative length function. Associated to these data we define
$$
o(H_\lambda)=\{ h\in H_\lambda \mid \dl (1,h)<\infty \} .
$$
\end{defn}

\begin{rem}\label{oHl}
In general, $o(H_\lambda)$ strongly depends on the choice of $X$. For instance, if $H=G$ and $X=\emptyset$, then $o(H)=\{ 1\}$. On the other hand, if $H\le \langle X\rangle $, then $o(H)=H$. Indeed for every $h\in H$ there is an admissible path in $\G $ connecting $1$ to $h$ labelled by a word in the alphabet $X$.
\end{rem}

\begin{thm}\label{elemhe}
Suppose that a group $G$ is weakly hyperbolic relative to $X$ and $\Hl $. Assume that for some $\lambda  \in \Lambda $ the following conditions hold.
\begin{enumerate}
\item[(a)] $H_\lambda $ is unbounded with respect to $\dl$.
\item[(b)] There exists an element $a\in X$ such that $|H_\lambda  ^a\cap H_\lambda|<\infty $.
\end{enumerate}
Then there exists an element $h\in H_\lambda $ such that $ah$ is a loxodromic element satisfying the WPD condition with respect to the action of $G$ on $\G $. In particular, $\{ E(ah)\} \h G$.

Moreover, if
\begin{enumerate}
\item[(a$^\prime$)] $o(H_\lambda )$ is unbounded with respect to $\dl$,
\end{enumerate}
then for every positive integer $k$, there are elements $h_1, \ldots , h_k\in H_\lambda $ such that $ah_1, \ldots, ah_k$  are non-commensurable loxodromic elements satisfying the WPD condition with respect to the action of $G$ on $\G $. In particular, $\{ E(ah_1), \ldots , E(ah_k)\} \h G$.
\end{thm}

\begin{figure}
  \centering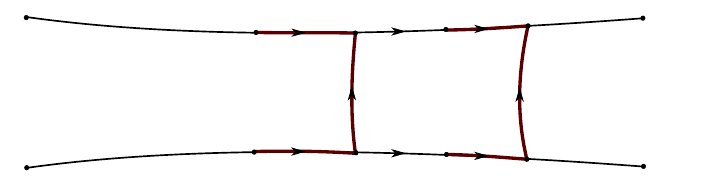\\
  \caption{}\label{61-f4}
\end{figure}

\begin{proof}
We first assume that (a$^\prime$) holds. Let us take $h_1\in o(H_\lambda) $ such that
\begin{equation}\label{h1}
\dl (1, h_1)>50D ,
\end{equation}
where $D=D(1,0)$ is provided by Proposition \ref{sn}. Since  $o(H_\lambda )$ is unbounded with respect to $\dl$, we can choose, by induction, $h_i\in o(H_\lambda) $ such that
\begin{equation}\label{hi}
\dl (1, h_i)> \dl (1, h_{i-1})+ 8D, \; n=2, \ldots, k.
\end{equation}
Let $f_i=ah_i$.

Note that for every $i\in \{ 1, \ldots , k\}$ and every integer $N\ne 0$, the word $(ah_i)^N$ satisfies conditions (W$_1$)-(W$_3$) from Lemma \ref{w}. Hence every path in $\G$ labelled by $(ah_i)^N$ is $(4,1)$-quasi-geodesic. This means that all $f_i$'s are loxodromic.

Let us verify the WPD condition. Fix $i\in \{ 1, \ldots , k\} $ and $\e>0$. Let $$K=|H_\lambda \cap H_\lambda ^a|+2$$ and let $R=R(\e, K)$ be given by Lemma \ref{w}. Let $N>R/2$ be an integer. Suppose that $g\in G$ moves
both $1$ and $f_i^N$ by at most $\e$. Let $p$ be a path in $\G $ starting at $1$ and labelled by $(ah_i)^N$ and let $q=g(p)$. Then $p$ and $q$ are oriented $\e $-close. Since $\ell (p)=2N>R$, by Lemma \ref{w} there exist subpaths $p_1r_1 \ldots p_Kr_K$ of $p$ and $q_1s_1\ldots q_Ks_K$ of $q$ such that $r_j$  and $s_j$  are edges labelled by $a$,  $p_j$  and $q_j$  are edges ($H_\lambda$-components) labelled by $h_i$, and $p_j$ is connected to $q_j$, $j=1, \ldots, K$. Let $e_j$ denote an empty path or an edge in $\G$ connecting $(p_j)_+$ to $(q_j)_+$ and labelled by an element $c_j\in H_\lambda \setminus\{ 1\} $. Reading labels of the loops $e_js_{j}q_{j+1}e_{j+1}^{-1} p_{j+1}^{-1}r_{j}^{-1}$ (Fig. \ref{61-f4}) yields $c_j\in H_\lambda \cap H_\lambda ^a$ for $j=1, \ldots , K-1$.

Since $K-1=|H_\lambda \cap H_\lambda ^a|+1$, there exist $j_1, j_2\in \{ 1, \ldots, K-1\} $ such that $c_{j_1}=c_{j_2}=c$. Let $d=|j_1-j_2|$. Again reading the labels of suitable loops it is easy to see that $[c, f_i^d]=1$ and $g=f_i^ac f_i^b$ for some integers $a,b$. By the former equality we can assume that $|b|<d<K<|H_\lambda \cap H_\lambda ^a|$.  Since $f_i$ is loxodromic, there are only finitely many integers $a$ satisfying $|f_i^ac f_i^b|_{X\cup \mathcal H}\le \e$ for any fixed $b$. Hence there are only finitely many choices for $g$ and thus $H_i$ satisfies the WPD condition for every $i\in \{ 1, \ldots , k\} $.

\begin{figure}
  \centering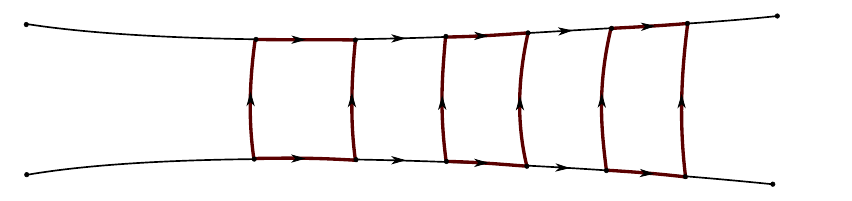\\
  \caption{}\label{61-f5}
\end{figure}

It remains to prove that $f_i$ and $f_j$ are non-commensurable whenever $i\ne j$. Suppose that $f_i^m=(f_j^n)^t$ for some $t\in G$. Let $p$ be the path in $\G $ starting at $1$ and labelled by $(ah_i)^m$, and let $q$ be the path starting at $t^{-1}$ and labelled by $(ah_j)^n$. Passing to multiples of $m$ and $n$ if necessary, we can assume that $|m|, |n|$ are sufficiently large. Applying Lemma \ref{w} for $K=3$ and $\e= |t|_{X\cup \mathcal H}$ as in the previous paragraph, we can find subpaths $p_1r_1 p_2r_2 p_3$ of $p$ and $q_1s_1q_2s_2 q_3$ of $q$ such that $r_1,r_2, s_1, s_2$  are edges labelled by $a^{\pm 1}$, $p_1, p_2,p_3$ and $q_1,q_2,q_3$  are edges ($H_\lambda$-components) labelled by $h_i^{\pm 1}$ and $h_j^{\pm 1}$, respectively, and $p_n$ is connected to $q_n$, $n=1,2,3$. Let $e_n$ (respectively, $g_n$) denote the edge in $\G$ or the trivial path connecting $(p_n)_+$ to $(q_n)_+$ (respectively, $(p_n)_-$ to $(q_n)_-$) and labelled by an elements of $H_\lambda $. Note that $g_2$ is an isolated component in the loop $g_2s_1^{-1}e_1^{-1}r_1$. Indeed otherwise two distinct components of $p$, namely $p_1$ and $p_2$, would be connected, which contradicts Lemma \ref{w}. Hence $\widehat\ell (g_2)\le 4D$ by Proposition \ref{sn}. Similarly $\widehat\ell (e_2)\le 4D$. Reading the label of the cycle $g_2q_2e_2^{-1}p_2^{-1}$ and applying the triangle inequality, we obtain $$|\dl (1, h_i)-\dl(1, h_j)| \le \widehat\ell (e_2)+\widehat\ell (g_2)\le 8D,$$ which contradicts (\ref{hi}).

Thus $f_1, \ldots , f_k$ are non-commensurable WPD loxodromic elements with respect to the action of $G$ on $\G $. To complete the proof it remains to apply Theorem \ref{wpd}.

Finally, if we only have (a), then we choose  any $h\in H_\lambda$ that satisfies $\dl (1, h)>50D $ (in particular, we may have $\dl (1, h)=\infty $). Then the same arguments as above show that $f=ah$ is a WPD loxodromic element with respect to the action of $G$ on $\G $.
\end{proof}

We record one corollary of Theorem \ref{elemhe} and Proposition \ref{malnorm} for the future use. Note the the last claim of the corollary follows from the proof of Theorem \ref{elemhe} and

\begin{cor}\label{elemhe1}
Let $G$ be a group, $X\subseteq G$, $H\h (G,X)$ a non-degenerate subgroup. Then for every $a\in G\setminus H$, there exists $h\in H$ such that $ah$ is  loxodromic and satisfies WPD with respect to the action on $\Gamma (G, X\sqcup H)$.

If, in addition, $H$ is finitely generated, then for  every integer $k>0$, there exist $h_1, \ldots , h_k\in H$ such that $ah_1, \ldots , ah_k$ are non-commensurable, loxodromic, and satisfy WPD. In particular, $\{ E(ah_1), \ldots , E(ah_k)\} \h G$. Moreover, if $H$ contains an element $h$ of infinite order, then we can choose $h_1, \ldots , h_k$ to be powers of $h$.
\end{cor}

\begin{proof}
If $H$ is non-degenerate, the local finiteness of $H$ with respect to the metric $\widehat \d$ implies that $(H, \widehat d)$ is unbounded. On the other hand $|H  ^a\cap H|<\infty $ for any $a\in G\setminus H$ by Proposition \ref{malnorm}. Thus Theorem \ref{elemhe} gives us a loxodromic WPD element of the form $ah$, where $h\in H$.

If $H$ is finitely generated, we can assume that $X$ contains a generating set of $H$ by Corollary \ref{he-indep}.  As we noticed in Remark \ref{oHl}, in this case we have $o(H)=H$. Hence the condition (a$^\prime$) from Theorem \ref{elemhe} holds and we get what we want again. Finally the fact that $h_1, \ldots , h_k$ can be chosen to be powers powers of an element $h$ of infinite order in $H$ follows immediately from the proof of Theorem \ref{elemhe} since $\langle h\rangle \le o(H)$ and $\langle h\rangle$ is unbounded with respect to $\widehat d$.
\end{proof}

In Section \ref{appl} we will prove some general results about the class of groups with hyperbolically embedded subgroups. The next corollary shows that this class is closed under taking certain subgroups in the following sense.

\begin{cor}\label{he subgroup}
Suppose that a group $G$ contains a non-degenerate hyperbolically embedded subgroup $H$. Let $K$ be a subgroup of $G$ such that $|K\cap H|=\infty $ and $K\setminus H\ne \emptyset $. Then $K$ contains a non-degenerate hyperbolically embedded subgroup.
\end{cor}

\begin{proof}
Let $a\in K\setminus H$. Arguing as in the previous corollary, we can find $h\in K\cap H$ such that $ah$ is a loxodromic WPD element and hence $E\h K$, where $E=E(ah)$. If $E=K$, then $K$ is elementary and hence every infinite subgroup has finite index in $K$. In particular, so does $K\cap H$. Therefore, there is a subgroup $N\le K\cap H$ that is normal of finite index in $E$. For every $g\in E$ we have $N\le H^g\cap H$. By Proposition \ref{malnorm} this implies that $E=K\le H$, which contradicts our assumption. Hence $E$ is a proper subgroup of $K$. Since $E$ is infinite, it is non-degenerate.
\end{proof}

\subsection{Hyperbolically embedded virtually free subgroups}

The goal of this section is to prove the following. Many of the ideas used here are due to Olshanskii \cite{Ols93} and Minasyan \cite{Min} (see also \cite{AMO}).

\begin{thm}\label{vf}
Suppose that a group $G$ contains a non-degenerate hyperbolically embedded subgroup. Then the following hold.
\begin{enumerate}
\item[(a)] There exists a maximal finite normal subgroup of $G$, denoted $K(G)$.
\item[(b)] For every infinite subgroup $H\h G$, we have $K(G)\le H$.
\item[(c)] For any $n\in \mathbb N$, there exists a subgroup $H\le G$ such that $H\h G$ and $H\cong F_n \times K(G)$, where $F_n$ is a free group of rank $n$.
\end{enumerate}
\end{thm}

Note that in every group, a maximal finite normal subgroup is unique if exists. Also note that claim (c) is, in a sense, the best possible according to (b). The proof will be divided into a sequence of lemmas.

\begin{proof}[Proof of Theorem \ref{vf}]
By Corollary \ref{elemhe1}, we can assume without loss of generality that $G$ contains three infinite elementary subgroups $H_1,H_2,H_3$ such that $$\{ H_1,H_2, H_3\}\h (G, X)$$ for some $X\subseteq G$. We denote by $\mathcal L_{WPD}=\mathcal L_{WPD}(G,X, \mathcal H)$ \label{i-LWPD} the set of all loxodromic elements of $G$ satisfying the WPD condition with respect to the action of $G$ on $\G $. By Corollary \ref{elemhe1}, we have $\mathcal L_{WPD}\ne \emptyset$. We start by proving parts (a) and (b) of the theorem.

Let
 $$K(G)=\bigcap\limits_{g\in \mathcal L_{WPD}} E(g).$$\label{i-KG}

\begin{lem}\label{KG}
$K(G)$ is the maximal finite normal subgroup of $G$. For every infinite subgroup $H\h G$, we have $K(G)\le H$.
\end{lem}
\begin{proof}
Note first that $K(G)$ is finite. Indeed by Corollary \ref{elemhe1}, there are two elements $g_1, g_2\in \mathcal L_{WPD}$ such that $\{ E(g_1), E(g_2)\} \h G$. Then by the definition $K(G)\le E(g_1)\cap E(g_2)$. The later intersection is finite by Proposition \ref{malnorm}. It is also easy to see that $K(G)$ is normal as the action of $G$ by conjugation simply permutes the set $\{ E(g)\mid g\in \mathcal L_{WPD} (G)\} $. Indeed the WPD condition is conjugation invariant, and every conjugate of a maximal elementary subgroups is also maximal elementary. Finally observe that for every finite normal subgroup $N\le G$ and every $g\in \mathcal L_{WPD}(G)$, there exists positive integer $n$ such that $N\le C_G(g^n)$. Hence $N\le E(g)$ for every $g\in \mathcal L_{WPD}(G)$. This implies $N\le K(G)$ and thus $K(G)$ is maximal. Finally we note that for every $H\le G$, a finite index subgroup of $H$ centralizes $K(G)$. This and Proposition \ref{malnorm} imply the second claim of the lemma.
\end{proof}

Let $$\mathcal L^+_{WPD}=\{ g\in \mathcal L_{WPD} \mid E(g)=E^+(g)\} . \label{i-L+WPD}$$ The proofs of the following three results are similar to proofs of their analogues for relatively hyperbolic groups (see \cite{AMO,Osi10}).

\begin{lem}\label{l+}
The set $\mathcal L^+_{WPD}$ contains infinitely many pairwise non-commensurable elements.
\end{lem}

\begin{proof}
It suffices to find $k$ non-commensurable elements in $\mathcal L^+_{WPD}$  for all $k\in \mathbb N$. Let $a\in H_1$, $b\in H_2$, be elements satisfying
\begin{equation}\label{dfg}
\min\{ \widehat\d_1(1,a), \widehat\d_2(1,b)\} >50D,
\end{equation}
where $D=D(1,0) $ is given by Proposition \ref{sn}. Note that $ab\notin H_3$. Indeed otherwise both $a$ and $b$ are labels of isolated components in a loop of length $3$ in $\G $, which contradicts (\ref{dfg}) by Proposition \ref{sn}. Hence by Corollary \ref{elemhe1} there exist $h_i\in H_3$, $i=1, \ldots, k$  such that
$f_1=abh_1,\ldots , f_k=abh_k$ are non-commensurable elements of $\mathcal L_{WPD}$. The last assertion of Corollary \ref{elemhe1} allows us to assume that
\begin{equation}\label{d3hi}
\widehat\d_3 (1, h_i)>50D,\; i=1,\ldots, k.
\end{equation}
Indeed since $(H_3, \widehat d_3)$ is locally finite, there is an element of infinite order $h\in H_3$ such that every non-trivial power of $h$ has length $\widehat\d_3 (1, h^n)>50D$.

Let us show that every $f_i$ satisfies $E(f_i)=E^+(f_i)$.
To this end it suffices to show that no element $t\in G$ and no $n\in \mathbb N$ satisfy
\begin{equation}\label{tfn}
t^{-1} f_i^nt=f_i^{-n}
\end{equation}
(see Corollary \ref{elemrem}). Arguing by a contradiction, let $t\in G$ and $n\in \mathbb N$ satisfy (\ref{tfn}). Let $\e =|t|_{X\cup\mathcal H}$. Then there exist oriented $\e$-close paths $p$ and $q$ in $\G $ labelled by $(abh_i)^n$ and $(abh_i)^{-n}$, respectively. Note that by (\ref{dfg}) and (\ref{d3hi}) these labels satisfy conditions (W$_1$)-(W$_3$) of Lemma \ref{w}. Let $R=R(\e, 2)$ be given by part (b) of the lemma. Passing to a multiple of $n$ if necessary, we can assume that $p$ is long enough so that $\ell (p)\ge R$. Then by Lemma \ref{w} there exist $2$ consecutive components of $p$ that are connected to $2$ consecutive components of $q$. However this is impossible, actually because the sequences $123123\ldots $ and $321321\ldots $ contain no common subsequence of length $2$.
\end{proof}

\begin{lem}\label{e+}
There exist non-commensurable $h_1, h_2\in \mathcal L^+_{WPD}$ such that $K(G)=E(h_1)\cap E(h_2)$.
\end{lem}
\begin{proof}
By Lemma \ref{l+}, $\mathcal L^+_{WPD}$ contains two non-commensurable elements $f$ and $g$.
We claim now that there exists $x \in G$ such that
$E (x^{-1}fx) \cap E (g)=K(G)$.  Note that for every $x\in G$ we have $K(G)
\subseteq E (x^{-1}fx) \cap E (g)$  by the definition of $K(G)$.

To obtain the inverse
inclusion, arguing by the contrary, suppose that for each $x \in
G$ we have
\begin{equation} \label{eq:inter-e}
(E(x^{-1}fx) \cap E(g)) \setminus K(G) \neq \emptyset.
\end{equation}
For any $h\in \mathcal L^+_{WPD}$,  the set of
all elements of finite order in $E(h)$ form a finite subgroup
$T(h) \le E(h)$.  This is a well-known and easy to prove property of groups, all of
whose conjugacy classes are finite; note that we use $E(h)=E^+(h)$ here.

Since the elements $f$ and $g$
are not commensurable, $E(f)\cap E(g)$ is finite by Proposition \ref{malnorm}. Hence every element of $E(f)\cap E(g)$ has finite order and we obtain
\begin{equation}\label{eq:Exfx}
E(x^{-1}fx) \cap E(g) =T(x^{-1}fx) \cap T(g)=x^{-1}T(f)x \cap T(g).
\end{equation}
Let $P=T(f) \times (T(g)\setminus K(G))$. For each pair of elements $(s,t) \in P$ we choose $y=y(s,t) \in G$ such that $y^{-1}sy=t$ if such $y$ exists; otherwise we set $y(s,t)=1$.

Note that
$$\displaystyle G=\bigcup_{(s,t) \in P} y(s,t) C_G(t).$$ Indeed given any $x\in G$, by (\ref{eq:inter-e}) and (\ref{eq:Exfx}) there exists $t\in T(g)\setminus K(G)$ such that $xtx^{-1}\in T(f)$. Let $y=y(xtx^{-1},t)$ be as above. Then $y^{-1}x \in C_G(t)$ and hence $x\in yC_G(t)$. Recall that by a well-known theorem of B. Neumann \cite{Neumann}, if a group is covered by finitely many cosets of some subgroups, then one of the subgroups has finite index. Thus there exists $t \in T(g) \setminus K(G)$ such that $|G:C_G(t)|<\infty$. Consequently, $t \in E(h)$ for every
$h \in \mathcal L _{WPD}$. Hence $t \in K(G)$, a contradiction.

Thus $E(xfx^{-1}) \cap E(g)=K(G)$ for some $x \in G$. Since $f$, $g$ are non-commensurable and belong to $\mathcal L^+_{WPD}$, so are $xfx^{-1}$ and $g$. It remains to set $h_1=xfx^{-1}$, $h_2=g$.
\end{proof}

\begin{lem}\label{suitable}
For every positive integer $k$, there exist subgroups $\{ E_1, \ldots , E_k\} \h G$ such that $E_i=\langle g_i\rangle \times K(G) $ for some $g_i\in G$, $i=1, \ldots , k$.
\end{lem}
\begin{proof}
Let $h_1, h_2$ be non-commensurable elements of $ \mathcal L^+_{WPD}$ such that $K(G)=E(h_1)\cap E(h_2)$. By Corollary \ref{elemrem}, after passing to powers of $h_i$ if necessary, we can assume that $E(h_i)=C_G(h_i)$, $i=1,2$. By Theorem \ref{wpd}, $\{ E(h_1), E(h_2)\} \h (G, Y)$ for some $Y\subseteq G$.  Let $\mathcal E=(E(h_1)\setminus\{1\})\sqcup(E(h_2)\setminus\{1\})$. Let $\widehat\d_1, \widehat\d_2$ be the corresponding metrics on $E(h_1), E(h_2)$ defined using $\Gk $.

Let $a\in E(h_1)$ be a power of $h_1$ satisfying $\widehat\d_1 (1,a)>50D$, where $D=D(1,0) $ is given by Proposition \ref{sn} applied to the Cayley graph $\Gk $. Obviously $a\notin E(h_2)$ and Corollary \ref{elemhe1} allows us to choose $b_i\in \langle h_2\rangle $, $i=1, \ldots,k$,  such that $g_1=ab_1,\ldots, g_k=ab_k$ are non-commensurable loxodromic elements with respect to the action on $\Gk $, $\{ E(g_1), \ldots, E(g_k)\}\h G$ and
\begin{equation}\label{d2bi}
\min\{ \widehat\d_2 (1,b_i), \widehat\d_2 (1, b_i^2)\} >50D,\; i=1, \ldots, k.
\end{equation}

Let us show that for every $g_i$ we have $E(g_i)=\langle g_i\rangle \times K(G) $.

We are arguing as in the second paragraph of the proof of Lemma \ref{l+}. Let $t\in E(g_i)$ and $\e =|t|_{X\cup\mathcal H}$. Then $tg_i^{\pm n}=g_i^nt$ for some $n\in \mathbb N$. Hence there exist oriented $\e$-close paths $p$ and $q$ in $\Gk $ labelled by $(ab_i)^n$ and $(ab_i)^{\pm n}$, respectively, such that $p_-=1$ and $q_-=t$. By the choice of $a$ and $b_i$'s, $\Lab (p)$ and $\Lab (q)$ satisfy conditions (W$_1$)-(W$_3$) of Lemma \ref{w}. Let $R=R(\e, 3)$ be given by part (b) of the lemma. Passing to a multiple of $n$ if necessary, we can assume that $p$ is long enough so that $\ell (p)\ge R$. Then by Lemma \ref{w} there exists $3$ consecutive components $p_1, p_2, p_3$ of $p$ that are connected to $2$ consecutive components $q_1,q_2,q_3$ of $q$ (Fig. \ref{61-f6}).

\begin{figure}
  \centering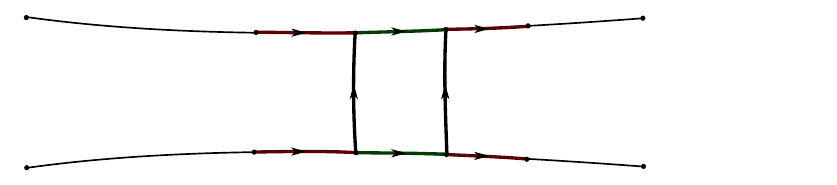\\
  \caption{}\label{61-f6}
\end{figure}

Without loss of generality we can assume that $p_1,p_3, q_1,q_3$ are $E(h_1)$-components while $p_2,q_2$ are $E(h_2)$-components. Let $e_j$ be a path connecting $(p_j)_+$ to $(q_j)_+$ in $\G $ and let $z_j$ be the element of $G$ represented by $\Lab (e_j)$, $j=1,2$. Then $z_j\in E(h_1)\cap E(h_2)=K(G)$.
 Note also that $z_j\in E(h_1)\cap E(h_2)$ implies
\begin{equation}\label{d2zi}
\d _2 (1,z_j) \le 2D, \; j=1,2
\end{equation}
by Proposition \ref{sn}. If $\Lab (q)=(ab_i)^{-n}$, then reading the label of the loop $e_1q_2e_2^{-1}p_2^{-1}$, we obtain $z_1b_i^{-1}z_2^{-1}b_i^{-1}=1$. Recall that $h_1$, $h_2$ are central in $E(h_1)$, $E(h_2)$, respectively, $a$ is a power of $h_1$ and $b_i$ is a power of $h_2$. Hence $z_1$ and $z_2$ commute with $a$ and $b_i$. In particular we obtain $z_1z_2=b_i^2$, which contradicts (\ref{d2bi}) and (\ref{d2zi}).

Thus $\Lab (q)=(ab_i)^{n}$. Reading the labels of the segment of $p$ from $1$ to $(p_1)_+$, $e_1$, and the segment of $q^{-1}$ from $(q_1)_+$ to $t$, we obtain $t= g_i^lz_1g_i^m$ for some $l,m\in \mathbb Z$. Since $z_1$ commutes with $g_i=ab_i$, we obtain $t\in \langle g_i\rangle K(G)$. Thus $E(g_i)\le \langle g_i\rangle K(G) $. Since $K(G)\le E(g_i)$ by Lemma \ref{l+}, we have $E(g_i)= \langle g_i\rangle K(G) $. Since $\langle g_i\rangle $ and $K(G)$ commute and intersect trivially, we obtain $E(g_i)\cong \langle g_i\rangle \times K(G)$.
\end{proof}

We are now ready to complete the proof of Theorem \ref{vf}. We prove it for $n=2$, other cases only differ by notation.

By Lemma \ref{suitable}, there exist subgroups $\{ E_1, \ldots , E_6\} \h (G,Y)$ for some $Y\subseteq G$ such that $E_i\cong \langle g_i\rangle \times K(G) $ for some $g_i\in G$, $i=1, \ldots, 6$. Let $\mathcal E=(E_1\setminus\{1\})\sqcup \ldots \sqcup (E_6\setminus\{1\})$.  Let $\widehat\d_1, \ldots , \widehat\d_6$ be the metrics on $E_1, \ldots , E_6$ constructed using the Cayley graph $\Gk $. Choose $n\in \mathbb N$ such that
\begin{equation}\label{dgi}
\d_i (1, g_i^n)>50D, \; i=1, \ldots , 6,
\end{equation}
where $D=D(1,0)$ is given by Proposition \ref{sn} applied to the Cayley graph $\Gk$.

Let $x= g_1^ng_2^ng_3^n$, $y= g_4^ng_5^ng_6^n$. We will verify that the subgroup $\langle x,y\rangle$ is free of rank $2$ and that $H=\langle x,y\rangle \times K(G)$ satisfies the assumptions of Theorem \ref{crit} with respect to the action of $G$ on $\Gk$.

First consider an arbitrary freely reduced word $W=W(x,y)$ in $\{x^{\pm 1}, y^{\pm 1}\}$. Let $r=r_1\ldots r_k$ be a path in $\Gk$ with $\Lab (r)\equiv W(g_1^ng_2^ng_3^n, g_4^ng_5^ng_6^n)$, where $\Lab (p_i)\in \{ (g_1g_2g_3)^{\pm 1}, (g_4g_5g_6)^{\pm 1}\}$ for $i=1, \ldots , k$. Here we think of $g_i^n$ as letters in $\mathcal E$. Then $\Lab (p)$ satisfies conditions (W$_1$)-(W$_3$) of Lemma \ref{w} and therefore $p$ is $(4,1)$-quasi-geodesic. In particular, it is not a loop in $\Gk$, which means that $W\ne 1$ in $G$. Thus $\langle x, y\rangle $ is free of rank $2$. Moreover it follows that $\langle x, y\rangle $ is $\kappa$-quasiconvex, where $\kappa =\kappa (\delta, 4,1)$ is given by Lemma \ref{qg} and $\delta $ is the hyperbolicity constant of $\Gk$.  It also follows that the action of $\langle x, y\rangle $ on $\Gk$ is proper and hence so is the action of $H$ as $|H: \langle x, y\rangle |<\infty $. This verifies conditions (a) and (b) from Theorem \ref{crit}.

To verify (c), fix $\e>0$ and let $R=R(\e, 4)$ be given by Lemma \ref{w}. Assume that for some $\e>0$, and $g\in G $, we have ${\rm diam} (H\cap (gH)^{+\e })>R$ in $\Gk$. Then there exist oriented $\e$-close  paths $p$, $q$ in $\Gk$ such that their labels are words obtained from some freely reduced words $U$, $V$ in $\{ x,y\}^{\pm 1}$, respectively, by substituting $x=g_1^ng_2^ng_3^n$ and $y=g_4^ng_5^ng_6^n$,  and
\begin{equation}\label{p-q-}
p_-\in H, \;\;\; q_-\in gH.
\end{equation}
By (\ref{dgi}), $\Lab (p)$ and $\Lab (q)$ satisfy  conditions (W$_1$)-(W$_3$) of Lemma \ref{w}. Therefore there exist at least $4$ consecutive components of $p$ connected to $4$ consecutive components of $q$. Taking into account the structure of the labels of $p$ and $q$, it is straightforward to derive that there exist consecutive edges $p_1$, $p_2$ of $p$ and $q_1$, $q_2$ of $q$, such that the following conditions hold. Note that to ensure (**) we essentially use that $x=g_1^ng_2^ng_3^n$ and $y=g_4^ng_5^ng_6^n$; taking $x=g_1^ng_2^n$ and $y=g_3^ng_4^n$ would not suffice.
\begin{enumerate}
\item[($\ast$)] The component $p_i$ of $p$ is connected to the component $q_i$ of $q$, $i=1,2$.

\item[($\ast\ast$)] There exist decompositions $U\equiv U_1U_2$, $V\equiv V_1V_2$ such that the initial subpath of $p$ corresponding to $U_1$ ends with $p_1$ and the initial subpath of $q$ corresponding to $V_1$ ends with $q_1$.
\end{enumerate}
Let $e$ be the empty path or an edge in $\Gk$ connecting $(p_1)_+$ to $(q_1)_+$. By ($\ast$), $\Lab (e)$ represents an element of $E_i\cap E_j$ in $G$ for some $i\ne j$. Hence $c\in K(G)$. Since $c$ commutes with $g_1, \ldots , g_6$, using (\ref{p-q-}) and ($\ast\ast$) it is easy to obtain $g=zc$, where $z\in \langle x,y\rangle $. Thus $g\in H$ and $H$ is geometrically separated. It remains to apply Theorem \ref{crit}.
\end{proof}

\subsection{Combination theorems}

In this section we mention two analogues of the combination theorems for relatively hyperbolic groups first proved by the first named author
 in \cite{Dah}. Our proofs are based on the approach  suggested in \cite{Osi06c}.

\begin{thm}\label{HNN}
Let $H$ be a group, $\Hl \cup \{ K\}$ a collection of subgroups, $X$ a subset of $H$. Suppose that $\Hl \cup \{ K\}  \h (H,X)$. Assume also
that $K$ is finitely generated and for some $\nu \in \Lambda $,
there exists a monomorphism $\iota :K\to H_{\nu }$. Let $G$ be the
HNN--extension
$$
\langle H,t\; |\; t^{-1}kt=\iota (k),\; k\in K\rangle .
$$
Then $\Hl \h (G,X\cup \{ t\})$.
\end{thm}
\begin{proof}
The proof is almost identical to the proof of Theorem 1.2 from \cite{Osi06c}. Instead of copying it here, we only indicate the necessary changes. First, throughout the proof the words ``finite relative presentation" should be replaced with ``strongly bounded relative presentation" and references to Lemma 2.1 from \cite{Osi06c} should be replaced with references to Lemma \ref{omega} from our paper. After these substitutions, the proof given in \cite{Osi06c} starts with a strongly bounded relative presentation $\mathcal P$ of $H$ with respect to $\Hl \cup \{ K\} $ and $X$, and produces a strongly bounded relative presentation $\mathcal Q$  of $G$ with respect to $\Hl$ and the relative generating set $X\cup Y\cup \{ t\} $, where $Y$ is any finite generating set of $K$. It is proved in \cite{Osi06c} that if $\gamma (n)$ is a relative isoperimetric function of $\mathcal P$, then there exist constants $C_1,C_2,C_3$ such that $C_1\overline\gamma \circ \overline \gamma (C_2n)+C_3n$ is a relative  isoperimetric function of $\mathcal Q$, where
$$
\overline \gamma(n)=\max_{i=1,\dots , n}\left(\max_{a_1+\dots + a_i=n,\
a_i\in {\mathbb N}} \left(\gamma(a_1)+\dots +\gamma(a_i)\right)\right) .
$$
In our case $\gamma $ is linear since $\Hl \cup \{ K\}  \h (H,X)$. Hence the proof yields a linear relative isoperimetric inequality for $\mathcal Q$. This means that $\Hl \h (G,X\cup Y\cup \{ t\})$. Note that $G$ is also generated by $X\cup\{t\}\cup \mathcal H$ as $Y\subset K\le \langle t, H_\nu\rangle $. As $Y$ is finite, Corollary \ref{he-indep} implies that $\Hl \h (G,X\cup \{ t\})$.
\end{proof}

Similarly for amalgamated products, we have the following.

\begin{thm}\label{Am}
Let $A$ (respectively, $B$) be a group, $\Am \cup \{ K\} $
(respectively, $\Bn $) a collection of subgroups, $X$ (respectively, $Y$) a subset of $A$ (respectively, $B$).
Suppose that $\Am \cup \{ K\} \h (A,X)$, $\Bn \h (B,Y)$.  Assume also that $K$ is finitely generated and for some $\eta\in N$, there is a monomorphism $\xi : K\to B_\eta $. Then $\Am \cup\Bn\h (A\ast _{K=\xi (K)} B, X\cup Y) $.
\end{thm}

Theorem \ref{Am} can be derived from Theorem \ref{HNN} by using the standard ``retraction trick".

\begin{lem}\label{retr}
Let $W$ be a group, $\Ul$ a
collection of subgroups, $X$ a subset of $W$. Let $V\le W$ be a retract of $W$, $\e\colon W\to V$ a retraction.
Suppose that $V$ contains all subgroups from the set $\Ul $ and $\Ul \h (W,X)$. Then $\Ul \h (V, \e (X))$.
\end{lem}

\begin{proof}
Let
$$
\mathcal U =\bigsqcup\limits_{\lambda\in \Lambda} (U_\lambda ).
$$
Obviously $\e $  defines a retraction $\hat\e $ between the corresponding Cayley graphs $\Gamma (W, X\sqcup\mathcal U)$ and $\Gamma (V, \e(X)\sqcup \mathcal U)$. Hence $\Gamma (V, \e(X)\sqcup \mathcal U)$ is hyperbolic. Note also that if $p$ is a path in $\Gamma (W, X\sqcup\mathcal U)$ and $q_1, q_2$ are components of $p$, then $\hat\e(q_1), \hat\e(q_2)$ are components of $\hat\e (p)$ and if $q_1, q_2$ are connected, then so are $\hat\e(q_1), \hat\e(q_2)$. Using this observation it is straightforward to verify that local finiteness of $(U_\lambda, \dl^\prime )$ easily follows from that of  $(U_\lambda, \dl )$, where $\dl ^\prime $ and $\dl $ are the distance functions defined as in Definition \ref{maindef} using $\Gamma (V, X\sqcup\mathcal U)$ and $\Gamma (W, X\sqcup\mathcal U)$, respectively. Hence the claim.
\end{proof}

\begin{proof}[Proof of Theorem \ref{Am}]
Recall that the amalgamated product $P=A\ast _{K=\xi (K)} B$ is isomorphic to a
retract of the HNN--extension $G$ of the free product $A\ast B$ with
the associated subgroups $K$ and $\xi (K)$ \cite{LS}. More precisely, $A\ast _{K=\xi (K)} B$ is isomorphic to the subgroup $\langle A^t,B\rangle \le G$ via the isomorphism sending $A$ to $A^t$ and $B$ to $B$, where $t$ is the stable letter. It is obvious from the isoperimetric characterization of hyperbolically embedded subgroups (see Theorem \ref{ipchar}) that $\Am\cup \Bn\cup \{ K\} \h A\ast B$. Then by Theorem \ref{HNN}, $\Am\cup \Bn\h G$. Further applying Proposition \ref{heconj} we conclude that $\{A_\mu^t\}_{\mu\in M }\cup \Bn\h G$. Consequently $\{A_\mu^t\}_{\mu\in M }\cup \Bn\h  \langle A^t, B\rangle $ by Lemma \ref{retr}. Passing from $\langle A^t, B\rangle $ to $P$ via the isomorphism, we obtain the claim.
\end{proof}

    \subsection{Rotating families from small cancellation subgroups}

\paragraph{Small cancellation subgroups}
Recall that if $\calr$ is a family of groups acting on $\X$,
we defined the injectivity radius of $\calr$ as
$$\inj_\X(\calr)=\inf_{H\in\calr}  \inf   \{\d(x,gx), \, g\in H\setminus\{1\}, x\in \X  \} $$
(see Section \ref{sec_coneoff}).
This invariant was relevant for to get a very rotating family in the cone-off in Proposition \ref{prop;ex_of_VR}.
Also recall that if $Q,Q'$ are two $10\delta$-\sqc\ subspaces, we defined their fellow-traveling constant by
$$\Delta(Q,Q')=\diam(Q^{+20\delta}\cap Q^{\prime+20\delta}).$$

      \begin{defn}\label{dfn_sc_subgroup}
        Let $G$ be a group acting on a $\delta$-hyperbolic space $\X$ with $\delta>0$.
        Consider a family of subgroups $\calr$  of $G$ stable under conjugation. 
        We say that $\calr$ satisfies the \emph{$(A,\eps)$-small cancellation condition} if the following hold.
        \begin{enumerate}
        \item[(a)] For each subgroup $H\in \calr$ there is a $10\delta$-\sqc\ subspace $Q_H\subset \X$ such that
$Q_{gHg\m}=gQ_{H}$.
        \item[(b)] $\inj_\X(\calr)\geq A \delta$.
        \item[(c)] For all $H\neq H^\prime\in\calr$,
$\Delta(Q_H,Q_{H^\prime})\leq \eps\cdot \inj_\X(\calr)$
        \end{enumerate}
      \end{defn}

Note that this condition does not change under rescaling of the metric.
Moreover, if $\X$ is a simplicial tree, we can choose $\delta$ as small as we want, so the assumption on the injectivity radius
means that all elements of $H$ should be hyperbolic, and the last assumption is (a
 strengthening of)
the usual $C^\prime(\eps)$ small condition.

Recall that
the universal constants $r_U, \delta_c,\Delta_c$ were defined in the beginning of section \ref{sec;cstes} and Theorem \ref{theo;cone-off_hyp_loc}.
We choose $r_0 \geq r_U$, and $\inj_c (r_0)$ is the  constant defined in Proposition \ref{prop;ex_of_VR}.

      \begin{prop}\label{prop_sc_subgroup}
        There are  constants
$A_0=\frac{\inj_c(r_0)}{\delta_c},\eps_0=\frac{\Delta_c}{\inj_c(r_0)}$ such that
        if a group $G$  acts on a $\delta$-hyperbolic graph $\X$ with $\delta>0$,
        if $\calr$ is a family of subgroups satisfying the $(A_0,\eps_0)$-small cancellation condition,
        then $\calr$ defines a  $2r_0$-separated very rotating family on the  ($\du$-hyperbolic)
        cone-off  $\dot \X=C(\lambda \X,(Q_H),r_0)$
        where $\lambda \X$ is the space $\X$  rescaled by the factor  $\lambda=\min(\frac{\delta_c}{\delta},\frac{\Delta_c}{\Delta})$.
      \end{prop}

      \begin{proof}
        Let $\X^\prime=\lambda \X$ be the rescaled space.
        Clearly, $\X^\prime$ is $\delta_c$-hyperbolic, and $\Delta(H,H^\prime)\leq \Delta_c$
        for all $H\neq H^\prime\in \calr$. Corollary \ref{coro;cone-off_hyp} implies that $\dot \X$ is $\du$-hyperbolic.
        Morever, $\inj_{\X^\prime}(\calr)=\min(\frac{\delta_c}{\delta},\frac{\Delta_c}{\Delta})\inj_\X(\calr)
        \geq \min(A_0\delta_c,\frac{\Delta_c}{\eps_0})$ since $\calr$  satisfies the $(A_0,\eps_0)$-small cancellation condition.
        This last quantity  is  $\inj_c$ by choice of $A_0,\eps_0$.
        Then, Proposition \ref{prop;ex_of_VR} applies, showing that $\calr$ is a $2r_0$-separated very-rotating family.
      \end{proof}

The assumption that $\X$ is a graph allows us to get particularly nice constants in the proposition above. However, if we do not care about constants, it is not very restrictive. Indeed if a group $G$ acts of a length space $\X$, we can define a graph $\Gamma(\X)$ by taking elements of $\X$ as vertices and connecting two vertices $x,y$ by an edge iff $\d (x,y)\le 1$. The action of $G$ on $\X$ induces the action on $\Gamma (\X)$ and it is obvious that $\X$ and $\Gamma (X)$ are $G$-equivariantly quasi-isometric. This allows us to apply Proposition \ref{prop_sc_subgroup} to actions on hyperbolic spaces as well. In particular, since  $C(\lambda \X,(Q_H),r_0)$ is $\delta_U$-hyperbolic (by Corollary \ref{coro;cone-off_hyp}), we obtain the following.

\begin{cor}
For any $\alpha$, there exist $A>0,\eps>0$, such that the following holds. Suppose that a group $G$ admits an action on a hyperbolic space and the family of conjugates of a subgroup $H\le G$ satisfies the $(A,\eps)$ small cancellation condition with respect to this action. Then $H$ is an $\alpha$-rotating subgroup of $G$.
\end{cor}

Now we specialize the previous results to the case of cyclic groups.
The \emph{axis} of a loxodromic element $g$,  is the $20\delta$-neighborhood of the set
of points $x$ at which $\d (gx,x) \leq \inf_y \d (gy,y) +\delta$.
We denote it by $\Axis(g)$. Note that it is not quite the same as the quasi-geodesic axis previously introduced.

For two loxodromic elements $g,h$, we write $\Delta(g,h) = \Delta(\Axis(g),\Axis(h))$
(as defined earlier in section \ref{sec_coneoff}).

\begin{defn}\label{dfn_sc_cyclic}
  Let $G$ be a group acting on a $\delta$-hyperbolic space $\X$ with $\delta>0$.
  Consider a family $\calr$ of loxodromic elements of $G$ stable under conjugation.
        We say that $\calr$ satisfies \emph{$(A,\eps)$-small cancellation} if the following conditions hold.
        \begin{enumerate}
        \item[(a)] $\inj(\calr) \geq A \delta$; here by the injectivity radius of the family $\calr$ we mean the injectivity radius of the corresponding family of cyclic subgroups. 
        \item[(b)] For all $g\neq h^{\pm 1}\in\calr$, $\Delta(g,h)\leq \eps \cdot \inj(\calr).$
        \end{enumerate}
      \end{defn}

 Small cancellation implies that $\Delta(\calR)$ is finite. In particular,
if $\Axis(g)$ and $\Axis(h)$ satisfy $\Delta(g,h)=\infty$, then $g=h^{\pm 1}$.
Applying Proposition \ref{prop_sc_subgroup}, one immediately gets

\begin{prop}\label{prop_sc_cyclic}
  For any $\alpha$, there exist $A>0, \eps>0$,
  such that the following holds. Let $G$ be a group acting on a $\delta$-hyperbolic space $\X$ and let
  $\calr$ be a family of loxodromic elements of $G$ stable under conjugation satisfying the $(A,\eps)$-small cancellation condition. Then
  $\calr$ is an $\alpha$-rotating family of $G$ with respect to the induced action of $G$ on a certain cone-off of $\X$.
\end{prop}

\paragraph{Small cancellation from acylindricity}

We now show how acylindricity implies that large powers of elements give small cancellation families.
We prove similar (but less uniform) assertions under Bestvina and Fujiwara's WPD condition.
See Definitions \ref{dfn_acyl} and \ref{WPD} for the notions of acylindricity and WPD.

    We will use the following facts concerning the stable norm $||g||$ of an element $g\in G$.
  Recall that the stable norm is defined as
    $\| g\| = \lim \frac{1}{n} \d (g^nx,x)$, see \cite{CDP}.

\begin{lem}[see {\cite[Prop. 3.1]{Del_Duke}}]\label{lem_norm}
  There exist $K_0,K_1,K_2$ such that
for any $g$  such that $[g]\geq K_0\delta$, the following hold.
For any $i>0$,
 $\Axis(g)$ and $\Axis(g^i)$ are at Hausdorff distance at most $K_1\delta$.
Moreover, for any $x\in \Axis(g)$,
  and all $i\in \bbN$, $$i||g||\leq \d (x,g^ix) \leq i||g||+K_2\delta.$$
\end{lem}

\begin{proof}
First recall that for any $g\in G$, $\|g\|\leq [g] \leq \|g\| +16\delta$, where $[g]=\inf\{d(x,gx)|x\in \X\}$
(see \cite[10.6.4]{CDP}).

Let $x,x'\in \X$ be such that $\d(x,gx)\leq [g]+\delta$
and $\d(x',g^ix')\leq [g^i]+\delta$.
Then $l=\cup_{n\in\bbZ} g^n[x,gx]$ and $l'=\cup_{n\in\bbZ} g^{ni}[x',g^ix']$ are two $K$-local $(1,\delta)$-quasigeodesic
with $K=\min\{[g],[g^i]\}$.
Consider $\alpha,\lambda,\mu$ such that $\alpha\delta$-local $(1,\delta)$-quasigeodesics are global $(\lambda,\mu)$-geodesics.
Define $K_0=\alpha+16$ and assume that $[g]\geq K_0\delta$. Then $l$ and $l'$ are global $(\lambda,\mu)$-geodesics.
Since $\d_{Hau}(l,l')<\infty$, there exists $K'_1$ (depending only on $\lambda,\mu$)
$\d_{Hau}(l,l')\leq K'_1\delta$ (see Lemma \ref{qg}).
Since $\Axis(g)$ and $\Axis(g')$ are at Hausdorff distance at most $20\delta$ from $l$ and $l'$,
the first assertion follows.

To prove the second assertion,
consider any $x\in l$, and $x'\in l'$ with $\d(x,x')\leq K'_1\delta$.
Then $d(x,g^ix)\geq d(x',g^ix')-2K'_1\delta \geq ||g^i||-16\delta-2K'_1\delta$.
Since $\|g^i\|=i\|g\|$, and since $\Axis(g)$ is at Hausdorff distance at most $20\delta$ from $l$,
the second assertion follows.
\end{proof}

We also note the following consequence of acylindricity.

\begin{lem}\label{lem_norm_acyl}\cite{B_acyl}
If $G$ acts acylindrically on $\X$, then
  there exists $\eta>0$ such that the stable norm of all loxodromic elements is at least $\eta$.\qed
\end{lem}

    We now explain how to obtain families satisfying small
    cancellation conditions from the acylindricity of the action.

    \begin{prop}[Small cancellation from acylindricity]\label{prop;SC_from_acyl}
     Let $G\actson \X$ be an acylindrical action on a geodesic $\delta$-hyperbolic space.
      Then, for all $A,\eps>0$, there exists $n\in \mathbb N$ such that the
      following holds.
      Let  $\calR_0$ be a conjugacy closed family of loxodromic elements of $G$
      having the same positive stable norm.
      Then the family $\calr_0^{n}=\{ g^{n}, g\in \calR_0 \}$ satisfies the $(A,\eps)$-small cancellation condition.
    \end{prop}

    \begin{rem}
      The statement of the proposition extends to the following situation: assuming acylindricity, given $A,\eps$ and $L$, there exists $n$ such that
      the following holds. Assume that $\calr_0$ is
      family of loxodromic elements of stable norm at most $L$, closed under conjugacy,
      and such that any pair of elements $g,h\in\calr_0$ having axes at finite Hausdorff distance satisfy $||g||=||h||$.
      Then
      the family $\{ g^{n}, g\in \calR_0 \}$
      satisfies the $(A,\eps)$-small cancellation condition.
    \end{rem}

    \begin{rem}
      One easily checks that if $\calr$ satisfies the $(A,\eps)$-small cancellation condition,
then so does $\calr^k$ for all $k\geq 1$. In particular, if $\calr_0$ is as in the proposition, then
$\calr_0^{nk}$ satisfies the $(A,\eps)$-small cancellation condition. However,
 in presence of torsion,
it might not be the case that $\calr_0^k$ satisfies the $(A,\eps)$-small cancellation condition for
all $k$ large enough.
    \end{rem}

    Let us briefly explain the argument in the case of an action on a
    tree.
    First, we argue that if  $g,h \in
    \calR$ have different axis of translation in the tree, then the
    common segment $\sigma$ of the two axis has length controlled by $L$ and
    the constants of acylindricity. Actually, restricted on a
    subsegment of $\sigma$
    far from it ends, $[g^i, h^j]$ is trivial (since $g^i$ and $h^j$
    are merely translations on a same axis), and one can find a
    contradiction with acylindricity, if the possible $i$ and $j$ are
    numerous.

    The second point is that if $g,h$ have same axis, and same
    translation length, then $h^ig^{-i}$
    fixes the whole axis. Again by acylindricity, $(h^ig^{-i}) =
    (h^jg^{-j})$ for two different bounded indices, and therefore
    $h^k=g^k$ for some controlled power $k$.

We start with a well known technical lemma.

\begin{lem}\label{lem_commut}
  There is a universal constant $K$ such that the following holds.
  Let $g,h$ be loxodromic elements in a $\delta$-hyperbolic space, and $N\in \bbN\setminus \{0\}$.
  Assume  $\Delta(g,h)\geq ||g^N||+||h||+50\delta$.
  Consider $x,y \in \Axis^{+20\delta}(g)\cap \Axis^{+20\delta}(h)$, with $\d (x,y)=\Delta(g,h)$.
  Without loss of generality up to changing $g,h$ to their inverses, assume that $g\m x$ and $h\m x$ are at distance
  at most $50\delta$ from $[x,y]$.
  Let $p\in [x,y]$ be the point at distance $\Delta(g,h)- ||g||-||h||$ from $y$.

Then for all $i\in \{1,\dots, N\}$,
the commutator $[g^i,h]=g^i h g^{-i} h\m$ moves all points in $[p,y]$ by at most $K\delta$.
Moreover, if $||h||=||g||$, then $i\in \{1,\dots, N\}$, $g^i h^{-i}$ moves all points in $[p,y]$ by at most $K\delta$.
\end{lem}

The first assertion is in the last claim of \cite{Pau_arboreal}.
The second follows from the
second point of Lemma \ref{lem_norm}.

\begin{proof}[Proof of Proposition \ref{prop;SC_from_acyl}]
      Let us fix the constants.
      Let $K$ be as in Lemma \ref{lem_commut}.
      By acylindricity, there
      exists $N$ and $R$, such that for all $x,y$ at distance $\geq R$,
      at most $N$ different elements of $G$ send them at distance at most $K\delta$ from themselves.
      By
      Lemma \ref{lem_norm_acyl},
      consider $\eta>0$ such that the stable norm of any loxodromic element is $\geq \eta$.

Recall the constants $K_0,K_1$ from Lemma \ref{lem_norm}.
      Fix $A\geq K_0+16$ and $\eps>0$.
      Let $m_0\geq \max(\frac{A\delta}{\eta},\frac{R+(N+2)L+(100+K_1)\delta}{\eps\eta})$.
      Define $n$ as the smallest
      multiple of $N!$ greater than $m_0$.

      Clearly, for all $m\geq m_0$, $||g^{m}|| \geq A\delta$
so $\inj_\X(\calr^{m})\geq A\delta$
      as required by the definition of $(A,\eps)$-small cancellation.

      Next, we claim that for all  $g,h \in \calR_0$ such that
      $\Delta(g,h)\geq R+(N+2)L+100\delta$, then
      $\Axis(g)$ and  $\Axis(h)$ are
      at bounded Hausdorff distance from each other.
      Indeed, by Lemma \ref{lem_commut}, there exists two points $p,y$ at distance $\geq R$ such all commutators $[g^i,h]$ for $i=1,\dots, N+1$ move
      $p$ and $y$ by at most $K\delta$.
      By acylindricity, there exists $i\neq j$ such that $[g^i,h]=[g^j,h]$, so $[g^{j-i},h]=1$.
      It follows that
      $g^{j-i}$ preserves the axis of $h$, and that $\Axis(g),\Axis(h)$ are at finite Hausdorff distance.

      If follows that for all $g,h\in \calr_0$, either $\Axis(g)$ is at finite Hausdorff distance from $\Axis(h)$, or
$\Delta(g^{m_0},h^{m_0})\leq \Delta(g,h)+K_1\delta \leq R+(N+2)L+(100+K_1)\delta$.
      Note that for all $m\geq m_0$, $\inj_\X(\calr_0^m)\geq m_0\eta \geq \frac1\eps (R+(N+2)L+(100+K_1)\delta)$ by choice of $m_0$.
      It follows that $\Delta(R_0^{m})\leq \eps \inj_\X(R_0^{m})$.

      We claim that for all $g,h$ such that $\Axis(g)$ is at finite Hausdorff distance from $\Axis(h)$,
$g^{N!}=h^{\pm N!}$. The small cancellation condition will follow.
      By assumption, $||g||=||h||$, so by Lemma \ref{lem_commut}, up to  changing $h$ to $h\m$,
      all elements $g^ih^{-i}$ move points of $\Axis^{+20\delta}(g)\cap \Axis^{+20\delta}(h)$ by at most $K\delta$.
      By acylindricity, there exists $i\neq j \in \{0,\dots, N\}$ such that $g^ih^{-i}=g^jh^{-j}$.
      It follows that $g^{i-j}=h^{i-j}$ so $g^{N!}=h^{N!}$, which concludes the proof.

    \end{proof}

Recall that the WPD condition was defined in Definition \ref{WPD}.

\begin{prop}\label{prop_WPD_SC}
Let $G$ be a group acting on a $\delta$-hyperbolic space $\X$.
Let $h_1,\dots h_n\in G$ be pairwise non-commensurable loxodromic elements satisfying the WPD condition.
Then for any $A,\eps$, there exists $m\in \mathbb N$ such that the conjugacy closed set  $\calr^{m}=\{h_1^m,\dots,h_n^m\}^G$ satisfies the $(A,\eps)$-small cancellation.
\end{prop}

\begin{proof}
  Let $\calr=\{ h_1, \ldots, h_n\}^G$. 
  Let $\eta, L$ be the minimal and maximal stable norms of the elements $h_1,\dots h_n$.
  Consider $C=K\delta$ as in Lemma \ref{lem_commut},
Denote by $\calc_{a}(x,y)$ the set of elements $g\in G$ that move $x$ and $y$ by at most $a$.
Consider $p_i$ such that for all $x\in \X$,
the set $\calc_{2K\delta}(x,h_i^{p_i}x)$ is finite.

Since any $x\in \Axis^{+20\delta}(h_i)$ is at distance at most
$20\delta$    from
$h_i^\bbZ.[x_0,h_i x_0]$, we see that $\calc_{K\delta}(x,h_i^{p_i}x)$
is bounded by some number $N_i$ independent of $x\in \Axis^{+20\delta}(h_i)$.
Consider $N=\max N_i$.

Given $A$ and $\eps>0$,
define $m_0\geq \max(\frac{A\delta}{\eta},\frac{pL+(N+2)L+(100+K_1)\delta}{\eps\eta})$.

Consider $g,h\in \calr$.
If $\Delta(g,h)\geq  pL+(N+2)L+100\delta$, then
 by Lemma \ref{lem_commut}, for all $i=1,\dots,N+1$, all commutators $[g^i,h]$ for $i=1,\dots, N+1$ move
  $y$ and $h^{-p}y$ by at most $K\delta$.
As in the previous section, this implies that some power of $g$ commutes with $h$,
so    $\Axis(g)$ and  $\Axis(h)$ are
      at bounded Hausdorff distance from each other.
It follows that $\grp{g,h}$ is virtually cyclic, so $g$ and $h$ are conjugate of the same $h_i$ by assumption.
In particular $||g||=||h||$.
Arguing as above, we see that there exists $i\leq N$ such that $g^i=h^i$, and $g^{N!}=h^{N!}$.
\end{proof}

Let us record one application of the previous discussion. Note that the constant $n$ in the proposition below is independent of $g$ in the case of an acylindrical action. For simplicity, we state the result for a single element $g$ and leave the (obvious) generalization to several element to the reader.

\begin{prop} \label{prop;Acyl_free} Let $G$ be a group acting on a hyperbolic space $\X$ and let $\alpha$ be a positive number.
\begin{enumerate}
\item[(a)] For any pairwise non-commensurable loxodromic WPD elements $g_1, \ldots, g_n\in G$, there exists $m=m(\alpha, g_1, \ldots, g_n)\in \mathbb N$ such that the collection of subgroups $\{\langle g_i^m\rangle\mid i=1, \ldots, n\}$ is $\alpha$-rotating with respect to the induced action of $G$ on a certain cone-off of $\X$.
\item[(b)] If the action of $G$ on $\X$ is acylindrical, then there exists $n=n(\alpha)$ such that for every loxodromic $g\in G$ the subgroup $\langle g^n\rangle$ is $\alpha$-rotating with respect to the induced action of $G$ on a certain cone-off of $\X$.
\end{enumerate}
\end{prop}

\begin{proof}
Let us choose $A=A(\alpha)$ and $\e =\e(\alpha)$ according to Proposition \ref{prop_sc_cyclic}. 
Now to prove (a) it suffices to apply Proposition \ref{prop_WPD_SC} to the elements $g_1, \ldots, g_n\in G$.
Similarly to prove (b) we note that all conjugates of $g$ have the same positive stable norm. We can thus apply Proposition \ref{prop;SC_from_acyl}.
\end{proof}

\subsection{Back and forth}

In this section we discuss a canonical way of constructing rotating families from normal subgroups of hyperbolically embedded subgroups. Our first result is the following.

\begin{thm}\label{he-vrf}
Let $G$ be a group, $\Hl$ a collection of subgroups of $G$, and $X$ a subset of $G$ such that $\G$ is hyperbolic. Then for every $\alpha >0$, there exists $D=D(\alpha)$ such that the following holds. Suppose that a collection of subgroups $\{ N_\lambda \}_{ \lambda \in \Lambda }$, where $ N_\lambda \lhd H_\lambda$, satisfies $\dl(1,h)>D$ for every nontrivial element $h\in N_\lambda $ for all $\lambda\in \Lambda$. Then $\{ N_\lambda \}_{ \lambda \in \Lambda }$ is $\alpha $-rotating.
\end{thm}

Before proceeding with the proof of Theorem \ref{he-vrf}, we discuss some corollaries. The first one is an immediate consequence of the theorem and the definition of a hyperbolically embedded collection of subgroups. 

\begin{cor}\label{cor-he-vrf}
Let $\Hl$ be a hyperbolically embedded collection of subgroups of a group $G$. Then for every $\alpha >0$ there exists finite subsets $\mathcal F_\lambda \subseteq H_\lambda \setminus \{ 1\}$ such that any collection $\{ N_\lambda \}_{ \lambda \in \Lambda }$, where $ N_\lambda \lhd H_\lambda$ and $N_\lambda \cap \mathcal F_\lambda =\emptyset $ for every $\lambda\in \Lambda $,  is $\alpha $-rotating.
\end{cor}

For example, this corollary together with Proposition \ref{he-rh} can be applied to construct very rotating families in relatively hyperbolic groups. We mention yet another particular case.

\begin{cor}\label{theo;WPD_free}
Let $G$ be a group acting on a hyperbolic space and let $\{ E_1, \ldots , E_k\}\h G$ be a collection of infinite virtually cyclic subgroups. Let $g_i\in E_i$ be elements of infinite order. Then for every $\alpha >0$, there exists $n\in \mathbb N$ such that the collection of  subgroups $\{\langle g_i^{n}\rangle\mid i=1, \ldots, k\} $ is $\alpha$-rotating.
\end{cor}

\begin{proof}
Since every $E_i$ is virtually cyclic, there exists $m\in \mathbb N$ such that $\langle g_i^m\rangle $ is normal in $E(g_i)$ for all $i$. Moreover, for every finite subset $\mathcal F\subseteq G\setminus\{ 1\}$, we can always ensure the condition $\langle g^{dm}\rangle\cap \mathcal F=\emptyset$ by choosing $d$ large enough. Hence by Corollary \ref{cor-he-vrf}, for every $\alpha >0$ there exists $n\in \mathbb N$ (namely, $n=dm$ for a sufficiently large $d$) such that the collection of subgroups $\{\langle g_i^{n}\rangle\mid i=1, \ldots, k\} $ is $\alpha$-rotating.
\end{proof}

In particular, Corollary \ref{theo;WPD_free} and Theorem \ref{WPD} provide an alternative way of constructing rotating families starting from WPD elements (cf. Proposition \ref{prop;Acyl_free}).

The proof of Theorem \ref{he-vrf} is divided into a series of lemmas. From now on and until the end of the proof, we work under the assumptions of Theorem \ref{he-vrf}.

We start by defining combinatorial horoballs introduced by Groves and Manning \cite{Gr_Ma}, which play an important role in our construction.  \label{i-combhorob}

\begin{defn}
Let $\Gamma$ be any graph.
The \emph{combinatorial horoball based on $\Gamma$}, denoted
$\mc{H}(\Gamma)$, is the graph formed as follows:
\begin{enumerate}
\item[1)] The vertex set $\mc{H}^{(0)}(\Gamma )$ is $\Gamma^{(0)}\times \left( \{0\}\cup \mathbb N \right)$.
\item[2)] The edge set $\mc{H}^{(1)}(\Gamma )$ contains the following three types of edges.  The
  first two types are called \emph{horizontal}, and the last type is
  called \emph{vertical}.
\begin{enumerate}
\item[(a)] If $e$ is an edge of $\Gamma$ joining $v$ to $w$ then there is a
  corresponding edge $\bar{e}$ connecting $(v,0)$ to $(w,0)$.
\item[(b)] If $k>0$ and $0<\d _{\Gamma}(v,w)\leq 2^k$, then there is a single edge
  connecting $(v,k)$ to $(w,k)$.
\item[(c)] If $k\geq 0$ and $v\in \Gamma^{(0)}$, there is an edge  joining
  $(v,k)$ to $(v,k+1)$.
\end{enumerate}
\end{enumerate}
\end{defn}

Given $r\in \mathbb N$, let $\mc D_r$ be the full subgraph of $\mc H(\Gamma )$ with vertices $\{ (y,n)\mid  n\ge r, \; y\in Y\} $. By $\d_\Gamma $ and $\d_{\mc H(\Gamma )} $ we denote the combinatorial metrics on $\Gamma $ and $\mc H (\Gamma )$ respectively. The following results were proved in \cite{Gr_Ma}. (The first one is Theorem 3.8 and the other two follow easily from Lemma 3.10 in \cite{Gr_Ma}.)

\begin{thm}[Groves-Manning]\label{GM}
\begin{enumerate}
\item[(a)] There exists $\delta >0$ such that for every connected graph $\Gamma $, $\mc H (\Gamma )$ is $\delta $-hyperbolic.
\item[(b)] For every $r\in \mathbb N$, $\mathcal D_r$ is convex.

\item[(c)] For every two vertices $a,b\in \Gamma $, we have $$\d_\Gamma (a,b)\le 2^{3(\d_{\mc H(\Gamma )} (a,b)-3)/2}.$$
\end{enumerate}
\end{thm}

Let $\Sigma $ be a graph. For a loop $c$ in $\Sigma $, we denote
by $[c]$ its homology class in $H_2(\Sigma , \mathbb Z)$. By
$\ell (c)$ and $\diam (c)$ we denote the length and the diameter of $c$
respectively. The next proposition is a homological variant of the
characterization of hyperbolic graphs by linear isoperimetric
inequality. It can be found in \cite{B_hyp}.

\begin{prop}\label{IP}
For any graph $\Sigma $ the following conditions are equivalent.
\begin{enumerate}
\item[(a)] $\Sigma $ is hyperbolic.

\item[(b)] There are some positive constants $M$, $L$ such that
if $c$ is a loop in $\Sigma $, then there exist loops $c_1,
\ldots , c_k$ in $\Sigma $ with $\diam(c_i)\le M$ for all $i=1,
\ldots, k$ such that
\begin{equation}\label{c}
[c]=[c_1]+\ldots +[c_k]
\end{equation}
and $k\le L\ell (c)$.
\end{enumerate}
\end{prop}

\begin{rem}\label{recoverML}
Clearly replacing ``$c$ is a loop" in (b) with ``$c$ is a simple loop" leads to an equivalent condition. It is also easy to see from the proof given in \cite{B_hyp} that the hyperbolicity constant of $\Sigma $ can be recovered from $M$ and $L$ and vice versa.
\end{rem}

 In the following definition we are combinatorially coning-off $\mathcal D_r$.

\begin{defn}
Given a graph $\Gamma $ and $r\ge 1$, we denote by $\mc H_r(\Gamma )$ the graph obtained from $\mc H(\Gamma )$ by adding one vertex $v$ and edges connecting $v$ to all vertices of $\mathcal D_r$. We call $v$ the {\it apex}. The  additional edges are called the {\it cone edges} of $\mc H_r(\Gamma )$.
\end{defn}

\begin{lem}\label{hypcone1}
There exists $\delta >0$ such that for every (not necessarily connected) graph $\Gamma $ and every $r\in \mathbb N\cup \{ 0\}$, $\mc H_r (\Gamma )$ is $\delta $-hyperbolic.
\end{lem}

\begin{figure}
  \centering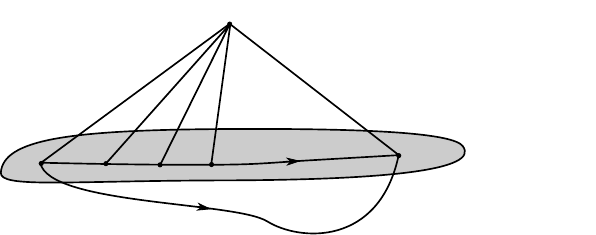\\
  \caption{}\label{46-f1}
\end{figure}

\begin{proof}
The statement follows easily from Theorem \ref{GM}. Indeed let $c$ be a simple loop in $\mc H_r (\Gamma )$. Since $c$ is simple, it passes through $v$ at most once. Hence $c$ can be decomposed as $c=ab$, where $a$ is a path in $\mc H (\Gamma )$ and $b$ is a path of length at most $2$ such that all edges of $b$ (if any) are cone edges of $\mc H _r (\Gamma )$. Let $d$ be a geodesic path in $\mc H (\Gamma )$ connecting $b_-$ to $b_+$.  Note that
\begin{equation}\label{decc1}
[c]=[ad]+[d^{-1}b].
\end{equation}
Since $\mathcal D_r$ is convex, $d$ belongs to $\mathcal D_r$. Hence $[d^{-1}b]$ can be decomposed into the sum of at most $\ell (d)$ homology classes loops of length $3$ (see Fig. \ref{46-f1}). Note that  $\ell (d)\le \ell (a) \le \ell (c)$.

By Theorem \ref{GM}, connected components of $\mc H (\Gamma )$ are hyperbolic with some universal hyperbolicity constant. By Remark \ref{recoverML} there exist  $M$ and $L$ such that all connected components of $\mc H (\Gamma )$ satisfy the condition (b) of Proposition \ref{IP}. Since $ad$ belongs to such a component, its homology class can be decomposed into a sum of at most $L \ell (ad)\le 2L\ell (a) \le 2L\ell (c)$ homology classes of loops of length at most $M$.

Now taking together the decompositions for the classes in the right side of (\ref{decc1}), we obtain a decomposition of $[c]$ into at most $(2L+1)\ell (c)$ classes of loops of length at most $M^\prime =\max\{ M, 3\}$. Thus $\mc H_r (\Gamma )$ satisfies condition (b) from the Proposition \ref{IP} with constants $M^\prime $ and and $L^\prime =2L+1$. Applying Remark \ref{recoverML} an Proposition \ref{IP} again, we obtain the claim.
\end{proof}

Lemma \ref{pres} provides us with a bounded reduced relative presentation
\begin{equation}
G=\langle X, \mathcal H  \mid  \mathcal S\cup \mathcal R\rangle
\label{rp2}
\end{equation}
with linear relative isoperimetric  function. Let $Y_\lambda\subseteq H_\lambda $ be the set of all letters from $H_\lambda \setminus\{ 1\} $ that appear in words from $\mathcal R$. Let $$Y = \bigcup\limits_{\lambda \in \Lambda }Y_\lambda.$$ Fix also any $r\in \mathbb N\cup\{ 0\}$. To these data we associate a graph $\mathbb K=\mathbb K (G, X, Y, \Hl, r)$ as follows.

\begin{defn} \label{KGXH}
Let $\Gamma (G, X\cup Y)$ be the Cayley graph of $G$ with respect to the set $X\cup Y$. Note that $\Gamma (G, X\cup Y)$ is not necessarily connected. Indeed it is connected iff $X$ is a relative generating set of $G$ with respect to the subgroups $\langle Y_\lambda\rangle$, which is not always the case. Let $\Gamma (H_\lambda, Y_\lambda)$ be the Cayley graph of $H_\lambda $ with respect to $Y_\lambda$. Again we stress that  $\Gamma (H_\lambda, Y_\lambda)$ is not necessarily connected. In what follows, $g\Gamma (H_\lambda, Y_\lambda)$ denotes image of $\Gamma (H_\lambda, Y_\lambda)$ under the left action of $G$ on $\Gamma (G, X\cup Y)$. For each $\lambda \in \Lambda $ we fix a set of representatives $T_\lambda $ of left cosets of $H_\lambda $ on $G$. Let
$$
\mathcal Q =\{ g\Gamma (H_\lambda, Y_\lambda) \mid \lambda\in \Lambda , g\in T_\lambda \}.
$$
Let $\mathbb K _r(G, X, Y, \Hl)$ be the graph obtained from $\Gamma (G, X\cup Y)$ by attaching $\mathcal H_r(Q)$ to every $Q\in \mathcal Q$ via the obvious attaching map $(q,0)\mapsto q$, $q\in Q$.
\end{defn}

The next lemma is similar to Theorem 3.23 from \cite{Gr_Ma}.
\begin{figure}
  \centering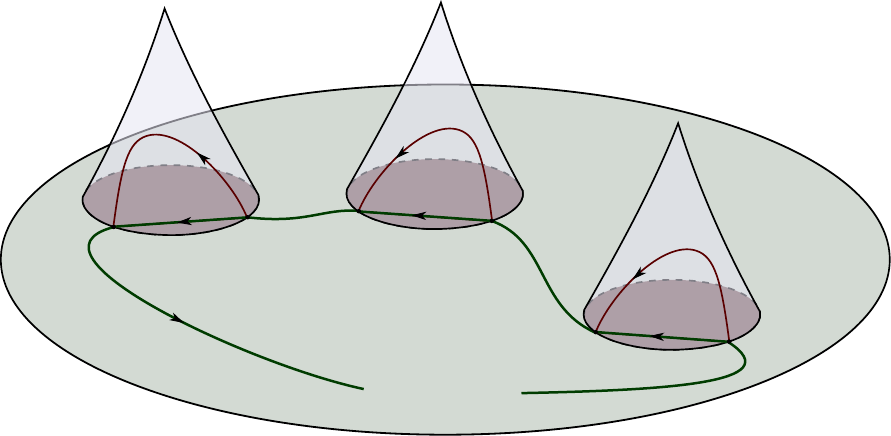\\
  \caption{}\label{46-f2}
\end{figure}

\begin{lem}\label{hypcone2}
There exists $\delta >0$ such that for every $r\in \mathbb N\cup\{ 0\}$, the graph $\mathbb K _r=\mathbb K (G, X, Y, \Hl)$ is $\delta $-hyperbolic.
\end{lem}

\begin{proof}
Observe that $\mathbb K _r$ is connected as left cosets of $H_\lambda$'s belong to connected subsets in $\mathbb K $ and $X$ generates $G$ relative to $\Hl $.

We will use Proposition \ref{IP} again. To each simple loop $c$ in $\mathbb K _r$ we associate a loop in $\Gamma (G, X\cup Y)\subseteq \G $  as follows. Let $b_1, \ldots , b_k$ be the set of all maximal subpaths of $c$ such that each $b_i$ belongs to $\mc H_r (Q_i) \setminus \Gamma ^{(1)}(G, X\cup Y)$ for some $\lambda _i\in \Lambda $ and $Q_i\in \mathcal Q$. We replace each $b_i$ with the edge  $e_i$ in $\G $ connecting $(b_i)_-$ to $(b_i)_+$ and labelled by an element of $H_{\lambda _i}$. Let $c^\prime $ be the resulting loop in $\G $.

Consider a van Kampen diagram $\Delta $ over (\ref{rp2}) such that:
\begin{enumerate}
\item[(a)] The boundary label of $\Delta $ is $\Lab (c^\prime)$.
\item[(b)] $\Delta $ has minimal number of $\mathcal R$-cells among all diagrams satisfying (a).
\item[(c)] $\Delta $ has minimal number of $\mathcal S$-cells among all diagrams satisfying (a) and (b).
\end{enumerate}
In what follows we identify $\partial \Delta $ with $c^\prime$.

The maps $e_i\mapsto b_i$ naturally induce a continuous map $\phi $ from $c^\prime$ to $\mathbb K_r$ whose image is $c$. Observe that (b) and (c) imply that every internal edge of $\Delta $ belongs to an $\mathcal R$-cell. Hence every such an edge is labelled by some element of $X\cup Y$ by the definition of $Y_\lambda$'s and the fact that the presentation (\ref{rp2}) is reduced. Thus we can naturally extend $\phi $ to the $1$-skeleton of $\Delta $. Note also that the total length of boundaries of all $\mathcal S$-cells of $\Delta $  does not exceed the total lengths of boundaries of all $\mathcal R$-cells. Let $f(n)=Cn$ be a relative isoperimetric function of (\ref{rp2}) and $M=\max\limits_{R\in \mathcal R} \| R\| $. Note that $M<\infty $ as (\ref{rp2}) is bounded.  Then $[c]$ decomposes into the sum of at most $C\ell (c^\prime)\le C\ell (c)$ homotopy classes of loops of length at most $M$ (corresponding to $\mathcal R$-cells of $\Delta $) plus $[s_1]+\cdots + [s_m]$, where $s_i$'s are images of boundaries of $\mathcal S$-cells and
\begin{equation}\label{tlsi}
\sum\limits_{i=1}^m \ell (s_i)\le MC\ell (c^\prime )+\ell (c)\le (MC+1)\ell (c).
\end{equation}

Note that every $s_i$ is a loop in some $\mc H_r(Q)$ and hence by Lemma \ref{hypcone1} there exist some constants $A,B$ independent of $r$ such that $[s_i]$ decomposes into the sum of at most $A\ell (s_i)$ homotopy classes of loops of length at most $B$. Hence $[c]$ decomposes into the sum of at most $(C+A(MC+1))\ell(c)$ homotopy classes of loops of length at most $\max \{ M, B\}$. Hence by Proposition \ref{IP}, $\mathbb K _r$ is $\delta$-hyperbolic, where $\delta $ is independent of $r$.
\end{proof}

We are now ready to prove the main result of this section.

\begin{figure}
  \centering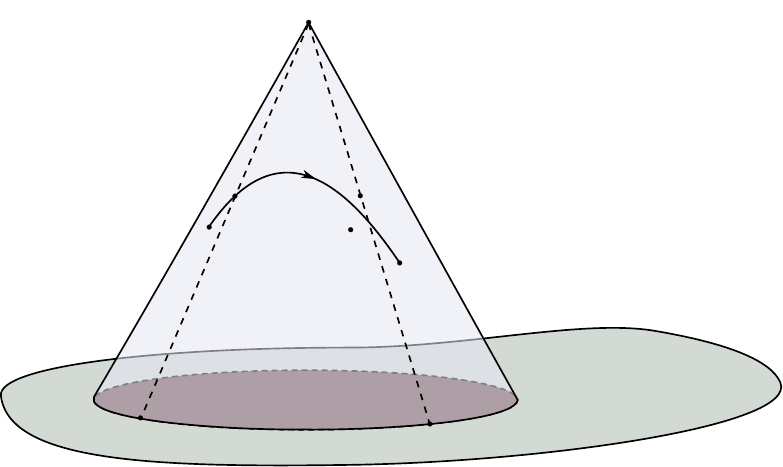\\
  \caption{}\label{46-f3}
\end{figure}

\begin{proof}[Proof of Theorem \ref{he-vrf}]
By Lemma \ref{hypcone2}, there exists $\delta >0$ such that $\mathbb K _r$ is $\delta $-hyperbolic for any $r$. Without loss of generality we may assume that $\delta \ge 1$. We take
\begin{equation}\label{choice_of_r}
r>\delta \max \{\alpha/2, 100\}
\end{equation}
Denote the combinatorial metric on $\mathbb K _r$ by $d$.

The left action of the group $G$ on $\Gamma (G, X\cup Y)$ can be extended to the action on $\mathbb K_r$ in a natural way. Namely given $g\in G$ and any vertex $(x,n)$ of $\mathcal H_r(Q)$ for some $Q\in \mathcal Q$ and $n\in \mathbb N$, we define  $g(x,n)$ to be the vertex $(gx,n)$ of $\mathcal H_r(gQ)$. Further we denote by $a_Q$ the apex of $\mathcal H_r(Q)$ and define $g(a_Q)=a_{gQ}$. This gives an action of $G$ on the set of vertices of $\mathbb K_r$. It is straightforward to check that this action preserves adjacency of vertices and hence extends to the action on $\mathbb K_r$.

Let $C=\{ a_Q\}_{Q\in \mathcal Q}$. If $Q=g\Gamma (H_\lambda, Y_\lambda )$ for some $g\in G$ and $\lambda\in \Lambda$, let $G_{a_Q}=gN_\lambda g^{-1}$. It is easy to verify that $(C, \{ G_c\}_{c\in C})$ is a rotating family. Clearly $C$ is $2r$-separated.  In particular, $C$ is $\alpha \delta $-separated by (\ref{choice_of_r}). To complete the proof it remains to show that $(C, \{ G_c\}_{c\in C})$ is very rotating.

Let $c\in C$ and let $x,y\in \mathbb K_r$, $g\in G_c\setminus\{1\}$, be as in the definition of a very rotating family. That is, suppose that
$$20\delta \le \d (x,c),\d (y,c)\le 40 \delta $$
and $$\d (gx,y)\le 15\delta .$$ Without loss of generality we may assume that $c$ is the apex of $\mc H_r(Q)$, where $Q=\Gamma (H_\lambda, Y_\lambda)$ for some $\lambda $ and thus $G_c=N_\lambda $.
Since $r\ge 100\delta $ and $\delta >1$, we have $x,gx,y\in \mc H(Q)\subset \mc H_r(Q)$ (see Fig. \ref{46-f3}).

Suppose that a geodesic $\gamma $ in $\mathbb K _r$ connecting $x$ and $y$ does not pass through $c$. This means that $\gamma $ does not intersect any cone edge of $\mc H_r(Q)$. On the other hand, $\gamma $ does not intersect $\Gamma (G, X\cup Y)$ as
$$\d (x,y) \le \d (x,c) +\d  (c,y)= \le 80\delta, $$ while any path between $x$ and $y$ intersecting $\Gamma (G, X\cup Y)$ would have length at least $$r-\d(c,x)+r-\d(c,y)\ge 100\delta - 40\delta +100\delta - 40 \delta >80 \delta .$$  Thus $\gamma $ entirely belongs to $\mc H(Q)$ and hence
$\d _{\mc H(Q)} (x,y)\le 80\delta $, where $\d _{\mc H(Q)} $ denotes the combinatorial metric on $\mc H(Q)$.
Similarly $\d _{\mc H(Q)} (gx,y)\le 15\delta $.

Note that $x$ is not necessary a vertex of $\mc H(Q)$ (it can be an internal point of an edge). Let $x^\prime \in \gamma $ be the vertex of $\mc H(Q)$ closest to $x$. We have
$$
\d _{\mc H(Q)} (x^\prime ,gx^\prime ) \le \d _{\mc H(Q)} (x,gx)+2\le \d _{\mc H(Q)}(x, y)+\d _{\mc H(Q)}(y, gx)+2\le 80\delta +15\delta +2\le 97\delta .
$$
Let $x^\prime =(a,n)$ for some $n\in \mathbb N$ and $a\in H_\lambda $. Then $gx^\prime=(ga,n)$. Recall that the vertex $a$ of $Q$ is identified with the vertex $(a,0)$ of $\mc H (Q)$. Thus we obtain
$$
\d _{\mc H(Q)} (a, ga) \le \d _{\mc H(Q)}(a, x^\prime) + \d _{\mc H(Q)}(x^\prime, gx^\prime ) + \d _{\mc H(Q)}(gx^\prime ,ga) \le 2n + 97\delta \le 2r+97\delta.
$$
Our sets $Y_\lambda$ are chosen in the same way as in the proof of Lemma \ref{omega} (see the first paragraph of the proof). Hence by part (b) of Lemma \ref{omega} there exists a constant $K$ such that $\dl (u,v)\le K\d _{Y_\lambda} (u,v)$ for every $u,v\in H_\lambda $. Applying part (c) of Theorem \ref{GM} we obtain
$$
 \dl (1, a^{-1}ga) =\dl (a, ga)\le K\d _{Y_\lambda} (a, ga) \le 2^{3(2r+97\delta -3)/2}K.
$$
Since $N_\lambda $ is normal in $H_\lambda$, we have $a^{-1}ga\in N_\lambda \setminus\{1\}$. This leads to a contradiction if
$D> 2^{3(2r+97\delta -3)/2}K$.
\end{proof}

In the other direction, we note that every $\alpha $-rotating subgroup of a group $G$ remains $\alpha $-rotating in $G\times \mathbb Z$ via the obvious induced action of $G\times \mathbb Z$. However $G\times \mathbb Z$ does not have any non-degenerate \he subgroups by Corollary \ref{center}. Thus, in general, passing from rotating families to \he subgroups is impossible. However, we show that, under good circumstances, very rotating subgroups are hyperbolically embedded.

Let $Y$ be a hyperbolic space, $\calc$ a $G$-invariant set of points,
and $\calc_0\subset\calc$ be a set of representatives of $\calc/G$.
Fix $R>0$ and $Y_0=Y\setminus \calc^{+R}$
the complement of the $R$-neighborhood of $\calc$
endowed with its intrinsic path metric $d_{Y_0}$.

\begin{lem}\label{lem_criterion}
Assume that the action of $G$ on $Y$ is cobounded.
Consider $x_0\in Y_0$, and assume  that
\begin{enumerate}
\item \label{it_proper} for each $c\in\calc$, $\Stab_G(c)$ acts properly on  $Y\setminus B_{R}(c)$
with its intrinsic metric;
\item \label{it_paths} for each $c\in\calc_0$ there is a path $q_c$ joining $x_0$ to $c$ and avoiding $(\calc\setminus\{c\})^{+R}$
\item \label{it_S} there is a (maybe infinite) set $S\subset G$ such that
  \begin{enumerate}
\item \label{it_rad} for each $D>0$, the set of elements of $G$ moving $x_0$ by at most $D$ for the metric $d_Y$
is contained in a ball of finite radius of $G$ for the word metric over $S\cup \{\Stab_G(c)\}_{c\in\calc_0}$

  \item \label{it_rad2} all elements of $S$
move $x_0$ by a bounded amount for the intrinsic metric $d_{Y_0}$
  \end{enumerate}
\end{enumerate}

Then $\{\Stab_G(c)\}_{c\in \calc_0}$ is hyperbolically embedded in $G$ with respect to $S$.
\end{lem}

\begin{rem}
 In the first assumption, one can replace $Y\setminus B_{R}(c)$
by the smaller set
$B_{R+20\delta}(c)\setminus B_{R}(c)$ that plays the role of the link around $c$.
This follows from the divergence of geodesics and the fact that the closest point projection to the convex set $B_{R+20\delta}(c)$ in the
hyperbolic space $Y$ is almost length decreasing.
\end{rem}

\begin{proof}
Consider $\calh= \bigsqcup\limits_{c \in \calc_0} \Stab_G(c)\setminus\{1\}$,
and the Cayley graph $Z=\Gamma(G,S\sqcup \calh)$.
Since $d_Y\leq d_{Y_0}$, all elements of $S\cup \calh$
move $x_0$ at bounded distance away for $d_Y$,
so the map $Z\ra Y$ sending $g$ to $g.x_0$ is Lipschitz.
Since $G$ acts coboundedly on $Y$, Assumption
\ref{it_rad}
ensures that this map is a quasi-isometry, so $Z$ is hyperbolic.

Given $c_0\in\calc_0$ and $n>0$, we need to check that
there are only finitely many elements $g\in \Stab_G(c_0)$ that can be written as $g=s_1\dots s_n$
where the corresponding path in the Cayley graph $\Gamma(G,S\sqcup\calh)$ does not contain
any edge of $\Gamma(\Stab_G(c_0),\Stab_G(c_0))$.
Denoting by $w_i=s_1\dots s_i$, this amounts to ask that
whenever $w_i\in G(c_0)$, then  $s_i$ is from the set
$S\cup\bigsqcup\limits_{c \in \calc_0\setminus\{c_0\}} \Stab_G(c)\setminus\{1\}$.

To such a word, we associate a path $p_1\dots p_n$ of bounded length joining $x_0$ to $gx_0$ in $Y\setminus B_{R}(c_0)$.
By Assumption \ref{it_proper}, this will imply that there are finitely many such elements $g$, concluding the proof.
If $s_i\in S$, Assertion
\ref{it_rad2} gives us a path $p_{s_i}\subset Y_0$ of bounded length joining $x_0$ to $s_ix_0$,
and we take $p_i=w_{i-1}p_{s_i}$.
If $s_i$ is from the alphabet $\Stab_G(c)\setminus\{1\}$, we use the path $q_c$ given
by Assumption \ref{it_paths} to construct
the path $q= q_c.s_i\bar q_c$ joining $x_0$ to $s_i x_0$ and avoiding
$\calc\setminus\{c\})^{+R}$, and we take $p_i=w_{i-1}q$.
To prove that $p_i$ avoids $B_{R}(c_0)$, we check that $w_{i-1}c\neq c_0$.
If $w_{i-1}c=c_0$, then
since $c,c_0$ lie in the set of representatives $\calc_0$,
we get that $c=c_0$.
It follows that $w_{i-1}\in\Stab_G(c_0)$,
and since $s_i\in\Stab_G(c_0)$, this contradicts the form of the word $s_1\dots s_n$.
\end{proof}

Although less general that Theorem \ref{crit} because of the coboundedness assumption,
the following corollary is more direct.

\begin{cor}\label{cor;RFtoHE}
  Let $\X$ be a hyperbolic hyperbolic space, with a cobounded action of $G$,
and $\calQ\subset \X$ a $G$-invariant, $G$-finite family of quasiconvex subspaces.
Assume that $\calQ$ is geometrically separated: for all $Q\neq Q'$, $Q, Q'\in\calQ$, and all $\eps>0$, there exists $R$
such that $\diam(Q^{+\eps}\cap Q'^{+\eps})\leq R$.
Let $(Q_\lambda)_{\lambda\in\Lambda}$ be a family of representatives $\calQ$ modulo $G$.

If for each $\lambda\in\Lambda$,  $\Stab_G(Q_\lambda)$ acts properly and coboundedly on $Q_\lambda$,
then  $\{\Stab_G(Q_{\lambda})\}_{\lambda\in\Lambda}$ is hyperbolically embedded in $G$.
\end{cor}

\begin{proof}
Let $\kappa$ be such that every $Q\in \calq$ is $\kappa$-quasiconvex.
Up to changing each $Q\in \calQ$ to $Q^{+\kappa}$, we can assume that each $Q\in\calQ$
is $2\delta$-\sqc\ (see Lemma \ref{lem_4delta}). Up to rescaling the metric on $\X$, we can assume
that $\delta\leq \delta_c$ and $\Delta(\calQ)\leq \Delta_c$ where $\delta_c,\Delta_c$ are the constants
appearing in Theorem \ref{theo;cone-off_hyp_loc}, which guarantees that the cone-off of $\X$ over $\calQ$ is hyperbolic.

Fix a basepoint $x_0\in \X$, and $D_0$ such that any point in $\X$ lies at distance at most $ D_0$
from the orbit of $x_0$.
Up to changing our choice of representatives, we can assume that $d_\X(x_0,Q_\lambda)\leq D_0$
for all $\lambda\in\Lambda$.
Let $S$ be the set of elements of $G$ moving $x_0$ by at most $3D_0$.
Let $Y=C(\X,\calQ,r_0)$ be the cone-off of $\X$ along $\calQ$ for $r_0\geq \max(\ru,40\delta_U)$.
By Corollary \ref{coro;cone-off_hyp}, $Y$ is hyperbolic.
We denote by $\calc\subset Y$ the
set of apices, $c_Q$ the apex corresponding to $Q\in \calQ$, and by $C(Q)\subset Y$
the cone on $Q$.
We take $\calc_0=\{c_{Q_\lambda}\}_{\lambda\in\Lambda}$.
As in Lemma \ref{lem_criterion}, we consider $Y_0=Y\setminus \calc^{+20\delta_U}$.

We check that the hypotheses of Lemma \ref{lem_criterion} are satisfied.
The action of $G$ on $Y$ is clearly cobounded and Assumption \ref{it_paths} is also clear.
Since  $X\subset Y_0\subset Y$,
for all $g\in S$, $d_{Y_0}(x_0,gx_0)\leq d_{X}(x_0,gx_0)\leq 3D_0$.
Assumption \ref{it_rad2} follows.

Let us check that $\Stab_G(c)$ acts properly on  $B_{r_0}(c)\setminus B_{20\delta_U}(c)$
for its intrinsic metric.
As noted above, this will imply that the first assumption of Lemma \ref{lem_criterion} is satisfied.
Consider the radial projection $p_c:B_{r_0}(c)\setminus\{c\}\ra X$ defined above Proposition \ref{prop;cone_BH}.
It easily follows from \cite[Prop. 2.1.4]{Coulon} that this map is locally Lipschitz: there exists $L>0$ such that
if $x,y\in B_{r_0}(c)\setminus B_{20\delta_U}(c)$ are at distance at most $10\delta_U$, then
$d_X(p_c(x),p_c(y))\leq L d_Y(x,y)$.
Since the action of $\Stab_G(c)$ on the corresponding subspace in $\calQ$ is proper,
Assumption \ref{it_proper} of Lemma \ref{lem_criterion} follows.

Let $D'$ be such that for each $\lambda\in \Lambda$, the group
 $\Stab_G(Q_{\lambda}$
acts $D'$-coboundedly on $Q_{\lambda}$.
To prove Assumption \ref{it_rad}, fix any $D>0$ and consider $g\in G$ such that $d_Y(x_0,gx_0)\leq 3D$.
If a geodesic $[x_0,gx_0]$ in $Y$ avoids  $\calc^{+20\delta_U}$, the radial projection of this geodesic
gives a path showing that $d_X(x_0,gx_0)\leq 3DL$.
In general, write $[x_0,gx_0]$ as a concatenation of paths $p_0q_1\dots q_np_n$
where for each $i$, $p_i$ avoids $\calc^{+20\delta_U}$,
and $q_i$ is a path contained in a cone $C(Q_{i})$, with endpoints in $X$, and intersecting
$B_{20\delta_U}(c_i)$.
As above, the length of the radial projection of $p_i$ is at most $3DL$.
Since the length of $q_i$ is at least $2(r_0-20\delta_U)$, the number of
paths $q_i$ is bounded.
Since $\Stab_G(c_{Q_i})$ acts $D'$-coboundedly on $Q_i$,
one easily gets that $g$ can be written as a product of a bounded number of elements of
$S\cup \{\Stab_G(Q_\lambda)\}_{\lambda\in\Lambda}$.
Assumption  \ref{it_rad} follows, and we can apply Lemma \ref{lem_criterion}.
\end{proof}

\subsection{Some particular groups}\label{subsec:spg}

In this section we discuss some particular examples.

We begin with mapping class groups. Recall that every mapping class group admits an action on the so-called curve complex. The definition of the curve complex is not essential for our goals and we refer the interested reader to \cite{MM99}.  The following lemma is due to Masur-Minsky \cite{MM99} and Bowditch \cite{B_hyp, B_acyl}.

\begin{lem}\label{actionMCG}
Let $\Sigma$ be a $p\geq 0$ times punctured closed orientable surface of genus $g$ such that $3g+p-4>0$ and let $\calM\calC\calG(\Sigma)$ denote its mapping class group. Let also $\mathcal C$ denote the curve complex of $\Sigma$. Then the following conditions hold.
\begin{enumerate}
\item[(a)] (Masur-Minsky \cite{MM99}, Bowditch \cite{B_hyp}) $\mathcal C$ is hyperbolic.
\item[(b)] (Bowditch \cite{B_acyl}) The action of $\calM\calC\calG(\Sigma)$ on $\mathcal C$ is acylindrical.
\end{enumerate}
\end{lem}

\begin{thm}\label{ex-MCG}
Let $\Sigma$ be a (possibly punctured) closed orientable surface. 
\begin{enumerate}
\item[(a)] For every collection of pairwise non-commensurable pseudo-Anosov elements $a_1, \ldots , a_k\in \calM\calC\calG(\Sigma)$, we have $\{ E(a_1), \ldots , E(a_k)\}\h \calM\calC\calG(\Sigma)$, where $E(a_i)$ is the unique maximal elementary subgroup containing $a_i$, $i=1, \ldots , k$. Furthermore, for every $\alpha >0$, there exists $n\in \mathbb N$ such that the collection $\{ \langle a_i^n\rangle \mid i=1, \ldots , k\}$ is $\alpha$-rotating.
\item[(b)] For every $\alpha >0$, there exists $n\in \mathbb N$ such that for every pseudo-Anosov element $a\in \calM\calC\calG(\Sigma)$, the cyclic subgroup $\langle a^{n}\rangle $  is $\alpha$-rotating.
\end{enumerate}
\end{thm}

\begin{proof}
We first observe that in all exceptional cases (i.e., when $3g+p-4\le 0$), $\calM\calC\calG(\Sigma)$ is hyperbolic and pseudo-Anosov elements have infinite order. In this situation the first claim of the theorem is well known (see, e.g., \cite{Bow}) and also follows immediately from Theorem \ref{wpd} as the action of a hyperbolic group on its Cayley graph with respect to a finite generating set is acylindrical and thus all infinite order elements are loxodromic WPD elements.

Suppose now that $3g+p-4> 0$ and consider the action of $\calM\calC\calG(\Sigma)$ on the curve complex $\mathcal C$, which is hyperbolic by part (a) of Lemma \ref{actionMCG}. Then part (b) of Lemma \ref{actionMCG} obviously implies the WPD property for every loxodromic element. Recall also that pseudo-Anosov elements  are precisely the loxodromic elements with respect to this action. Thus the first claim in (a) follows from Theorem \ref{wpd}; the second claim in (a) follows from either Corollary \ref{theo;WPD_free} or directly from part (a) of Proposition \ref{prop;Acyl_free}. Note that the constant $n$ in part (a) {\it a priori}  depends on the elements $a_1, \ldots, a_k$.  The more uniform version of this statement in (b) follows immediately from Lemma \ref{actionMCG}(b) and Proposition \ref{prop;Acyl_free} (b). Note, however, that (b) only applies to a single element.
\end{proof}

A result similar to part (a) of the previous theorem  also holds for outer automorphism groups of free groups, with iwip elements in place of pseudo-Anosov. Recall that an element $g\in Out(F_n)$ is {\it irreducible with irreducible powers} (or {\it iwip}, for brevity) if none  of its non-trivial powers preserve the conjugacy class of a proper free factor of $F_n$.
Given $F_n$, and a finite family $I$ of iwip elements in $Out(F_n)$, Bestvina and Feighn \cite{BFe,BFe2} constructed hyperbolic spaces on which $Out(F_n)$ acts so that the action of the elements of the family $I$ is loxodromic and satisfies the WPD condition; alternatively, we can use the free factor complex (see \cite[Th. 9.3]{BFe2}). Arguing as in the proof of Theorem \ref{ex-MCG} (a), we obtain the following.

\begin{thm}\label{outfn}
Let $F_n$ be the free group of rank $n$, $g_1, \ldots , g_k$ a collection of pairwise non-commensurable iwip elements in $Out(F_n)$. Then  $\{ E(g_1), \ldots , E(g_k)\}\h Out(F_n)$, where $E(g_i)$ is the unique maximal elementary subgroup containing $g_i$, $i=1, \ldots , k$. Furthermore, for every $\alpha >0$, there exists $n\in \mathbb N$ such that the collection of cyclic subgroups $\{ \langle g_i^{n}\rangle \mid i=1, \ldots , k\} $ is $\alpha $-rotating.
\end{thm}

\label{i-Cremo}
A similar argument  works for the Cremona groups. Recall that the $n$-dimensional Cremona group over a field $\bf k$ is the group ${\bf Bir}(\mathbb P^n_{\bf k})$ of birational transformations of the projective space $\mathbb P^n_{\bf k}$. In \cite{CL}, Cantat and Lamy used the Picard-Manin space to construct a hyperbolic space $\mathbb H_{\bar{\mathcal Z}}$ on which the group ${\bf Bir}(\mathbb P^n_{\bf k})$ acts. In fact, $\mathbb H_{\bar{\mathcal Z}}$ is the infinitely dimensional hyperbolic space in the classical sense.

Further, Cantat and Lamy introduce the notion of a tight element of a group $G$ acting on a hyperbolic space $S$, which can be restated as follows (see paragraph 2.3.3  and [Lemma 2.8] in \cite{CL}).
\begin{defn}\label{tightdef}
An element $g\in {\bf Bir}(\mathbb P^2_{\mathbb C})$ is \emph{tight} if the following conditions hold.
\begin{enumerate}
\item[(T$_1$)] $g$ acts on $S$ loxodromically and has an invariant geodesic axes $Ax(g)$.
\item[(T$_2$)] There exists $C >0$ ($C=2\theta $ in the notation of \cite{CL}) such that for every $\e\ge C$ there exists $B>0$ such that if $${\rm diam} (Ax(g)^{+\e} \cap f(Ax(g))^{+\e})\ge B$$ for some $f\in G$, then $f(Ax(g))=Ax(g)$.
\item[(T$_3$)] If for some $f\in G$ we have $f(Ax(g))=Ax(g)$, then $f^{-1}gf=g^{\pm 1}$.
\end{enumerate}
\end{defn}
In \cite{CL} it is shown that generic (in a certain precise sense) transformations from ${\bf Bir}(\mathbb P^2_{\mathbb C})$ are tight with respect to the action on the hyperbolic space $\mathbb H_{\bar{\mathcal Z}}$.

We will also need the following result proved in \cite{BC}

\begin{lem}[{\cite[Corollary 4.7]{BC}}]\label{BC}
Let $g\in {\bf Bir}(\mathbb P^2_{\mathbb C})$ be a loxodromic element with respect to the action on $\mathbb H_{\bar{\mathcal Z}}$. Then the centralizer of $g$ in ${\bf Bir}(\mathbb P^2_{\mathbb C})$ is virtually cyclic.
\end{lem}

Let now $g\in {\bf Bir}(\mathbb P^2_{\mathbb C})$ be a tight element and let $$E(g)=\{ f\in {\bf Bir}(\mathbb P^2_{\mathbb C})\mid f(Ax(g))=Ax(g)\} .$$ Condition (T$_3$) implies that the centralizer of $g$ in ${\bf Bir}(\mathbb P^2_{\mathbb C})$ has index at most $2$ in $E(g)$. Hence by Lemma \ref{BC}, $E(g)$ is virtually cyclic. This means that $\langle g\rangle $ has finite index in $E(g)$, which in turn implies that the action of $E(g)$ on $\mathbb H_{\bar{\mathcal Z}}$ is proper since $g$ is loxodromic (see (T$_1$)). Further let $s$ be any point of $Ax(g)$. Then $\d_{Hau} (E(g)(s), Ax(g))<\infty $ and hence $E(g)(s)$ is quasi-convex. Finally observe that (T$_2$) implies that $E(g)$ is a geometrically separated subgroup of ${\bf Bir}(\mathbb P^2_{\mathbb C})$ with respect to the action on $\mathbb H_{\bar{\mathcal Z}}$. Thus Theorem \ref{crit} applies. Combining it with Corollary \ref{theo;WPD_free} we obtain the following.

\begin{cor}\label{crem}
Let $g$ be a tight element of the Cremona group ${\bf Bir}(\mathbb P^2_{\mathbb C})$. Then there exists an elementary subgroup $E(g)$ of ${\bf Bir}(\mathbb P^2_{\mathbb C})$ which contains $g$ and is hyperbolically embedded in ${\bf Bir}(\mathbb P^2_{\mathbb C})$. Furthermore, for every $\alpha>0$ there exists $n\in \mathbb N$ such that the cyclic subgroup $\langle g^{n}\rangle $ is $\alpha $-rotating.
\end{cor}


\section{Dehn filling}


\subsection{Dehn filling via rotating families}

Recall a definition of relative hyperbolicity, which is equivalent to
Definition \ref{rhg} for countable groups (see \cite[\S 2, \S 5,
Theorem 5.1]{Hru} for this equivalence; we also borrow the following definition
of horoball from there).

If $\bbX$ is a hyperbolic space, and $\xi\in \partial \bbX$, a
horofunction at $\xi$ is a function $h:\bbX\to \bbR$ such that there
exists a constant $D_0$ for which, for all
geodesic triangle of vertices $\xi$ and $x,y \in \bbX$, and all $w$ at
distance at most $\delta$ from each side of the triangle, one has
$\vert (h(x) - d(x,w)) - (h(y) - d(y,w)) \vert <D_0$. An horoball
centered at $\xi$  is a subset $H$ of $\bbX$  for which there is an
horofunction $h$ centered at $\xi$, and $D_1$,  such that    $\forall x\in H, h(x) \geq
- D_1 $ and    $\forall x\in \bbX\setminus H, h(x) \leq
  D_1 $. Note that combinatorial horoballs are horoballs in this
  sense.  \label{i-horob}

     \begin{defn}\label{dfn_relh} Let $G$ be a countable group, and $\calP$ a family of subgroups, closed under conjugacy.

       One says that $G$ is hyperbolic relative to $\calP$ (or to a set of conjugacy representative of  $\calP$ in $G$)
       if $G$ acts properly discontinuously by
       isometries on a proper geodesic $\delta$-hyperbolic
 graph $\X$, such that, for all $L>0$,
       there exists a $G$-invariant  family of  closed horoballs $\calH$ of $\X$ such that

       \begin{enumerate}
         \item[(a)] $\calh$ is $L$-separated: any two points in two different horoballs of $\calH$ are at distance at least $L$
         \item[(b)] the map  $\phi: \calH \to \calP$ defined by $\phi(H)=Stab_G(H)$ is a bijection
         \item[(c)] $G$ acts co-compactly on
$\X\setminus \left(  \bigcup_{H\in \calH} \rond H \right)$.
       \end{enumerate}
     \end{defn}

     As before, we can assume that $\X$ is a metric graph whose edges have the same length.
     The horoballs can be chosen so that they don't intersect any  ball given in advance,
     and they can be assumed to be
     $4\delta$-\sqc\ subgraphs (see Lemma  \ref{lem_4delta}).

     The family $\calP$ is finite up to conjugacy in $G$, and it is convenient to consider
     representatives $P_1, \dots , P_n$ of the conjugacy classes. We will also say that  $G$ is hyperbolic relative to $\{P_1, \dots , P_n\}$.

In the following,  we propose a specific cone-off construction over such a space $\X$, and proceed to
an argument for the Dehn filling theorem \cite{Osi07,Gr_Ma} through the
construction of very rotating families. \label{i-df2} Recall that this
theorem generalizes a construction of Thurston on hyperbolic
manifolds, and states that for all group $G$ that is hyperbolic
relative to $\{P_1, \dots, P_n\}$, there exists a finite set $F
\subset G\setminus\{1\}$ such that whenever one considers groups $N_i
\normal P_i$ avoiding $F$,        the quotient $\bar{G} =
G/\langle\langle \cup_i N_i \rangle\rangle$ is  again relatively
hyperbolic, relative to the images of the parabolic groups which are
$P_i/N_i$.
 In fact, this can be viewed as a variation on the small cancellation condition (see Lemma \ref{lem_RH_inj_rad} below).

Our motivation for this construction of  rotating family is to get a good control on the spaces appearing in the proof, and in particular the hyperbolic space on which the quotient group $\bar{G}$ acts. Indeed, consider for instance the case of groups $N_i$ of finite index in  $P_i$. Even though in this case the Dehn fillings are hyperbolic when the theorem applies, there is in principle no good control on the hyperbolic constant of their Cayley graph (for the image of a fixed generating set of $G$). In fact since big finite subgroups appear, the hyperbolicity constant has to go to infinity with the index of $N_i$ in $P_i$. On the contrary, the original construction of Thurston, on finite volume hyperbolic manifolds, provides hyperbolic compact manifolds of controlled volume. This is the phenomenon that we want to capture here, in statements, even if it was already implicitly present in the proofs of the Dehn filling theorems for relatively hyperbolic groups \cite{Osi07, Gr_Ma}. It turns out that rotating families are well suited for that. This aspect will be used in the forthcoming work of the two first named authors characterizing the isomorphism class of
a relatively hyperbolic group
in terms of its Dehn fillings.

If $\X$ is a $\delta_c$-hyperbolic space and $\calh$ a
$50\delta_c$-separated system of horoballs,
its fellow traveling constant $\Delta(\calH)$ is zero (as defined in Section \ref{sec_coneoff}),
and coning off the horoballs of $\calh$ yields a hyperbolic space:
for all $r_0\geq \ru$, $\dot \X=C(\X,\calh,r_0)$ is $\du$-hyperbolic by  Corollary \ref{coro;cone-off_hyp},
with $\delta_c,\ru$ as in Theorem \ref{theo;cone-off_hyp_loc}.

The assumption that $\X$ is $\delta_c$-hyperbolic is not a restriction thanks to rescaling (once given $r_0$).
However, this does not produce a very rotating family on $\dot \X$.
Indeed, for any parabolic element $g$, there are points very deep in a horoball of $\X$ moved by $g$ by a small amount.
This prevents $g$ to be part of a very rotating family on the cone-off
$\dot \X$.
   This is why we are going consider a subset of the cone-off where we remove all those bad points.
We will call this subset the \emph{parabolic cone-off}.
\\

So start with $\X$, a $\delta_c$-hyperbolic space and $\calh$ a $50\delta_c$-separated system of horoballs $\calh$.
For each horoball $H\in \calh$ of $\X$,
denote by $\partial H=H\setminus \rond H$ the corresponding horosphere.
Now consider the constant  $r_U$ given by Theorem \ref{theo;cone-off_hyp_loc}, and fix $r_0\geq r_U$.
Now let  $\dot \X=C(\X,\calh,r_0)$ be the cone-off of $\X$ along $\calh$
 which is $\du$-hyperbolic by  Corollary \ref{coro;cone-off_hyp}.
Recall that $\dot \X$ is obtained  by gluing on $\X$ a hyperbolic cone $\Cone(H,r_0)$ on each horoball $H$.
We denote by $c_H$ the apex of this cone.

For each geodesic $[p,q]$ of $\Cone(H,r_0)$ (for its intrinsic metric) avoiding $c_H$
and with endpoints in the horosphere $\partial H$,
we consider the filled triangle $T_{[p,q]}\subset \Cone(H,r_0)$
bounded by the three geodesics $[c_H,p],[c_H,q],[p,q]$.
When $p=q\in\partial H$, we define $T_{[p,q]}=[c_H,p]$.
When $[p,q]$ contains the apex $c_H$ (\ie when $\d_{H}(p,q)\geq \pi\sinh r_0$ by Proposition \ref{prop;cone_BH}),
we define $T_{[p,q]}=[p,c_H]\cup[c_H,q]=[p,q]$.

We define $B_H=\bigcup_{[p,q]} T_{[p,q]}$
as the union of all those triangles
where $[p,q]$ describes all geodesics of $\Cone(H,r_0)$
with endpoints in $\partial H$, and such that $\d_H(p,q)<\pi\sinh r_0$.
Note that we would get the same set if we dropped the condition $\d_H(p,q)<\pi\sinh r_0$.
Also note that $B_H$ is star-shaped: for all $x\in B_H$, $[c,x]\subset B_H$.

We claim that $B_H$ is isometric to a $2$-complex with finitely many isometry classes of triangles.
Indeed, given an edge $e$ of $H$, denote by $C_e\subset \Cone(H,r_0)$ the cone over $e$.
The intersection $T_{[p,q]}\cap C_e$ is determined the position of the edge $e$ in the radial projection
 of $[p,q]$, \ie by $\d_H(p,e)$ and $\d_H(q,e)$.
Since $\d_H(p,q)<\pi\sinh r_0$, $\d_H(p,e)$ and $\d_H(q,e)$ take only finitely many values as $p$ and $q$ vary,
so $B_H\cap C_e$ is a finite union of convex geodesic triangles containing $c_H$,
and $B_H\cap C_e$  can be written as a union of
finitely many convex geodesic triangles intersecting each other along radial segments.
Moreover, as $e$ varies, there are only finitely many possibilities for $B_H\cap C_e$ up to isometry
 which proves the claim.

A similar argument using local compactness of $\X$ shows that $B_H\setminus\{c_H\}$ is locally compact.
Indeed, any edge $e$ or vertex $v$ of  $H$ is contained in only finitely many segments
with endpoints in $\partial H$ and  of length at most $\pi\sinh r_0$.

\label{i-pconeoff}
     \begin{defn}\label{dfn_coneoff_para}
       Let $\ru,\delta_c$ be the constants as in Theorem \ref{theo;cone-off_hyp_loc}.
Let $\X$ be a $\delta_c$-hyperbolic space, and $\calh$ a $50\delta_c$-separated system of horoballs.
Fix $r_0\geq \ru$.

The \emph{parabolic cone-off} $C^\prime(\X,\calh,r_0)$ is the subset of $\dot \X=C(\X,\calh,r_o)$ defined as
$$
C^\prime(\X,\calh,r_0)
= \left(\dot \X\setminus \bigcup_{H\in\calh}\Cone(H,r_0)\right) \cup \left(\bigcup_{H\in \calh} B_H \right).$$
     \end{defn}

Denoting by $\X_0=\X\setminus \bigcup_{H\in \calh} \rond{H}$ the complement of the horoballs in $\X$,
the parabolic cone-off can also described as
$$C^\prime(\X,\calh,r_0)= \X_0 \cup  \left(\bigcup_{H\in\calh}  B_H \right).$$

We endow $C^\prime(\X,\calh,r_0)$ with the induced path metric.
Since $C^\prime(\X,\calh,r_0)$ has finitely many isometry classes of triangles,
this makes  $C^\prime(\X,\calh,r_0)$  a geodesic space
as in the proof of Theorem \ref{theo;cone-off_hyp_loc}.

     \begin{rem}
       Given that we start with  a $\delta_c$-hyperbolic space $\X$, there is no
       rescaling involved for defining the parabolic
       cone-off. In particular, modifying the choice of our system of horoballs $\calh$ does not imply any further rescaling.
     \end{rem}

\newcommand{\dP}{{\delta_P}}

\begin{lem}
The parabolic cone-off $C^\prime(\X,\calh,r_0)$ is $2\du$-quasiconvex in $\dot \X$,
and its intrinsic metric $\d_{C'}$ satisfies
$$\forall x,y\in C^\prime(\X,\calh,r_0),\quad \d_{\dot\X}(x,y)\leq \d_{C'}(x,y)\leq \d_{\dot\X}(x,y)+4\du.$$
In particular, it is $\dP$-hyperbolic with $\dP=16\du$.
\end{lem}

\begin{proof}
Denote by $\dot\X''\subset C^\prime(\X,\calh,r_0)$
the union of $\X_0=\X\setminus \bigcup_{H\in \calh} \rond{H}$
with all radial segments of the form $[c_H,x]$ with $x\in \partial H$.

We first claim that for all $x,y\in \dot\X''$,
every geodesic $[x,y]_{\dot \X}$ of $\dot \X$ is contained in $C^\prime(\X,\calh,r_0)$.

Assume first that $[x,y]_{\dot \X}$ is contained in $\Cone(H,r_0)$ for some $H\in\calh$.
If this geodesic contains $c_H$, then $[x,y]_{\dot \X}=[x,c_H]\cup [c_H,y]\subset \dot\X''$ and we are done.
If not, then $[x,y]_{\dot \X}$ is a geodesic of $\Cone(H,r_0)$ avoiding $c_H$,
so the radial projections $p,q$ of $x,y$  satisfy  $\d_{H}(p,q)< \pi\sinh r_0$.
Since $x,y\in \X''$, $p,q\in\partial H$,
and $[x,y]_{\dot \X}$ is contained in a triangle $T_{[p,q]}\subset B_H$,
so $[x,y]_{\dot \X}\subset C^\prime(\X,\calh,r_0)$.

If $[x,y]_{\dot \X}$ is not contained in a cone,
consider $[x',y']$ a connected component of
the intersection of $[x,y]_{\dot \X}$ with a cone $\Cone(H,r_0)$.
If $x'\neq x$, then $x'$ lies in $\partial H$ as this is the boundary of
$\Cone(H,r_0)$ in $\dot \X$, and so does $y'$ if $y'\neq y$.
In all cases, $x',y'\in \X''$.
The argument above shows that $[x',y']\subset  C^\prime(\X,\calh,r_0)$. Since this holds
for every connected component of the intersection of $[x,y]$ with a cone, this proves our claim.

Next, given $H\in \calh$,
every triangle $T_{[p,q]}$ occurring in the definition of $B_H$ is contained in the $\du$-neighborhood of
$[p,c_H]\cup [c_H,q]\subset \dot\X''$.
Thus, for each $x\in T_{[p,q]}$ there exists $x_0\in \dot\X''$
and a path in $T_{[p,q]}$ of length at most $\du$ joining $x$ to $x_0$,
and in particular, $\d_{C'}(x,x_0)\leq\du$.

To conclude, consider $x,y\in C^\prime(\X,\calh,r_0)$, and $x_0,x_0\in \dot\X''$
with $\d_{C'}(x,x_0)\leq\du$ and $\d_{C'}(y,y_0)\leq\du$.
Since $[x,y]_{\dot \X}$ is contained in the $3\du$-neighborhood of $[x_0,y_0]_{\dot \X}$ which is
itself contained in $C^\prime(\X,\calh,r_0)$,
$C^\prime(\X,\calh,r_0)$ is $3\du$-quasiconvex.
Moreover, we have
$$\d_{\dot \X}(x,y)\leq \d_{C'}(x,y)\leq \d_{C'}(x_0,y_0)+2\du=\d_{\dot\X}(x_0,y_0)+2\du
\leq  \d_{\dot\X}(x,y)+4\du.$$

These estimates for $\d_{C'}$ imply that it satisfies the  $4\du$-hyperbolic inequality,
so $C^\prime(\X,\calh,r_0)$ is $16\du$-hyperbolic.
  \end{proof}

The following lemma is similar to Proposition \ref{prop;ex_of_VR} saying that a family
of subgroups acting on quasiconvex subspaces with sufficiently large injectivity radius
provides a very rotating family on the cone-off.

  \begin{lem}\label{lem_RH_inj_rad}
     Let $G$ be countable group, hyperbolic relatively to
       $\{P_1, \dots , P_n\}$, action on a $\delta_c$-hyperbolic space $\X$
with $\calh$ a $50\delta_c$-separated family of horoballs as above.
Let $H_i\in \calh$ be the horoball stabilized by $P_i$.

 For each $i\in\{1,\dots,n\}$, consider a normal subgroup $N_i\normal P_i$
 such that
$$\forall g\in N_i\setminus \{1\} \forall x\in \partial H_i\quad \d_\X(x,gx)\geq 4\pi\sinh r_0.$$

Then the family  $\calr$  of $G$-conjugates
       of $N_1,\dots,N_n$ defines a $2r_0$-separated very rotating family
        on $C^\prime(\X,\calh,r_0)$.
  \end{lem}

\begin{rem}
In this section, we never use the fact that  $H\in \calh$ is a horoball:
 any family of $50\delta_c$-separated $10\delta_c$-\sqc\ subgraphs would work as well.
Moreover, we
did not use  locally compact or proper discontinuity up to now
(except to prove the local compactness of $B_H\setminus \{c_H\}$ which we did not use yet),
but they will used in the results below.

In contrast to Proposition \ref{prop;ex_of_VR}, we ask in this lemma that the fellow traveling constant is zero
(this is the requirement that $\calh$ should be $50\delta_c$-separated),
and the assumption on the injectivity radius is replaced
by a condition asking only that points on the \emph{boundary} of our subspaces are moved by a large amount.
\end{rem}

     \begin{proof}
Since the horoballs $H_i$
are in distinct orbits, and since $N_i$ is normal in $P_i$, one can unambiguously assign to the horoball $g.H_i\in\calh$
the group $gN_ig\m$.
It follows that $\calr$ is a rotating family on the set of apices of $C^\prime(\X,\calh,r_0)$.
By equivariance, it is enough to prove the very rotation condition at the apex $c_i$ of $\Cone(H_i,r_0)$.

Denote by $\d_{C'}$ the path metric on $C^\prime(\X,\calh,r_0)$.
Consider $x,y\in C^\prime(\X,\calh,r_0)$ such that $20\dP\leq \d_{C'} (x,c_i),\d_{C'} (y,c_i)\leq 40\dP$,
and $\d_{C'} (x,gy)\leq 15\dP$ for some $g\in N_i\setminus\{1\}$.
In particular, since $40\dP\leq 10^4\du\leq \frac{r_0}{100}$ (see beginning of Section \ref{sec;cstes}),
$x,y\in B_{H_i}$ (where $B_{H_i}$ was introduced above Definition \ref{dfn_coneoff_para}).
This also implies that the geodesics of $\dot \X$ joining $x$ to $y$ are exactly the geodesics
of $\Cone(H_i,r_0)$ joining $x$ to $y$.
Since $B_{H_i}$ is star-shaped, the radial segments $[c_i,x],[c_i,y]$ are contained $B_{H_i}$.
We claim that $[x,c_i]\cup [c_i,y]$ is geodesic in $\dot \X$.
By Proposition \ref{prop;cone_BH}, this will ensure that there is no other geodesic in $\Cone(H_i,r_0)$,
hence in $C^\prime(\X,\calh,r_0)$, and the very rotating condition will follow.

Consider $p_x,p_y\in H_i$ the radial projections of $x$ and $y$.
By definition of $B_{H_i}$, $y$ lies in a triangle $T_{[q,q']}$ for some geodesic
$[q,q']\subset \Cone(H_i,r_0)$
joining two points of $q,q'\in\partial H_i$ and avoiding $c_i$.
Since by Proposition \ref{prop;cone_BH},
the radial projection of $[q,q']$ is a geodesic of $H_i$ of length at most $\pi\sinh(r_0)$,
$\d_\X(p_y,q)\leq \pi\sinh(r_0)$.
This implies that $\d_\X(gp_y,p_y)\geq \d_\X(gq,q)-2\pi\sinh(r_0)\geq 2\pi\sinh(r_0)$.

On the other hand, denoting by $\d_C$ the path metric on $\Cone(H_i,r_0)$,
$\d_{C}(p_x,gp_y)\leq \d_{C}(p_x,x)+\d_{C'}(x,gy)+\d_{C}(gy,gp_y)
\leq (r_0-20\dP)+(15\dP)+(r_0-20\dP)<2r_0$.
This implies that
no geodesic of $\Cone(H_i,r_0)$ joining $p_x$ to $gp_y$ contains $c$,
so $\d_\X(p_x,gp_y)\leq \d_{H_i}(p_x,gp_y)< \pi\sinh(r_0) $.
It follows that $\d_\X(p_x,p_y)\geq \d_\X(p_y,gp_y)-\d_\X(gp_y,p_x) >\pi\sinh r_0 $.
By Lemma \ref{lem_lip}, $[p_x,c_i]\cup [c_i,p_y]$ is a geodesic in $\dot \X$.
This implies that $[x,c_i]\cup [c_i,y]$ is geodesic in $\dot \X$, as claimed.
\end{proof}

The following proposition is based on the fact that each $P_i$ acts
properly and cocompactly on the horosphere $\partial H_i$.

     \begin{prop}\label{prop;RH_VRF}
       Let $G$ be countable group, hyperbolic relatively to
       $\{P_1, \dots , P_n\}$. Let $\X$ be a proper $\delta_c$-hyperbolic graph
       and  $\calh$ a $50\delta$-separated system of horoballs
       as in Definition \ref{dfn_relh}.
       Consider $r_0\geq \ru$,
and  $C^\prime(\X,\calh,r_0)$ the parabolic cone-off.

      Then there exists a finite subset $S\subset G\setminus\{1\}$,
       such that,
 given for each $i\in\{1,\dots,n\}$ a normal subgroup $N_i\normal P_i$ avoiding $S$,
 the family  $\calr$  of $G$-conjugates
       of $N_1,\dots,N_n$ defines a $2r_0$-very rotating family
         on $C^\prime(\X,\calh,r_0)$.
     \end{prop}

     \begin{proof}
       Since $G$ acts cocompactly on $\X\setminus (\cup_{H\in \calh}
       \rond{H}) $, $P_i$ acts cocompactly on $\partial H_i$.  Let
       $K_i\subset \partial H_i$ be a compact set such that
       $P_iK_i=\partial H_i$.  Let $S_i\subset P_i\setminus\{1\}$ be
       the set of elements $g$ such that there exists some $x\in K_i$
       with $\d_\X (x, gx) \leq 4\pi\sinh r_0$.  Since the action of
       $G$ on $\X$ is proper, $S_i$ is finite. We take
       $S=S_1\cup\dots\cup S_n$. To conclude,  note that if $N_i$ is a normal
       subgroup of $P_i$ avoiding $S$, then any $g\in
       N_i\setminus\{1\}$ moves any point $q\in\partial H_i$ by at
       least $4\pi\sinh r_0$ (for the metric $\d_\X$).
       Thus Lemma \ref{lem_RH_inj_rad} applies.
     \end{proof}

\begin{cor}\label{cor_inject}
Under the assumptions of Proposition \ref{prop;RH_VRF},
consider $N=\grp{\grp{N_1,\dots N_n}}\normal G$,
$\dot\X'=C^\prime(\X,\calh,r_0)$ the parabolic cone-off,
and $\pi: \dot\X'\onto \dot\X'/N$ the quotient map.

Consider $p\in \X$ and $r$ such that $B_\X(p,r)$ is disjoint from $\calh$.
Then $\pi$ is injective in restriction to $B_\X(p,r)$, and for any $g\in N\setminus\{1\}$,
$g.B_\X(p,r)\cap B_\X(p,r)= \es$.

Moreover,
$\pi$ is isometric in restriction to $B_\X(p,r/3)$:
 for $x,y\in B_\X(p,r/3)$, $\d_{\X}(x,y)=\d_{\dot\X'}(x,y)=\d_{\dot\X'/N}(\pi(x),\pi(y))$.

Finally, if each $N_i$ has finite index in $P_i$, then $\dot\X'/N$ is locally compact,
and $G/N$ acts on $\dot\X'/N$ properly discontinuously and cocompactly. In particular, $G/N$ is a hyperbolic group.
\end{cor}

\begin{proof}
First note that for any $x\in B_\X(p,r)$, any path of length $\leq r$ in $\dot\X'$ with origin $p$
cannot exit $B_\X(p,r)$, so $\d_\X(p,x)=\d_{\dot\X'}(p,x)$, and
$B_\X(p,r)=B_{\dot\X'}(p,r)$.
To prove the first assertion, consider on the contrary
$x,y\in B_\X(p,r)$, such that $x=gy$ for some $g\in N\setminus\{1\}$.
By  the qualitative Greendlinger Lemma \ref{lem;green_pointed},
any geodesic $[x,y]$ in $\dot\X'$ contains an apex. Since $B_{\dot\X'}(p,r)$ is $2\dP$ quasiconvex,
this apex is at distance at most $r+2\dP$ from $p$, a contradiction since $\dP$ is small compared to $r_0$.
This proves the first assertion.
The second assertion is a consequence.

For the third assertion,
recall that for each $H\in \calh$, $B_H\setminus \{c_H\}$ is locally compact.
It follows that $\dot\X'$ and $\dot\X'/N$ are locally compact on the complement of the apices.
Since $\partial H_i/P_i$ is compact and $N_i$ has finite index in $P_i$,
 $\partial H_i/N_i$ is compact. This implies that the link of any apex in $\dot\X'/N$ is compact,
so $\dot\X'/N$ is locally compact.
Since the action of $G$ on $\X$ is proper,
and since $P_i/N_i$ is finite, vertex stabilizers of
the action of $G/N$ on $\dot\X'/N$ are finite. The third assertion follows.
\end{proof}

  \begin{thm}\label{theo;DF}
  Let $G$ be a group hyperbolic relatively to $\{P_1, \dots, P_n\}$.
Let $\{g_1,\dots, g_n\}\subset G$ be a finite generating set, $R>0$,
and $B_R(G)$ the ball of radius $R$ in the corresponding Cayley graph of $G$.

Then, there exists  a finite set $S \subset G\setminus\{1\}$ such that whenever
  $N_i\normal P_i$ is of finite index
  and avoids $S$,
  the quotient $G/\langle\langle  \cup_i N_i\rangle\rangle$ is hyperbolic, and the quotient map is injective
in restriction to $B_R(G)$.

  Moreover $\langle\langle  \cup_i N_i\rangle\rangle$ is a free
  product of conjugates of $N_i$'s, and its elements are either contained in some conjugate of $N_i$ or
are loxodromic (as elements of the relatively hyperbolic group $G$).
  \end{thm}

\begin{proof}
  Let $r_0=\ru$, and consider a hyperbolic space $\X$ associated to the relatively
  hyperbolic group  $(G,\calP)$.
Assume without  loss of generality that $\X$ is $\delta_c$ hyperbolic.
Let $p\in \X$ be a base point.
Let $d$ be such that $\d_\X(p,g_ip)\leq d$ for  each generator $g_i$ of $G$.
Choose a system of horoballs that is
$50\delta_c$-separated, and that avoids the ball $B_\X(p,Rd)$.
Let $\dot\X'$ be the corresponding parabolic cone-off.
Consider $S$ a finite set satisfying the conclusions of Proposition \ref{prop;RH_VRF} and Corollary \ref{cor_inject}.
Consider $N_i\normal P_i$ with finite index, and $N_i\cap S=\es$.

Proposition \ref{prop;RH_VRF} says that the groups $N_i$ define a  $2r_0$-separated very rotating family of $\dot\X'$.
 Let  $\bar \X^\prime=\dot\X'/N$ where $N$ the normal subgroup of $G$ generated by $\cup_i N_i$.
Theorem \ref{theo;app_wind} about very rotating families then says that $N$ is a free product of conjugate of $N_i$'s,
and that any element of $N$ not conjugate to some $N_i$ is loxodromic in $\dot\X'$.
Such an element is necessarily loxodromic in $\X$, so the last assertion follows.

By Corollary \ref{cor_inject}, $G/N$ is hyperbolic,
and there remains to prove that the ball in the Cayley graph of $G$ injects in $\bar G$.
Consider  $u,v\in B_R(G)$ two words of length $\leq R$ with $uv\m\in N$. Since
$up,vp\in B_{\X}(p,dR)$, Corollary \ref{cor_inject} prevents that $uv\m\in N\setminus\{1\}$, so $u=v$, and
the injectivity follows.
\end{proof}

\subsection{Diagram surgery}

The goal of this section is to prove some auxiliary results about van Kampen diagrams over Dehn fillings of groups with hyperbolically embedded subgroups. These results will be used in the next section to prove the general version of the group theoretic Dehn filling theorem in the context of weak relative hyperbolicity (Theorem \ref{CEP-0}). Our exposition follows closely \cite{Osi07}. In fact, we could refer to \cite{Osi07} for proofs and just list the few necessary changes. However, since Theorem \ref{CEP-0} is one of the main results of our paper we decided to reproduce the proofs here for convenience of the reader.

Throughout this section, let $G$ be a group weakly hyperbolic with respect to a collection of subgroups $\Hl$ and a subset $X\subseteq G$. By Lemma \ref{pres} there exists a bounded reduced relative presentation
\begin{equation}\label{CEP-Gfull}
G=\langle X, \mathcal H \mid \mathcal R \cup \mathcal S\rangle
\end{equation}
of $G$ with respect to $\Hl $ and $X$ with linear relative isoperimetric function. Recall that $\mathcal S$ is the set of all relations in the alphabet $$\mathcal H=\bigsqcup_{\lambda\in \Lambda } (H_\lambda \setminus \{1\})$$ that hold in the free product $\ast_{\lambda\in \Lambda } H_\lambda $ and $\mathcal R\subseteq F$ normally generates the kernel of the homomorphism $F\to G$, where $F=F(X)\ast (\ast_{\lambda\in \Lambda } H_\lambda)$. We refer the reader to Section \ref{sec-GPIF} and Section \ref{sec-WRHBP} for details.

Given a collection $\N =\{ N_\lambda \} _{\lambda \in \Lambda }$, where
$N_\lambda $ is a normal subgroup of $H_\lambda $, we denote by $N$
the normal closure of $\mbox{$\bigcup_{\lambda \in \Lambda }$}
N_\lambda $ in $G$ and let $\bar G=G/N $.

We fix the following
presentation for $\bar G$
\begin{equation}\label{GNfull}
\bar G =\langle X, \mathcal H\; |\;\mathcal S\cup\mathcal R\cup \mathcal Q\rangle ,
\end{equation}
where $\mathcal Q=\bigcup_{\lambda \in \Lambda } \mathcal
Q_\lambda $ and $\mathcal Q_\lambda $ consists of all words (not
necessary reduced) in the alphabet $H_\lambda \setminus \{ 1\} $
representing elements of $N_\lambda $ in $G$.

In this section we consider van Kampen diagrams over (\ref{CEP-Gfull})
of a certain type. More precisely, we denote by $\mathcal D$ the
set of all diagrams $\Delta $ over (\ref{CEP-Gfull}) such that: \label{i-D}

(D1) \label{i-dd} Topologically $\Delta $ is a disc with $k\ge 0$ holes. More
precisely, the boundary of $\Delta $ is decomposed as $\partial
\Delta =\partial_{ext} \Delta \sqcup \partial_{int} \Delta $,
where $\partial _{ext}\Delta $ is the boundary of the disc and
$\partial _{int}\Delta $ consists of disjoint cycles ({\it
components}) $c_1, \ldots c_k$ that bound the holes.

(D2) For any $i=1, \ldots , k$, the label $\Lab (c_i)$ is a word
in the alphabet $H_\lambda$ for some $\lambda \in
\Lambda $ and this word represents an element of $N_\lambda $ in
$G$.

The following lemma relates diagrams of the described type to the
group $\bar G$.

\begin{lem}\label{cutting}
A word $W$ in $X\sqcup \mathcal H$ represents $1$ in $\bar G$ if and
only if there is a diagram $\Delta \in \mathcal D$ such that $\Lab
(\partial _{ext} \Delta )\equiv W$.
\end{lem}

\begin{proof}
Suppose that $\Sigma $ is a disc van Kampen diagram over
(\ref{GNfull}). Then by cutting off all essential cells labeled by
words from $\mathcal Q$ ({\it $\mathcal Q$--cells}) and passing to a
$0$--refinement if necessary we obtain a van Kampen diagram $\Delta
\in \mathcal D$ with $\Lab (\partial_{ext} \Delta )\equiv \Lab
(\partial \Sigma )$. Conversely, each $\Delta \in D$ may be
transformed into a disk diagram over (\ref{GNfull}) by attaching
$\mathcal Q$--cells to all components of $\partial _{int} \Delta $.
\end{proof}

In what follows we also assume the diagrams from $\mathcal D$ to
be endowed with an additional structure.

(D3) Each diagram $\Delta \in \mathcal D$ is equipped with a {\it
cut system} \label{i-cutsys} that is a collection of disjoint paths ({\it cuts})
$T=\{ t_1, \ldots , t_k\} $ without self--intersections in $\Delta $
such that $(t_i)_+, (t_i)_-$ belong to $\partial \Delta $, and after
cutting $\Delta $ along $t_i$ for all $i=1, \ldots , k$ we get a
connected simply connected diagram $\widetilde{\Delta }$.

By $\kappa\colon \widetilde{\Delta }\to\Delta $ we denote the
natural map that 'sews' the cuts. We also fix an arbitrary point $O$
in $\widetilde{\Delta }$. Recall that $\mu $ denotes the map from the $1$-skeleton of $\widetilde{\Delta }$ to $\G $ described in Remark \ref{rem-mu}.

\begin{lem}\label{trans}
Suppose that $\Delta \in \mathcal D$. Let $a,b$ be two vertices on
$\partial \Delta $, $\tilde a, \tilde b$ some vertices on
$\partial \widetilde{\Delta }$ such that $\kappa (\tilde a)=a$,
$\kappa (\tilde b)= b$. Then for any paths $r$ in $\G $ such that
$r_-=\mu (\tilde a)$, $r_+=\mu (\tilde b)$, there is a diagram
$\Delta _1\in \mathcal D$ endowed with a cut system $T_1$ such
that the following conditions hold:
\begin{enumerate}
\item[(a)] $\Delta _1$ has the same boundary and the same cut system as $\Delta $.
By this we mean the following. Let $\Gamma _1$ (respectively $\Gamma
$) be the subgraph of the $1$-skeleton of $\Delta_1$ (respectively
of the $1$-skeleton of $\Delta $) consisting of $\partial \Delta _1$
(respectively $\partial \Delta $) and all cuts from $T_1$
(respectively $T$). Then there is a graph isomorphism $\Gamma _1\to
\Gamma $ that preserves labels and orientation and maps cuts of
$\Delta _1$ to cuts of $\Delta $ and $\partial _{ext} \Delta _1$ to
$\partial _{ext} \Delta $.

\item[(b)] There is a paths $q$ in $\Delta _1$ without
self--intersections such that $q_-=a$, $q_+=b$, $q$ has no common
vertices with cuts $t\in T_1$ except for possibly $a$,$b$, and $\Lab
(q)\equiv \Lab (r)$.
\end{enumerate}
\end{lem}

\begin{proof}
Let us fix an arbitrary path $\tilde t$ in $\widetilde{\Delta }$
without self--intersections that connects $\tilde a$ to $\tilde b$
and intersects $\partial \widetilde{\Delta }$ at the points $\tilde
a$ and $\tilde b$ only. The last condition can always be ensured by
passing to a $0$--refinement of $\Delta $ and the corresponding
$0$--refinement of $\widetilde{\Delta }$. Thus $t=\kappa (\tilde t)$
connects $a$ to $b$ in $\Delta $ and has no common points with cuts
$t\in T$ except for possibly $a$,$b$. Note that
$$\Lab (t)\equiv \Lab (\tilde t)\equiv \Lab(\mu (\tilde t))$$
as both $\kappa $, $\mu $ preserve labels and orientation.

Since  $\mu (\tilde t)$ connects $\mu (\tilde a)$ to $\mu (\tilde
b)$ in $\G$, $\Lab (\mu (\tilde t))$ represents the same element
of $G$ as $\Lab (r)$. Hence there exists a disk diagram $\Sigma
_1$ over (\ref{CEP-Gfull}) such that $\partial \Sigma _1=p_1q^{-1}$,
where $\Lab (p_1)\equiv \Lab (t)$ and $\Lab (q)\equiv \Lab (r)$.
Let $\Sigma _2$ denote its mirror copy. We glue $\Sigma _1$ and
$\Sigma _2$ together by attaching $q$ to its mirror copy. Thus we
get a new diagram $\Sigma $ with boundary $p_1p_2^{-1}$, where
$\Lab (p_1)\equiv\Lab (p_2)\equiv \Lab (t)$. The path in $\Sigma $
corresponding to $q$ in $\Sigma _1$ and its mirror copy in $\Sigma
_2$ is also denoted by $q$.

We now perform the following surgery on the diagram $\Delta $.
First we cut $\Delta $ along $t$ and denote the new diagram by
$\Delta _0$. Let $t_1$ and $t_2$ be the two copies of the path $t$
in $\Delta _0$. Then we glue $\Delta _0$ and $\Sigma $ by
attaching $t_1$ to $p_1$ and $t_2$ to $p_2$ (Fig. \ref{7-f6}) and
get a new diagram $\Delta _1$. This surgery does not affect cuts
of $\Delta $ as $t$ had no common points with cuts from $T$ except
for possibly $a$ and $b$. Thus the system of cuts in $\Delta _1$
is inherited from $\Delta $ and $\Delta _1$ satisfies all
requirements of the lemma.
\end{proof}

\begin{figure}
\vspace{1mm}
\hspace{1.4cm}\includegraphics[]{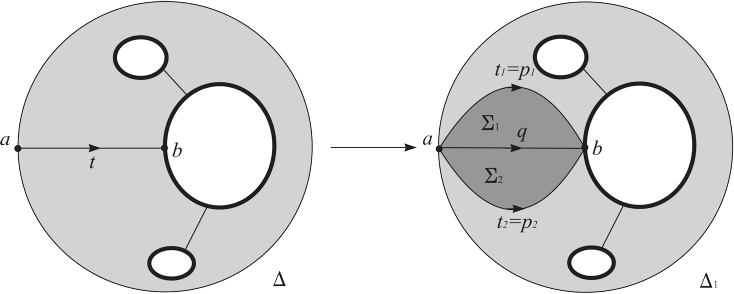}\\
\vspace{-2mm}
  \caption{}\label{7-f6}
\end{figure}

\begin{defn}
By an {\it $H_\lambda $--path} in $\Delta \in \mathcal D$ or in
$\widetilde{\Delta }$ we mean any paths whose label is a
(nontrivial) word in $H_\lambda \setminus \{ 1\} $. We say that two
such paths $p$ and $q$ in $\Delta \in \mathcal D$ are {\it
connected} if they are $H_\lambda $--paths for the same $\lambda \in
\Lambda $ and there are $H_\lambda $--paths $a$, $b$ in
$\widetilde{\Delta }$ such that $\kappa (a)$ is a subpaths of $p$,
$\kappa (b)$ is a subpaths of $q$, and $\mu (a)$, $\mu (b)$ are
connected in $\G $, i.e., there is a path in $\G $ that connects a
vertex of $\mu (a)$ to a vertex of $\mu (b)$ and is labelled by a
word in $H_\lambda \setminus \{ 1\} $. We stress that the equalities
$\kappa (a)=p$ and $\kappa (b)=q$ are not required. Thus the
definition makes sense even if the paths $p$ and $q$ are cut by the
cuts of $\Delta $ into several pieces.
\end{defn}

\begin{defn}\label{typeD}
We also define the {\it type} of a diagram $\Delta \in
\mathcal D$ by the formula
$$
\tau (\Delta )=\left( k, \sum\limits_{i=1}^k l(t_i)\right) ,
$$
where $k$ is the number of holes in $\Delta $. We fix the standard
order on the set of all types by assuming $(m,n)\le (m_1, n_1)$ is
either $m< m_1$ or $m=m_1 $ and $n\le n_1$.
\end{defn}

For a word $W$ in the alphabet $X\sqcup\mathcal H$, let $\mathcal
D(W)$ denote the set of all diagrams $\Delta \in \mathcal D$ such
that $\Lab(\partial _{ext} \Delta )\equiv W$. In the proposition
below we say that a word $W$ in $X\sqcup \mathcal H$ is {\it
geodesic} if any (or, equivalently, some) path in $\G $ labelled
by $W$ is geodesic.

\begin{prop}\label{DS}
Suppose that $W$ is a word in $X\sqcup\mathcal H$ representing $1$
in $\bar G$, $\Delta $ is a diagram of minimal type in $\mathcal
D(W)$, $T$ is the cut system in $\Delta $, and $c$ is a component
of $\partial _{int}\Delta $. Then:
\begin{enumerate}
\item[(a)] For each cut $t\in T$, the word $\Lab (t)$ is geodesic.

\item[(b)] The label of $c$ represents a nontrivial element in $G$.

\item[(c)] The path $c$ can not be connected to an $H_\lambda
$--subpath of a cut.

\item[(d)] The path $c$ can not be connected to another component of
$\partial _{int} \Delta $
\end{enumerate}
\end{prop}

\begin{figure}
 \hspace{3mm}
 \includegraphics[]{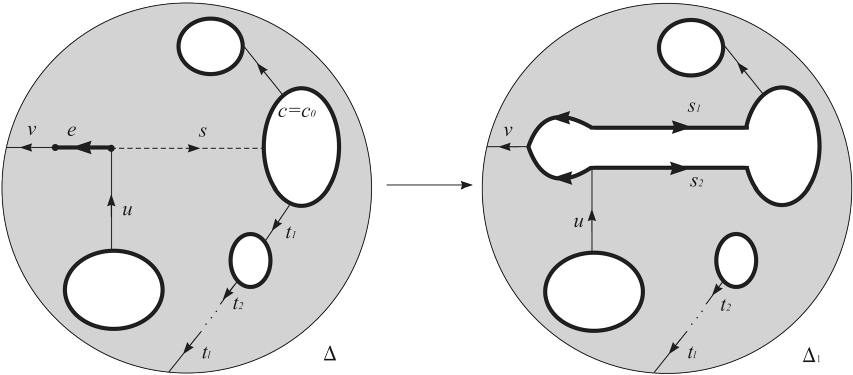}
\begin{center} a) \end{center}
\vspace{5mm}
\hspace{3mm}
\includegraphics[]{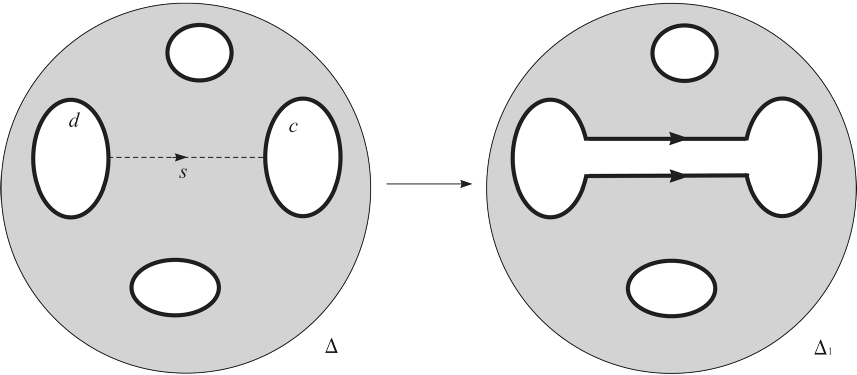}\\
\begin{center} b) \end{center}
\caption{} \label{7-f7}
\end{figure}

\begin{proof}
Assume that for a certain path $t\in T$, $\Lab (t)$ is not geodesic.
Let $\tilde a$, $\tilde b$ be vertices in $\widetilde \Delta $ such
that $\kappa (\tilde a)=t_-$, $\kappa (\tilde b)=t_+$. Let also $r$
be a geodesic paths in $\G $ that connects $\mu (\tilde a)$ to $\mu
(\tilde b)$. Applying Lemma \ref{trans}, we may assume that there is
a path $q$ in $\Delta $ such that $q_-=t_-$, $q_+=t_+$, and $\Lab
(q)\equiv \Lab (r)$, i.e., $\Lab (q)$ is geodesic. In particular,
$l(q)<l(t)$. Now replacing $t$ with $q$ in the cut system we reduce
the type of the diagram. This contradicts the choice of $\Delta $.

The second assertion is obvious. Indeed if $\Lab (c)$ represents $1$
in $G$, there is a disk diagram $\Pi $ over (\ref{CEP-Gfull}) with
boundary label $\Lab (\partial \Pi)\equiv \Lab (c)$. Attaching $\Pi
$ to $c$ does not affect $\partial _{ext}\Delta $ and reduces the
number of holes in the diagram. This contradicts the minimality of
$\tau (\Delta )$ again.

Further assume that $c$ is connected to an $H_\lambda $--subpath $e$
of some $r\in T$. Then $c$ is an $H_\lambda $--path for the same
$\lambda \in \Lambda$. Let $r=uev$. Cutting $\Delta $ along $e$ (to
convert $e$ into a boundary component), applying Lemma \ref{trans},
and gluing the copies of $e$ back, we may assume that there is a
path $s$ without self--intersections in $\Delta $ such that
$s_-=e_-$, $s_+\in c$, and $\Lab (s)$ is a word in $H_\lambda
\setminus \{ 1\} $. Moreover passing to a $0$--refinement, we may
assume that $s$ has no common vertices with the boundary of the
diagram, paths from $T\setminus \{ r\} $, $u$, and $v$ except for
$s_-$ and  $s_+$. Now we cut $\Delta $ along $s$ and $e$. Let $s_1$,
$s_2$ be the copies of $s$ in the obtained diagram $\Delta _1$. The
boundary component of $\Delta _1$ obtained from $c$ and $e$ has
label $\Lab(c)\Lab (s)^{-1}\Lab(e)\Lab(e)^{-1}\Lab{s}$ that is a
word in $H_\lambda \setminus\{ 1\} $ representing an element of
$N_\lambda $ in $G$. Note also that our surgery does not affect cuts
of $\Delta $ except for $r$. Thus the system of cuts $T_1$ in
$\Delta _1$ may obtained from $T$ as follows. Since $\widetilde
{\Delta } $ is connected and simply connected, there is a unique
sequence
$$
c=c_0,\; t_1,\; c_1,\; \ldots ,\; t_l,\; c_l=\partial _{ext} {\Delta
},
$$
where $c_0, \ldots , c_l$ are (distinct) components of $\partial
\Delta $, $t_i\in T$, and (up to orientation) $t_i$ connects
$c_{i-1} $ to $c_i$, $i=1, \ldots , l$ (Fig. \ref{7-f7}a). We set
$T_1=(T\setminus \{ r, t_1\} )\cup \{ u, v\} $. Thus $\Delta _1\in
\mathcal D(W)$ and $\tau (\Delta _1)<\tau (\Delta )$. Indeed
$\Delta _1$ and $\Delta $ have the same number of holes and
$\sum\limits_{t\in T_1} l(t)\le \sum\limits_{t\in T_1} l(t)-1$.
This contradicts the choice of $\Delta $.

Finally suppose that $c$ is connected to another component $d$ of
$\partial _{int} \Delta $, $d\ne c$. To be definite, assume that $c$
and $d$ are labelled by words in $H_\lambda \setminus\{ 1\} $. Again
without loss of generality we may assume that there is a path $s$
without self--intersections in $\Delta $ such that $s_-\in d$,
$s_+\in c$, $\Lab (s)$ is a word in $H_\lambda \setminus \{ 1\} $,
and $s$ has no common points with $\partial \Delta $ and paths from
$T$ except for $s_-$ and $s_+$. Let us cut $\Delta $ along $s$ and
denote by $\Delta _1$ the obtained diagram (Fig. \ref{7-f7}b). This
transformation does not affect $\partial _{ext} \Delta $ and the
only changed internal boundary component has label $\Lab
(c)\Lab(s)^{-1}\Lab (d)\Lab (s)$, which is a word in $H_\lambda
\setminus\{ 1\} $. This word represents an element of $N_\lambda $
in $G$ as $N_\lambda\lhd H_\lambda $. We now fix an arbitrary system
of cuts in $\Delta _1$. Then $\Delta _1\in \mathcal D(W)$ and the
number of holes in $\Delta _1$ is smaller that the number of holes
in $\Delta $. We get a contradiction again.
\end{proof}

\subsection{The general case}

The aim of this section is to prove the general version of the group theoretic Dehn filling theorem. We start by recalling the general settings.

Let $G$ be a group, $\Hl $ a collection of subgroups of $G$, $X$ a subset of $G$ that generates $G$ together with the union of $H_\lambda $'s.  As usual, $\dl $ denotes the corresponding distance function on $H_\lambda $ defined using $\G$. Given a collection $\N =\Nl$ of subgroups of $G$ such that $N_\lambda \lhd H_\lambda $ for all $\lambda \in \Lambda $, we define
$$
s (\N) = \min\limits_{\lambda \in \Lambda }\min\limits_{h\in N_\lambda\setminus \{ 1\} } \dl (1,h).
$$
The {\it Dehn filling} \label{i-df3} of $G$ associated to this data is the quotient group $$\bar G =G/\ll \bigcup\limits_{\lambda \in \Lambda }N_\lambda \rr ^G.$$ Let $\bar X $ be the natural image of $X$ in $\bar G$ and let
$$
\mathcal{\bar H} =\bigsqcup\limits_{\lambda\in \Lambda} H_\lambda/N_\lambda .
$$

Our main result is the following result. When talking about loxodromic, parabolic, or elliptic elements of the group $G$ or its subgroups (respectively, $\bar G$) we always mean that these elements are loxodromic, parabolic, or elliptic with respect to the action on $\G $ (respectively, $\Gamma (\bar G, \bar X\sqcup\mathcal{\bar H})$).

\begin{thm}\label{CEP-0}
Suppose that a group $G$ is weakly hyperbolic relative to a collection of subgroups $\Hl $ and a relative generating set $X$. Then there exists a constant $R>0$ such that for every collection
$\N =\{ N_\lambda \lhd H_\lambda \mid \lambda \in \Lambda \}$ satisfying
\begin{equation}\label{snr}
s(\N)>R,
\end{equation}
the following hold.
\begin{enumerate}
\item[(a)] The natural map from $H_\lambda /N_\lambda $ to $\bar G$ is injective for every $\lambda \in \Lambda $.

\item[(b)] $\bar G$ is weakly hyperbolic relative to $\mathcal{\bar H}$ and $\bar X$.

\item[(c)] The natural epimorphism $\e \colon G\to \bar G$ is injective on $X$.

\item[(d)] Every element of $Ker (\e )$ is either conjugate to an element of $N_\lambda $ for some $\lambda \in \Lambda $ or is loxodromic. Moreover, translation numbers of loxodromic elements of $Ker (\e)$
(with respect to the action on $\G $) are uniformly bounded away from zero.

\item[(e)] $Ker (\e)=\ast_{\lambda\in \Lambda }\ast_{t\in T_{\lambda} } N_{\lambda }^t $ for some subsets $T_\lambda \subseteq G$.

\item[(f)] Every  loxodromic (respectively, parabolic or elliptic) element of $\bar G$ is the
image of a loxodromic (respectively, parabolic or elliptic) element of $G$.
\end{enumerate}
\end{thm}

The proof of parts (a)-(c) of Theorem \ref{CEP-0} repeats the proof of the main result of \cite{Osi07}. It consists of a sequence of lemmas, which are proved by induction on the rank of a diagram defined as follows. We assume that the reader is familiar with the terminology and notation introduced in the previous section. Let
\begin{equation}\label{r4d}
R=4D,
\end{equation}
where $D=D(2,0)$ be the constant from Proposition \ref{sn}.

\begin{defn}\label{theta}
Given a word $W$ in the alphabet $X\sqcup \mathcal H$ representing
$1$ in $\bar G$, we denote by $q(W)$ the minimal number of holes
among all diagrams from $\mathcal D(W)$. Further we define the
{\it type} of $W$ by the formula $\theta (W)=(q(W), \| W\| )$. The
set of types is endowed with the natural order (as in Definition
\ref{typeD}).
\end{defn}
The next three results are
proved by common induction on $q(W)$. Recall
that a word $W$ in $X\sqcup \mathcal H$ is called $(\lambda ,
c)$--quasi--geodesic (in $G$) for some $\lambda \ge 1$, $c\ge 0$,
if some (or, equivalently, any) path in $\G $ labelled by $W$ is
$(\lambda , c)$--quasi--geodesic.

\begin{lem}\label{CEP1}
Suppose that $W$ is a word in the alphabet $X\sqcup \mathcal H$
representing $1$ in $\bar G$ and $\Delta $ is a diagram of minimal
type in $\mathcal D (W)$. Then:

\begin{enumerate}
\item[(a)] Assume that for some $\lambda \in \Lambda $, $p$ and $q$ are
two connected $H_\lambda $--subpaths of the same component $c$ of
$\partial _{int}\Delta $, then there is an $H_\lambda $--component
$r$ of $\partial\widetilde{\Delta }$ such that $p$ and $q$ are
subpaths of $\kappa (r)$.

\item[(b)]  If $W$ is $(2,0)$--quasi--geodesic and $q(W)>0$, then some
component of $\partial _{int} \Delta $ is connected to an
$H_\lambda $--subpath of $\partial _{ext}\Delta $ for some
$\lambda \in \Lambda $.

\item[(c)] If $W$ is a word in the alphabet $H_\lambda \setminus \{ 1\}
$ for some $\lambda \in \Lambda $, then $W$ represents an element
of $N_\lambda $ in $G$.
\end{enumerate}
\end{lem}

\begin{proof}
For $q(W)=0$ the lemma is trivial. Assume that $q(W)>0$.

\begin{figure}
  \vspace{1mm}
 \hspace{13mm}\includegraphics[]{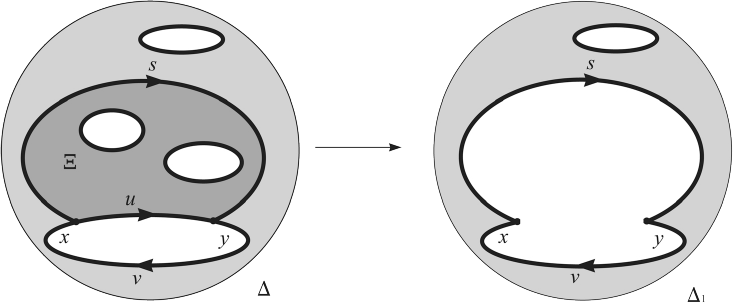}
  \caption{}\label{7-f8}
\end{figure}

Let us prove the first assertion. Let $x$ (respectively $y$) be an
ending vertex of a certain essential edge of $p$ (respectively
$q$). Passing to a $0$--refinement of $\Delta $, we may assume
that $x$ and $y$ do not belong to any cut from the cut system $T$
of $\Delta $. Applying Lemma \ref{trans} we get a paths $s$ in
$\Delta $ connecting $x$ to $y$ such that $\Lab (s)$ is a word in
the alphabet $H_\lambda \setminus \{ 1\}$ and $s$ does not
intersect any path from $T$. Let us denote by $\Xi $ the
subdiagram of $\Delta $ bounded by $s$ and the segment $u=[x,y]$
of $c^{\pm 1}$ such that $\Xi $ does not contain the hole bounded
by $c$ (Fig. \ref{7-f8}).

Note that $V\equiv \Lab (\partial \Xi )$ is a word in the alphabet
$H_\lambda \setminus \{ 1\}$ and $q(V)< q(W)$. By the third
assertion of our lemma, $V$ represents an element of $N_\lambda $.
Up to a cyclic shift, the label of the external boundary component
of the subdiagram $\Sigma =\Xi \cup c$ of $\Delta $ is a word in
$H_\lambda \setminus \{ 1\}$ representing the same element as
$\Lab (c^{\pm 1})\Lab(u)^{-1}V^{\pm 1}\Lab (u)$ in $G$. As
$N_\lambda $ is normal in $H_\lambda $ and $\Lab (u)$ represents
an element of $H_\lambda $ in $G$, $\Lab (\partial _{ext} \Sigma
)$ represents an element of $N_\lambda $ in $G$. If $\Xi $
contains at least one hole, we replace $\Sigma$ with a single hole
bounded by $\partial _{ext}\Sigma $ (Fig. \ref{7-f8}). This reduces
the number of holes in $\Delta $ and we get a contradiction.
Therefore $\Xi $ is simply connected. In particular, the path $u$
does not intersect any cut from $T$. This means that $p$ and $q$
are covered by the image of the same $H_\lambda $--component of
$\partial\widetilde{\Delta }$.

To prove the second assertion we suppose that for every component
$c_i$ of $\partial _{int} \Delta $, no $H_\lambda $--subpath of
$\partial _{ext}\Delta $ is connected to $c$. Then Proposition
\ref{DS} and the first assertion of our lemma imply that each
component $c_i$ of $\partial _{int}\Delta $ gives rise to
$H_\lambda $--components $a_{i1}, \ldots, a_{il}$ of $\partial
\widetilde{\Delta }$  for some $l=l(i)$ such that $\kappa
(a_{ij})\in c_i$, $j=1, \ldots , l$, and $\mu (a_{i1}),\ldots ,\mu
(a_{il})$ are isolated $H_\lambda $--components of the cycle
$\mathcal P=\mu (\partial \widetilde{\Delta })$ in $\G $.

For each component $c_i$ of $\partial _{int} (\Delta
)$, we fix a vertex $o_i\in c_i$ such that $o_i=t_-$ or $o_i=t_+$
for some $t\in T$ and denote by $g_i$ the element represented by
$\Lab (c_i)$ when we read this label starting from $o_i$. Clearly $g_i\in H_{\lambda_i}$ for some $\lambda _i\in \Lambda $ and
\begin{equation}\label{gi}
\widehat \d_{\lambda_i} (1, g_i) \le \sum\limits_{j=1}^{l(i)} \widehat\ell (\mu
(a_{ij})).
\end{equation}

The path $\mathcal P$ may be considered as an $n\le 4q(W)$--gon
whose sides (up to orientation) are of the following three types:
\begin{enumerate}
\item[(1)] sides corresponding to parts of $\partial _{ext} \Delta
$;

\item[(2)] sides corresponding to cuts in $\Delta $;

\item[(3)] components corresponding to $\partial _{int}\Delta $.
\end{enumerate}

The sides of $\mathcal P$ of type (1) are $(2,0)$--quasi--geodesic
in $\G $ as $W$ is $(2,0)$--quasi--geodesic. The sides of type (2)
are geodesic in $\G $ by the first assertion of Proposition
\ref{DS}. Hence we may apply Proposition \ref{sn} to the
$n$--gon $\mathcal P$, where the set of components $I$ consists of
sides of type (3). Taking into account (\ref{gi}), we obtain
$$
\sum\limits_{i=1}^{q(W)} |g_i|_\Omega \le \sum\limits_{p\in
I}l_{\Omega } (p) \le Dn\le 4D q(W),
$$
where $D=D(2,0)$ is provided by Proposition \ref{sn}.  Hence at
least one element $g_i\in N_{\lambda_i}$ satisfies $\widehat \d_{\lambda_i} (1,g_i) <4D $. According to
(\ref{r4d}) and (\ref{snr}) this implies $g_i=1$ in $G$. However this contradicts the second assertion of Proposition
\ref{DS}.

To prove the last assertion we note that it suffices to deal with
the case when $W$ is geodesic as any element of $H_\lambda $ can
be represented by a single letter. Let $\Delta $ be a diagram of
minimal type in $\mathcal D (W)$. By the second assertion of the
lemma, some component $c$ of $\partial _{int}\Delta $ labelled by
a word in $H_\lambda \setminus \{ 1\} $ is connected to
$\partial_{ext}\Delta $. Applying Lemma \ref{trans} yields a path
$s$ in $\Delta $ connecting $\partial_{ext}\Delta $ to $c$ such
that $\Lab (s) $ is a word in the alphabet $H_\lambda \setminus \{
1\}$. Let us cut $\Delta $ along $s$ and denote the new diagram by
$\Delta _1$. Obviously the word
$$
\Lab (\partial_{ext}\Delta _1)\equiv \Lab (s)\Lab(c)\Lab (s^{-1})\Lab
(\partial_{ext}\Delta)
$$
is a word in the alphabet $H_\lambda \setminus \{ 1\}$ and $q(\Lab
(\partial_{ext}\Delta _1))<q(W)$. By the inductive assumption, $\Lab (\partial_{ext} \Delta
_1)$ represents an element of $N_\lambda $ in $G$. Since $\Lab
(c)$ represents an element of $N_\lambda $ and $N_\lambda \lhd
H_\lambda $, the word $\Lab (\partial_{ext}\Delta )$ also represents an element
of $N_\lambda $.
\end{proof}

For a word $W$ in the alphabet $X\sqcup \mathcal H$ representing $1$
in $\bar G$, we set $$\AA (W)=\min\limits_{\Delta \in \mathcal
D(W)} N_\mathcal R (\Delta ).$$ It is easy to see that for any two
words $U$ and $V$ in $X\sqcup \mathcal H$ representing $1$ in $\bar G$, we have
\begin{equation}\label{AA}
\AA (UV)\le \AA (U) +\AA (V).
\end{equation}

\begin{lem}\label{CEP2}
For any word $W$ in $X\sqcup \mathcal H$ representing $1$ in $\bar G$, we have $\AA (W)\le 3C\| W\| $, where $C$ is the relative isoperimetric constant of (\ref{CEP-Gfull}).
\end{lem}

\begin{proof}
If $q(W)=0$, then $W=1$ in $G$ and the required
estimate on $\AA (W)$ follows from the relative hyperbolicity of
$G$. We now assume that $q(W)>1$.

First suppose that the word $W$ is not $(2,0)$--quasi--geodesic in
$G$. That is, up to a cyclic shift $W\equiv W_1W_2$, where $W_1=U$
in $G$ and $\| U\| < \| W_1\| /2$. Note that $q(W_1U^{-1})=0$,
$q(UW_2)=q(W)$, and $\| UW_2\| \le \| W\| -\| W_1\| /2$. Hence
$\theta (UW_2)<\theta (W)$. Using the inductive assumption and
(\ref{AA}), we obtain
$$
\begin{array}{rl}
\AA (W)\le & \AA (W_1U^{-1})+\AA (UW_2)< \\ & \\ & \frac32C \|
W_1\| + 3C \left(\| W\| -\frac12 \| W_1\| \right) = 3C\| W\|.
\end{array}
$$

Now assume that $W$ is $(2,0)$--quasi--geodesic. Let $\Delta $ be
a diagram of minimal type in $\mathcal D(W)$. By the second
assertion of Lemma \ref{CEP1}, some component $c$ of $\partial
_{int} \Delta $ is connected to an $H_\lambda $--subpath $p$ of
$\partial _{ext} \Delta $ for some $\lambda \in \Lambda $.
According to Lemma \ref{trans}, we may assume that there is a path
$s$ in $\Delta $ connecting $c$ to $p_+$ such that $\Lab (s)$ is a
word in the alphabet $H_\lambda \setminus \{ 1\} $. We cut $\Delta
$ along $s$ and denote by $\Delta _1$ the obtained diagram. Up to
cyclic shift, we have $W\equiv W_0\Lab (p)$ and $$\Lab (\partial_{ext}\Delta
_1)\equiv W_0\Lab (p)\Lab (s)^{-1}\Lab (c)\Lab (s).$$ Let $h$ be
the element of $H_\lambda $ represented by $\Lab (p)\Lab
(s)^{-1}\Lab (c)\Lab (s)$ in $G$. Observe that $q(W_0h)=q(\phi
(\partial_{ext}\Delta _1))<q(W)$. Further since $h^{-1}\Lab (p)$ is a word in
$H_\lambda \setminus \{ 1\} $ representing $1$ in $\bar G$, we
have $h^{-1}\Lab (p)\in \mathcal Q$ and hence $\AA (h^{-1}\Lab
(p))=0$. Applying the inductive assumption we obtain
$$
\begin{array}{rl}
\AA (W)= & \AA (W_0h)+\AA (h^{-1}\Lab (p)) = \\ & \\ & \AA (W_0h)
\le 3C \| W_0h\| \le 3C \| W\| .
\end{array}
$$
\end{proof}

\begin{proof}[Proof of Theorem \ref{CEP-0}]
Lemma \ref{CEP1} gives part (a).

Part (b) follows from Lemma \ref{CEP2} in the same way as in \cite{Osi07}. Indeed let $\e_1 \colon F(\N )\to
\bar G$ be the natural homomorphism, where $F(\N )=F(X)\ast (\ast
_{\lambda \in \Lambda }H_\lambda /N_\lambda )$. Let $\e _0$ denote
the natural homomorphism $F\to F(\N )$, where $F$ is given by
(\ref{F}). Part (a) of the theorem implies that
$Ker\, \e _1=\langle \e _0 (\mathcal R)\rangle ^{F(\N )}$. Now let
$U$ be an element of $F(\N )$ such that $\e _1 (U)=1$, $W\in F$ a
preimage of $U$ such that $\| W\| =\| U\| $. Lemmas \ref{CEP2} and
\ref{cutting} imply that
\begin{equation}\label{WW1}
W=_F\prod\limits_{i=1}^{k} f_i^{-1}R_i^{\pm 1}f_i,
\end{equation}
where $f_i\in F$, $R_i\in \mathcal R\cup \mathcal Q$, and the
number of multiples corresponding to $R_i\in \mathcal R$ is at
most $3C\| W\| $. Applying $\e _0 $ to the both sides of
(\ref{WW1}) and taking into account that $\e_0 (f_i^{-1}R_if_i)=1$
in $F(\N )$ whenever $R_i\in \mathcal Q$, we obtain
$$
U=_{F(\N )} \prod\limits_{i=1}^{l} g_i^{-1}P_i^{\pm 1}g_i,
$$
where $g_i\in F(\N )$, $P_i\in \e _0 (\mathcal R)$, and $l\le 3C
\| W\| =3C\| U\| $.

This shows that $\bar G$ has a relative presentation
\begin{equation}\label{GNrel}
\bar G =\langle \bar X, \; \mathcal{\bar H} |\; \mathcal S^\prime \cup\mathcal \e_0 (R) \rangle ,
\end{equation}
with linear relative isoperimetric function. Hence the corresponding relative Cayley graph is hyperbolic by Lemma \ref{pres}, i.e., $\bar G$ is weakly hyperbolic relative to the collection $\{ H_\lambda /N_\lambda \mid \lambda \in \Lambda \} $ and the image of $X$ in $\bar G$.

To prove (c), suppose that $x=y$ in $\bar G$ for some $x,y\in X$. Assume that
$xy^{-1}\ne 1$ in $G$. Then $q(xy^{-1})>0$. Let $\Delta $ be a diagram of
minimal type in $\mathcal D (xy^{-1})$. Since $xy^{-1}$ is a $(2,0)$-quasi-geodesic word in
$G$, some component of $\partial _{int}\Delta $ must be connected to an
$H_\lambda $--subpath of $\partial _{ext}\Delta $ by the second assertion of Lemma
\ref{CEP1}. However $\partial _{ext} \Delta $ contains no
$H_\lambda $--subpaths at all and we get a contradiction.

Parts (d)-(f) can be derived from Corollary \ref{cor-he-vrf} and the corresponding results about $\alpha $-rotating families. Indeed by Corollary \ref{cor-he-vrf} we can assume that the collection $\Nl$ is $\alpha$-rotating with respect to the action of $G$ on the hyperbolic space $\mathbb K$ provided by Theorem \ref{he-vrf}. Recall that the space $\mathbb K$ constructed in the proof of Theorem \ref{he-vrf} contains $\G $ as a subspace and it is obvious from the construction that $\d_{Hau} (\G, \mathbb K)<\infty $. Thus the inclusion of $\G $ in $\mathbb K$ is a $G$-equivariant quasi-isometry and hence an element $g\in G$ is loxodromic (respectively, parabolic or elliptic) with respect to the action on $\G $ if and only if it is loxodromic (respectively, parabolic or elliptic) with respect to the action on $\mathbb K $. Thus Theorem \ref{theo;app_wind} yields parts (d)and (e). Similarly an element of $\bar G$ is loxodromic (respectively, parabolic or elliptic) with respect to the action on $\Gamma (\bar G, \bar X\sqcup\mathcal{\bar H})$ are also loxodromic (respectively, parabolic or elliptic) with respect to the action on $\mathbb K /Rot$ and we obtain (f) by applying Proposition \ref{prop;quotient_isom}.
\end{proof}

For hyperbolically embedded collections, we obtain the following.

\begin{thm}\label{CEP}
Let $G$ be a group, $X$ a subset of $G$,  $\Hl $ a collection of subgroups of $G$. Suppose that $\Hl\h (G,X)$. Then for any finite subset $Z\subseteq G$, there exists a family of finite subsets $\mathcal F_\lambda \subseteq H_\lambda \setminus \{ 1\} $  such that for every collection $\N =\{ N_\lambda \lhd H_\lambda \mid \lambda \in \Lambda \}$ satisfying $N_\lambda \cap \mathcal F_\lambda =\emptyset$ the following hold.
\begin{enumerate}
\item[(a)] The natural map from $H_\lambda /N_\lambda $ to $\bar G$ is injective for every $\lambda \in \Lambda $.

\item[(b)] $\{ H_\lambda/N_\lambda \}_{\lambda\in \Lambda }\h \bar G $.

\item[(c)] The natural epimorphism $\e \colon G\to \bar G$ is injective on $Z$.

\item[(d)] Every element of $Ker (\e )$ is either conjugate to an element of $N_\lambda $ for some $\lambda \in \Lambda $ or is loxodromic. Moreover, translation numbers of loxodromic elements of $Ker (\e)$
(with respect to the action on $\G $) are uniformly bounded away from zero.

\item[(e)] $Ker (\e)=\ast_{\lambda\in \Lambda }\ast_{t\in T_{\lambda} } N_{\lambda }^t $ for some subsets $T_\lambda \subseteq G$.

\item[(f)] Every  loxodromic (respectively, parabolic or elliptic) element of $\bar G$ is the
image of a loxodromic (respectively, parabolic or elliptic) element of $G$.
\end{enumerate}
\end{thm}

\begin{proof}
Let $R$ be the constant chosen as in the proof of Theorem \ref{CEP-0} (see (\ref{r4d})). Note that
$$
F_\lambda =\{ h\in N_\lambda \setminus \{ 1\} \mid \dl (1,h)\le R\}
$$
is finite as $\Hl \h G$. Then parts (a) and (d)-(f) follow from the corresponding parts of Theorem \ref{CEP-0}. To prove (b) note that in the notation of the proof of Theorem \ref{CEP-0}, we can assume that (\ref{CEP-Gfull}) is strongly bounded and hence so is (\ref{GNrel}). Therefore, $\{ H_\lambda/N_\lambda \}_{\lambda\in \Lambda }\h (\bar G,\bar X) $. Finally note that we can assume that $Z\subseteq X$ without loss of generality (see Corollary \ref{he-indep}). This and Theorem \ref{CEP-0} (c) give part (c).
\end{proof}


\section{Applications}\label{appl}

\subsection{Largeness properties}

The main purpose of this section is to obtain some general results about groups with non-degenerate hyperbolically embedded subgroups. For the definitions and a survey of related results we refer to Section \ref{subsec:appl}.

\begin{thm}\label{large}
Suppose that a group $G$ contains a non-degenerate hyperbolically embedded subgroup. Then the following hold.
\begin{enumerate}
\item[(a)] The group $G$ is SQ-universal. Moreover, for every finitely generated group $S$ there is a quotient group $Q$ of $G$ such that $S\h Q$.
\item[(b)] The group $G$ contains a non-trivial free normal subgroup.
\item[(c)] ${\rm dim\, }\widetilde {QH} (G)=\infty $, where $\widetilde {QH} (G)$ is the space of homogeneous quasimorphisms. In particular, ${\rm dim\,} H_b^2(G, \mathbb R)=\infty $ and $G$ is not boundedly generated.
\item[(d)] The elementary theory of $G$ is not superstable.
\end{enumerate}
\end{thm}

\begin{proof}
We start with (a). Note first that SQ-universality of $G$ follows easily from Theorems \ref{vf} and \ref{CEP}. Recall the following definition.
\begin{defn}
A subgroup $A$ of a group $B$ satisfies the {\it congruence extension property} (or CEP) if for every normal subgroup $N\lhd A$ one has $A\cap \ll N\rr ^B=N$ (or, in other words, the natural map from $A/N$ to $B/\ll N\rr ^B$ is injective.
\end{defn}
Obviously, the CEP is transitive: if $A\le B\le C$, $A$ has the CEP in $B$, and $B$ has the CEP in $C$, then $A$ has the CEP in $C$.

Let $F_n$ denote a finitely generated free group of rank $n$. By Theorem \ref{vf}, there exists a hyperbolically embedded subgroup $H$ of $G$ such that $H\cong F_2\times K(G)$. Obviously $F_2$ has the CEP in $H$. It is well known that for every $n$ and $R>0$, one can find a subgroup $F_n\le F_2$ with the CEP such that the lengths of the shortest nontrivial element of the normal closure of $F_n$ in $F_2$ with respect to a fixed finite generating set of $F_2$ is at least $R$ (see, e.g., \cite{Ols}). Obviously $F_n$ also has CEP in $H$. Using transitivity of the CEP and (a) of Theorem  \ref{CEP} we conclude that $F_n$ has CEP in $G$ if $R$ is big enough. Let $S=F_n/N$. Then $S$ embeds in $Q=G/\ll N\rr ^{G}$.

To make this embedding hyperbolic, we have to be a bit more careful. We will need two auxiliary results. The first one generalizes a well-known property of relatively hyperbolic groups.

\begin{lem}\label{quotbyfin}
Let $\Hl \h (G,X)$ and let $N$ be a finite normal subgroup of $G$. Then $\{ H_\lambda N/N \}_{\lambda \in \Lambda } \h (G/N, {\bar X})$, where  $\bar X$ is the natural image of $X$ in $G/N$.
\end{lem}
\begin{proof}
Let $$\mathcal H=\bigsqcup\limits_{\lambda\in \Lambda } H_\lambda  $$ and $$\bar{\mathcal H}=\bigsqcup\limits_{\lambda\in \Lambda } \bar H_\lambda   ,$$ where $\bar H_\lambda=H_\lambda N/N\le G/N$. Since $|N|<\infty $, the map $G\to G/N$ obviously extends to a quasi-isometry $\G \to \Gamma (G/N, \bar X\sqcup\bar{\mathcal H})$. In particular, $\Gamma (G/N, \bar X\sqcup\bar{\mathcal H})$ is hyperbolic.

Further let $\dl $ and $\dl^\prime$ be the distance functions on $H_\lambda $ and $\bar H_\lambda $ defined using the Cayley graphs $\G $ and $\Gamma (G/N, \bar X\sqcup\bar{\mathcal H})$, respectively. We have to show that $(\bar H_\lambda , \dl^\prime)$ is locally finite for every $\lambda\in \Lambda $. Fix $\lambda \in \Lambda $. If $|H_\lambda| <\infty$, we are done, so assume that $H_\lambda $ is infinite. In this case $N\le H_\lambda $ by Theorem \ref{vf}. Let us fix any section $\sigma \colon G/N\to G$. Note that $\sigma (\bar H_\lambda )\subseteq H_\lambda $. Thus $\sigma $ naturally extends to a map from the set of words in the alphabet $\bar X\sqcup\bar{\mathcal H}$ to the set of words in the alphabet $X\sqcup\mathcal H$. We denote this extension by $\sigma $ as well.

Let $\bar p$ be a path in $\Gamma (G/N, \bar X\sqcup\bar{\mathcal H})$ connecting $1$ to some $\bar x\in \bar H_\lambda $. Define $p$ to be the path in $\G $ starting at $1$ with label $\Lab (p)\equiv \sigma(\Lab (\bar p))$. Then $(p_+) =x$ for some $x\in H_\lambda N=H_\lambda $. It is straightforward to see that if $\bar p$ contains no edges of the subgraph $\Gamma (\bar H_\lambda, \bar H_\lambda )$ of $\Gamma (G/N, \bar X\sqcup\bar{\mathcal H})$ , then $p$ contains no edges of the subgraph $\Gamma (H_\lambda, H_\lambda )$ of $\G $. Thus $\dl (1,x) \le \dl ^\prime (1, \bar x)$. Therefore locall finiteness of $(H_\lambda, \dl)$ implies local finiteness of $(\bar H_\lambda , \dl^\prime)$.
\end{proof}

The next lemma is an exercise on small cancellation theory over free products.

\begin{lem}\label{FreeSQ}
Let $H$ be a non-abelian free group, $\mathcal F$ a subset of $H$, $S$ a finitely generated group. Then $S$ embeds into a quotient group $K$ of $H$ such that $K$ is hyperbolic relative to $S$ and the natural homomorphism $H\to K$ is injective on $\mathcal F$.
\end{lem}

\begin{proof}
Since $H$ is free and non-cyclic, we can decompose it as $H=A\ast B$, where $A$ and $B$ are nontrivial. Let $\{ s_1, \ldots, s_k\} $ be a generating set of $S$. Let $K=\langle A,B, S \mid x_i=w_i,\, i=1, \ldots , k\rangle$, where $w_i\in A\ast B$ and $x_i^{-1}w_i$ satisfy the $C^\prime (1/6)$ condition over the free product $A\ast B\ast S$. Note that $K$ is generated by the images of $A$ and $B$ and hence is a quotient of $H$. It is well-known that $S$ embeds in $K$ \cite[Corollary 9.4, Ch. V]{LS} and it follows immediately from the Greendlinger Lemma for free products \cite[Theorem 9.3, Ch. V]{LS} that the relative Dehn function of $K$ with respect to $S$ is linear. Hence $K$ is hyperbolic relative to $S$. The Greendlinger Lemma also implies that if the elements $w_i$ are long enough with respect to the generating set $A\cup B$ of $H$, then $H\to K$ is injective on $\mathcal F$.
\end{proof}

Let now $G$ be a group with a non-degenerate hyperbolically embedded subgroup. Recall that $K(G)$ denote the maximal normal finite subgroup of $G$ (see Theorem \ref{vf}). Indeed let $H$ be a non-degenerate \he subgroup of $G$. Then $K(G)\le H$ by Theorem \ref{vf} and hence the image of $H$ in $G/K(G)$ is also non-degenerate (i.e., proper and infinite). By Lemma \ref{quotbyfin} the image of $H$ in $G/K(G)$  is \he in $G/K(G)$. Thus passing to $G/K(G)$ if necessary and using Lemma \ref{quotbyfin}, we can assume that $K(G)=\{ 1\} $.

Again by Theorem \ref{vf} there exists a hyperbolically embedded free subgroup $H$ of rank $2$ in $G$. Let $R>0$ be the constant provided by Theorem \ref{CEP} and let $\mathcal F$ be the set of all nontrivial elements $h\in H$ such that $\widehat\d (1,h)\le R$. By Lemma \ref{FreeSQ}, $S$ embeds in a quotient group $K$ of $H$ such that $K$ is hyperbolic relative to $S$ and the natural homomorphism $H\to K$ is injective on $\mathcal F$. In particular, $S\h K$ by Proposition \ref{he-rh}. Let $N=\Ker (H\to K)$, and let $G_1= G/\ll N\rr$. Since $H\to K$ is injective on $\mathcal F$, $N$ satisfies the assumptions of Theorem \ref{CEP}. Hence $K=H/N\h G_1$. Since $S\h K$ we have $S\h G_1$  by Proposition \ref{transitive}. This completes the proof of the part (a) of Theorem \ref{large}.

The proof of (b) follows the standard line. By Theorem \ref{vf}, there exists an infinite elementary subgroup $E\h G$. Let $g\in E$ be an element of infinite order such that $\langle g\rangle \lhd E$. Then for sufficiently large $n\in \mathbb N$, we can apply Theorem \ref{CEP} to the group $G$, the subgroup $E\h G$, and the normal subgroup $\langle g^n\rangle $. In particular, $\ll N\rr ^G$ is free.

Recall that a \textit{quasi-morphism} of a group $G$ is a map $\phi\colon G\to \mathbb R$ such that
$$ \sup_{g,h\in G} |\phi (gh)-\phi(g)-\phi (h)|<\infty .$$ Trivial examples of quasi-morphisms are bounded maps and homomorphisms. Note that the set $QH(G)$ of all quasi-morphisms has a structure of a linear vector space and $\ell ^\infty (G)$ and $Hom(G,\mathbb R)$ are subspaces of $QH(G)$. By definition, the \textit{space of non-trivial quasi-morphisms} is the quotient space
$$
\widetilde{QH} (G) =QH(G)/(\ell^\infty (G)\oplus Hom (G, \mathbb R)).
$$

The third part of Theorem \ref{large} follows easily from Corollary \ref{elemhe1}, Proposition \ref{malnorm}, and \cite[Theorem 1]{BF}. Indeed suppose that $H$ is a non-degenerate subgroup of $G$ such that $H\h (G,X)$ for some $X\subseteq G$. Consider  the action of $G$ on $\Gamma (G, X\sqcup H)$. By Corollary \ref{elemhe1}, there exist two loxodromic elements $g,h\in G$ such that $\{ E(g), E(h)\} \h G$. By the characterization of elementary subgroups obtained in Lemma \ref{elem1}, $g$ and $h$ are independent in the terminology of \cite{BF}. Furthermore, $g\not\sim h$ in the notation of \cite{BF} by Proposition \ref{malnorm}. (Recall that $g\sim h$ if and only if some positive powers of $g$ and $h$ are conjugate, see the remark after the definition of the equivalence on p. 72 of \cite{BF}.) Now Theorem 1.1 from \cite{BF} gives ${\rm dim\, }\widetilde {QH} (G)=\infty $. The fact that $\dim H^2_b(G,\mathbb R)=\infty $ follows from the well-known observation that the space $\widetilde{QH} (G)$ can be naturally identified with the kernel of the canonical map $H^2_b(G,\mathbb R) \to H^2(G,\mathbb R)$ of the second bounded cohomology space to the ordinary second cohomology. It is also well-known and straightforward to prove that for every boundedly generated group $G$, the space $\widetilde{QH} (G)$ is finite dimensional.

To prove (d) we need the following lemma, which is a simplification of \cite[Corollary 1.7]{bau}.
\begin{lem}[Baudisch, \cite{bau}]\label{Baudisch}
Let  $G$ be an infinite superstable group.  Then there are subgroups $1=H_0 \lhd H_1 \lhd \dots \lhd H_n = G$ such that every quotient $H_{i+1}/H_{i}$ is either abelian or simple.
\end{lem}

On the other hand, we have the following.

\begin{lem}\label{g-normal}
Let $G$ be a group that contains a non-degenerate \he subgroup. Then every infinite subnormal subgroup of $G$ contains a non-degenerate \he subgroup.
\end{lem}

\begin{proof}
Clearly it suffices to prove the theorem for normal subgroups; then the general case follows by induction. Let $N\lhd G$.

By Theorem \ref{vf}, there exists an infinite elementary subgroup $E$ such that $E\h (G,Y)$ for some $Y\subseteq G$. If $|N\cap E|=\infty$,  then $N$ is finite by Proposition \ref{malnorm}, which contradicts our assumption. Thus $N\setminus E$ is non-empty. Let $a\in N\setminus E$ and let $g\in E$ be an element of infinite order.

Let $\widehat \d$ denote the metric on $E$ defined using $\Gamma (G, Y\sqcup E)$. Without loss of generality we can assume that $a\in Y$ (see Corollary \ref{he-indep})).  Take $f\in \langle g\rangle $ such that $\widehat\d(1,f)>50 D$, where $D=D(1,0) $ is given by Proposition \ref{sn}. Let $w=faf^{-1}a$. Clearly $w\in N$. On the other hand, the word $faf^{-1}a$ in the alphabet $Y\sqcup E$, where $f$ is interpreted as a letter from $E$, satisfies the conditions (W$_1$)--(W$_3$) of Lemma \ref{w} applied to $G$, $Y$, and the collection $\{ E \}$. Hence $w$ acts loxodromically on $\Gamma (G, Y\sqcup E)$. Moreover, the WPD condition can be verified for $w$ exactly in the same way as in the third paragraph of the proof of Theorem \ref{elemhe} (with $E$ in place of $H_\lambda $ and $Y$ in place of $X$). Now   Theorem \ref{wpd} applied to the group $N$ acting on $\Gamma (G, Y\sqcup E)$ yields an elementary subgroup $E_1$ containing $w$ such that $E_1\h N$. Similarly applying Theorem \ref{wpd} to the action of $G$, we obtain a maximal elementary subgroup $E_2$ of $G$ containing $w$, which is \he in $G$. Obviously $E_1\le E_2$. Thus if $N=E_1$, we get a contradiction with Proposition \ref{malnorm}. Hence $N\ne E_1$ and we are done.
\end{proof}

We now observe that part (c) of Theorem \ref{large} follows easily from Lemma \ref{Baudisch} and Lemma \ref{g-normal}. Indeed if $G$ was superstable, it would contain either infinite finite-by-abelian or infinite finite-by-simple subnormal subgroup by Lemma \ref{Baudisch}. The first case contradicts Lemma \ref{g-normal} and the first part of Theorem \ref{large} as no finite-by-abelian can be $SQ$-universal. In the second case we get a contradiction with the existence of free normal subgroups.
\end{proof}

\subsection{Subgroups in mapping class groups and $Out(F_n)$}\label{sec-MCG}

All theorems in this section are formulated for normal closures of a single element for simplicity. We leave the (obvious) generalization to the case of several elements to the reader.

We begin with a general result about normal closures of high powers of loxodromic WPD elements and its uniform version for acylindrical actions.

\begin{thm}\label{wpd-free}
Let $G$ be a group acting on a hyperbolic space $\X$.
\begin{enumerate}
\item[(a)] For every loxodromic WPD element $g\in G$, there exists $n\in \mathbb N$ such that the normal closure $\ll g^n\rr$ in $G$ is free.
\item[(b)] If the action of $G$ is acylindrical, then there exists $n\in \mathbb N$ such that for every loxodromic element $g\in G$, the normal closure $\ll g^n\rr$ in $G$ is free.
\end{enumerate}
Moreover, in both cases every non-trivial element of $\ll g^n\rr $ is loxodromic with respect to the action on $\X$.
\end{thm}

\begin{proof}
To prove (a), we recall that, by Proposition \ref{prop;Acyl_free} (a), there exists $n>0$ such that the  cyclic subgroup $\grp{g^n}$ is $200$-rotating with respect to the induced action of $G$ on a certain cone-off $C(\X)$ of $\X$. Hence the subgroup $\ll g^n\rr$ is a free product of cyclic groups by Corollary \ref{coro;app_wind}, i.e. it is a free group. The same argument with a reference to Proposition \ref{prop;Acyl_free} (b) proves (b). Finally, Theorem \ref{theo;app_wind} (b) applied to the action of $G$ on $C(\X)$ implies (in both (a) and (b)) that every element $h\in \ll g^n\rr$ is either loxodromic with respect to the action of $G$ on $C(\X)$ or is conjugate to a power of $g$. Clearly $h$ is loxodromic with respect to the action on $\X$ in both cases.
\end{proof}

In particular, Lemma \ref{actionMCG} allows us to apply this result to mapping class groups.

\begin{thm}\label{theo;MCG}
Let $\Sigma$ be a (possibly punctured) orientable closed surface. Then there exists $n$ such that for any pseudo-Anosov element  $g\in \calM\calC\calG(\Sigma)$, the normal closure of $g^{n}$ is free and purely pseudo-Anosov.
\end{thm}

\begin{proof}
If $\Sigma $ is exceptional (i.e., $3g+p-4\le 0$), then $\calM\calC\calG(\Sigma)$ is hyperbolic and acts acylindrically on its Cayley graph with respect to a finite generating set. If $\Sigma$ is non-exceptional, then $\calM\calC\calG(\Sigma)$ acts acylindrically on the corresponding hyperbolic curve complex (see Lemma \ref{actionMCG}). In both cases loxodromic elements with respect to the action are exactly the pseudo-Anosov elements. Thus Theorem \ref{wpd-free} immediately gives the result.
\end{proof}

Recall that a subgroup $H<\MCG$ is \emph{reducible} if it contains no pseudo-Anosov elements \cite{Iva92}.
In the spirit of some constructions of infinite periodic groups, we can also obtain the following.

  \begin{thm}\label{theo;MCG_period}
Let $\Sigma$ be a closed orientable surface, possibly with punctures.
 Then, there exists a quotient of its Mapping Class group $\pi:  \calM\calC\calG(\Sigma) \to Q$ such that,
\begin{enumerate}
 \item[(a)]   $\pi$ is injective on each reducible subgroup and
\item[(b)] for all element $g \in  \calM\calC\calG(\Sigma)$, either $\pi(g)$ has finite order, or $\pi(g)\in\pi(H)$ for some reducible subgroup $H<\MCG$.
\end{enumerate}
\end{thm}

To prove Theorem \ref{theo;MCG_period}, we construct by induction a sequence of quotients
using repeatedly the argument of Theorem \ref{theo;MCG}.

\begin{proof}
Observe that in the exceptional cases,  (i.e., when $3g + p - 4 \leq 0$),  $\calM\calC\calG(\Sigma)$ is hyperbolic.
Moreover, the reducible subgroups are the subgroups of the stabilizers of multicurves, which consist of finitely many conjugacy classes of finite or
virtually cyclic subgroups. Thus, the result is well known in this case.
 We assume  $3g + p - 4 > 0$.

 Let $(g_n)_{n\geq 1} $ be an enumeration of the pseudo-Anosov elements of   $\calM\calC\calG(\Sigma)$.  Let $Q_0 = \calM\calC\calG(\Sigma)$.  We want to prove that, for all $n\geq 1$, there is a quotient  $\pi_{n}:  Q_{n-1} \to Q_{n}$ injective on (the image of) each  reducible subgroup, and such that the image of $g_{n}$ in $Q_{n}$ is either of finite order or equals the image of a reducible element. Indeed, if such a quotient is found, The theorem holds with $Q=G/Q_\infty$ where
 $Q_\infty=\bigcap_{n\geq1} \ker \pi_n\circ\dots\circ\pi_1$.

Our induction hypothesis is the following. The group $Q_n$ acts acylindrically, co-boundedly on a hyperbolic graph $\mathcal{K}_n$, and the elliptic elements are precisely the  images of the reducible elements of  $\calM\calC\calG(\Sigma)$.

This is satisfied for $n=0$,  by theorems of Masur-Minsky, and Bowditch (recalled in Lemma \ref{actionMCG}).

Assume it is satisfied for $n-1$. Consider $g_n$, and its image $\bar g_n$ in $Q_{n-1}$. If it is  elliptic on  $\mathcal{K}_{n-1}$, then taking $Q_n = Q_{n-1}$ and $\pi_n$ to be the identity is suitable.

Assume then that  $\bar g_n$ is loxodromic in  $\mathcal{K}_{n-1}$ (the argument that we are going to give now is similar to that of  Theorem \ref{theo;MCG}, but with $\mathcal{K}_{n-1}$ replacing the curve complex).  The action of   $Q_{n-1}$ on  the graph $\mathcal{K}_{n-1}$ is acylindrical, therefore  by Proposition \ref{prop;SC_from_acyl}, we can choose $m$ so that the family of conjugates of  $\bar g_n^m$ satisfy the $(A_0,\epsilon_0)$-small-cancellation condition (the constants are those of  Proposition \ref{prop_sc_subgroup}).
Then,  Proposition \ref{prop_sc_subgroup} can be applied, which ensures that, for the constants defined there (which are universal),  the cone-off space  $\dot{\mathcal{K}}_{n-1} = C(\lambda \mathcal{K}_{n-1}, Q_{\bar g_n^m}, r_0)$ along the axis of $\bar g_n^m$ (and its conjugates)
is $\du$-hyperbolic and
 carries a  $2r_0>100\du$-separated very rotating family consisting of conjugates of $\bar g_n^m$.
The group generated by this family is denoted by $Rot_n$.

The action of $Q_{n-1}$ on  $\dot{\mathcal{K}}_{n-1}$ is still acylindrical, by  Proposition \ref{prop;cone_acyl}.

  Then  we define $Q_n = Q_{n-1}/ Rot_n$ and  $\mathcal{K}_n'= \dot{\mathcal{K}}_{n-1}/Rot_n$. By Proposition
  \ref{prop;quotient_hyp},
   $\mathcal{K}_n'$ is hyperbolic.  By construction, the action of $Q_n$ on  $\mathcal{K}_n'$ is also co-bounded. Also, by Theorem \ref{theo;app_wind}
any element of $Rot_n\setminus\{1\}$ is either conjugate to a power of $\bar g_n$, or acts loxodromically on $\dot{\mathcal{K}}_{n-1}$,
so $Rot_n\setminus\{1\}$ contains no element elliptic in $\mathcal{K}_{n-1}$.
It follows that the quotient map  $\pi_n:Q_{n-1}\ra Q_n$ is injective on the image of each reducible subgroups in $Q_{n-1}$.

Proposition  \ref{prop;quotient_isom}
  ensures that elliptic elements in the quotient are images of elliptic
  elements in the cone-off, namely elliptic elements on $\mathcal{K}_{n-1}$ or
elements conjugate in the maximal virtually cyclic group containing  $g_n$.

 Since we showed that the action of   $Q_{n-1}$ on  $\dot{\mathcal{K}}_{n-1}$ is acylindrical,   by Proposition  \ref{prop;quotient_acyl}, the action of  $Q_n$ on $\mathcal{K}_n'$ also is acylindrical. Finally, $\mathcal{K}_n'$ is not a graph but one can replace it by a graph $\mathcal{K}_n$ thanks to Lemma \ref{lem_graph}.
Clearly, $\mathcal{K}_n$ is hyperbolic, $Q_n$ still acts coboundedly and acylindrically on $\mathcal{K}_n$, and the elements elliptic in $\mathcal{K}'_n$ and $\mathcal{K}_n$  are the same.
\end{proof}

The next theorem is useful for proving results about subgroups of mapping class groups.

\begin{thm}\label{thm_MCG_HE}
Let $\Sigma$ be a (possibly punctured) closed orientable surface.
Let $G<\MCG$ be a subgroup, that is not virtually abelian.
Then $G$ has a finite index subgroup
having a quotient $Q$ such that
$Q$ contains a non-degenerate hyperbolically embedded cyclic subgroup.
\end{thm}

\begin{proof}
Suppose that our surface has genus $g$ and $p\ge 0$ punctures. The proof is by induction on the complexity. We first take care of surfaces
for which $3g+p-4\leq 0$. In these cases,
$\MCG$ is finite for $(g,p)\in\{(0,0),(0,1),(0,2),(0,3)\}$
and virtually free for $(g,p)\in\{(0,4),(1,0),(1,1)\}$.
The result is thus clear in these cases (see for instance Theorem \ref{WPD}).

Assume now that $3g+p-4>0$.
Since $\MCG$ is virtually torsion free, we can assume that $G$ is torsion-free.
If $G$ contains a pseudo-Anosov element, then by Theorem \ref{WPD} and Lemma \ref{actionMCG}
$G$ contains a hyperbolically embedded infinite cyclic subgroup (it is proper since $G$ is not cyclic). Note that we use here the well-known (and easy to prove) fact that a torsion free virtually cyclic group is cyclic.

If $G$ does not contain a pseudo-Anosov element, we use Ivanov's theorem which states that any subgroup of $\MCG$ containing no pseudo-Anosov element is either finite, or preserves a multicurve \cite{Iva92}.
The stabilizer of a multicurve has a finite index subgroup $R_0$
such that there is a homomorphism $\psi:R_0\ra\prod_i \calM\calC\calG(\Sigma_i)$, where $\ker\psi$ is abelian,
and $\Sigma_i$ are surfaces with lower complexity. 
Denote by $G_i$ the natural projection of $\psi(G\cap R_0)$ to $\calM\calC\calG(\Sigma_i)$.
If all the groups $G_i$ are virtually abelian, then $G$ is virtually solvable, and
hence it is virtually abelian
by the Tits alternative for $\MCG$ \cite{Iva92}.
Otherwise, by induction some $G_i$ has a finite index subgroup
having a quotient $Q$ satisfying the conclusion of the theorem,
and the result follows.
\end{proof}

We will see below that Theorem \ref{thm_MCG_HE} implies that a subgroup
of $\MCG$ that is not virtually abelian is SQ-universal. This allows to reprove various (well-known) non-embedding theorems for lattices in mapping class groups. Compare the following corollary to \cite{Farb_Masur}.

  \begin{cor}\label{coro;HJI}
Let $\Sigma$ be a (possibly punctured) closed orientable surface. Then every subgroup of $\MCG$ is either virtually abelian or $SQ$-universal. In particular, every homomorphism from an irreducible lattice in a connected semisimple Lie group of $\mathbb R$-rank at least $2$ with finite center to $\MCG $ has finite image.
  \end{cor}

\begin{proof}
Let $G\le MCG$. Suppose that $G$ is not virtually abelian.  By Theorem \ref{thm_MCG_HE}, $G$ has a finite index subgroup $G_0$ having a quotient $Q$ containing a non-degenerate \he subgroup. By Theorem \ref{large} proved below, $Q$ (and hence $G_0$) is SQ-universal. By a theorem of P. Neumann \cite{PNeu} (who attributes the result to Ph. Hall), a group containing an SQ-universal subgroup of finite index is itself SQ-universal. Hence $G$ is SQ-universal.

The claim about lattices easily follows from the Margulis normal subgroup theorem. Indeed the latter says that every normal subgroup of an irreducible lattice $\Gamma $ in a connected semisimple Lie group of $\mathbb R$-rank at least $2$ is either finite or of finite index. In particular, $\Gamma $ contains only countably many normal subgroups. On the other hand, every countable SQ-universal group $G$ has uncountably many normal subgroups. Indeed, every single quotient of $G$ has only countably many finitely generated subgroups while the number of isomorphism classes of finitely generated groups is continuum.  Thus the definition of SQ-universality implies that $G$ has continuously many quotients. Thus the image of $\Gamma $ in $\MCG $ is virtually abelian and consequently it is finite finite (say, by the same Margulis theorem).
\end{proof}

Similarly to Theorem  \ref{theo;MCG}, we obtain the following.

\begin{thm}\label{theo;Out}
  Let $Out(F_n)$ be the outer automorphism group of a free group. For any  iwip element $g\in Out(F_n)$, there exists $m$ such that the normal closure of $g^{m}$ is free and purely iwip.
\end{thm}

\begin{proof}
We consider the action of $Out(F_n)$ on the free factor complex originally introduced by Hatcher and Vogtmann \cite{HaVo}. This complex is hyperbolic and every element of $Out(F_n)$ acting loxodromically also satisfies WPD \cite[Theorem 9.3]{BFe2}. Recall also that all loxodromic elements with respect to this action are iwip (by definition of the free factor complex, an element that is not iwip has a finite orbit). In these settings, theorem follows immediately from Theorem \ref{wpd-free}.\end{proof}

We also have a weak version of Theorem \ref{thm_MCG_HE} for $\Out(F_n)$.
\begin{thm}
  Let $G<\Out(F_n)$ be a subgroup containing an iwip element.
If $G$ is not virtually cyclic, then $G$ is SQ-universal.
\end{thm}

\begin{proof}
  Let $g\in G$ be an iwip element.
  As above, we use the fact that $\Out(F_n)$ acts on the free factor complex $\X$ in which $g$ acts loxodromically with the WPD property.
  This also holds for the action of $G$ on $\X$. By Theorem \ref{wpd},  $G$ contains a hyperbolically embedded virtually cyclic subgroup,
  so $G$ is SQ-universal by Theorem \ref{large} below.
\end{proof}

\subsection{Inner amenability and $C^\ast $-algebras}

The main goal of this section is to characterize groups with non-degenerate hyperbolically embedded subgroups that are inner amenable or have simple reduced $C^\ast $-algebra with unique trace.

\begin{thm}\label{cstar}
Suppose that a group $G$ contains a non-degenerate hyperbolically embedded subgroup. Then the following conditions are equivalent.
\begin{enumerate}
\item[(a)] $G$ has no nontrivial finite normal subgroups.
\item[(b)] $G$ contains a proper infinite cyclic hyperbolically embedded subgroup.
\item[(c)] $G$ is ICC.
\item[(d)] $G$ is not inner amenable.
\end{enumerate}
If, in addition, $G$ is countable, the above conditions are also equivalent to
\begin{enumerate}
\item[(e)] The reduced $C^\ast $-algebra of $G$ is simple.
\item[(f)] The reduced $C^\ast $-algebra of $G$ has  a unique normalized trace.
\end{enumerate}
\end{thm}

The rest of the section is devoted to the proof of Theorem \ref{cstar} so we assume that $G$ contains a non-degenerate hyperbolically embedded subgroup.

We will show first that (c) $\Rightarrow$ (a) $\Rightarrow$ (b) $\Rightarrow$ (c). The implication (c) $\Rightarrow $ (a) is obvious. Further by Theorem \ref{vf} applied in the case $n=1$ we obtain (a) $\Rightarrow$ (b). Let us show that (b) $\Rightarrow$ (c). Let $C=\langle c\rangle \h G$ be an infinite cyclic subgroup and let  $g\in G\setminus \{1\}$. We want to show that the conjugacy class of $g$ in $G$ is infinite. If the set $\{ g^{c^n} \mid n\in \mathbb Z\}$ is infinite, we are done. Otherwise $g^{c^m}=g$ for some $m\in \mathbb N$. Hence $C^g\cap C$ is infinite and $g\in C$ by Proposition \ref{malnorm}. Now if $g^h=g^f$ for some $f,h\in G$, then $g^{fh^{-1}}=g$ and we similarly obtain $fh^{-1}\in C$, i.e., $f$ and $h$ belong to the same right coset of $C$. As every group containing a non-degenerate hyperbolically embedded subgroup is non-elementary (say, by Theorem \ref{vf}), the index of $C$ in $G$ is infinite. Hence the conjugacy class of $g$ in $C$ is infinite. Thus conditions (a)-(c) are equivalent.

To relate (a)-(c) to properties of $C^\ast $-algebras we need the following results.

\begin{lem}[{\cite[Theorem 3]{AL}}]\label{Ake1}
If a countable group $G$ contains a $C^\ast $-simple normal subgroup $N$ with trivial centralizer, then $G$ is $C^\ast $-simple.
\end{lem}

Suppose now that $G$ satisfies (b). Let $C=\langle c\rangle \h G$ be an infinite cyclic subgroup. By Theorem \ref{CEP}, there exists $n\in \mathbb N$ such that the normal closure of $c^n$ in $G$ is free. We denote this normal closure by $F$. Observe that $F$ satisfies the assumptions of Lemma \ref{Ake1}. Indeed for every $g\in C_G(F)$ we have $[g,c^n]=1$. Hence by Proposition \ref{malnorm}, we have $g\in C$. Further let us take any $a\in G\setminus C$. Since $(c^n)^a\in F$, we have $[g, (c^n)^a]=1$, which can be rewritten as $[g^{a^{-1}}, c^n]=1$. Again by Proposition \ref{malnorm} we obtain $g^{a^{-1}}\in C$ or, equivalently, $g\in C^a$. One more application of Proposition \ref{malnorm} gives $|g| \le |C\cap C^a|<\infty $, which is only possible if $g=1$ as $C\cong \mathbb Z$. Thus we obtain (e) and (f) by Lemma \ref{Ake1}. On the other hand it is well-known that discrete group with simple reduced $C^\ast $-algebra (or with a non-unique trace) can not have a non-trivial finite (and even amenable) normal subgroup (see, e.g., \cite{BeHa}). Thus either of (e), (f) is equivalent to to (a)-(c).

Finally let us prove that (d) is equivalent to the other conditions. The implication (d) $\Rightarrow$ (c) is obvious since every group $G$ with a finite nontrivial conjugacy class $g^G$ admits the natural conjugation invariant finitely additive measure on $G\setminus \{ 1\}$ such that $\mu (G\setminus \{1\})=1$. Namely, given $A\subseteq G\setminus\{ 1\}$, we let $\mu (A)=|A\cap [g]|/|[g]|$, where $[g]$ is the conjugacy class of $g$ in $G$.

To complete the proof of the theorem, we will prove the implication (a) $\Rightarrow$ (d). The proof is more technical and uses a variant of Tarski paradoxical decomposition.

By Lemma \ref{suitable} there exist infinite cyclic subgroups $H_1, \ldots , H_4\le  G$ such that $\{ H_1, \ldots , H_4\} \h (G,X)$ for some $X\subseteq G$. Note that by Proposition \ref{malnorm}, we have
\begin{equation}\label{HiHj}
H_i\cap H_j=\emptyset
\end{equation}
for every $i\ne j$, $i,j\in \{ 1,2,3,4\}$.

We define $\mathcal H$ in the usual way by $$\mathcal H=\bigsqcup _{i=1}^4 (H_i ).$$ Denote by $A$ the set of all elements $g\in G\setminus \{ 1\}$ satisfying the following property: there exists a geodesic $\gamma $ going from $1$ to $g$ in $\G $ such that the first edge of $\gamma $ is an $H_i $-component for $i\in \{ 1, 2\}$. Further let $B=G\setminus (A\cup \{ 1\} )$. Let $D=D(1,0)$ be the constant provided by Proposition \ref{sn}. Since for every $i\in \{ 1, \ldots, 4\}$, $H_i$ is infinite and hyperbolically embedded, there exists $h_i \in H_i $ such that
\begin{equation}\label{h1h4}
\widehat\d_i (1, h_i)>6D.
\end{equation}
Let $$A_1=A^{h_3},\;\; A_2=A^{h_4},\;\; B_1=B^{h_1},\;\; B_2=B^{h_2}.$$ We are going to show that the sets $A_1$, $A_2$, $B_1$, $B_2$ are pairwise disjoint.

\begin{lem}\label{A1A2}
$A_1\cap A_2=\emptyset $.
\end{lem}

\begin{figure}
  \centering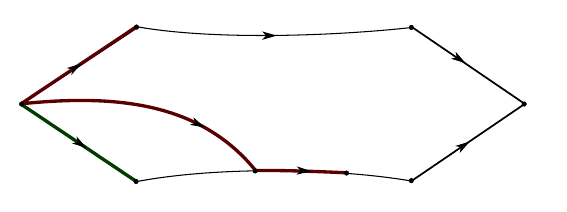\\
  \caption{}\label{83-f1}
\end{figure}

\begin{proof}
Suppose that there is $g\in A_1\cap A_2$. Then $g^{h_3^{-1}}\in A$ and $g^{h_4^{-1}}\in A$. Thus there exist be geodesic words $U_1$, $U_2$ in $X\sqcup \mathcal H$ representing $g^{h_3^{-1}}$ and $g^{h_4^{-1}}$, respectively, such that the first letters of $U_1$ and $U_2$ belong to $H_1\cup H_2$.  (Recall that a word is geodesic if it has shortest length among words representing the same element or, alternatively, every path in $\G$ labelled by this word is geodesic.) Let $p_1$, $p_2$ be paths in $\G $ starting at $1$ and having labels $\Lab (p_1)\equiv h_3^{-1}U_1h_3$ and $\Lab (p_2)\equiv h_4^{-1}U_2h_4$, respectively. Clearly $(p_1)_+=(p_2)_+=g$.

Since the first letter in $U_1$ belongs to $H_1\cup H_2$, the first edge $a$ of $p_1$ labelled by $h_3^{-1}$ is an $H_3$-component of the cycle $q=p_1p_2^{-1}$. Note that $q$ consists of $6$ geodesic segments and hence by (\ref{h1h4}) and Proposition \ref{sn} $a$ can non be isolated in $q$. Observer first that $a$ can not be connected to an $H_3$-component of $p_1$. Indeed this would mean that $U_1\equiv WU$, where $U$ may be trivial while $W$ is nontrivial and represents a (nontrivial) element of $H_3$. Since every (nontrivial) element of $H_3$ can be represented by a single letter from $H_3 $ and $U_1$ is geodesic, we conclude that $W$ consists of a single letter. By the choice of $U_1$ this letter is from $H_1$ or $H_2$. Hence one of the intersections $H_3\cap H_1$ or $H_3\cap H_2$ is nontrivial, which contradicts (\ref{HiHj}).

Thus $a$ is connected to an $H_3$-component $b$ of $p_2$. Let $e$ be a path in $\G$ of length at most $1$ labelled by an element of $H_3$ and going from $1=a_-$ to $b_-$. Repeating the arguments from the previous paragraph, we obtain that the first edge $c$ of $p_2$ is an $H_4$-component of $q$, which is isolated in $p_2$. In particular, $c$ is isolated in the cycle $ce[h_4^{-1},b_-]^{-1}$, where $[h_4^{-1},b_-]$ is the segment of $p_2$ from $h_4^{-1}$ to $b_-$. Note that $ce[h_4^{-1},b_-]^{-1}$ is composed of at most $3$ geodesics. Hence $\widehat\d _4 (1, h_4^{-1})\le 3D$ by Proposition \ref{sn}, which contradicts (\ref{h1h4}).
\end{proof}

\begin{lem}\label{A1B1}
$A_i\cap B_j=\emptyset $ for any $i,j\in \{1, 2\}$.
\end{lem}
\begin{proof}
We assume that $i=j=1$. The proof for other pairs $i,j$ is identical. Suppose that there exists $g\in A_1\cap B_1$. Then $g^{h_3^{-1}}\in A$ and $g^{h_1^{-1}}\in B$. Let $U_1$, $U_2$ be geodesic words in $X\sqcup \mathcal H$ representing $g^{h_3^{-1}}$ and $g^{h_1^{-1}}$, respectively, such that the first letter of $U_1$ belongs to $H_1\cup H_2$ while the first letter of $U_2$ does not belong to $H_1\cup H_2$. Let $p_1$, $p_2$ be paths in $\G $ starting at $1$ and having labels $\Lab (p_1)\equiv h_3^{-1}U_1h_3$ and $\Lab (p_2)\equiv h_1^{-1}U_2h_1$, respectively. Clearly $(p_1)_+=(p_2)_+=g$.

Let $a$ and $c$ be the first edges of $p_1$ and $p_2$ respectively. As in the proof of Lemma \ref{A1A2}, we prove that $a$ is an $H_3$-component of $q=p_1p_2^{-1}$, which is isolated in $p_1$. Hence, as above, we conclude that $a$ is connected to an $H_3$-component $b$ of $p_2$. Let $e$ be a path in $\G$ of length at most $1$ labelled by an element of $H_3$ and going from $1=a_-$ to $b_-$.

Since the first letter of $U_2$ does not belong to $H_1$, $c$ is an $H_1$-component of $p_2$. Since $U_2$ is geodesic, $c$ can not be connected to an $H_1$-component of the segment $[h_1^{-1}, b_-]$ of $p_1$. Hence $c$ is isolated in the cycle $ce[h_1^{-1}, b_-]^{-1}$, and we obtain $\widehat\d _1(1, h_1^{-1})$, which contradicts (\ref{h1h4}) again.
\end{proof}

\begin{lem}\label{B1B2}
$B_1\cap B_2=\emptyset $.
\end{lem}
\begin{proof}
Suppose that $g\in B_1\cap B_2$. Then $g^{h_1^{-1}}\in B$ and $g^{h_2^{-1}}\in B$. Again let $U_1$, $U_2$ be geodesic words in $X\sqcup \mathcal H$ representing $g^{h_1^{-1}}$ and $g^{h_2^{-1}}$, respectively, such that the first letters of $U_1$ and $U_2$ do not belong to $H_1\cup H_2$. Let $p_1$, $p_2$ be paths in $\G $ going from $1$ to $g$ and having labels $\Lab (p_1)\equiv h_1^{-1}U_1h_1$ and $\Lab (p_2)\equiv h_2^{-1}U_2h_2$, respectively.

Let $a$ and $c$ be the first edges of $p_1$ and $p_2$, respectively. Again it is easy to see that $a$ and $c$ are components of $q=p_1p_2^{-1}$. Suppose $a$ is connected to another $H_1$-component $d$ of $p_1$. As $U_1$ is geodesic, $d$ must be the last edge of $p_2$. Hence $U_1$ represents an element of $H_1$, i.e., $g^{h_1^{-1}}\in H_1$. However this means that $g^{h_1^{-1}}\in A$ by the definition of $A$. A contradiction. Thus $a$ is isolated in $p_1$. Similarly, $c$ is isolated in $p_2$. The rest of the proof is identical to that of Lemmas \ref{A1A2} and \ref{A1B1}.
\end{proof}

Now we are ready to complete the proof of Theorem \ref{cstar}. Assuming (a), suppose also that the group $G$ is inner amenable. That is, there exists a finitely additive conjugation invariant measure defined on all subsets of $G\setminus \{1\}$ such that $\mu (G\setminus \{1\})=1$. Since $A\sqcup B=G\setminus \{1\}$, $\mu (A)+\mu (B)=1$. On the other hand, by Lemmas \ref{A1A2} - \ref{B1B2} we have
$$
1=\mu (G\setminus \{1\}) \ge \mu (A_1)+\mu(A_2)+\mu (B_1)+\mu (B_2)=2\mu (A)+2\mu (B)=2 .
$$
A contradiction. Hence $G$ is not inner amenable. This completes the proof of the theorem.


\section{Some open problems}


In this section we discuss some natural open problems about hyperbolically embedded subgroups and rotating families. Since the first version of this paper was published in arXiv, most of the problems from the list below were solved partially or completely. We keep this section in the new version of our paper for historical reason and add footnotes describing the recent progress.

We start with problems which ask whether the ``hyperbolic properties" of groups considered in this paper are geometric. Recall that if a finitely generated group $G_1$ is hyperbolic relative to a collection of proper subgroups, then so is any finitely generated group $G_2$ quasi-isometric to $G_1$. In the full generality this fact was proved by Drutu in \cite{Dr09} (see also \cite{DS} for a particular case). For a survey of some other classical and more recent quasi-isometric rigidity results we refer to \cite{Dr11}.

\begin{prob}Is the existence of non-degenerate \he subgroups a quasi-isometry invariant? That is, suppose that a finitely generated group $G_1$ contains a non-degenerate \he subgroup $H_1$ and $G_2$ is a finitely generated group quasi-isometric to $G_1$.
\begin{enumerate}
\item[(a)] Does $G_2$ contain any non-degenerate \he subgroup?
\item[(b)] Does $G_2$ contain a hyperbolically embedded subgroup $H_2$ which is within a finite Hausdorff distance from the image of $H_1$ under the quasi-isometry between $G_1$ and $G_2$?
\end{enumerate}
\end{prob}

Similar questions make sense for rotating families. There are several ways to make these questions precise. We suggest just one of them.
Except in degenerate cases, groups with $\alpha $-rotating subgroups for $\alpha >>1$ contain non-abelian free subgroups, and are therefore non-amenable. Recall that two finitely generated non-amenable groups $G_1$, $G_2$ are quasi-isometric if and only if they are bi-Lipschitz equivalent, i.e., there exists a map $f\colon G_1\to G_2$ such that
$$
\frac 1C \d (g,h) \le \d (f(g), f(h)) \le C \d (g,h)
$$
for some fixed constant $C>0$. We call a map $f\colon G_1\to G_2$ satisfying the above property {\it $C$-bi-Lipschitz}.

\begin{prob}
Let $G_1$ be a finitely generated group that contains an $\alpha $-rotating subgroup for some sufficiently large $\alpha $ and let $G_2$ be another finitely generated group. Suppose there exists a $C$-bi-Lipschitz map $G_1\to G_2$ for some $C>0$. Is it true that $G_2$ contains an $\alpha ^\prime$-rotating subgroup, where $\alpha ^\prime =\alpha ^\prime (C, \alpha )$ only depends on $\alpha $ and $C$ and satisfies $\lim \limits_{\alpha \to \infty} \alpha ^\prime (C, \alpha)=\infty $ for every fixed $C>0$?
\end{prob}

Recall that a finitely generated group is {\it constricted} if every its asymptotic cone has cut points. Examples of constricted groups include relatively hyperbolic groups \cite{DS}, all but finitely many mapping class groups \cite{Ber}, $Out (F_n)$ for $n\ge 2$ \cite{Alg}, and many ``exotic" groups such as  Tarski Monsters \cite{OOS}. Constricted groups share many common properties with groups containing non-degenerate \he subgroups. For instance, constricted groups do not satisfy any nontrivial law \cite{DS}. Existence of cut points in asymptotic cones of a group $G$ is an important tool in studying outer automorphisms of $G$ and proving ``non-embeddability" theorems (see \cite{BDS, BDSadd, DMS} for examples).

A geodesic $l$ in a Cayley graph $\Gamma (G,X)$ of a group $G$ generated by a finite set $X$ is called {\it Morse} if for every $(\lambda ,c)$ there exists $B$ such that every $(\lambda, c)$-quasi-geodesic in $\Gamma (G,X)$ with endpoints on $l$ is contained in the closed $B$-neighborhood of $l$. It is not hard to show that existence of a Morse geodesics in $\Gamma (G,X)$ implies that $G$ is constricted.

\begin{prob}\footnote{A. Sisto answered affirmatively both parts of this question as well as the next one, see \cite[Theorem 1]{Sis13}.}
\begin{enumerate}
\item[(a)] Is every group with a non-degenerate \he subgroup constricted?
\item[(b)] Does every group $G$ with a non-degenerate \he subgroup contain a Morse quasi-geodesic?
\end{enumerate}
\end{prob}

More precisely, let $E$ be an infinite elementary subgroup such that $E\h G$, which always exists by Corollary \ref{elemhe1}. Let $g\in E$ be an element of infinite order.
\begin{prob}\footnote{Solved by A. Sisto, see the comment to the previous problem.}
Is it true that any bi-infinite $g$-invariant line in any Cayley graph of $G$ (with respect to a finite generating set) is a Morse quasi-geodesic?
\end{prob}

Let $G$ be a $1$-relator group. If $G=BS(m,n)=\langle a,b \mid (a^m)^b=a^n\rangle $ for some $m,n\in \mathbb Z\setminus\{0\} $ or $G= \langle a,b\mid a^m=b^n\rangle$, then it is easy to show that $G$ does not contain any hyperbolically embedded subgroup. Other examples of $1$-relator groups which do not  have any hyperbolically embedded subgroups are groups with infinite center (for particular examples and a structure theory of such groups we refer to \cite{Piet}). However it seems that a generic $1$-relator group must contain a non-degenerate hyperbolically embedded subgroup and, moreover, we do not know any examples of $1$-relator groups without non-degenerate hyperbolically embedded subgroups except for the groups from the two classes described above. Thus we ask the following.

\begin{prob}\footnote{Some progress towards solution of this problem is made in \cite{MO}. In particular, it is proved that every $1$-relator group with at least $2$ generators contains non-degenerate hyperbolically embedded subgroups.}
Classify $1$-relator group which do not contain non-degenerate hyperbolically embedded subgroups. Is it true that every such a group is either a Baumslag-Solitar group $BS(m,n)$ for some $m,n\in \mathbb Z\setminus\{0\} $ or has infinite center?
\end{prob}

This problem is closely related to the old conjecture by P. Neumann saying that all $1$-relator groups other than the Baumslag-Solitar groups $BS(m,n)$ defined below are $SQ$-universal \cite{Sch}. For the discussion of this problem see \cite{MO}.

It follows from Theorem \ref{wpd} that every group which admits a non-elementary acylindrical action on a hyperbolic metric space contains a non-degenerate hyperbolically embedded subgroup. Note that if a subgroup $H$ is a hyperbolically embedded in a group $G$ with respect to a subset $X\subseteq G$, then (unlike in the case when $G$ is hyperbolic relative to $H$) the action of $G$ on $\G $ is not necessary acylindrical. Here is the easiest counterexample. Let $G=(K\times  \mathbb Z)\ast H$, where $K$ is an infinite group. Let $X=K\cup \{x\}$, where $x$ is a generator of $\mathbb Z$. It is easy to verify that $H\h (G,X)$. However the action of $G$ on $\G $ is not acylindrical, as any element of $K$ moves any vertex of the infinite geodesic ray in $\G $ starting from $1$ and labelled by the infinite power of $x$ by a distance at most $1$.

However it seems plausible to modify $\G $ so that the action becomes acylindrical. For instance, in the above example the action of $G$ on $\Gamma (G, Y\sqcup H)$, where $Y=K\cup \mathbb Z$ is acylindrical. Thus we propose the following.

\begin{conj}\footnote{Proved in \cite{Osi13}.}
A group $G$ contains a non-degenerate hyperbolically embedded subgroup if and only if it admits a non-elementary acylindrical action on a hyperbolic space.
\end{conj}

If the conjecture holds, we obtain an alternative definition of the class of groups with hyperbolically embedded subgroups, which does not use subgroups at all. The conjecture would also yield an alternative proof of the following result obtained in \cite{HO}: Every group $G$ with a non-degenerate hyperbolically embedded subgroup is in the Monod-Shalom class $\mathcal C_{reg}$. Indeed every group admitting a non-elementary acylindrical action on a hyperbolic space is in $\mathcal C_{reg}$ by a result of Hamenst\"adt \cite{Ham}.

\begin{theindex}

\addcontentsline{toc}{section}{Index}

\indexspace
\item $0$-edges \hspace{10pt} \pageref{i-0edge}

\indexspace
\item $0$-cells \hspace{10pt} \pageref{i-0c}

\indexspace
\item $0$-refinement \hspace{10pt} \pageref{i-0ref}

\indexspace
\item $\Delta (\mathcal Q)$ \hspace{10pt} \pageref{i-ftc1}
\indexspace
\item $\Delta (Q_1, Q_2)$ \hspace{10pt} \pageref{i-ftc}
\indexspace
\item $\G$ \hspace{10pt} \pageref{i-Gxh1}, \pageref{i-Gxh2}
\indexspace
\item $\mu$ \hspace{10pt} \pageref{rem-mu}
\indexspace
\item $\mathcal D$ \hspace{10pt} \pageref{i-D}
\indexspace
\item $\partial _{ext}\Delta$ \hspace{10pt} \pageref{i-dd}
\indexspace
\item $\partial _{int}\Delta$ \hspace{10pt} \pageref{i-dd}

\indexspace
\item $\widehat\d$ \hspace{10pt} \pageref{i-dhat}
\indexspace
\item $\dl $ \hspace{10pt} \pageref{maindef}
\indexspace
\item $\d _{Hau}$ \hspace{10pt} \pageref{i-dHau}
\indexspace
\item $E(h)$ \hspace{10pt} \pageref{i-Eh1}
\indexspace
\item $E^+(h)$ \hspace{10pt} \pageref{elemrem}
\indexspace
\item $H_\lambda$-component \hspace{10pt} \pageref{defcomp}
\indexspace
\item $H_\lambda$-subpath \hspace{10pt} \pageref{defcomp}
\indexspace
\item $inj_{\mathbb X} (\mathcal R)$ \hspace{10pt} \pageref{i-sc}, \pageref{i-inj}
\indexspace
\item $\hat\ell $ \hspace{10pt} \pageref{defcomp}
\indexspace
\item $K(G)$ \hspace{10pt} \pageref{i-KG}
\indexspace
\item $\ell_Y$ \hspace{10pt} \pageref{i-lY}
\indexspace
\item $\Lab $ \hspace{10pt} \pageref{i-lab1}, \pageref{i-lab2}
\indexspace
\item $\mathcal L_{WPD}$ \hspace{10pt} \pageref{i-LWPD}
\indexspace
\item $\mathcal L_{WPD}^+$ \hspace{10pt} \pageref{i-L+WPD}
\indexspace
\item $o(H_\lambda)$ \hspace{10pt} \pageref{i-oH}
\indexspace
\item $s_{\mu, c} (n)$ \hspace{10pt} \pageref{i-smcn}
\indexspace

\item Acylindrical action \hspace{10pt} \pageref{dfn_acyl}

\indexspace
\item Bi-Lipschitz equivalence \hspace{10pt} \pageref{i-bL1}, \pageref{i-bL2}

\indexspace

\item Cartan Hadamard \hspace{10pt} \pageref{i-CarH}

\indexspace
\item Cayley graph \hspace{10pt} \pageref{i-Cg}

\indexspace
\item Cell (in a van Campen diagram)
\subitem essential \hspace{10pt} \pageref{i-0c}
\subitem $\mathcal R$-cell \hspace{10pt} \pageref{i-RSc}
\subitem $\mathcal S$-cell \hspace{10pt} \pageref{i-RSc}

\indexspace
\item Components  \hspace{10pt} \pageref{defcomp}
\subitem connected \hspace{10pt} \pageref{defcomp}
\subitem isolated \hspace{10pt} \pageref{defcomp}

\indexspace
\item Cone-off \hspace{10pt} \pageref{sec_coneoff}
\subitem parabolic \hspace{10pt} \pageref{i-pconeoff}

\indexspace
\item Cremona group \hspace{10pt} \pageref{i-Cremo}

\indexspace
\item Cut system \hspace{10pt} \pageref{i-cutsys}

\indexspace
\item Dehn filling
\subitem in $3$-manifolds \hspace{10pt} \pageref{i-dfgeom}
\subitem group theoretic \hspace{10pt} \pageref{i-df1}, \pageref{i-df2}, \pageref{i-df3}

\indexspace
\item Elements
\subitem commensurable \hspace{10pt} \pageref{i-comm}
\subitem elliptic \hspace{10pt} \pageref{i-ell}
\subitem loxodromic \hspace{10pt} \pageref{i-lox}
\subitem parabolic \hspace{10pt} \pageref{i-par}

\indexspace
\item Fellow traveling constant \hspace{10pt} \pageref{i-ftc}

\indexspace
\item Greendlinger \hspace{10pt} \pageref{i-Greend},  \pageref{i-Greend2}

\indexspace
\item Group
\subitem hyperbolic \hspace{10pt} \pageref{defhypg}
\subitem relatively hyperbolic \hspace{10pt} \pageref{i-rhg}, \pageref{dfn_relh}
\subitem weakly relatively hyperbolic \hspace{10pt} \pageref{i-wrh}
\subitem elementary \hspace{10pt} \pageref{i-elem}

\indexspace
\item Horoball \hspace{10pt} \pageref{i-horob}
\subitem combinatorial \pageref{i-combhorob}

\indexspace
\item Hyperbolic cone \hspace{10pt} \pageref{i-hc}

\indexspace
\item Hyperbolic space \hspace{10pt} \pageref{i-hyps}

\indexspace
\item Injectivity radius \hspace{10pt} \pageref{i-inj}

\indexspace
\item Isoperimetric function \hspace{10pt} \pageref{i-ip}

\indexspace
\item Iwip \hspace{10pt} \pageref{i-iwip}

\indexspace
\item Lipschitz quasi-retract \hspace{10pt} \pageref{i-Lqr}

\indexspace
\item Mapping Class Group  \hspace{10pt} \pageref{i-MCG}

\indexspace
\item Presentation
\subitem bounded \hspace{10pt} \pageref{i-brp}
\subitem relative \hspace{10pt} \pageref{i-rp}
\subitem reduced \hspace{10pt} \pageref{i-brp}
\subitem strongly bounded \hspace{10pt} \pageref{i-sbp}

\indexspace
\item Projection complex \hspace{10pt} \pageref{projc}

\indexspace
\item Quasi-geodesic axis (of a loxodromic element) \hspace{10pt} \pageref{i-qga}

\indexspace
\item Quasi-geodesic path \hspace{10pt} \pageref{i-qg}

\indexspace
\item Radial projection \hspace{10pt} \pageref{i-radproj}

\indexspace
\item Relative generating set \hspace{10pt} \pageref{i-rgs}

\indexspace
\item Relative area \hspace{10pt} \pageref{i-ra}

\indexspace
\item Relative isoperimetric function \hspace{10pt} \pageref{i-rip}

\indexspace
\item Relative metric \hspace{10pt} \pageref{maindef}

\indexspace
\item Rotating family \hspace{10pt} \pageref{i-rot}, \pageref{i-rot1}
\subitem $\rho$-separated \hspace{10pt} \pageref{i-sep}, \pageref{i-sep1}
\subitem very rotating \hspace{10pt} \pageref{i-vrot}, \pageref{i-vrot1}

\indexspace
\item  Shortening pair  \hspace{10pt} \pageref{i-shorteningpair}

\indexspace
\item $SQ$-universality \hspace{10pt} \pageref{i-SQ}

\indexspace
\item Subgroups
\subitem geometrically separated  \hspace{10pt} \pageref{GeomSep_mainres}, \pageref{GeomSep}
\subitem hyperbolically embedded \hspace{10pt} \pageref{i-hypemb}, \pageref{hes}
\subitem $\alpha $-rotating \hspace{10pt} \pageref{i-arot}, \pageref{i-arot1}

\indexspace Subset
\subitem quasi-convex \hspace{10pt} \pageref{dfn_sqc}
\subitem strongly quasi-convex \hspace{10pt} \pageref{dfn_sqc}

\indexspace
\item Van Kampen diagram \hspace{10pt} \pageref{i-vK}

\indexspace
\item Windmill \hspace{10pt} \pageref{dfn_windmill}

\indexspace
\item Weak proper discontinuity (WPD) \hspace{10pt} \pageref{WPD}

\indexspace
\item Word length \hspace{10pt} \pageref{i-wl}

\indexspace
\item Word metric \hspace{10pt} \pageref{i-wd}
\end{theindex}

\newpage

\vspace{1cm}
\noindent \textbf{Fran\c{c}ois Dahmani: } Institut Fourier, 100 rue des maths, Universit\'e de Grenoble (UJF),  BP74.
38 402 Saint Martin d'H\`eres, Cedex France.\\
E-mail: \emph{francois.dahmani@ujf-grenoble.fr}

\bigskip

\noindent \textbf{Vincent Guirardel: } Universit\'e de Rennes 1 263 avenue du G\'en\'eral Leclerc, CS 74205. F-35042 RENNES Cedex France.\\
E-mail: \emph{vincent.guirardel@univ-rennes1.fr}

\bigskip

\noindent \textbf{Denis Osin: } Department of Mathematics, Vanderbilt University, Nashville 37240, USA.\\
E-mail: \emph{denis.osin@gmail.com}


\addcontentsline{toc}{section}{Bibliography}

\begin{thebibliography}{99}

\bibitem{Ake}
C.A. Akemann,
Operator algebras associated with Fuchsian groups,
\textit{Houston J. Math},
\textbf{7} (1981), no. 3, 295--301.

\bibitem{AL}
C.A. Akemann, T.-Y. Lee,
Some Simple C$^\ast $-algebras associated with free groups,
\textit{Indiana Univ. Math. J.}
\textbf{29} (1980), 505--511.


\bibitem{Alg}
Y. Algom-Kfir,
Strongly Contracting Geodesics in Outer Space,
\textit{Geom. Top.} \textbf{15} (2011) no.4, 2181--2233.


\bibitem{Alo}
J.M. Alonso,
Finiteness conditions on groups and quasi-isometries,
\textit{J. Pure Appl. Algebra}
\textbf{95} (1994), 121--129.

\bibitem{APW}
J.M. Alonso, X. Wang, S.J. Pride,
Higher-dimensional isoperimetric (or Dehn) functions of groups,
\textit{J. Group Theory}
\textbf{2} (1999), 81--112.

\bibitem{AAS}
J.W. Anderson, J. Aramayona, K. J. Shackleton,
An obstruction to the strong relative hyperbolicity of a group,
\textit{J. Group Theory}
\textbf{10} (2007), no. 6, 749--756.

\bibitem{AnMS}
Y. Antolin, A. Minasyan, A. Sisto, Commensurating endomorphisms of acylindrically hyperbolic groups and applications, arXiv:1310.8605.

\bibitem{ArzhDel}
G. Arzhantseva, T. Delzant,
Examples of random groups,
preprint, 2010.


\bibitem{AM}
G. Arzhantseva, A. Minasyan,
Relatively hyperbolic groups are $C^\ast$-simple,
\textit{J. Funct. Anal.}
 \textbf{243} (2007), no. 1, 345--351.

\bibitem{AMO}
G. Arzhantseva, A. Minasyan, D. Osin,
The SQ--universality and residual properties of relatively hyperbolic
groups,
\textit{J. Algebra},
\textbf{315} (2007), no. 1, 165--177.

\bibitem{Bal82}
W. Ballmann, Axial isometries of manifolds of nonpositive curvature, \emph{Math. Ann.} \textbf{259} (1982), no. 1, 131-144.

\bibitem{Bal95}
W. Ballmann, Lectures on spaces of nonpositive curvature. With an appendix by M. Brin. DMV Seminar, Vol. 25, Birkh\"auser Verlag, Basel, 1995,

\bibitem{BBE}
W. Ballmann, M. Brin, P. Eberlein, Structure of manifolds of nonpositive curvature. I. \emph{Ann. of Math.}  \textbf{122} (1985), no. 1, 171-203.

\bibitem{BBS}
W. Ballmann, M. Brin, R. Spatzier, Structure of manifolds of nonpositive curvature. II. \emph{Ann. of Math.} \textbf{ 122} (1985), no. 2, 205-235.


\bibitem{BB08}
W. Ballmann, S. Buyalo, Periodic rank one geodesics in Hadamard spaces, \emph{Contemp. Math.} \textbf{469} (2008), 19-27.

\bibitem{bau}
A. Baudisch,
On superstable groups,
\textit{J. {L}ondon {M}ath. {S}oc.},
\textbf{42} (1990), 452--464.

\bibitem{BP}
B. Baumslag, S.J. Pride,
Groups with two more generators than relators,
{\it J. London Math. Soc. (2)}
{\bf 17} (1978), no. 3, 425--426.

\bibitem{Ber}
J.  Behrstock,
Asymptotic geometry of the mapping class group and Teichm\"uller space,
\textit{Geom. Topol.} \textbf{10} (2006), 1523--1578.

\bibitem{BDS}
J. Behrstock, C. Dru\c{t}u, M.  Sapir,
Median structures on asymptotic cones and homomorphisms into mapping class groups.
{\it Proc. Lond. Math. Soc. (3)}
{\bf 102} (2011), no. 3, 503--554.

\bibitem{BDSadd}
J. Behrstock, C. Drutu, M. Sapir,
Addendum: Median structures on asymptotic cones and homomorphisms into mapping class groups,
\textit{Proc. Lond. Math. Soc. (3)}
\textbf{102} (2011), no. 3, 555--562.

\bibitem{BeHa}
B. Bekka, P. de la Harpe,
Groups with simple reduced C$^\ast$-algebras,
\textit{Expo. Math}.
\textbf{18} (2000), no. 3, 215--230.

\bibitem{BB}
M. Bestvina, N. Brady,
Morse theory and finiteness properties of groups,
\textit{Invent. Math.}
\textbf{129 }(1997), no. 3, 445--470.

\bibitem{BBF}
M. Bestvina, K. Bromberg, K. Fujiwara,
The asymptotic dimension of mapping class groups is finite,
preprint, arXiv:1006.1939.

\bibitem{BBF13}
M. Bestvina, K. Bromberg, K. Fujiwara, Bounded cohomology via quasi-trees,  arXiv:1306.1542.


\bibitem{BFe}
M. Bestvina, M. Feighn,
A hyperbolic $Out(F_n)$-complex,
\textit{Groups Geom. Dyn.} \textbf{4} (2010), no. 1, 31--58.

\bibitem{BFe2}
M. Bestvina, M. Feighn,
Hyperbolicity of the complex of free factors,
 \textit{Adv. Math.} \textbf{256} (2014) 104 -- 155.


\bibitem{BF}
M. Bestvina, K. Fujiwara,
Bounded cohomology of subgroups of mapping class groups,
{\it Geom. Topol.}
{\bf 6} (2002), 69--89.

\bibitem{BLM}
J.S. Birman, A. Lubotzky, J. McCarthy,
Abelian and solvable subgroups of the mapping class group,
{\it Duke Math. J.}
{\bf 50} (1983), no. 4, 1107--1120.

\bibitem{BC}
J. Blanc, S. Cantat,
Dynamical degrees of birational transformations of projective surfaces, arXiv:1307.0361.

\bibitem{B_hyp}
B.H. Bowditch,
Intersection numbers and the hyperbolicity of the curve complex,
\textit{J. Reine Angew. Math.}
\textbf{598}  (2006), 105--129.





\bibitem{B_acyl}
B.H. Bowditch,
Tight geodesics in the curve complex.
\textit{Invent. Math.}
\textbf{171} (2008), no. 2, 281--300.


\bibitem{Bow}
B.H. Bowditch,
Relatively hyperbolic groups, \textit{Internat. J. Algebra Comput.}
\textbf{22} no.3, 1250016, 66 pp.







\bibitem{Bri}
M. Bridson,
Polynomial Dehn functions and the length of asynchronously automatic structures,
\textit{Proc. London Math. Soc. (3)}
\textbf{85} (2002), no. 2, 441--466.

\bibitem{BH_metric}
M. Bridson, A. Haefliger,
\textit{Metric spaces of non-positive curvature}.
Grundlehren der Mathematischen Wissenschaften [Fundamental Principles of Mathematical Sciences], 319. Springer-Verlag, Berlin, 1999. xxii+643 pp.



\bibitem{BrHa}
M. Bridson, P. de la Harpe,
Mapping class groups and outer automorphism groups of free groups are C$^\ast $-simple,
\textit{J. Funct. Anal.}
\textbf{212} (2004), 195--205.

\bibitem{BW}
M. Bridson, R. Wade,
 Actions of higher-rank lattices on free groups,
\textit{Compos. Math.}
\textbf{147} no. 5,  1573--1580.

\bibitem{Bro}
K.S. Brown,
\textit{Cohomology of groups},
Graduate Texts in Mathematics, 87. Springer-Verlag, New York-Berlin, 1982.

\bibitem{BS87}
K. Burns, R. Spatzier, Manifolds of nonpositive curvature and their buildings, \emph{IHES Publ. Math.} (1987), no. 65, 35-59.

\bibitem{C13}
S. Cantat, The Cremona group in two variables. Proc. of the sixth European Congress of Math., 211-225. (Europ. Math. Soc., 2013)

\bibitem{C11}
S. Cantat,
Sur les groupes de transformations birationnelles des surfaces,
\textit{Ann. of Math.}
\textbf{174} (2011), no. 1, 299--340.

\bibitem{CL}
S. Cantat, S. Lamy,
Normal subgroups in the Cremona group. With an appendix by Y. de Cornulier.
\textit{Acta Math.} \textbf{210} (2013) no. 1, 31--91.

\bibitem{CR}
P.-E. Caprace, B. Remy,
Simplicity and superrigidity of twin building lattices,
\textit{Invent. Math.}
\textbf{176}  (2009),  no. 1, 169-221.

\bibitem{CS}
P.-E. Caprace, M. Sageev,
Rank rigidity for $CAT(0)$ cube complexes,
\textit{Geom. Funct. Anal.}
\textbf{21} no. 4 (2011) pp. 851-891.

\bibitem{CK}
D. Carter, G. Keller,
Bounded elementary generation of ${\rm SL}\sb{n}({\cal O})$,
\textit{ Amer. J. Math.} {\bf 105} (1983), no. 3, 673-687.

\bibitem{Ch} C. Champetier,
Petite simplification dans les groupes hyperboliques,
 \textit{Ann. Fac. Sci. Toulouse Math.}
\textbf{3}  (1994),  no. 2, 161--221.

\bibitem{Cha}
R. Charney, An introduction to right-angled Artin groups, \emph{Geom. Dedicata} \textbf{125} (2007), 141-158.

\bibitem{Chay}
V. Chaynikov, On the generators of the kernels of hyperbolic group presentations,
\emph{Algebra Discrete Math.} \textbf{11} (2011), no. 2, 18-50.

\bibitem{CDP}
M. Coornaert, T. Delzant, A. Papadopoulos,
\textit{G\'eom\'etrie et th\'eorie des groupes.
Les groupes hyperboliques de Gromov.}
Lecture Notes in Mathematics, 1441. Springer-Verlag, Berlin, 1990. x+165 pp.



\bibitem{Coulon}
R. Coulon,
Asphericity and small cancellation theory for rotation family of groups,
\textit{Groups Geom. Dyn.},
\textbf{5}, no. 4, 2011, 729--765.

\bibitem{Coulon_notes}
R. Coulon,
Small cancellation theory and Burnside problem. To appear in
\textit{Internat. J. Algebra Comput.}. arXiv 1302.6933.




\bibitem{Dah}
F. Dahmani,
Combination of convergence groups,
\textit{Geom. Topol.}
\textbf{7} (2003), 933--963.

\bibitem{DG}
F. Dahmani, V. Guirardel,
Presenting parabolic subgroups,
\textit{Alg. \& Geom. Top.} \textbf{13} (2013) 3203--3222.

\bibitem{Del_Duke} T. Delzant,
Sous-groupes distingu\'es et quotients des groupes hyperboliques,
\textit{Duke Math. J. }
\textbf{83}  (1996),  no. 3, 661--682.

\bibitem{Del_Gro} T. Delzant, M. Gromov,
Courbure m\'esoscopique et th\'eorie de la toute petite simplification,
\textit{J. Topol.} \textbf{1}  (2008),  no. 4, 804--836.

\bibitem{Dr09}
C. Dru\c{t}u,
Relatively hyperbolic groups: geometry and quasi-isometric invariance,
\textit{Comment. Math. Helv.}
\textbf{84} (2009), no. 3, 503--546.

\bibitem{Dr11}
C. Dru\c{t}u, Quasi-isometry rigidity of groups.  G\'eom\'etries \`a courbure n\'egative ou nulle, groupes discrets et rigidit\'es, 321--371,
{\it S\'emin. Congr.},
\textbf{18}, Soc. Math. France, Paris, 2009.

\bibitem{DMS}
C. Dru\c{t}u, S. Moses, M.  Sapir,
Divergence in lattices in semisimple Lie groups and graphs of groups.
{\it Trans. Amer. Math. Soc. (3)}
{\bf 362} (2010), no. 5, 2451--2505




\bibitem{DS}
C. Dru\c{t}u, M. Sapir,
Tree-graded spaces and asymptotic
cones of groups. With an appendix by D. Osin and M. Sapir.
{\it Topology}
{\bf 44} (2005), no. 5, 959--1058.

\bibitem{Efr}
E.G. Effros,
Property $\Gamma $ and inner amenability,
\textit{Proc. Amer. Math. Soc.}
\textbf{47} (1975), 483--486.


\bibitem{EF}
D.  Epstein, K. Fujiwara,
The second bounded cohomology of word-hyperbolic groups,
\textit{Topology}
\textbf{36} (1997), no. 6, 1275--1289.



\bibitem{F}
B. Farb,
Relatively hyperbolic groups,
{\it Geom. Funct. Anal.}
{\bf 8} (1998), 810--840.

\bibitem{F_book}
B. Farb,
Some problems on mapping class groups and moduli space,
in \textit{ Problems on mapping class groups and related topics},
11--55, Proc. Sympos. Pure Math., 74, Amer. Math. Soc., Providence, RI.

\bibitem{Farb_Masur}
B. Farb, H. Masur,
Superrigidity and mapping class groups,
\textit{Topology}
\textbf{37} (1998), no. 6, 1169--1176.

\bibitem{FT}
B. Fine, M. Tretkoff,
{ On the SQ-universality of HNN groups},
{\it Proc. Amer. Math. Soc.}
{\bf 73} (1979), no. 3, 283--290.


\bibitem{GP}
V. Gerasimov, L. Potyagailo,
Quasiconvexity in the Relatively Hyperbolic Groups,
preprint arXiv:1103.1211.





\bibitem{GdH} E. Ghys, P. de la Harpe,
\textit{Sur les groupes hyperboliques d'apr\`es Mikhael Gromov} (Bern, 1988),
Progress in Mathematics, 83. Birkhäuser Boston, Inc., Boston, MA, 1990. xii+285 pp.




\bibitem{Gro}
M. Gromov,
Hyperbolic groups, Essays in Group Theory, MSRI Series,
Vol.8, (S.M. Gersten, ed.), Springer, 1987, 75--263.

\bibitem{Gro_cat}
M. Gromov,
${\rm CAT}(\kappa)$-spaces: construction and concentration,
\textit{J. Math. Sci. (N. Y.)}
\textbf{119}  (2004),  no. 2, 178--200.

\bibitem{G-rnd}
M. Gromov,
Random walk in random groups,
{\it Geom. Funct. Anal.}
{\bf  13} (2003), no. 1, 73--146.

\bibitem{Gro_Meso}
M. Gromov,
Mesoscopic curvature and hyperbolicity.
Global differential geometry: the mathematical legacy of Alfred Gray (Bilbao, 2000),
58--69, Contemp. Math., 288, Amer. Math. Soc., Providence, RI, 2001.






\bibitem{Gr_Ma}
D. Groves, J.F. Manning,
Dehn filling in relatively hyperbolic groups,
{\it Israel J. Math.}
{\bf 168} (2008), 317--429.

\bibitem{GM07}
D. Groves, J.F. Manning,
Fillings, finite generation, and direct limits of relatively hyperbolic groups, {\it Groups Geom. Dyn.} {\bf 1} (2007), no. 3, 329--342.


\bibitem{Gui_pcmi} V. Guirardel. Geometric small cancellation.
{\it IAS/Park City Math. Ser.}, to appear.


\bibitem{H}
M. Hamann, Group actions on metric spaces: fixed points and free subgroups, arXiv:1301.6513.


\bibitem{Ham}
U. Hamenst\" adt,
Bounded cohomology and isometry groups of hyperbolic spaces,
\textit{J. Eur. Math. Soc.}
\textbf{10} (2008), no. 2, 315--349.

\bibitem{Ham09}
U. Hamenst\" adt, Rank-one isometries of proper CAT(0)-spaces,
\emph{Contemporary Math.} \textbf{501} (2009), 43-59.

\bibitem{HanMos}
M. Handel, L. Mosher,
The free splitting complex of a free group I: hyperbolicity,
\textit{Geom. Topol} \textbf{17} (2013) no. 3, 1581--1672.


\bibitem{Har07}
P. de la Harpe,
On simplicity of reduced C$^\ast$-algebras of groups,
\textit{ Bull. London Math. Soc.}
\textbf{39} (2007), no. 1, 1-26.

\bibitem{Har95}
P. de la Harpe,
Operator algebras, free groups and other groups,
Recent advances in operator algebras (Orl\'eans, 1992),
\textit{Ast\'erisque} \textbf{232} (1995), 121-153.

\bibitem{Har88}
P. de la Harpe,
Groupes hyperboliques, alg\`ebres d'op\'erateurs et un th\'eor\`eme de Jolissaint,
\textit{C.R. Acad. Sci. Paris,} I,
\textbf{307} (1988), 771--774.

\bibitem{HaPre}
P. De la Harpe, J.-P. Pr\'eaux,  C$^\ast$-simple groups: amalgamated
free products, HNN extensions, and fundamental groups of 3-manifolds,
\textit{J. Topol. Anal.} \textbf{3} (2011) no. 4, 451--489.


\bibitem{HaSk}
P. de la Harpe, G. Skandalis,
Les r\'eseaux dans les groupes semi-simples ne sont pas int\'erieurement moyennables,
\textit{L'Ens. Math\'ematique}
\textbf{40} (1994), 291--311.

\bibitem{HaVo}
A. Hatcher, K. Vogtmann, The complex of free factors of a free group, \emph{Quart. J. Math.} \textbf{49} (1998), no. 196, 459-468.

\bibitem{HiHo_hyperbolicity}
Arnaud Hilion, Camille Horbez,
The hyperbolicity of the sphere complex via surgery paths,
preprint, 	arXiv:1210.6183 [math.GT].

\bibitem{HNN}
G. Higman, B.H. Neumann, H.Neumann,
{Embedding theorems for groups},
{\it  J. London Math. Soc.}
{\bf 24} (1949), 247--254.

\bibitem{HK}
C. Hodgson, S. Kerckhoff, 
Rigidity of hyperbolic cone-manifolds and hyperbolic Dehn surgery,
\emph{J. Differential Geom.} \textbf{48} (1998), no. 1, 1–-59. 

\bibitem{How}
J. Howie,
{On the SQ-universality of $T(6)$-groups},
\textit{Forum Math.}
{\bf 1} (1989), no. 3, 251--272.






\bibitem{Hru}
C. Hruska,
Relative hyperbolicity and relative quasiconvexity for countable groups,
{\it Algebr. Geom. Topol.}
{\bf  10} (2010) 1807-1856.

\bibitem{Hull}
M. Hull, Small cancellation in acylindrically hyperbolic groups,  arXiv:1308.4345.

\bibitem{HO}
M. Hull, D. Osin, Induced quasi-cocycles on groups with hyperbolically embedded subgroups, \emph{Alg. \& Geom. Topol.} \textbf{13} (2013) 2635-2665.


\bibitem{Iva84}
N. Ivanov,
Algebraic properties of Teichm\"uller modular group,
\textit{Dokl. Akad. Nauk SSSR}
\textbf{275} (1984), no. 4, 786--789.

\bibitem{Iva92}
N. Ivanov,
Subgroups of Teichm\"uller Modular Groups,
\textit{Translations of Math. Monographs},
\textbf{115}, Amer. Math. Soc. 1992.

\bibitem{Iv}
N. Ivanov,
Fifteen problems about the mapping class groups,
in \textit{Problems on mapping class groups and related topics},
71--80, Proc. Sympos. Pure Math., 74, Amer. Math. Soc., Providence, RI, 2006.

\bibitem{J}
P. Jolissaint,
Moyennabilit\'e in\'erieure du groupe $F$ de Thomson,
\textit{C.R. Acad. Sci. Paris}, I,
\textbf{325} (1997), 61--64.

\bibitem{LeeSong}
S.-J. Lee, W.-T. Song,
The kernel of ${\rm Burau}(4)\otimes Z_p$ is all pseudo-Anosov.
{\it Pacific J. Math.}
{\bf 219} (2005), no. 2, 303--310.

\bibitem{Los}
K.I. Lossov,
 SQ-universality of free products with amalgamated finite subgroups,
\textit{Siberian Math. J.}
\textbf{27} (1986), no. 6, 890–899.


\bibitem{LubSeg}
A. Lubotzky, D. Segal,
\textit{Subgroup growth}.
Progress in Mathematics,
212. Birkh\"auser Verlag, Basel, 2003

\bibitem{LS}
R.C.  Lyndon and P.E.  Schupp,
\textit{Combinatorial group theory}.
Springer-Verlag, Berlin, 1977. Ergebnisse der Mathematik und ihrer Grenzgebiete,
Band 89.


\bibitem{MacSis}
J.M. Mackay, A. Sisto,
Quasi-hyperbolic planes in relatively hyperbolic groups,
preprint arXiv:1111.2499.

\bibitem{Mar}
D. Marker,
\textit{Model Theory: An Introduction.}
Graduate Texts in Mathematics 217, Springer, 2002.


\bibitem{MM99}
H. Masur, Y. Minsky,
Geometry of the complex of curves I: hyperbolicity,
\textit{Invent. Math.}
\textbf{138} (1999), 103--149.

\bibitem{McCW_windmills}
Jon McCammond and Daniel Wise.
 Windmills and extreme 2-cells.
 {\em Illinois J. Math.}, 54(1):69--87, 2010.


\bibitem{Mc85}
J.D. McCarthy,
A ``Tits alternative" for subgroups of surface mapping class groups,
{\it Trans. Amer. Math. Soc.}
{\bf 291} (1985), no.2, 583--612.

\bibitem{Min}
A. Minasyan, Groups with finitely many conjugacy classes and their automorphisms, \emph{Comment. Math. Helv.} \textbf{84} (2009), no. 2, 259-296.

\bibitem{MO}
A. Minasyan, D. Osin, Acylindrical hyperbolicity of groups acting on trees,  arXiv:1310.6289.

\bibitem{Mon}
N. Monod,
An invitation to bounded cohomology.
International Congress of Mathematicians,
Vol. II, 1183-1211, Eur. Math. Soc., Z\"urich, 2006.


\bibitem{MS}
N. Monod, Y. Shalom,
Orbit equivalence rigidity and bounded cohomology,
\textit{Ann. of Math.}
\textbf{164} (2006), no. 3, 825--878.


\bibitem{MvN}
F.J. Murray, J. von Neumann,
On Rings of Operators IV,
\textit{Ann. of Math.}
\textbf{44} (1943), 716--808.

\bibitem{Neumann}
B.H. Neumann,
Groups covered by permutable subsets,
{\it J. London Math. Soc.}
{\bf 29} (1954), 236--248.

\bibitem{PNeu}
P.M. Neumann, The SQ-universality of some finitely presented groups,
Collection of articles dedicated to the memory of Hanna Neumann, I.
\emph{J. Austral. Math. Soc. } \textbf{16 }(1973), 1-6.

\bibitem{Ols-book}
A.Yu.  Ol'shanskii,
\textit{Geometry of defining relations in groups.}
Mathematics and its Applications (Soviet Series), 70. Kluwer
Academic Publishers Group, Dordrecht, 1991.

\bibitem{Ols92}
A.Yu. Ol'shanskii,
Periodic quotient groups of hyperbolic groups
\textit{Math. USSR-Sb.}
\textbf{72}  (1992),  no. 2, 519--541.

\bibitem{Ols93}
A.Yu. Olshanskii, {On residualing homomorphisms and
$G$-subgroups of hyperbolic groups}, \emph{Int. J. Alg. Comp.} {\bf 3}
(1993), 4, 365-409.

\bibitem{Ols}
A.Yu. Ol'shanskii,
SQ-universality of hyperbolic groups,
\textit{Sb. Math.} {\bf 186} (1995), no. 8, 1199--1211.

\bibitem{OO}
A.Yu. Olshanskii, D. Osin, $C^\ast$-simple groups without free subgroups,  arXiv:1401.7300.

\bibitem{OOS}
A.Yu. Olshanskii, D. Osin, M. Sapir,
Lacunary hyperbolic groups.
With an appendix by M. Kapovich and B. Kleiner.
{\it Geom. Topol.}
{\bf  13} (2009), no. 4, 2051--2140.

\bibitem{Osi13}
D. Osin, Acylindrically hyperbolic groups, \emph{arXiv:1304.1246}.

\bibitem{Osi10}
D. Osin,
Small cancellations over relatively hyperbolic groups and embedding theorems,
\textit{Ann. of Math.}
{\bf 172} (2010), no. 1, 1--39.

\bibitem{Osi07}
D. Osin,
Peripheral fillings of relatively hyperbolic groups,
\textit{Invent. Math.}
{\bf 167} (2007), no. 2, 295--326.

\bibitem{Osi06a}
D. Osin,
Relatively hyperbolic groups: Intrinsic geometry, algebraic properties, and algorithmic problems,
{\it Memoirs Amer. Math. Soc.}
{\bf 179} (2006), no. 843.

\bibitem{Osi06b}
D. Osin,
Elementary subgroups of relatively hyperbolic groups and bounded generation,
\textit{Internat. J. Algebra Comput.},
{\bf 16} (2006), no. 1, 99--118.

\bibitem{Osi06c}
D. Osin,
Relative Dehn functions of HNN--extensions and amalgamated products,
\textit{Contemp. Math.}
{\bf 394} (2006), 209--220.

\bibitem{Osi05}
D. Osin,
Asymptotic dimension of relatively hyperbolic groups,
\textit{Internat. Math. Res. Notices}
{\bf 35} (2005), 2143--2162.

\bibitem{Osi04}
D. Osin,
Weak hyperbolicity and free constructions.
\textit{Contemp. Math.}
{\bf 360} (2004), 103--111.




\bibitem{Ould}
A. Ould Houcine,
On superstable CSA-groups,
\textit{Ann. Pure Appl. Logic}
\textbf{154} (2008), no. 1, 1--7.

\bibitem{Pau_arboreal}
F. Paulin,
Outer automorphisms of hyperbolic groups and small actions on $R$-trees.
Arboreal group theory (Berkeley, CA, 1988),  331--343,
Math. Sci. Res. Inst. Publ., 19, Springer, New York, 1991.

\bibitem{PP}
C. Petronio, J. Porti, Negatively oriented ideal triangulations and a proof of Thurston's hyperbolic Dehn filling theorem,   \emph{Expo. Math.} \textbf{18} (2000), no. 1, 1--35.

\bibitem{Pic}
 G. Picioroaga,
The inner amenability of the generalized Thompson group,
\textit{Proc. Amer. Math. Soc.}
\textbf{134} (2006), no. 7, 1995--2002.

\bibitem{Piet}
A. Pietrowski, The isomorphism problem for one-relator groups with non-trivial centre. \emph{Math. Z.} \textbf{136} (1974), 95-106.

\bibitem{P}
B. Poizat,
Groupes stables, avec types generiques reguliers,
\textit{J. Symbolic Logic}
\textbf{48} (1983), 339--355.

\bibitem{Pow}
R.T. Powers,
Simplicity of the $C^{\ast }$-algebra associated with the free group on two generators,
\textit{Duke Math. J. }
\textbf{42} (1975), 151--156.


\bibitem{Rap}
A.S. Rapinchuk,
The congruence subgroup problem for arithmetic groups of bounded generation (Russian),
\textit{Soviet Math. Dokl.} {\bf 42} (1991), no. 2, 664-668.

\bibitem{Reb}
D.Y. Rebbechi,
Algorithmic Properties of Relatively Hyperbolic Groups,
PhD thesis, arXiv:math/0302245

\bibitem{Sas}
G.S. Sacerdote, P.E. Schupp,
SQ-universality in HNN groups and one relator groups,
\textit{J. London Math. Soc. (2)}
{\bf 7} (1974), 733--740.

\bibitem{Sch}
P.E. Schupp,
A survey of SQ-universality,
Conference on Group Theory (Univ. Wisconsin-Parkside, Kenosha, Wis., 1972), pp.
183--188. Lecture Notes in Math. {\bf 319} (1973), Springer, Berlin.

\bibitem{Sel}
Z. Sela,
Diophantine Geometry over Groups VIII: Stability.
\textit{Ann. of Math. (2)} \textbf{177} (2013) no. 3, 787--868.


\bibitem{Ser}
J-P. Serre,
\textit{Trees}.
Springer-Verlag, Berlin-New York,  1980,  ix+142 pp.


\bibitem{Sh}
Y. Shalom,
Bounded generation and Kazhdan's property (T),
\textit{ Inst. Hautes Etudes Sci.  Publ. Math.}
\textbf{90 }(1999), 145--168.

\bibitem{Sh69}
S. Shelah,
Stable theories,
\textit{Israel J. Math.}
\textbf{7} (1969), no. 3, 187--202.

\bibitem{Sh85}
S. Shelah,
Classification of first order theories which have a structure theorem,
\textit{Bull. Amer. Math. Soc.}
\textbf{12} (1985), 227--232.

\bibitem{Sis13}
 A. Sisto, Quasi-convexity of hyperbolically embedded subgroups, To
 appear in \textit{Math. Z.}  arXiv:1310.7753.


\bibitem{Sis}
A. Sisto, Contracting elements and random walks, arXiv:1112.2666.

\bibitem{Stal}
Y. Stalder,
Moyennabilit\'e int\'erieure et extensions HNN,
\emph{Ann. Inst. Fourier}
\textbf{56} (2006) 309--323.

\bibitem{T}
I.O. Tavgen',
Bounded generability of Chevalley groups over rings of $S$-integer algebraic numbers (Russian), 
\textit{Math. USSR-Izv.}
{\bf 36} (1991), 1, 101--128

\bibitem{Th}
W.P. Thurston,
Three-dimensional manifolds, Kleinian groups and hyperbolic geometry.
{\it Bull. Amer. Math. Soc. (N.S.)}
{\bf 6} (1982), no. 3, 357--381.

\bibitem{V}
S. Vaes,
An inner amenable group whose von Neumann algebra does not have property Gamma,
\textit{Acta Math.} \textbf{208} (2012) no. 2, 389 -- 394.

\bibitem{Wag}
F. Wagner,
Stable groups.
\textit{Handbook of algebra}, {\bf 2}, 277--318, North-Holland, Amsterdam, 2000.

\bibitem{Whi}
K. Whittlesey,
Normal all pseudo-Anosov subgroups of mapping class groups,
\textit{Geom. Topol.}
\textbf{4}  (2000), 293--307.

\bibitem{Y}
A. Yaman,
A topological characterization of relatively hyperbolic
groups,
{\it J. Reine Angew. Math.}
{\bf 566} (2004), 41--89.




\end{thebibliography}
\end{document}